\documentclass[11pt]{article}
\usepackage{amssymb, latexsym, bbm, times}  
\usepackage{amsthm}  
\usepackage{mathptmx}
\usepackage{amscd}
\usepackage{xcolor}
\usepackage{url}

\usepackage{quiver}
\usetikzlibrary{nfold} 
\usepackage{graphicx}

\usepackage[all]{xy}
\usepackage{placeins}

\usepackage{mathtools} 
\usepackage{rotating}  

\newtheorem{Thm}{Theorem}[section]
\newtheorem{Lem}[Thm]{Lemma}
\newtheorem{Cor}[Thm]{Corollary}
\newtheorem{Prop}[Thm]{Proposition}
\newtheorem{Conj}[Thm]{Conjecture}

\theoremstyle{definition}
\newtheorem{Def}[Thm]{Definition}
\newtheorem{Rem}[Thm]{Remark}
\newtheorem{example}[Thm]{Example}

\newcommand{\Z}{\mathbb{Z}}

\newcommand{\C}{\mathbb{C}}
\newcommand{\bll}{\mathbbm{1}}

\newcommand{\M}{{\mathcal M}}

\newcommand{\End}{\operatorname{End}}

\usepackage{xcolor}
\usepackage{version}

\excludeversion{NB}

\newcommand{\bsm}{\begin{smallmatrix}}
\newcommand{\esm}{\end{smallmatrix}}

\newcommand{\bg}{\mathbf{g}}

\def\cqfd{\hfill $\Box$ \bigskip}

\def\resp{{\em resp.\ }}

\def\<{\langle\,}
\def\>{\,\rangle}

\def\uP{\underline{P}}

\def\g{\mathfrak g}

\def\<{\langle}
\def\>{\rangle}

\def\ra{\rightarrow}

\def\1{\mathbf 1}

\def\ii{{\mathbf i}}

\def\a{\alpha} 
\newcommand{\alp}{\alpha}  
\newcommand{\bet}{\beta}
\newcommand{\Psib}{{\Psi}}
\newcommand{\gam}{\gamma}
\newcommand{\del}{\delta}
\newcommand{\sig}{\sigma}
\newcommand{\hsig}{\hat{\sigma}}
\newcommand{\htau}{\hat{\tau}}
\newcommand{\Gam}{\Gamma}
\newcommand{\tGam}{\Gamma}
\newcommand{\sgn}{\mathrm{sgn}}

\newcommand{\abs}[1]{\lvert #1 \rvert}
\newcommand{\df}{\colon}
\newcommand{\xlrarrow}[1]{\xleftrightarrow{#1}} 
\newcommand{\diag}{\operatorname{diag}} 

\newcommand{\sfA}{A}  
\newcommand{\sfB}{B}
\newcommand{\sfC}{C}
\newcommand{\sfD}{D}
\newcommand{\sfE}{E}
\newcommand{\sfF}{F}
\newcommand{\sfG}{G}
\newcommand{\tsfF}{\widetilde{\sfF}} 
\newcommand{\tsfG}{\widetilde{\sfG}}

\newcommand{\tb}{\mathbf{\mathfrak{t}}}

\def\End{{\rm End}}

\def\de{\delta}

\def\la{\lambda}
\def\AA{\mathcal{A}}

\def\SS{\mathcal{S}}
\def\YY{\mathcal{Y}}

\def\M{{\mathfrak M}}
\def\O{{\mathcal O}}
\def\1{{\mathbf 1}}

\def\hg{\widehat{\mathfrak{g}}}

\def\G{\Gamma}
\def\tG{\widetilde{\Gamma}}

\def\be{\mathbf e}
\def\bz{\mathbf z}
\def\bx{\mathbf x}
\def\bb{\mathbf b}

\def\bi{\mathbf{i}}

\def\bk{\mathbf{k}}
\def\be{\mathbf{e}}

\def\bv{\mathbf{v}}
\def\bw{\mathbf{w}}

\def\btrois{3}
\def\bdeux{2}
\def\bun{1}

\def\Si{\Sigma}
\def\SS{\mathcal{S}}

\def\bd{\mathbf{d}}
\def\br{\mathbf{r}}

\def\QQ{\mathcal{Q}}

\def\bQ{\underline{Q}}

\newcommand{\red}{\color{red}}
\newcommand{\blue}{\color{blue}}
\newcommand{\green}{\color{green}}

\begin{document}

\title{\bf  Representations of shifted quantum affine \\ algebras and cluster algebras II: \\Non simply-laced types and cluster expansions}
\author{C. Geiss, D. Hernandez and B. Leclerc}

\date{}

\maketitle

\begin{abstract}
We extend to non simply-laced types the results of \cite{GHL} on cluster structures in the Gro\-then\-dieck rings of  categories $\mathcal{O}^{\mathrm{sh}}$ for shifted quantum affine algebras. We also show that, for all types, the formal power series expressions of the cluster monomials in terms of classes of positive prefundamental representations can be regarded as limits of cluster expansions, thus establishing their positivity.
\end{abstract}

\setcounter{tocdepth}{2}
{\small \tableofcontents}

\section{Introduction}
In this paper we continue our investigation, started in \cite{GHL}, of cluster algebra structures on the Gro\-then\-dieck rings of categories of modules over shifted quantum affine algebras.

Recall that the shifted quantum affine algebra $U_q^\mu(\widehat{\mathfrak{g}})$ 
was introduced by Finkelberg
and Tsymbaliuk \cite{FT} in their study of quantized $K$-theoretic Coulomb branches of $3d$ $N = 4$ SUSY
quiver gauge theories. Here $\mathfrak{g}$ denotes a finite-dimensional simple complex Lie algebra, and $\mu$ is 
an integral coweight of $\mathfrak{g}$.
When $\mu = 0$, the algebra $U_q^\mu(\widehat{\mathfrak{g}})$ is a central extension of the quantum loop algebra $U_q(L{\mathfrak{g}})$ and it has essentially the same representation theory. However, for $\mu \not = 0$, things become very different. If $\mu$ is dominant $U_q^\mu(\widehat{\mathfrak{g}})$ has a rich supply of finite-dimensional representations, wheras for non dominant $\mu$ there are no non-trivial finite-dimensional representations. In \cite{H}, the second author has started a systematic study of the representation theory of $U_q^\mu(\widehat{\mathfrak{g}})$. He has introduced a category
$\mathcal{O}^\mu$ containing infinite-dimensional representations, and shown that the Grothendieck group of
$\mathcal{O}^{\mathrm{sh}} := \oplus_{\mu\in P^\vee} \mathcal{O}^\mu$ has a natural ring structure coming from fusion product of representations.

In \cite{GHL}, we focused on the case when $\g$ is of simply-laced type $A, D, E$. 
We showed that the Grothendieck ring of $\mathcal{O}^{\mathrm{sh}}$ contains an infinite rank cluster algebra whose explicit initial seeds consists of certain distinguished elements $\underline{Q}_{w(\varpi^\vee_i),q^r}$ of the Grothendieck ring previously constructed in \cite{FH3}. Here $w$ denotes an element of the Weyl group of $\mathfrak{g}$, and $\varpi^\vee_i$ a fundamental coweight. We proved that this cluster algebra is dense for the topological ring structure of the Grothendieck ring. We also conjectured that the genuine cluster algebra (without taking closure) coincides with the Grothendieck group of the full subcategory of $\mathcal{O}^{\mathrm{sh}}$ whose objects are the finite length modules. This conjecture is now proved, by combining results of Hernandez and Zhang \cite{HZ} and Kamnitzer, Labelle, Leroux-Lapierre, Pinet and Weekes \cite{KLLLPW}. However, the monoidal categorification conjecture is not yet settled. It states that all cluster monomials are classes of simple objects of $\mathcal{O}^{\mathrm{sh}}$. This was checked for finite-dimensional representations \cite[\S 9.2]{GHL}, using previous results of \cite{H} and of Kashiwara, Kim, Oh and Park \cite{KKOP}. In \cite{KLLLPW} it was also established that the $Q$-variables $\underline{Q}_{w(\varpi^\vee_i),q^r}$ are classes of simple objects. But the general case remains open. 

The aim of this paper is to extend the results of \cite{GHL} to the non simply-laced types $B, C, F, G$. 
In \S\ref{sec:QClA} we associate with each Cartan type and each Weyl group element $w$ a cluster algebra $\mathcal{A}_w$, defined via the infinite quiver of an initial seed. Strictly speaking, this quiver depends on the choice of a reduced word for $w$, but we prove in \S\ref{mutequiv} and Appendix~\ref{App:G2-braid} that its mutation class depends only on $w$. Let $w_0$ denote the longest element of the Weyl group $W$. The sequel of the paper is devoted to the construction of an injective homomorphism $I$ from $\mathcal{A}_{w_0}$ into the Grothendieck ring of $\mathcal{O}^{\mathrm{sh}}$, and to the proof that its image $I(\mathcal{A}_{w_0})$ is dense. 

One of the key points is the existence of a stabilized $g$-vector for each cluster variable of our initial seed. Heuristically, the stabilized $g$-vector can be regarded as the $g$-vector with respect to the canonical initial seed of $\mathcal{A}_e$, where $e$ denotes the unit element of $W$, but since the two seeds are only connected via an infinite sequence of mutations, one needs to prove the stabilization of the corresponding sequences of $g$-vectors. In \cite[\S 4]{GHL} our proof of this stabilization process relied on a close connection in simply-laced type between initial seeds of $\mathcal{A}_{w_0}$   
and the Auslander-Reiten combinatorics of a Dynkin quiver $Q$ of the same type. This enabled us to obtain the stabilized $g$-vectors via an analogue of the well-known knitting algorithm for calculating the Auslander-Reiten quiver of $Q$. However, it does not seem possible to extend this method to the non simply-laced types, for which the description of the quivers of our initial seeds is significantly more involved.   

In \S\ref{sgvect}, we develop a new approach using a twisting of a braid group action introduced in \cite{FH3}. The main difference with the proof of \cite[\S 4]{GHL} lies in the choice of our infinite sequence of mutations. The sequence of the present paper allows a different type of induction, which can be applied to both simply-laced and non simply-laced types. In the end we obtain a generalization to all types of the knitting algorithm of \cite[\S 4]{GHL}. In \S\ref{stabg} we establish that the g-vector of any cluster variable stabilizes, and that this stabilized value does not depend on the choice of a reduced expression. 

With this in hand, the construction of the homomorphism $I$ and the proof that its image is dense follow the same lines as in \cite{GHL}. In \S\ref{sec:rings}, we introduce the rings of formal power series used for constructing the $q$-characters of objects of $\mathcal{O}^{\mathrm{sh}}$. In \S\ref{sec:QQ} we define following \cite{FH3} the $Q$-variables, and we recall the functional relations that they satisfy, called $QQ$-systems. Then in \S\ref{sec:main-thm} we define the ring homomorphism $I\colon \mathcal{A}_{w_0} \to K_\Z$ and show our main result (Theorem~\ref{main-Thm}). To prove that $I$ is independent of the choice of the initial seed of $\mathcal{A}_{w_0}$ which was used to define it (Theorem~\ref{uniF}), we extend the notion of stabilized $g$-vector to every cluster variable, and then show that this independence follows from a unicity result for solutions of the $QQ$-system.

The topological ring $K_\Z$ is topologically generated by the elements $\underline{Q}^{\pm 1}_{\varpi_i^\vee, q^r}$ 
(where the parity of $r\in \Z$ is determined by $i$).  
In \S\ref{sect-limitF} we show that the expansion of the image by $I$ of a cluster monomial of $\mathcal{A}_{w_0}$ in 
terms of the variables $\underline{Q}^{\pm 1}_{\varpi_i^\vee, q^r}$ can be interpreted as a cluster expansion with 
respect to the canonical initial seed of $\mathcal{A}_e$. That is, we prove that this power series expansion is the limit of the Laurent polynomial cluster expansions along an infinite sequence of mutations ending in this limit seed. 
Using the positivity theorem for cluster expansions, we deduce that this power series has positive coefficients
(Theorem~\ref{devtp}). Note that this question was not considered in \cite{GHL}, so this result is also new in simply-laced types $A, D, E$. 

\medskip

Finally, in \S\ref{sec:shifted-quantum} we connect our results with the representation theory of shifted quantum affine algebras. After reviewing the algebras $U_q^\mu(\widehat{\mathfrak{g}})$ and the category $\mathcal{O}^{\mathrm{sh}}$, we deduce from Theorem~\ref{main-Thm} the desired cluster structure on the Grothendieck ring of   
$\mathcal{O}^{\mathrm{sh}}$ (Theorem~\ref{isomgv}). We end this section by providing some evidence for the monoidal categorification conjecture (Conjecture~\ref{main-conj}). Note that if a cluster monomial is proved to be the class of a simple object $M$ of $\mathcal{O}^{\mathrm{sh}}$, then its limit cluster expansion given by Theorem~\ref{devtp} coincides with the $q$-character of $M$ (Theorem~\ref{qc}).

\medskip

The results of this work were presented in several conferences by the authors, in particular at the Workshop {\it Enveloping Algebras and Geometric Representation Theory} in Oberwolfach in May 2025 (see \cite{HMF}), at the \emph{Mathematical Congress of the Americas 2025} in Miami, at the \emph{ARTA~X} Conference in Cologne in September 2025,
and at the Conference {\it Quantum groups, monoidal categorification and related topics} in Jeju Island in December 2025.

\section{Definitions}

\subsection{Cartan matrix of a Lie algebra}\label{ssecCartan}
Let $\g$ be a simple Lie algebra over $\C$, with Cartan generators $e_i, f_i, h_i\ (i\in I)$.
We denote by $n := |I|$ the rank of $\g$.
Let $C=(c_{ij})_{i,j\in I}$ be the Cartan matrix of $\g$, which encodes the Serre presentation of~$\g$.
There is a diagonal matrix $D = \diag(d_i\mid i\in I)$ with entries in $\Z_{>0}$ 
such that the product
\[
B=D\,C = (b_{ij})_{i,j\in I}
\]
is symmetric. We normalize $D$ so that $\min\{d_i \mid i\in I\} = 1$
and call it the minimal symmetrizer of~$C$.

Put $t:=\max\{d_i \mid i\in I\}$.
Thus 
\[
t =
\left\{
\begin{array}{cl}
1 & \mbox{if $C$ is of type $A_n$, $D_n$, $E_6$, $E_7$ or $E_8$}, \cr
         2 & \mbox{if $C$ is of type $B_n$, $C_n$ or $F_4$}, \cr
         3 & \mbox{if $C$ is of type $G_2$}.
\end{array}
\right.
\]

\begin{example}\label{example_Cartan}
The Lie algebra $\g = \mathfrak{so}_7$, of type $B_3$, has
Cartan matrix
\[
C = \begin{pmatrix}
2 & -1 & 0 \cr
-1 & 2 & -1\cr
0 & -2 & 2\end{pmatrix}
\]
We have $D=\diag(2,2,1)$ and the symmetric matrix $B$ is given by
\[
B = \begin{pmatrix}
4 & -2 & 0 \cr
-2 & 4 & -2\cr
0 & -2 & 2 \end{pmatrix}
\]
\end{example}

\subsection{Weights, roots, and Weyl group}

Let $\mathfrak{h} := \oplus_{i\in I}\C h_i$ denote the Cartan subalgebra of $\g$, 
and let $\mathfrak{h}^*$ be its dual vector space. 
We denote by $(\varpi_i\mid i\in I)$ the basis of $\mathfrak{h}^*$ dual to 
$(h_i\mid i\in I)$, and call its elements the fundamental weights.
We denote by $\a_i\ (i\in I)$ the simple roots of $\g$, defined by
\begin{equation}\label{fundw_roots}
\a_i = \sum_{j\in I} c_{ji} \varpi_j,\qquad (i,j\in I). 
\end{equation}
They span respectively the root lattice and the weight lattice:
\[
 Q := \bigoplus_{i\in I} \Z\a_i\ \subset \  P := \bigoplus_{i\in I} \Z\varpi_i
 \ \subset \ \mathfrak{h}^*.
\]
Let $Q_+ := \bigoplus_{i\in I} \Z_{\ge 0}\a_i$.
The weight lattice $P$ is endowed with a partial ordering given by:
\[
 \la \ge \mu \quad \Longleftrightarrow \quad \la-\mu\in Q_+,\qquad (\la,\mu \in P).
\]

The Weyl group $W$ of $\g$ is generated by the simple reflexions $s_i$, which act on $\mathfrak{h}^*$ by
\[
 s_i(\la) = \la - \la(h_i) \a_i,\qquad (i\in I,\ \la\in \mathfrak{h}^*).
\]
In particular, we have
\[
s_i(\varpi_j) = \varpi_j - \de_{ij}\a_i,\quad s_i(\a_j) = \a_j - c_{ij}\a_i,\qquad (i,j\in I).
\]
Let $w_0$ be the longest element of $W$. It induces an involution $\nu : I \to I$ defined via the equality
\[
 w_0(\alpha_i) = -\alpha_{\nu(i)},\quad (i\in I).
\]

\subsection{Langlands dual and coweights}

Recall that the Langlands dual Lie algebra $\g^L$ has Cartan matrix $C^T$, the transpose of $C$.
Note that if we denote by $D'$ the minimal symmetrizer of $C^T$, we have that 
$DD'= \diag(t,\ldots,t)$.

\begin{example}\label{example_Cartan2}
We continue Example~\ref{example_Cartan}.  
The Langlands dual of $\g = \mathfrak{so}_7$ is is $\g^L = \mathfrak{sp}_6$, of type $C_3$,
with Cartan matrix
\[
C^T = \begin{pmatrix}
2 & -1 & 0 \cr
-1 & 2 & -2\cr
0 & -1 & 2 \end{pmatrix}
\]
and symmetrizer $D'=\diag(1,1,2)$.
\end{example}

We denote by $(\varpi^\vee_i\mid i\in I)$ and $(\alpha^\vee_i\mid i\in I)$ the fundamental weights and simple roots of
$\g^L$, and call them the fundamental coweights and simple coroots of $\g$. They are related by
\begin{equation}\label{cofundw_roots}
\a^\vee_i = \sum_{j\in I} c_{ij} \varpi^\vee_j,\qquad (i,j\in I). 
\end{equation} 
The two Lie algebras $\g$ and $\g^L$ have the same Weyl group $W$, which acts on the coweight lattice $P^\vee$ of $\g^L$ by
\[
s_i(\varpi^\vee_j) = \varpi^\vee_j - \de_{ij}\a^\vee_i,\quad s_i(\a^\vee_j) = \a^\vee_j - c_{ji}\a^\vee_i,\qquad (i,j\in I).
\]
Finally, we have a natural pairing between $P$ and $P^\vee$ given by
\[
 \a_i(\varpi^\vee_j) = \de_{ij},\qquad \a_i(\a^\vee_j) = c_{ji}, \qquad (i,j\in I).
\]

\subsection{Basic infinite quiver}
\label{subsec-basic-quiver}
\begin{figure}[t]
\[
\def\objectstyle{\scriptstyle}
\def\lablestyle{\scriptstyle}
\xymatrix@-1.0pc{
&{}\save[]+<0cm,1.5ex>*{\vdots}\restore&{}\save[]+<0cm,1.5ex>*{\vdots}\restore  
&{}\save[]+<0cm,1.5ex>*{\vdots}\restore
\\
&{(1,2)}\ar[rd]\ar[u]&
&\ar[ld] (3,2) \ar[u]
\\
&&\ar[ld] (2,1) \ar[rd]\ar[uu]&&
\\
&\ar[uu]{(1,0)}\ar[rd]&
&\ar[ld] (3,0) \ar[uu]
\\
&&\ar[uu]\ar[ld] (2,-1) \ar[rd]&&
\\
&\ar[uu](1,-2) \ar[rd] &&\ar[ld] (3,-2)\ar[uu]
\\
&&\ar[ld] \ar[uu](2,-3) \ar[rd]&&
\\
&\ar[uu](1,-4) \ar[rd] &&\ar[ld] (3,-4)\ar[uu]
\\
&&\ar[ld] \ar[uu](2,-5) \ar[rd]&&
\\
&\ar[uu](1,-6) &&\ar[uu] (3,-6) 
\\
&{}\save[]+<0cm,0ex>*{\vdots}\ar[u]\restore&{}\save[]+<0cm,0ex>*{\vdots}\ar[uu]\restore  
&{}\save[]+<0cm,0ex>*{\vdots}\ar[u]\restore
\\
}
\qquad
\def\objectstyle{\scriptstyle}
\xymatrix@-1.0pc{
&{}\save[]+<0cm,1.5ex>*{\vdots}\restore&{}\save[]+<0cm,1.5ex>*{\vdots}\restore  
&{}\save[]+<0cm,1.5ex>*{\vdots}\restore&{}\save[]+<0cm,1.5ex>*{\vdots}\restore 
&{}\save[]+<0cm,1.5ex>*{\vdots}\restore 
\\
&(\bun,3)\ar[rd]\ar[u]&&\ar[ld] (\btrois,3)\ar[u]& \ar[ld](\bdeux,3)\ar[u]\ar[rd] &&
\\
&&{(\bdeux,1)}\ar[rd]\ar[ld]\ar[uu] &\ar[u] (\btrois,1) \ar[rd]&&(\bun,1)\ar[ld]\ar[uu]&
\\
&(\bun,-1)\ar[rd]\ar[uu]&&\ar[ld] (\btrois,-1)\ar[u]& \ar[ld](\bdeux,-1)\ar[rd]\ar[uu] &&
\\
&&{(\bdeux,-3)}\ar[rd]\ar[ld]\ar[uu] &\ar[u] (\btrois,-3) \ar[rd]&&(\bun,-3)\ar[ld]\ar[uu]&
\\
&\ar[uu](\bun,-5)\ar[rd]&&\ar[u]\ar[ld] (\btrois,-5) &\ar[ld] (\bdeux,-5)\ar[uu] \ar[rd]  &&
\\
&&\ar[uu](\bdeux,-7)\ar[rd]\ar[ld]   & \ar[u](\btrois,-7) \ar[rd]&&(\bun,-7)\ar[ld]\ar[uu]&
\\
&(\bun,-9)\ar[rd]\ar[uu]&&\ar[u]\ar[ld] (\btrois,-9) &\ar[ld] (\bdeux,-9) \ar[uu]\ar[rd] &&
\\
&&\ar[uu](\bdeux,-11)\ar[ld]\ar[rd] & \ar[u](\btrois,-11)\ar[rd]&&\ar[ld](\bun,-11)\ar[uu]&
\\
&(\bun,-13)\ar[uu] && (\btrois,-13)\ar[u] & (\bdeux,-13) \ar[uu]
\\
&{}\save[]+<0cm,0ex>*{\vdots}\ar[u]\restore&{}\save[]+<0cm,0ex>*{\vdots}\ar[uu]\restore  
&{}\save[]+<0cm,0ex>*{\vdots}\ar[u]\restore&{}\save[]+<0cm,0ex>*{\vdots}\ar[u]\restore 
&{}\save[]+<0cm,0ex>*{\vdots}\ar[uu]\restore 
\\
}
\]
\caption{\label{Fig0} {\it The basic quivers $\G_C$ in type $A_3$ and $B_3$.}}
\end{figure}
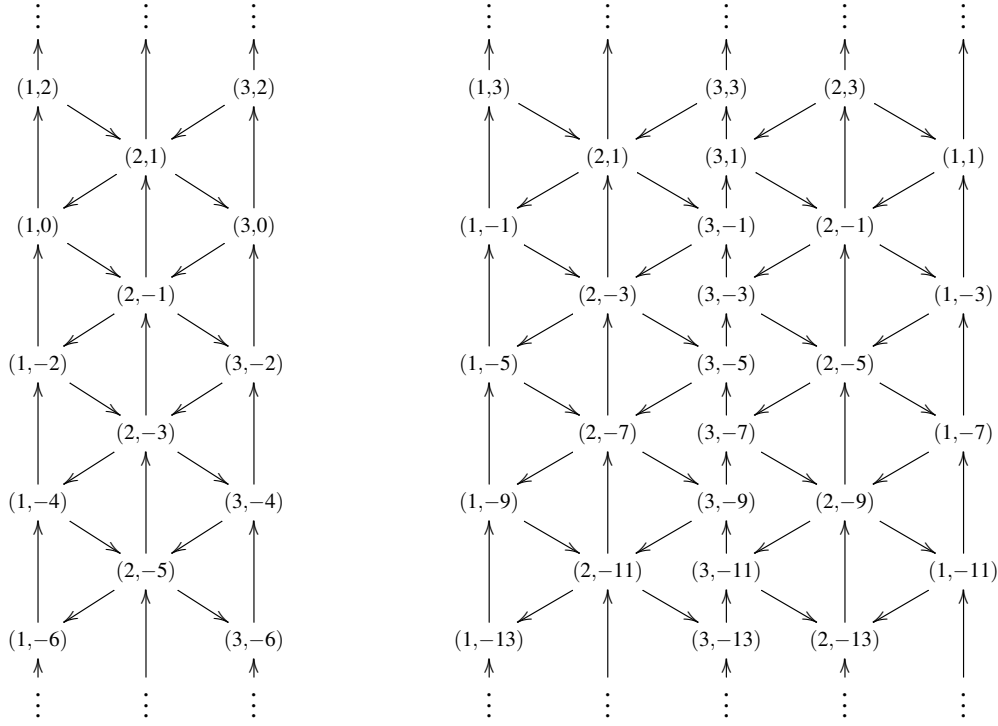

Following \cite{HL1}, we attach an infinite quiver to the Cartan matrix $C$.
Put $\widetilde{V} = I \times \Z$.
We introduce a quiver $\tG$ with vertex set $\widetilde{V}$.
Recall the symmetric matrix $B= (b_{ij})_{i,j\in I}$ of \S\ref{ssecCartan}. The arrows of $\tG$ are given by
\[
(i,r) \to (j,s) 
\quad \Longleftrightarrow \quad
(c_{ij}\not = 0 
\quad \mbox{and} \quad
s=r+b_{ij}).
\]
It is easy to check that the oriented graph $\tG$ has two isomorphic connected components. 
We pick one of them and call it $\Gamma$. The vertex set of $\Gamma$ is denoted by $V$ and the arrow set by $A$. We call $\Gamma = \Gamma_C$ the \emph{basic infinite quiver of $C$}.
Examples in type $A_3$ and $B_3$ are shown in Figure~\ref{Fig0}.

\section{Quivers and cluster algebras} \label{sec:QClA}

Let $U_q(\widehat{\g})$ denote the quantum affinization of the Lie algebra $\g$.
In \cite{HL2}, we have introduced the cluster algebra $\AA_\G$ defined by the basic infinite quiver $\G$, and we have shown that a completion of $\AA_\G$ is isomorphic to the Grothendieck rings of two tensor subcategories $\O^+_{\Z}$ and $\O^-_{\Z}$ of the category $\O$ of representations of the Borel subalgebra of $U_q(\widehat{\g})$ introduced in \cite{HJ}. More recently, it was shown in \cite{H} that $\AA_\G$ is isomorphic to the Grothendieck ring of the category $\mathcal{C}^{\mathrm sh}_\Z$ of finite-dimensional representations of shifted quantum affine algebras (with integral spectral parameter). The aim of this section is to introduce a more general class of cluster algebras $\AA_w$ labelled by $w\in W$, where $\AA_e := \AA_\G$ corresponds to the unit element $e$ of $W$. This was already explained in \cite{GHL} for the simply-laced types $A,D,E$, so 
we will mainly focus here on the new features of the non simply-laced types.

\subsection{The case $w=s_i$}
\label{subsec2-2}

In this section we introduce the cluster algebras $\AA_{s_i}$ associated with the simple reflections $s_i$. We start by defining the quivers $\G_{s_i,\,r}$.

\begin{Def}\label{Def2-7}
Let $i\in I$ and let $(i,r)$ be a vertex of $\G = \G_e$. The quiver $\G_{s_i,\,r}$ is obtained from 
$\G_e$ by performing the following operations:
\begin{enumerate}
 \item [(a)] If $d_i = 1$, then
\begin{enumerate}
 \item[(i)] insert a new vertex $*$ between vertices $(i,r)$ and $(i,r-2)$;
 \item[(ii)] replace the arrow $(i,r) \leftarrow (i,r-2)$ by a pair of arrows 
$(i,r) \to * \leftarrow (i,r-2)$;
 \item[(iii)] for $j$ with $c_{ij}<0$, replace the arrow $(i,r) \to (j,r+b_{ij})$ by an arrow $(* \to (j,r+b_{ij})$;
 \item[(iv)] if $c_{ij}= -2$, add two new arrows $(i,r) \leftarrow (j,r)$ and $*\to (j,r)$;
 \item[(v)] if $c_{ij}= -3$, add four new arrows $(i,r) \leftarrow (j,r\pm 1)$ and $*\to (j,r\pm 1)$;
 \item[(vi)] rename vertex $*$ as $(i,r-2)$, and rename all vertices $(i,p)$'s with $p<r$ as $(i,p-2)$.
 \end{enumerate}
\item [(b)] If $d_i = 2$, then
\begin{enumerate}
 \item[(i)] insert a new vertex $*$ between vertices $(i,r)$ and $(i,r-4)$,
 and a new vertex $\dagger$ between vertices $(i,r-2)$ and $(i,r-6)$;
 \item[(ii)] replace the arrow $(i,r) \leftarrow (i,r-4)$ (\resp $(i,r-2) \leftarrow (i,r-6)$) by a pair of arrows $(i,r) \to * \leftarrow (i,r-4)$ (\resp $(i,r-2) \to \dagger \leftarrow (i,r-6)$);
 \item[(iii)] for $j$ with $c_{ij}<0$, replace the arrow $(i,r) \to (j,r+b_{ij})$ 
 (\resp $(i,r-2) \to (j,r+b_{ij}-2)$) 
 by an arrow $* \to (j,r+b_{ij})$
 (\resp $\dagger \to (j,r+b_{ij}-2)$);
 \item[(iv)] rename vertex $*$ as $(i,r-4)$, vertex $\dagger$ as $(i, r-6)$, 
 and rename all vertices $(i,p)$'s with $p<r-2$ as $(i,p-4)$.
\end{enumerate}
\item [(b)] If $d_i = 3$, then
\begin{enumerate}
 \item[(i)] insert a new vertex $*$ between vertices $(i,r)$ and $(i,r-6)$,
a new vertex $\dagger$ between vertices $(i,r-2)$ and $(i,r-8)$, and 
a new vertex $\circ$ between vertices $(i,r-4)$ and $(i,r-10)$.
 \item[(ii)] replace the arrow $(i,r) \leftarrow (i,r-6)$ (\resp $(i,r-2) \leftarrow (i,r-8)$, $(i,r-4) \leftarrow (i,r-10)$) 
 by a pair of arrows 
$(i,r) \to * \leftarrow (i,r-6)$ (\resp $(i,r-2) \to \dagger \leftarrow (i,r-8)$, $(i,r-4) \to \circ \leftarrow (i,r-10)$;
 \item[(iii)] for $j$ with $c_{ij}<0$, replace the arrow $(i,r) \to (j,r+b_{ij})$ 
 (\resp $(i,r-2) \to (j,r+b_{ij}-2)$, $(i,r-4) \to (j,r+b_{ij}-4)$)
 by an arrow $* \to (j,r+b_{ij})$ (\resp $\dagger \to (j,r+b_{ij}-2)$, $\circ \to (j,r+b_{ij}-4)$);
 \item[(iv)] rename vertex $*$ as $(i,r-6)$, vertex $\dagger$ as $(i, r-8)$, vertex $\circ$ as $(i, r-10)$, 
 and rename all vertices $(i,p)$'s with $p<r-4$ as $(i,p-6)$.
\end{enumerate}
\end{enumerate}
\end{Def}

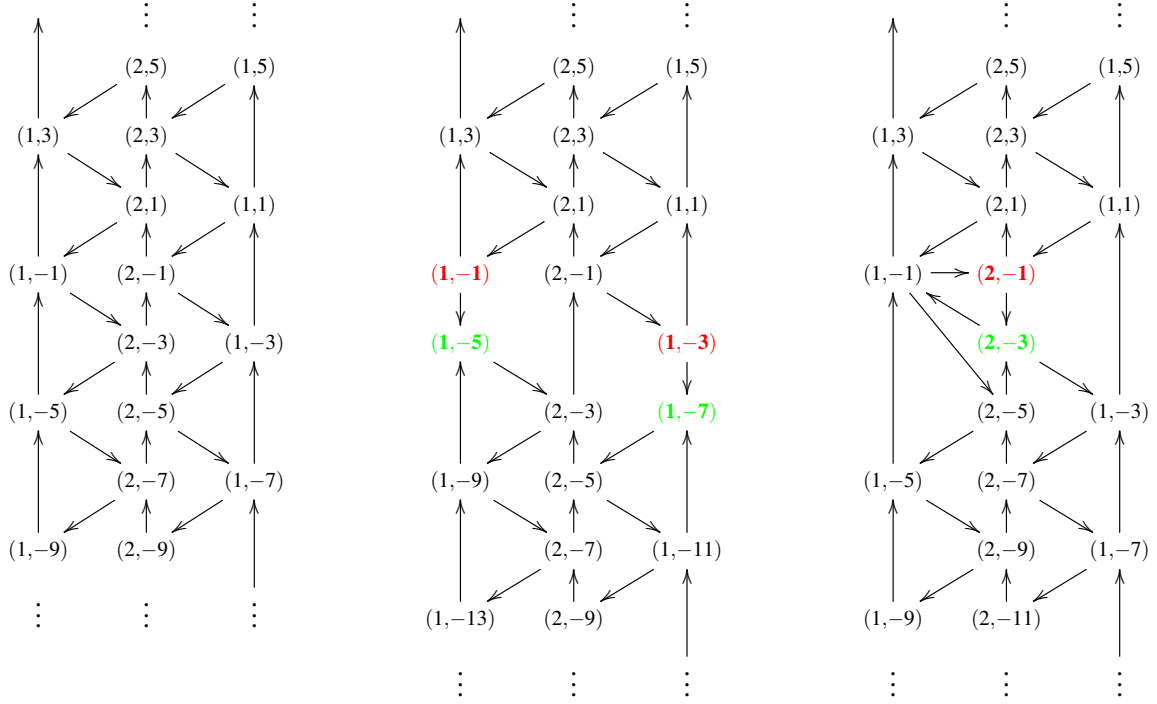
\begin{figure}[t]
\[
\def\objectstyle{\scriptstyle}
\def\lablestyle{\scriptstyle}
\xymatrix@-1.0pc{
&&&&\\
&&{}\save[]+<0cm,0ex>*{\vdots}\restore
& {}\save[]+<0cm,0ex>*{\vdots}\restore
\\
&&\ar[ld] (\bdeux,5)& \ar[ld](\bun,5)&
\\
&{(\bun,3)}\ar[rd]\ar[uu] &\ar[u] (\bdeux,3) \ar[rd]&&
\\
&&\ar[u]\ar[ld] (\bdeux,1) &\ar[ld] (\bun,1) \ar[uu]  &
\\
&\ar[uu](\bun,-1)\ar[rd]   & \ar[u](\bdeux,-1) \ar[rd]&&
\\
&&\ar[u]\ar[ld] (\bdeux,-3) &\ar[ld] (\bun,-3) \ar[uu] &
\\
&\ar[uu]\ar[rd](\bun,-5) & \ar[u]\ar[rd](\bdeux,-5)&&
\\
&&\ar[u]\ar[ld] (\bdeux,-7) &\ar[ld] (\bun,-7) \ar[uu]  &
\\
&\ar[uu](\bun,-9)   & \ar[u](\bdeux,-9) &&
\\
&{}\save[]+<0cm,0ex>*{\vdots}\restore  &{}\save[]+<0cm,0ex>*{\vdots}\restore&{}\save[]+<0cm,0ex>*{\vdots}\ar[uu]\restore 
\\
}
\quad
\xymatrix@-1.0pc{
&&&&\\
&&{}\save[]+<0cm,0ex>*{\vdots}\restore
& {}\save[]+<0cm,0ex>*{\vdots}\restore
\\
&&\ar[ld] (\bdeux,5)& \ar[ld](\bun,5)&
\\
&{(\bun,3)}\ar[rd]\ar[uu] &\ar[u] (\bdeux,3) \ar[rd]&&
\\
&&\ar[u]\ar[ld] (\bdeux,1) &\ar[ld] (\bun,1) \ar[uu]  &
\\
&\ar[d]\ar[uu]\mathbf{\red(\bun,-1)}   & \ar[u](\bdeux,-1) \ar[rd]&&
\\
&\mathbf{\green(\bun,-5)}\ar[rd] && \mathbf{\red(1,-3)}\ar[uu]\ar[d] 
\\
&&\ar[uu]\ar[ld] (\bdeux,-3) &\ar[ld] \mathbf{\green(\bun,-7)}  &
\\
&\ar[uu]\ar[rd](\bun,-9) & \ar[u]\ar[rd](\bdeux,-5)&&
\\
&&\ar[u]\ar[ld] (\bdeux,-7) &\ar[ld] (\bun,-11) \ar[uu]  &
\\
&\ar[uu](\bun,-13)   & \ar[u](\bdeux,-9) &&
\\
&{}\save[]+<0cm,0ex>*{\vdots}\restore  &{}\save[]+<0cm,0ex>*{\vdots}\restore&{}\save[]+<0cm,0ex>*{\vdots}\ar[uu]\restore 
\\
}
\quad
\def\objectstyle{\scriptstyle}
\def\lablestyle{\scriptstyle}
\xymatrix@-1.0pc{
&&&&\\
&&{}\save[]+<0cm,0ex>*{\vdots}\restore
& {}\save[]+<0cm,0ex>*{\vdots}\restore
\\
&&\ar[ld] (\bdeux,5)& \ar[ld](\bun,5)&
\\
&{(\bun,3)}\ar[rd]\ar[uu] &\ar[u] (\bdeux,3) \ar[rd]&&
\\
&&\ar[u]\ar[ld] (\bdeux,1) &\ar[ld] (\bun,1) \ar[uu]  &
\\
&\ar[uu](\bun,-1)\ar[rdd]\ar[r]   & \ar[d]\ar[u]\mathbf{\red(\bdeux,-1)} &&
\\
&   & \ar[lu]\mathbf{\green(\bdeux,-3)} \ar[rd]&&
\\
&&\ar[u]\ar[ld] (\bdeux,-5) &\ar[ld] (\bun,-3) \ar[uuu] &
\\
&\ar[uuu]\ar[rd](\bun,-5) & \ar[u]\ar[rd](\bdeux,-7)&&
\\
&&\ar[u]\ar[ld] (\bdeux,-9) &\ar[ld] (\bun,-7) \ar[uu]  &
\\
&\ar[uu](\bun,-9)   & \ar[u](\bdeux,-11) &&
\\
&{}\save[]+<0cm,0ex>*{\vdots}\restore  &{}\save[]+<0cm,0ex>*{\vdots}\restore&{}\save[]+<0cm,0ex>*{\vdots}\ar[uu]\restore 
\\
}
\]
\caption{\label{Fig7} {\it The quivers $\G_e$, $\G_{s_1,-1}$ and $\G_{s_2,-1}$ in type $B_2$.}}
\end{figure}

This is illustrated in Figure~\ref{Fig7}, where we display the quivers $\G_{s_1,-1}$ and $\G_{s_2,-1}$ in type $B_2$. 
The new inserted vertices are painted in green. The vertices immediately above the green vertices are painted in red. Hence the only vertical down arrows join the red vertices to the green vertices.

\begin{Lem} \label{Lem28} 
\begin{enumerate}
 \item[\rm (i)] If we mutate the quiver $\G_{s_i,\,r}$ at vertex $(i,r)$ if $d_i = 1$ (\resp at vertices $(i,r)$ and $(i,r-2)$ if $d_i = 2$; at vertices $(i,r)$, $(i,r-2)$ and $(i,r-4)$ if $d_i = 3$) 
 then we obtain the quiver $\G_{s_i,\,r+2d_i}$.
 \item[\rm (ii)] If we mutate the quiver $\G_{s_i,\,r}$ at vertex $(i,r-2)$ if $d_i = 1$ (\resp at vertices $(i,r-4)$ and $(i,r-6)$ if $d_i = 2$; at vertices $(i,r-6)$, $(i,r-8)$ and $(i,r-10)$ if $d_i = 3$) 
 then we obtain the quiver $\G_{s_i,\,r-2d_i}$.
\end{enumerate}
\end{Lem}

\begin{proof}
This is easily checked by inspection, using the quiver mutation rules. \cqfd  
\end{proof}

It follows from Lemma~\ref{Lem28} that, for a fixed $i\in I$, all the quivers $\G_{s_i,\,r}$ are mutation-equivalent, and they define the same cluster algebra. The next definition then makes sense.

\begin{Def}
\emph{We denote by $\AA_{s_i}$ the cluster algebra with quiver $\G_{s_i,\,r}$.}
\end{Def}

\subsection{General case with $w$ arbitrary}\label{SecAw}

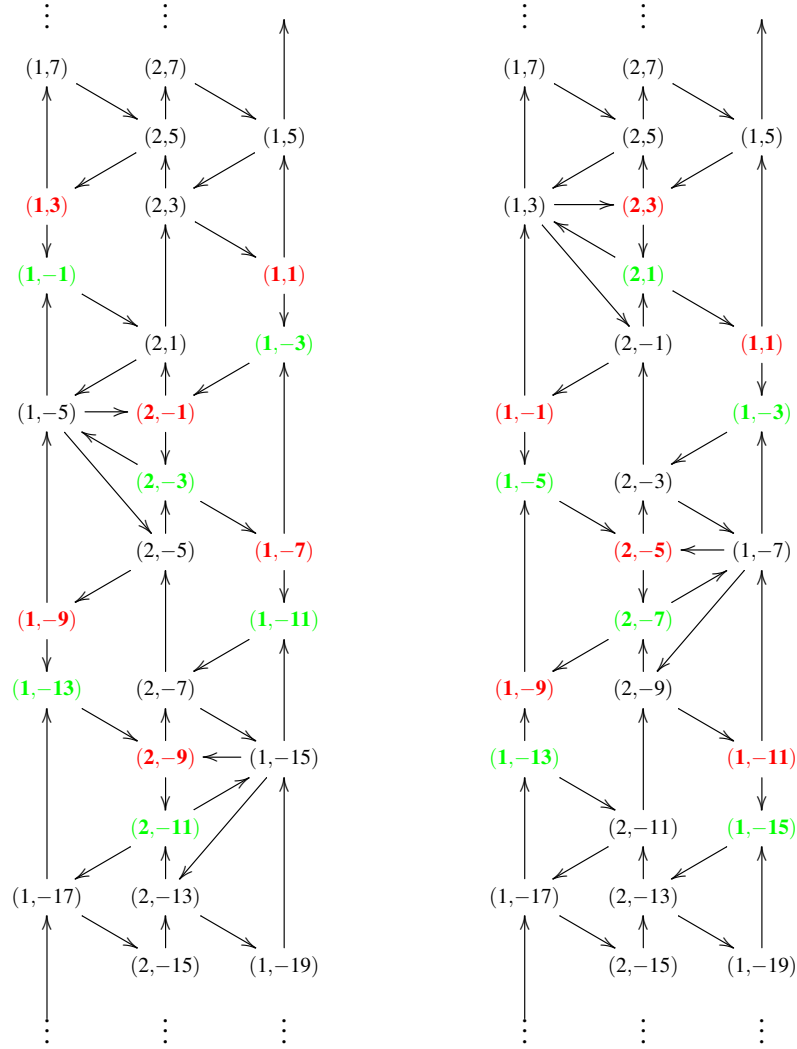
\begin{figure}[t]
\[
\def\objectstyle{\scriptstyle}
\def\lablestyle{\scriptstyle}
\def\objectstyle{\scriptstyle}
\def\lablestyle{\scriptstyle}
\xymatrix@-1.0pc{
&{}\save[]+<0cm,0ex>*{\vdots}\restore&{}\save[]+<0cm,0ex>*{\vdots}\restore
& 
\\
&(\bun,7)\ar[rd] & (\bdeux,7) \ar[rd]&&
\\
&&\ar[ld] (\bdeux,5)\ar[u]& \ar[ld](\bun,5)\ar[uu]&
\\
&\mathbf{\red(\bun,3)}\ar[uu]\ar[d] &\ar[u] (\bdeux,3) \ar[rd]&&
\\
&\mathbf{\green(\bun,-1)}\ar[rd] && \mathbf{\red(\bun,1)} \ar[uu]\ar[d] 
\\
&&\ar[uu]\ar[ld] (\bdeux,1) &\ar[ld] \mathbf{\green(\bun,-3)}   &
\\
&\ar[uu](\bun,-5)\ar[rdd]\ar[r]   & \ar[d]\mathbf{\red(\bdeux,-1)}\ar[u] &&
\\
&                          & \ar[lu]\mathbf{\green(\bdeux,-3)}\ar[rd]
\\
&&\ar[ld] (\bdeux,-5)\ar[u] & \mathbf{\red(\bun,-7)} \ar[uuu]\ar[d] &
\\
&\mathbf{\red(1,-9)}\ar[uuu]\ar[d] & &\ar[ld] \mathbf{\green(\bun,-11)}  &
\\
&\mathbf{\green\ar[rd](\bun,-13)} & \ar[uu]\ar[rd](\bdeux,-7)&&
\\
&&\ar[u] \mathbf{\red(\bdeux,-9)}\ar[d] &\ar[l]\ar[ldd] (\bun,-15) \ar[uu]  &
\\
&&\ar[ld] \mathbf{\green(\bdeux,-11)}\ar[ru] &&
\\
&\ar[uuu](\bun,-17)\ar[rd]   & \ar[u](\bdeux,-13)\ar[rd] &&
\\
&&(2,-15)\ar[u] & (1,-19)\ar[uuu]
\\
&\ar[uu]{}\save[]+<0cm,0ex>*{\vdots}\restore  &{}\save[]+<0cm,0ex>*{\vdots}\restore&{}\save[]+<0cm,0ex>*{\vdots}\restore 
\\
}
\qquad
\def\objectstyle{\scriptstyle}
\def\lablestyle{\scriptstyle}
\xymatrix@-1.0pc{
&{}\save[]+<0cm,0ex>*{\vdots}\restore&{}\save[]+<0cm,0ex>*{\vdots}\restore
& 
\\
&(\bun,7)\ar[rd] & (\bdeux,7) \ar[rd]&&
\\
&&\ar[ld] (\bdeux,5)\ar[u]& \ar[ld](\bun,5)\ar[uu]&
\\
&{(\bun,3)}\ar[rdd]\ar[uu]\ar[r] &\ar[u] \mathbf{\red(\bdeux,3)}\ar[d] &&
\\
& & \ar[lu]\mathbf{\green(2,1)}\ar[rd]
\\
&&\ar[u]\ar[ld] (\bdeux,-1) & \mathbf{\red(\bun,1)}\ar[d] \ar[uuu]  &
\\
&\ar[uuu]\ar[d]\mathbf{\red(1,-1)} && \ar[ld]\mathbf{\green(\bun,-3)}
\\
&\mathbf{\green(\bun,-5)}\ar[rd]   & \ar[uu](\bdeux,-3) \ar[rd]&&
\\
&&\ar[u] \mathbf{\red(\bdeux,-5)}\ar[d] &\ar[l]\ar[ldd] (\bun,-7) \ar[uu] &
\\
&& \ar[ld]\mathbf{\green(\bdeux,-7)}\ar[ru]
\\
&\ar[uuu]\mathbf{\red(\bun,-9)} & \ar[u]\ar[rd](\bdeux,-9)&&
\\
& \ar[rd]\ar[u]\mathbf{\green(1,-13)} && \mathbf{\red(1,-11)} \ar[uuu]\ar[d]
\\
&&\ar[uu]\ar[ld] (\bdeux,-11) &\ar[ld] \mathbf{\green(\bun,-15)}  &
\\
&\ar[uu](\bun,-17)\ar[rd]   & \ar[u](\bdeux,-13)\ar[rd] &&
\\
&&(2,-15)\ar[u] & (1,-19)\ar[uu]
\\
&\ar[uu]{}\save[]+<0cm,0ex>*{\vdots}\restore  &{}\save[]+<0cm,0ex>*{\vdots}\restore&{}\save[]+<0cm,0ex>*{\vdots}\restore 
\\
}
\]
\caption{\label{Fig8} {\it Two initial seeds of $\AA_{w_0}$ in type $B_2$.}}
\end{figure}

To define the cluster algebra $\AA_w$ for arbitrary $w$, we proceed exactly as in \cite[\S3.3]{GHL}.
Namely, we fix a reduced expression $w = s_{i_1}\cdots s_{i_l}$ of $w$, and a sequence of vertices 
$(i_1,r_1), \ldots , (i_l, r_l)$ of $\G_e$ such that $r_{s} \gg r_{s+1}$ for every $1\le s <l$
(the minimal descent steps $r_s - r_{s+1}$ will be specified in Definition~\ref{def:adapt} below). 
To this datum, we can attach a quiver $\G_{(i_1,\ldots, i_l),\,(r_1,\ldots, r_l)}$ by performing the same operations as in Definition~\ref{Def2-7} successively at each vertex $(i_1,r_1), \ldots , (i_l, r_l)$. 

Put $\bi:=(i_1,\ldots, i_l)$ and $\br := (r_1,\ldots, r_l)$.
It follows from Lemma~\ref{Lem28} that the cluster algebra defined by the quiver $\G_{\bi, \br}$ depends only on the sequence $\bi$, and not on the choices of $r_1, \ldots , r_l$ (subject to the above condition). Moreover, as in \cite[\S3.3]{GHL}, one can show the following crucial property of our family of quivers.

\begin{Thm}\label{Prop2-10}
The cluster algebra defined by the quiver $\G_{\bi,\br}$ depends only on the Weyl group element $w$, not on the choice of $\br$ and not on the choice of a reduced word $\bi$ for $w$. 
\end{Thm}

\begin{proof}
For $A, D, E$ types, this was proved in \cite[Proposition 3.5]{GHL}, via a reduction to type~$A_2$.
Similarly, for types $B, C, F$, this can be reduced to a verification in rank $2$. For instance, for type $B_2$, we exhibit in \S\ref{subsec-5.2} an explicit sequence of mutations from $\G_{(1,2,1,2),\,(r_1,r_2,r_3,r_4)}$ to $\G_{(2,1,2,1),\,(p_1,p_2,p_3,p_4)}$ for some $(r_1,r_2,r_3,r_4)$ 
and $(p_1,p_2,p_3,p_4)$. 

The detailed proof of Theorem~\ref{Prop2-10} will be given in Section \ref{mutequiv} (and in Appendix \ref{App:G2-braid} for type $G_2$).

\cqfd
\end{proof}

\begin{Rem} 
Alternatively, we have also found a general argument based on some results of \cite{FHOO} which reduces 
this theorem for non simply-laced types to the same statement for simply-laced types. 
For instance, this general argument allows to deduce the statement in type $G_2$ from 
\cite[Proposition 3.5]{GHL} in type $D_4$. We plan to present this general argument in a forthcoming paper.
\end{Rem}

\begin{Def}
\emph{The cluster algebra with defining quiver $\G_{\bi,\br}$ is denoted by $\AA_w$.}  
\end{Def}

In Figure~\ref{Fig8} we display two initial seeds of the cluster algebra $\AA_{w_0}$ in type $B_2$, corresponding, respectively to the two reduced expressions $s_1s_2s_1s_2$ and
$s_2s_1s_2s_1$ of $w_0$.

\section{Stabilized $g$-vectors}\label{sgvect}

In this section we extend to  all Cartan types the results of \cite[\S 4]{GHL} for types $A$, $D$, $E$.
The method of \cite[\S 4]{GHL}, based on Auslander-Reiten quivers combinatorics, does not generalize to the non simply-laced cases, so we use a different approach, which also provides new proofs for the results of \cite[\S 4]{GHL} in types $A$, $D$, $E$. 

\subsection{Fine description of the quivers $\G_{\bi,\br}$} \label{subsec-fine}

We first describe the quivers $\G_{\bi, \br}$ in more detail, introducing in particular
a notation for their arrows as well as some useful combinatorial functions of $(\bi, \br)$. 

The next lemma follows immediately by inspection of the Cartan matrices.


\begin{Lem}  \label{lem:CartanM}
For $i, j \in I$ there holds $|{c_{ij}}| \leq 3$. If $c_{ij} \leq -2$ then $c_{ji} = -1$, $d_i = 1$ and $d_j = -c_{ij}$. 
Conversely, if $b_{ij} < 0$ and $d_i = 1$,  then $c_{ij} = -2d_j$. \cqfd
\end{Lem}

Recall the basic quiver $\Gamma$ of \S\ref{subsec-basic-quiver}, 
with vertex set $V \subset I \times \Z$ and arrow set $A$. The arrows in $A$ starting at vertex $(i,k)\in V$ will be denoted by:
\[
\alpha_{ji}^{(k)} : (i,k) \to (j, k + b_{ij}),\qquad (b_{ij}\not = 0). 
\]

\begin{Def} \label{def:adapt}
Let $\bi = (i_1, \ldots, i_l)$ be a reduced word for 
$w \in W$. We say that
$\br = (r_1, \ldots, r_l) \in \Z^l$ is \emph{adapted} to  $\bi$
if the following holds:
\begin{eqnarray}
(i_{j}, r_j) \in V, && (1 \le j \le l), \label{eq:1.1a} \\
r_p \leq r_m + b_{i_m,i_p}, && (1 \le m < p \le l),  \label{eq:1.1b}\\
r_p \leq r_m +2b_{i_m,i_p}, && \mbox{ if moreover } c_{i_p, i_m}\leq -2. \label{eq:1.1c} 
\end{eqnarray}
\end{Def}
Note that 
Equation (\ref{eq:1.1c}) becomes redundant for pairs $(m,p)$ such that there exists $m'$
with $i_{m'}=i_m$ and $m<m'<p$. 
In the sequel we will always assume that $\br$ is adapted to $\bi$ without further mention.

\begin{example}\label{exa-3-3}
Choose the reduced word $\bi = (2,1,2,1)$ for $w_0$ in type $B_2$.
Then $\br = (4,2,-2,-4)$ is adapted to $\bi$, with minimal descending intervals.
This is illustrated in Figure~\ref{fig:1-3}.
The vertices $(i_k,r_k)\ (1\le k\le 4)$ appear in boldface in the left quiver $\Gamma$.

\end{example}

The following notation will be useful later. For $1\le p \le l$ and $j\in I$ such that $j\not = i_p$,
we put:
\[
 m^-(j,p) := \max\{1\le h<p \mid i_h = j\},
\]
if this set is nonempty, and $m^-(j,p) = 0$ otherwise.

\subsubsection{Vertices of $\G_{\bi,\br}$} \label{sub-sec-311}

It follows from Definition~\ref{Def2-7} that the vertices of the quivers $\Gamma_{\bi,\br}$
and $\Gamma$ are labelled by the same set $V\subset I\times \Z$. This is because 
after inserting a new vertex $*$ in position $(i,k)$, we rename $*$ as $(i,k-2d_i)$ and shift accordingly all 
the second components of the vertices below it.
This can be observed in Figure~\ref{fig:1-3}. To keep track of all these shifts, it will be convenient
to introduce the following combinatorial functions.
For $(i,r)\in V$, we set: 
\begin{eqnarray*}
\sigma(i, r) &:=& 2 d_i \cdot |\{1 \leq j \leq l \mid i_j = i \mbox{ and } r_j \geq r \}|,
\\
\tau(i, r) &:=& \sigma(i, r + 2 d_i).
\end{eqnarray*}
Note that if $(i,r) = (i_m,r_m)$ for some $1\le m \le l$, then 
$\sigma(i, r) = \tau(i,r) + 2d_i$.
For  $1\leq m\leq l$ we also set:
\begin{eqnarray*}
\rho_m &:=& r_m - \sigma(i_m, r_m),
\\
m^+ &:=& 
\left\{
\begin{array}{cl}
\min \{k > m \mid i_k = i_m \} & \mbox{if } \{k > m \mid i_k = i_m \}  \neq \emptyset, \\
\infty & \mbox{otherwise},
\end{array}
\right.
\\
\rho_m^+ &:=& 
\left\{
\begin{array}{cl}
\rho_{m^+} & \mbox{if } m^+ \leq l, \\
-\infty & \mbox{if } m^+ = \infty.
\end{array}
\right.
\end{eqnarray*}
We will see below that the pairs $(i_m,\rho_m)$ label the target vertices of (some of)
the vertical down arrows of $\G_{\bi,\br}$ corresponding to the components $i_m$ of the reduced word $\bi$.

For $1\le m \le l$, we introduce certain subsets $V_m$ and $V^-_m$ of the vertex set $V$ as follows.
Note that $\rho_m^+ \leq \rho_m$. We define:
\begin{eqnarray*}
V_m &:=& \{ (i_m, k) \in V \mid \rho_m^+ < k \leq \rho_m \},\\[2mm]
V_m^- &:=& V_m\setminus\{(i_m,\rho_m-2j)\mid 1\leq j\leq d_{i_m}\}. 
\end{eqnarray*}
Note that $|V_m| = \infty$ if and only if $m^+ = \infty$. 
We also set $V_0 := V\setminus (V_1\cup V_2\cup\cdots\cup V_l)$.

\begin{example}
We continue Example~\ref{exa-3-3}. In this case we have (see Figure~\ref{fig:1-3}):
\[
\begin{array}{ll}
 V_1 = \{(2,2), (2,0), (2,-2), (2,-4)\}, & V_2 = \{(1,-2), (1,-4), (1,-6), (1,-8), (1,-10)\},
\\[2mm]
V_3 = \{(2,k)\mid k \le -6\}, & V_4 = \{(1,k)\mid k \le -12\}.  
\end{array}
\]
\end{example}

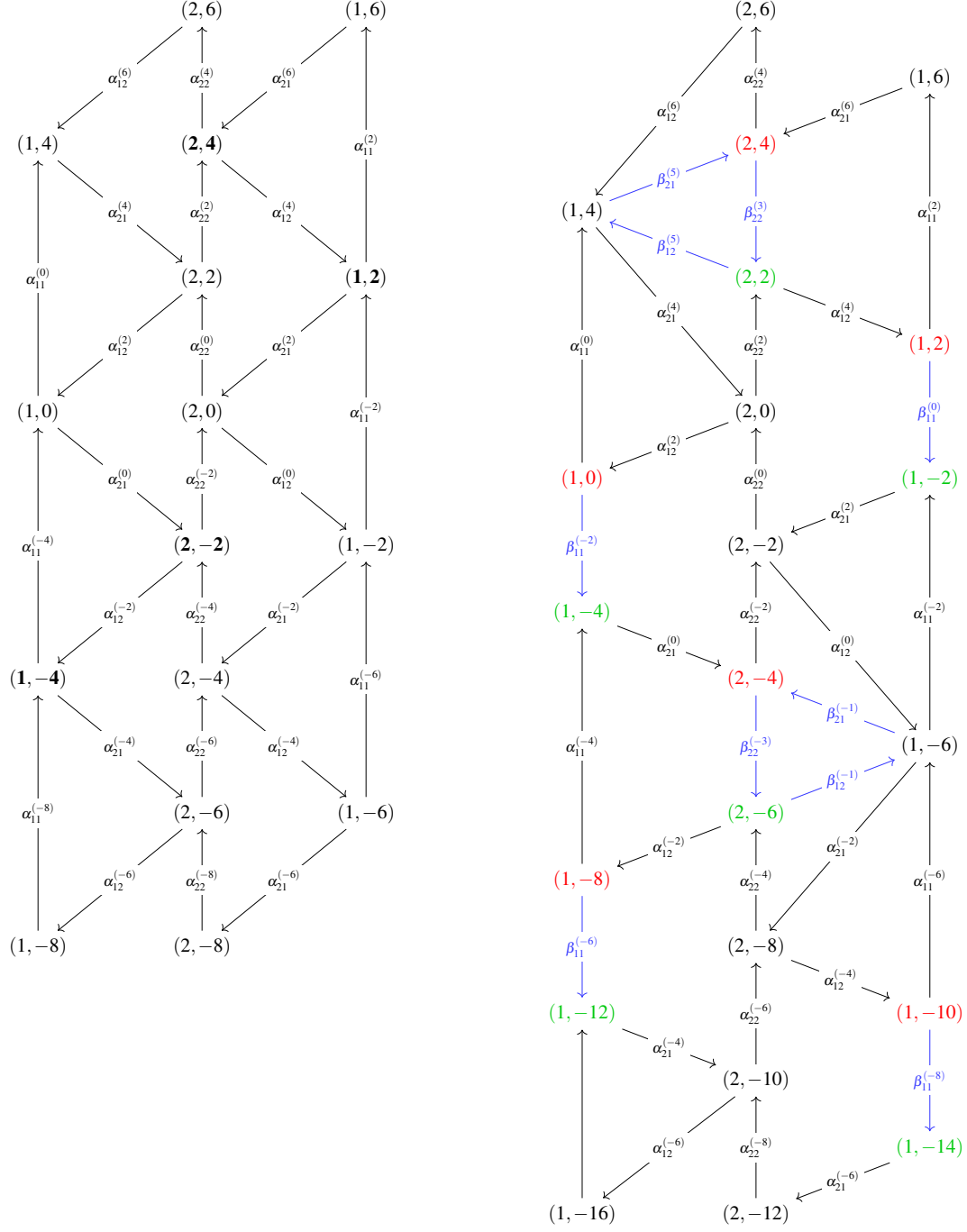
\begin{figure}[ht]
\[\scalebox{0.75}{\begin{tikzcd}[cramped]
	&& {(2,6)} && {(1,6)} &&&&& {(2,6)} && \\
	&&&&&&&&&&& {(1,6)} \\
	{(1,4)} && {\mathbf{(2,4)}} &&&&&&& \textcolor{rgb,255:red,255;green,5;blue,9}{{(2,4)}} \\
	&&&&&&& {(1,4)} \\
	&& {(2,2)} && {\mathbf{(1,2)}} &&&&& \textcolor{rgb,255:red,2;green,202;blue,15}{{(2,2)}} \\
	&&&&&&&&&&& \textcolor{rgb,255:red,255;green,5;blue,9}{{(1,2)}} \\
	{(1,0)} && {(2,0)} &&&&&&& {(2,0)} \\
	&&&&&&& \textcolor{rgb,255:red,255;green,5;blue,9}{{(1,0)}} &&&& \textcolor{rgb,255:red,2;green,202;blue,15}{{(1,-2)}} \\
	&& {\mathbf{(2,-2)}} && {(1,-2)} &&&&& {(2,-2)} \\
	&&&&&&& \textcolor{rgb,255:red,2;green,202;blue,15}{{(1,-4)}} \\
	{\mathbf{(1,-4)}} && {(2,-4)} &&&&&&& \textcolor{rgb,255:red,255;green,5;blue,9}{{(2,-4)}} \\
	&&&&&&&&&&& {(1,-6)} \\
	&& {(2,-6)} && {(1,-6)} &&&&& \textcolor{rgb,255:red,2;green,202;blue,15}{{(2,-6)}} \\
	&&&&&&& \textcolor{rgb,255:red,255;green,5;blue,9}{{(1,-8)}} \\
	{(1,-8)} && {(2,-8)} &&&&&&& {(2,-8)} \\
	&&&&&&& \textcolor{rgb,255:red,2;green,202;blue,15}{{(1,-12)}} &&&& \textcolor{rgb,255:red,255;green,5;blue,9}{{(1,-10)}} \\
	&&&&&&&&& {(2,-10)} \\
	&&&&&&&&&&& \textcolor{rgb,255:red,2;green,202;blue,15}{{(1,-14)}} \\
	&&&&&&& {(1,-16)} && {(2,-12)}
	\arrow["{\alpha_{1 2}^{(6)}}"{description}, from=1-3, to=3-1]
	\arrow["{\alpha_{2 1}^{(6)}}"{description}, from=1-5, to=3-3]
	\arrow["{\alpha_{1 2}^{(6)}}"{description}, from=1-10, to=4-8]
	\arrow["{\alpha_{2 1}^{(6)}}"{description}, from=2-12, to=3-10]
	\arrow["{\alpha_{2 1}^{(4)}}"{description}, from=3-1, to=5-3]
	\arrow["{\alpha_{2 2}^{(4)}}"{description}, from=3-3, to=1-3]
	\arrow["{\alpha_{1 2}^{(4)}}"{description}, from=3-3, to=5-5]
	\arrow["{\alpha_{2 2}^{(4)}}"{description}, from=3-10, to=1-10]
	\arrow["{\beta_{2 2}^{(3)}}"{description}, color={rgb,255:red,51;green,54;blue,255}, from=3-10, to=5-10]
	\arrow["{\beta_{2 1}^{(5)}}"{description}, color={rgb,255:red,51;green,54;blue,255}, from=4-8, to=3-10]
	\arrow["{\alpha_{2 1}^{(4)}}"{description}, from=4-8, to=7-10]
	\arrow["{\alpha_{2 2}^{(2)}}"{description}, from=5-3, to=3-3]
	\arrow["{\alpha_{1 2}^{(2)}}"{description}, from=5-3, to=7-1]
	\arrow["{\alpha_{1 1}^{(2)}}"{description}, from=5-5, to=1-5]
	\arrow["{\alpha_{2 1}^{(2)}}"{description}, from=5-5, to=7-3]
	\arrow["{\beta_{1 2}^{(5)}}"{description}, color={rgb,255:red,51;green,54;blue,255}, from=5-10, to=4-8]
	\arrow["{\alpha_{1 2}^{(4)}}"{description}, from=5-10, to=6-12]
	\arrow["{\alpha_{1 1}^{(2)}}"{description}, from=6-12, to=2-12]
	\arrow["{\beta_{1 1}^{(0)}}"{description}, color={rgb,255:red,51;green,54;blue,255}, from=6-12, to=8-12]
	\arrow["{\alpha_{1 1}^{(0)}}"{description}, from=7-1, to=3-1]
	\arrow["{\alpha_{2 1}^{(0)}}"{description}, from=7-1, to=9-3]
	\arrow["{\alpha_{2 2}^{(0)}}"{description}, from=7-3, to=5-3]
	\arrow["{\alpha_{1 2}^{(0)}}"{description}, from=7-3, to=9-5]
	\arrow["{\alpha_{2 2}^{(2)}}"{description}, from=7-10, to=5-10]
	\arrow["{\alpha_{1 2}^{(2)}}"{description}, from=7-10, to=8-8]
	\arrow["{\alpha_{1 1}^{(0)}}"{description}, from=8-8, to=4-8]
	\arrow["{\beta_{1 1}^{(-2)}}"{description}, color={rgb,255:red,51;green,54;blue,255}, from=8-8, to=10-8]
	\arrow["{\alpha_{2 1}^{(2)}}"{description}, from=8-12, to=9-10]
	\arrow["{\alpha_{2 2}^{(-2)}}"{description}, from=9-3, to=7-3]
	\arrow["{\alpha_{1 2}^{(-2)}}"{description}, from=9-3, to=11-1]
	\arrow["{\alpha_{1 1}^{(-2)}}"{description}, from=9-5, to=5-5]
	\arrow["{\alpha_{2 1}^{(-2)}}"{description}, from=9-5, to=11-3]
	\arrow["{\alpha_{2 2}^{(0)}}"{description}, from=9-10, to=7-10]
	\arrow["{\alpha_{1 2}^{(0)}}"{description}, from=9-10, to=12-12]
	\arrow["{\alpha_{2 1}^{(0)}}"{description}, from=10-8, to=11-10]
	\arrow["{\alpha_{1 1}^{(-4)}}"{description}, from=11-1, to=7-1]
	\arrow["{\alpha_{2 1}^{(-4)}}"{description}, from=11-1, to=13-3]
	\arrow["{\alpha_{2 2}^{(-4)}}"{description}, from=11-3, to=9-3]
	\arrow["{\alpha_{1 2}^{(-4)}}"{description}, from=11-3, to=13-5]
	\arrow["{\alpha_{2 2}^{(-2)}}"{description}, from=11-10, to=9-10]
	\arrow["{\beta_{2 2}^{(-3)}}"{description}, color={rgb,255:red,51;green,54;blue,255}, from=11-10, to=13-10]
	\arrow["{\alpha_{1 1}^{(-2)}}"{description}, from=12-12, to=8-12]
	\arrow["{\beta_{2 1}^{(-1)}}"{description}, color={rgb,255:red,51;green,54;blue,255}, from=12-12, to=11-10]
	\arrow["{\alpha_{2 1}^{(-2)}}"{description}, from=12-12, to=15-10]
	\arrow["{\alpha_{2 2}^{(-6)}}"{description}, from=13-3, to=11-3]
	\arrow["{\alpha_{1 2}^{(-6)}}"{description}, from=13-3, to=15-1]
	\arrow["{\alpha_{1 1}^{(-6)}}"{description}, from=13-5, to=9-5]
	\arrow["{\alpha_{2 1}^{(-6)}}"{description}, from=13-5, to=15-3]
	\arrow["{\beta_{1 2}^{(-1)}}"{description}, color={rgb,255:red,51;green,54;blue,255}, from=13-10, to=12-12]
	\arrow["{\alpha_{1 2}^{(-2)}}"{description}, from=13-10, to=14-8]
	\arrow["{\alpha_{1 1}^{(-4)}}"{description}, from=14-8, to=10-8]
	\arrow["{\beta_{1 1}^{(-6)}}"{description}, color={rgb,255:red,51;green,54;blue,255}, from=14-8, to=16-8]
	\arrow["{\alpha_{1 1}^{(-8)}}"{description}, from=15-1, to=11-1]
	\arrow["{\alpha_{2 2}^{(-8)}}"{description}, from=15-3, to=13-3]
	\arrow["{\alpha_{2 2}^{(-4)}}"{description}, from=15-10, to=13-10]
	\arrow["{\alpha_{1 2}^{(-4)}}"{description}, from=15-10, to=16-12]
	\arrow["{\alpha_{2 1}^{(-4)}}"{description}, from=16-8, to=17-10]
	\arrow["{\alpha_{1 1}^{(-6)}}"{description}, from=16-12, to=12-12]
	\arrow["{\beta_{1 1}^{(-8)}}"{description}, color={rgb,255:red,51;green,54;blue,255}, from=16-12, to=18-12]
	\arrow["{\alpha_{2 2}^{(-6)}}"{description}, from=17-10, to=15-10]
	\arrow["{\alpha_{1 2}^{(-6)}}"{description}, from=17-10, to=19-8]
	\arrow["{\alpha_{21}^{(-6)}}"{description}, from=18-12, to=19-10]
	\arrow[from=19-8, to=16-8]
	\arrow["{\alpha_{2 2}^{(-8)}}"{description}, from=19-10, to=17-10]
\end{tikzcd}
}\]    
    \caption{Basic quiver $\Gamma$ and $\Gamma_{(2,1,2,1), (4,2,-2,-4)}$ in type $B_2$}
    \label{fig:1-3}
\end{figure}

\FloatBarrier

\subsubsection{Arrows of $\G_{\bi,\br}$}

We will now give names to the arrows of $\G_{\bi,\br}$. There will be two types of arrow denoted by the letters
$\alpha_{ji}^{(k)}$ and $\beta_{ji}^{(k)}$. The $\alpha$ arrows are the ones coming from the basic quiver $\G$. 
They will be labelled by the same set of triples $(i,j,k)$ as in $\G$, but note that because of the shifts, their
starting vertex will not necessarily be $(i,k)$. The $\beta$ arrows are the new arrows inserted following Definition~\ref{Def2-7}. More precisely, the set of $\alpha$ arrows of $\G_{\bi,\br}$ is going to be
\[
 \left\{ \alpha_{ji}^{(k)}\mid ((i, k),j) \in V\times I \mbox{ and } b_{ij}\neq 0 \right\},
\]
the set of vertical $\beta$ arrows will be
\[
 \left\{ \beta_{i_p i_p}^{ (r_p- d_{i_p}-2k)} \mid 1 \leq p \leq l,\ {0\leq k\leq d_{i_p}-1} \right\}
\]
and the set of oblique $\beta$ arrows will be
\[
\begin{array}{ll}  
&\hspace{-3cm} \left\{ \beta_{ji_p}^{ (r_p+1)},\ \beta_{i_pj}^{(r_p+1)}\mid 
j\in I,\  1\leq p \leq l, c_{i_p,j}=-2  \right\}\ \\[.5ex]
 \bigcup\ &\left\{ \beta_{j i_p}^{(r_p+h)}, \beta_{i_p j}^{(r_p+h+1)} \mid j\in I,\ 1\leq p\leq l,\ h\in\{0,2\},\ c_{i_p,j}=-3  \right\}.
\end{array}
\]
We now define the source and target of these arrows. For an arrow $\gamma$, we denote by $s(\gamma)$ 
its source and $t(\gamma)$ its target. It will be convenient to abbreviate $\sigma'(i,r):=r-\sigma(i,r)$ and $\tau'(i,r)= r-\tau(i,r)$.
For $(i,r)\in V$, and for $j\in I$ with $b_{ij}<0$, we put
for the $\alpha$-arrows:
\[
\begin{array}{ll}
s\left(\alp_{ii}^{(r)}\right) :=(i,  \tau'(i,r)),\quad  
&t\left(\alp_{ii}^{(r)}\right) := (i,  b_{ii}\!+\!\tau'(i,r)), \\[3mm]   
s\left(\alp_{ji}^{(r)}\right) :=(i, \sig'(i,r)),\quad           
&t\left(\alp_{ji}^{(r)}\right) :=(j, \tau'(j,r+b_{ij})).
\end{array}
\]
Next, for $1\leq p\leq l$, we deal with the vertical  $\beta$-arrows:
\[
s\left(\bet_{i_p i_p}^{(r_p-d_{i_p}-2k)}\right):= (i_p, \rho_p + 2d_{i_p}-2k),\quad  
t\left( \bet_{i_p i_p}^{(r_p-d_{i_p}-2k)}\right):=(i_p, \rho_p-2k),\quad( 0\leq k\leq d_{i_p}-1).
\]
Finally, for $j\in I$ and $1\leq p\leq l$ we  deal with the
oblique $\beta$-arrows. Those arrows occur only when 
$c_{i_p j}\leq-2$. Recall, that in this case $d_{i_p}=1$ by Lemma~\ref{lem:CartanM}. In case $c_{i_p,j}=-3$ the additional parameter $h\in\{0,2\}$ is needed:
\[
\begin{array}{lll}

s\left(\bet_{j i_p}^{(r_p+1)}\right):= (i_p, \rho_p), \quad
&t\left(\bet_{j i_p}^{(r_p+1)}\right):=(j, \tau'(j,r_p)), &( c_{i_pj}=-2),\\[.5ex]
s\left(\bet_{i_p j}^{(r_p+1)}\right):= (j, \tau'(j,r_p)), \quad 
&t\left(\bet_{i_p j}^{(r_p+1)}\right):= (i_p, \rho_p + 2),  &(c_{i_pj}=-2),\\[.5ex]
s\left(\beta_{j i_p}^{(r_p+h)}\right) := (i_p, \rho_p), \quad
& t\left(\beta_{j i_p}^{(r_p+h)}\right) := (j, \tau'(j, r_p+h-1)), & (c_{i_p,j}=-3),\\[.5ex]
s\left(\beta_{i_p j}^{(r_p+h+1)}\right) :=  (j, \tau'(j, r_p+h-1))  , \quad
& t\left(\beta_{i_p j}^{(r_p+h+1)}\right) := (i_p, \rho_p+2), & (c_{i_p,j}=-3).\\
\end{array}
\]
These definitions are illustrated in Figures \ref{fig:1-1ab}, \ref{fig:1-1c} and~\ref{fig:1-1d}.
Note that, for $c_{i_pj}=-2$ we have by our definitions: 
\[
\begin{array}{ll}
s\left(\bet_{ji_p}^{(r_p+1)}\right) = t\left(\bet_{i_pi_p}^{(r_p-1)}\right),\qquad &
t\left(\bet_{ji_p}^{(r_p+1)}\right)= s\left(\alp_{jj}^{(r_p)}\right),\\[3mm]
s\left(\bet_{i_pj}^{(r_p+1)}\right) = t\left(\bet_{ji_p}^{(r_p+1)}\right), &
t\left(\bet_{i_pj}^{(r_p+1)}\right) = s\left(\bet_{i_pi_p}^{(r_p-1)}\right).
\end{array}
\]
Similarly, for $c_{i_p,j}=-3$ and $h\in\{0,2\}$ we have:
\[
\begin{array}{ll}
s\left(\bet_{ji_p}^{(r_p+h)}\right) = t\left(\bet_{i_pi_p}^{(r_p-1)}\right),\qquad &
t\left(\bet_{ji_p}^{(r_p+h)}\right)= s\left(\alp_{jj}^{(r_p+h-1)}\right),\\[3mm]
s\left(\bet_{i_pj}^{(r_p+h+1)}\right) = t\left(\bet_{ji_p}^{(r_p+h)}\right), &
t\left(\bet_{i_pj}^{(r_p+h+1)}\right) = s\left(\bet_{i_pi_p}^{(r_p-1)}\right).
\end{array}
\]
In particular, for $c_{i_p,j}=-2$ the oblique arrows
$\beta_{j i_p}^{r_p+1}$ and $\beta_{i_p j}^{(r_p+1)}$ 
(\resp for $c_{i_p,j}=-3$ and $h\in\{0,2\}$ the oblique arrows
$\beta_{j i_p}^{r_p+h}$ and $\beta_{i_p j}^{(r_p+h+1)}$) together
with the vertical arrow $\beta_{i_p i_p}^{(r_p-1)}$ form a 3-cycle.

\begin{Prop} \label{prop:Gam-basic}
Let $j\in I$ and $1\le p \le l$ be such that $b_{j {i_p}}<0$. Set $m=m^-(j,p)$. 
\begin{itemize}
\item[(a)]
We have $s\left(\alp_{j i_p}^{(r_p-2k)}\right)=(i_p,\rho_p-2k)
=t\left(\bet_{i_pi_p}^{(r_p-d_{i_p}-2k)}\right)=t\left(\alp_{i_pi_p}^{(r_p-2d_{i_p}-2k)}\right)$
and\\ 
$t\left(\alp_{ji_p}^{(r_p-2k)}\right)=(j,\tau'(r_p)-2k)\in V_m^-$ for $0\leq k\leq d_{i_p}-1$.
\item[(b)]
If $c_{i_pj}=-2$ we have 
$t\left(\bet_{ji_p}^{(r_p+1)}\right)= s\left( \alp_{j j}^{(r_p)}\right)  =s\left( \alp_{i_p j}^{(r_p)}\right)\in V_m^-$, and
$t\left(\alp_{jj}^{(r_p)}\right)\in V_m$.
\item[(c)] If $c_{i_pj}=-3$ we have 
$t\left(\bet_{ji_p}^{(r_p+h)}\right)= s\left(\alp_{j j}^{(r_p+h-1)}\right)=s\left(\alp_{i_p j}^{(r_p+h-1)}\right)\in V_m^-$, and\\
${t\left(\alp_{jj}^{(r_p+h-1)}\right)\in V_m}$ for $h\in\{0,2\}$.
\end{itemize}
\end{Prop}

\begin{figure}[t]
\[\scalebox{.9}{
\begin{tikzcd}[cramped]
	{\scriptstyle{(i_p,\rho_p+4d_{i_p}-2k)}} &&& \\
	&&& {\scriptstyle{(j,\sigma'(j,r_p-b_{i_p,j}-2k))}} \\
	\\
	\textcolor{rgb,255:red,255;green,0;blue,4}{{\scriptstyle{(i_p,\rho_p+2d_{i_p}-2k)}}} \\
	&&& \bullet \\
	&&& \bullet \\
	\textcolor{rgb,255:red,0;green,199;blue,7}{{\scriptstyle{(i_p,\rho_p-2k)}}} \\
	\\
	&&& {\scriptstyle{(j,\tau'(j,r_p+b_{i_p,j}-2k))}} \\
	{\scriptstyle{(i_p,\rho_p-2d_{i_p}-2k)}}
	\arrow["{\alpha_{i_p,j}^{(r_p-b_{i_pj}-2k)}}"{description}, from=2-4, to=4-1]
	\arrow["{\alpha_{i_p i_p}^{(r_p-2k)}}"{description}, from=4-1, to=1-1]
	\arrow["{\beta_{i_p i_p}^{(r_p-d_{i_p}-2k)}}"{description}, color={rgb,255:red,0;green,17;blue,255}, from=4-1, to=7-1]
	\arrow["{\alpha_{j,j}^{(r_p-b_{i_pj}-2k)}}"{description}, from=5-4, to=2-4]
	\arrow["{\alpha_{j i_p}^{(r_p-2k)}}"{description}, from=7-1, to=9-4]
	\arrow["{\alpha_{j j}^{(r_p+b_{ji_p}-2k)}}"{description}, from=9-4, to=6-4]
	\arrow["{\alpha_{i_pi_p}^{(r_p-2d_{i_p}-2k)}}"{description}, from=10-1, to=7-1]
\end{tikzcd}
}\]
    \caption{ Local configuration around $\bet_{i_pi_p}^{(r_p-d_{i_p}-2k)}$ 
  for $c_{i_p,j}=-1$ and $0\leq k\leq d_{i_p}-1$}
    \label{fig:1-1ab}
\end{figure}
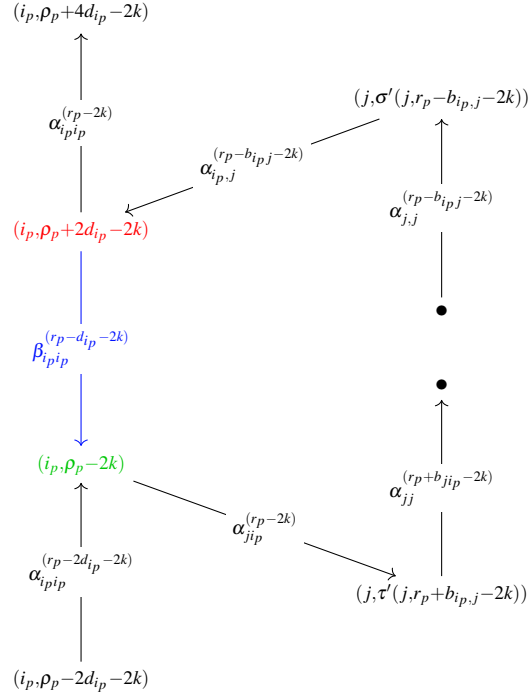

\begin{figure}[ht]
\[\scalebox{1.0}{
\begin{tikzcd}[cramped,sep=scriptsize]
	\bullet &&&& \\
	&& {\scriptstyle{(i_p,\rho_p+4)}} \\
	&&&& {\scriptstyle(j,\sigma'(j,r_p+2))} \\
	&& \textcolor{rgb,255:red,255;green,0;blue,4}{{\scriptstyle{(i_p,\rho_p+2)}}} \\
	{\scriptstyle{(j,\sigma'(j,r_p))}} \\
	&& \textcolor{rgb,255:red,0;green,204;blue,3}{{\scriptstyle{(i_p,\rho_p)}}} \\
	&&&& {\scriptstyle(j,\tau'(j,r_p-2))} \\
	&& {\scriptstyle{(i_p,\rho_p-2)}} \\
	\bullet
	\arrow[from=1-1, to=2-3]
	\arrow["{\alpha_{j i_p}^{(r_p+2)}}"{description}, from=2-3, to=5-1]
	\arrow["{\alpha_{i_p j}^{(r_p+2)}}"{description}, from=3-5, to=4-3]
	\arrow["{\alpha_{i_p i_p}^{(r_p)}}"{description}, from=4-3, to=2-3]
	\arrow["{\beta_{i_p i_p}^{(r_p-1)}}"{description}, color={rgb,255:red,4;green,0;blue,255}, from=4-3, to=6-3]
	\arrow["{\alpha_{j j}^{(r_p)}}"{description}, from=5-1, to=1-1]
	\arrow["{\beta_{i_p j}^{(r_p+1)}}"{description}, color={rgb,255:red,4;green,0;blue,255}, from=5-1, to=4-3]
	\arrow["{\alpha_{i_p j}^{(r_p)}}"{description}, from=5-1, to=8-3]
	\arrow["{\beta_{j i_p}^{(r_p+1)}}"{description}, color={rgb,255:red,4;green,0;blue,255}, from=6-3, to=5-1]
	\arrow["{\alpha_{j i_p}^{(r_p)}}"{description}, from=6-3, to=7-5]
	\arrow["{\alpha_{j j}^{(r_p-2)}}"{description}, from=7-5, to=3-5]
	\arrow["{\alpha_{i_p i_p}^{(r_p-2)}}"{description}, from=8-3, to=6-3]
	\arrow[from=8-3, to=9-1]
	\arrow["{\alpha_{jj}^{(r_p-4)}}"{description}, from=9-1, to=5-1]
\end{tikzcd}
}\]
\caption{Local configuration around $\bet_{i_pi_p}^{(r_p-1)}$ in case $c_{i_p j}=-2$}
    \label{fig:1-1c}
\end{figure}

\begin{figure}[ht]
\[\scalebox{0.9}{\begin{tikzcd}[cramped]
	&&&& \bullet & \\
	&& {\scriptstyle(i_p,\rho_p+2)} \\
	\bullet \\
	&& {\scriptstyle(i_p,\rho_p+4)} &&& {\scriptstyle(j,\sigma'(r_p+3))} \\
	\\
	&& \textcolor{rgb,255:red,255;green,0;blue,8}{{\scriptstyle(i_p,\rho_p+2)}} && {\scriptstyle(j,\sigma'(j,r_p+1))} \\
	\\
	{\scriptstyle(j,\sigma'(j,r_p-1))} && \textcolor{rgb,255:red,0;green,204;blue,14}{{\scriptstyle(i_p,\rho_p)}} \\
	\\
	&& {\scriptstyle(i_p,\rho_p-2)} &&& {\scriptstyle(j,\tau'(r_p-3))} \\
	&&&& \bullet \\
	&& {\scriptstyle(i_p,\rho_p-4)} \\
	\bullet
	\arrow[from=1-5, to=2-3]
	\arrow["{\alpha_{j i_p}^{(r_p+4)}}"{description}, from=2-3, to=6-5]
	\arrow[from=3-1, to=4-3]
	\arrow["{\alpha_{i_p i_p}^{(r_p+2)}}"{description}, from=4-3, to=2-3]
	\arrow["{\alpha_{j i_p}^{(r_p+2)}}"{description, pos=0.4}, from=4-3, to=8-1]
	\arrow["{\alpha_{i_pj}^{(r_p+3)}}"{description, pos=0.2}, from=4-6, to=6-3]
	\arrow["{\alpha_{i_p i_p}^{(r_p)}}"{description}, from=6-3, to=4-3]
	\arrow["{\beta_{i_pi_p}^{(r_p-1)}}"{description}, color={rgb,255:red,5;green,9;blue,255}, from=6-3, to=8-3]
	\arrow["{\alpha_{jj}^{(r_p+1)}}"{description}, from=6-5, to=1-5]
	\arrow["{\beta_{i_p j}^{(r_p+3)}}"{description}, color={rgb,255:red,5;green,9;blue,255}, from=6-5, to=6-3]
	\arrow["{\alpha_{i_p j}^{(r_p+1)}}"{description}, from=6-5, to=10-3]
	\arrow["{\alpha_{jj}^{(r_p-1)}}"{description}, from=8-1, to=3-1]
	\arrow["{\beta_{i_p j}^{(r_p+1)}}"{description, pos=0.6}, color={rgb,255:red,5;green,9;blue,255}, from=8-1, to=6-3]
	\arrow["{\alpha_{i_p j}^{(r_p-1)}}"{description, pos=0.6}, from=8-1, to=12-3]
	\arrow["{\beta_{j i_p}^{(r_p+2)}}"{description}, color={rgb,255:red,5;green,9;blue,255}, from=8-3, to=6-5]
	\arrow["{\beta_{j i_p}^{(r_p)}}"{description}, color={rgb,255:red,5;green,9;blue,255}, from=8-3, to=8-1]
	\arrow["{\alpha_{ji_p}^{(r_p)}}"{description, pos=0.8}, from=8-3, to=10-6]
	\arrow[from=10-3, to=8-3]
	\arrow[from=10-3, to=11-5]
	\arrow["{\alpha_{jj}^{(r_p-3)}}"{description}, from=10-6, to=4-6]
	\arrow["{\alpha_{jj}^{(r_p-5)}}"{description, pos=0.6}, from=11-5, to=6-5]
	\arrow[from=12-3, to=10-3]
	\arrow[from=12-3, to=13-1]
	\arrow["{\alpha_{jj}^{(r_p-7)}}"{description}, from=13-1, to=8-1]
\end{tikzcd}
}\]
\caption{ Local configuration around $\bet_{i_pi_p}^{(r_p-1)}$ in case $c_{i_p j}=-3$}
    \label{fig:1-1d}
\end{figure}
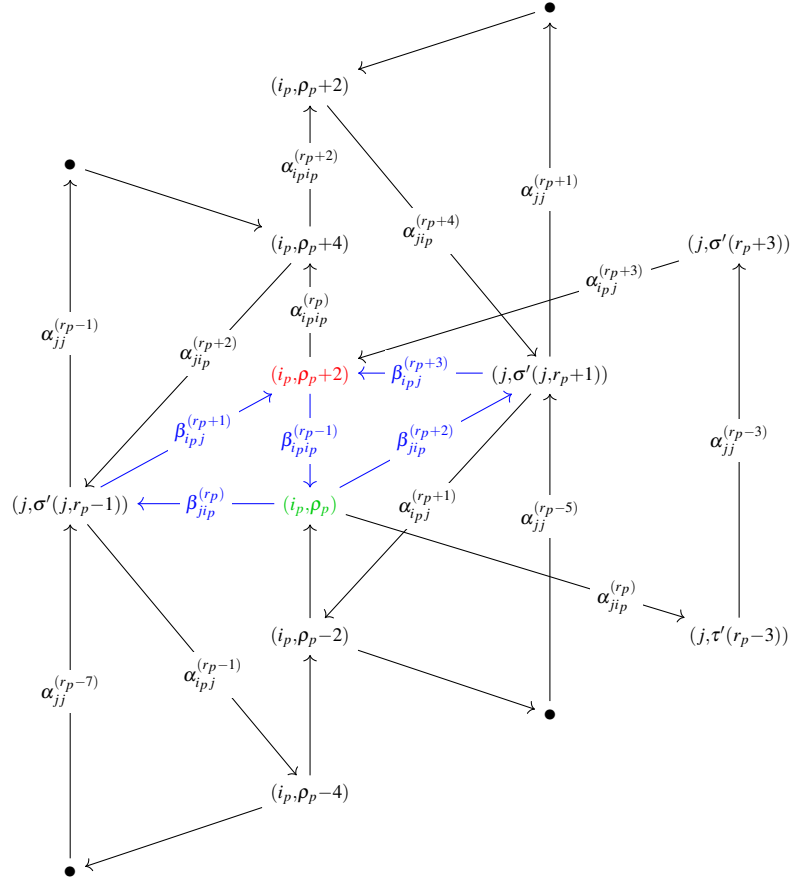

\begin{proof}
These are elementary consequences of the definitions and of the following remark:

\smallskip
Let $(j,r)\in V$ and suppose that $r_{m^+}< r+2d_j< r_m$ 
for some  $m\in\{1,2,\ldots, l\}$ with $i_m=j$. 
Then we have
$(j, r-\tau(j,m))\in V_m^-$.

\smallskip
We omit the details.
\cqfd
\end{proof}

\begin{Cor}\label{cor:sgam}
Let $1\le p \le l$.
\begin{itemize} 
 \item[(a)] The set of arrows of $\G_{\bi,\br}$ with source $(i_p,\rho_p)$ is equal to
\[
 \left\{\alp_{ji_p}^{(r_p)}\mid b_{i_pj}<0\right\}\ \bigcup\ 
\left\{\bet_{ji_p}^{(r_p+1)}\mid c_{i_pj}=-2\right\}\ \bigcup\ 
 \left\{\bet_{ji_p}^{(r_p)}, \bet_{ji_p}^{(r_p+2)} \mid c_{i_pj}=-3\right\}
\]
 \item[(b)] If $d_{i_p}\geq 2$, for $1\leq k\leq d_{i_p}-1$ the set of arrows with source $(i_p,\rho_p-2k)$ is equal to
\[
\left\{\alp_{ji_p}^{(r_p-2k)}\mid b_{i_pj}<0\right\}.
\]
 \item[(c)] Write $B(p) :=  \left\{ \beta_{i_p i_p}^{(r_p- d_{i_p}-2k)} \mid 0\leq k\leq d_{i_p}-1 \right\}$.
Then for $\bet\in B(p)$, every arrow $\gamma$ with source $s(\gamma) = t(\bet)$ has its target
$t(\gamma)$ in $V_m$ for some $m<p$.
\end{itemize}
\end{Cor}

\begin{proof}
This follows from Proposition~\ref{prop:Gam-basic} (a)--(c).
\cqfd
\end{proof}

\begin{example}
We continue Example~\ref{exa-3-3} and illustrate our main notation in this case.
Recall that $\bi = (2,1,2,1)$ is a reduced word for $w_0$ in type $B_2$, and we take 
$\br = (4,2,-2,4)$. We have $d_1 = 2$ and $d_2 = 1$. 
The functions $\sig$ and $\tau$ take the following values:
\[
\begin{array}{c|cccccccc}
r & \sig (1, r) & \tau (1, r) & \sig(2, r) & \tau (2, r)  \\
\hline
  6 & 0 & 0 & 0 & 0  \\
  4 & 0 & 0 & 2 & 0\\
  2 & 4 & 0 & 2 & 2  \\
  0 & 4 & 0 & 2 & 2  \\
 -2 & 4 & 4 & 4 & 2  \\
 -4 & 8 & 4 & 4 & 4 \\
 -6 & 8 & 4 & 4 & 4 \\
 -8 & 8 & 8 & 4 & 4 \\
-10 & 8 & 8 & 4 & 4 \\
\end{array}
\]
Thus we obtain:
\[
\begin{array}{c|rrrr}
p &   1 & 2 & 3 & 4\\\hline
i_p & 2 & 1 &  2 & 1\\
r_p & 4 & 2 & -2 & -4\\
\sig(i_p,r_p) & 2 & 4 & 4 & 8\\
\rho_p & 2 &-2 & -6 & -12
\end{array}
\]
We can now check on the right quiver of Figure~\ref{fig:1-3} that for example
\[
s\left(\alp_{ii}^{(r)}\right) :=(i, r - \tau(i,r)),\qquad
t\left(\beta^{(r_p-d_p)}_{i_pi_p}\right) = (p,\rho_p).
\]
\end{example}

\begin{figure}[ht]
\vspace*{-1.0cm}
\[\scalebox{0.7}{\input{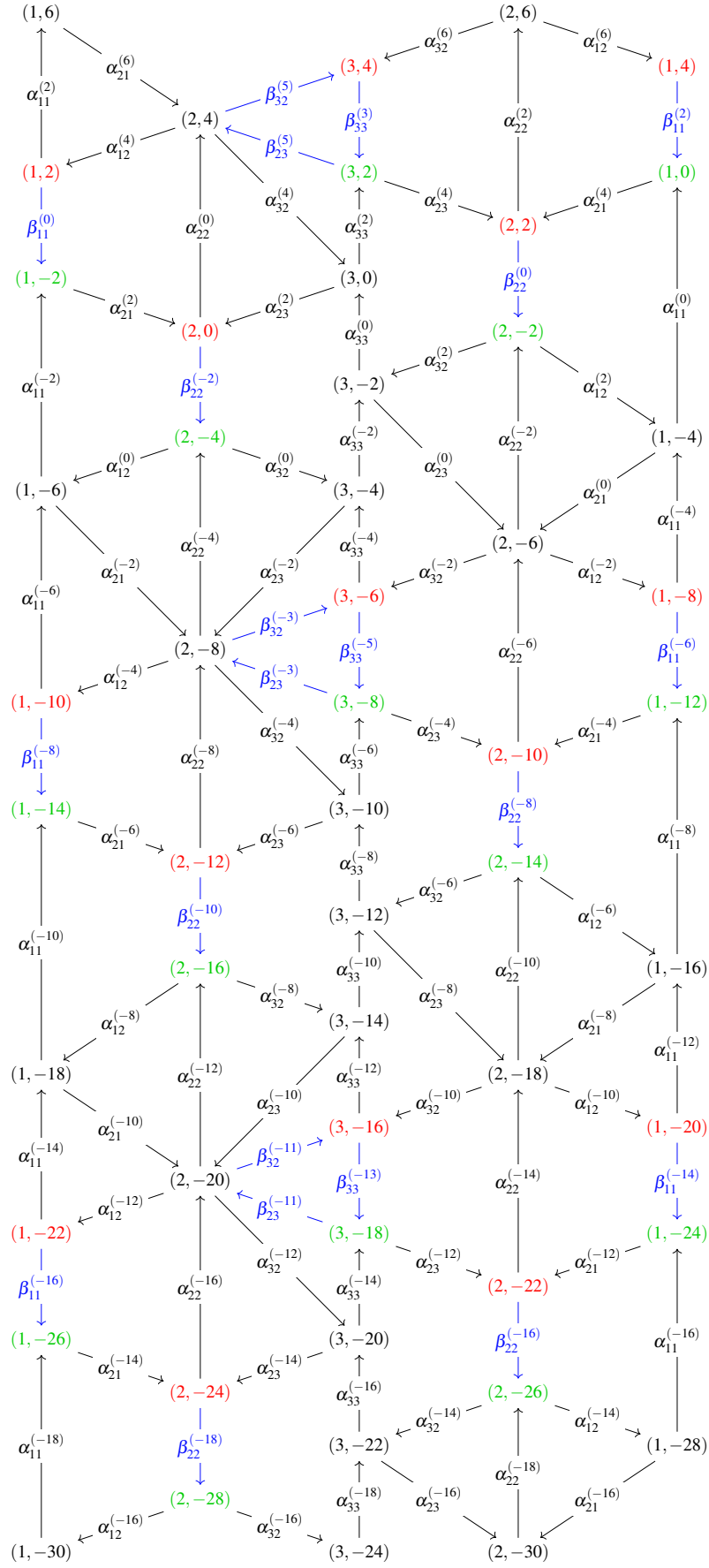}}\]    
    \caption{$\Gamma_{(3,1,2,3,1,2,3,1,2),(4,4,2,-4,-4,-6,-12,-12, -14)}$ in type $B_3$}
    \label{fig:1-2}
\end{figure}

\begin{example}\label{ex-3-7}
In type $B_3$ we take the reduced word $\bi = (3,1,2,3,1,2,3,1,2)$ for $w_0$, 
and we choose $\br = (4,4,2,-4,-4,-6,-12,-12, -14).$
The corresponding quiver $\G_{\bi,\br}$ is displayed in Figure~\ref{fig:1-2}.  
\end{example}

\FloatBarrier

\subsection{Braid group generators}

In this section we recall from \cite{FH3} the formulas for 
an action of the braid group associated with the Cartan matrix $C$. 
We introduce some twisted braid group generators depending on the datum $(\bi, \br)$ of a reduced word $\bi$
for $w_0$, and a tuple $\br$ adapted to $\bi$ as in Definition~\ref{def:adapt}. 
We then describe the matrices of these twisted generators in terms of $\Gamma_{\bi,\br}$
and show that they exhibit a certain sign-coherence property.

\subsubsection{Twisting datum}\label{sub-sec3-2-1}

We fix $(\bi, \br)$ as above, and we denote by $l$ the length of $w_0$.
The twisted generators will involve a family of integers $k_i^{(p)}\ (i\in I, 0\le p \le l)$
defined recursively as follows.

\begin{Def}
Let $\bk^{(p)}=\left(k_i^{(p)}\right)_{i\in I}\in\Z^I$ be the integer 
vectors defined by the initial condition $\bk^{(0)}:=0$ and the recurrence relation:
\[
\bk^{(p+1)} =\bk^{(p)}- 2d_{i_{p+1}} \be_{i_{p+1}},\qquad (0\leq p<l),
\]
where $\be_i\in\Z^I$ denotes the $i$-th unit vector.
\end{Def}
By construction, we obviously have $k^{(p)}_j=k^{(p-1)}_j$ if $j\neq i_p$, and
\begin{equation} \label{eq:2.1}
    k^{(p)}_j
\left\{
\begin{array}{ll}
-\sig(i_p,r_p)         &\mbox{if } j=i_p,\\[2mm]
-\sig\left(j, r_{m^-(j,p)}\right) &\mbox{if } j\neq i_p,
\end{array}                
\right.
\qquad (1\le p\le l).   
\end{equation}

The following lemma expresses the coordinates of the source and target of certain arrows of $\G_{\bi,\br}$
in terms of the integers $k_i^{(p)}$.
\begin{Lem}\label{lem-3-9}
Let $j\in I$ and $1\le p\le l$ such that $b_{ji_p}<0$. 
Then, for $0\leq k\leq d_{i_p}-1$ we have:
\begin{equation}
s\left(\alp^{(r_p-2k)}_{ji_p}\right)=\left(i_p,\, r_p+k^{(p)}_{i_p}-2k\right), 
\ \
t\left(\alp^{(r_p-2k)}_{ji_p}\right)=\left(j,\, b_{ji_p}+r_p+k^{(p-1)}_j-2k\right). \label{eq:2.2a}
\end{equation}
Moreover, if $c_{i_p j}=-2$, we have
\begin{equation}
s\left(\bet^{(r_p+1)}_{ji_p}\right) =\left(i_p,\,r_p+k^{(p)}_{i_p}\right),\
t\left(\bet^{(r_p+1)}_{ji_p}\right) =\left(j,\,  b_{ji_p}+r_p+2+k^{(p-1)}_j\right)
= \left(j,\, r_p+k^{(p-1)}_j\right), 
 \label{eq:2.2b}
\end{equation} 
and if $c_{i_p,j}=-3$,  for $h\in\{0,2\}$  we have
\begin{equation}
s\left(\bet^{(r_p+h)}_{ji_p}\right) =\left(i_p,\,r_p+k^{(p)}_{i_p}\right),
\ \ 
t\left(\bet^{(r_p+h)}_{ji_p}\right) 
= \left(j,\, r_p-1+h+k^{(p-1)}_j\right). 
 \label{eq:2.2g}
\end{equation} 
\end{Lem}
\begin{proof}
This follows from Equation (\ref{eq:2.1}) and Proposition~\ref{prop:Gam-basic}. We omit the details.
\cqfd
\end{proof}

For an arrow $\gamma$ of $\G_{\bi,\br}$ with source $(i,a)$ and target $(j,b)$, we define the 
degree of $\gamma$ by
\[
\deg(\gamma):= b-a.
\]
Lemma~\ref{lem-3-9} implies the following
\begin{Cor}\label{Cor3-10}
For $j\in I$ and $1\le p\le l$ such that $b_{ji_p}<0$. 
Then, for $0\leq k\leq d_{i_p}-1$ we have
\[
 \deg\left(\alp_{ji_p}^{(r_p-2k)}\right)=b_{ji_p}+k_j^{(p-1)}-k_{i_p}^{(p)}.
\]
Moreover,  if $c_{i_pj} = -2$ we have
\[
 \deg\left(\bet_{ji_p}^{(r_p+1)}\right)=k_j^{(p-1)}-k_{i_p}^{(p)}=\deg\left(\alp_{ji_p}^{(r_p)}\right)+2,
\]
and if $c_{i_p j}=-3$, for $h\in\{0,2\}$  we have 
\[
 \deg\left(\bet_{ji_p}^{(r_p+h)}\right)=k_j^{(p-1)}-1+h-k_{i_p}^{(p)}=\deg\left(\alp_{ji_p}^{(r_p+2)}\right)+h+2.
\]
\cqfd
\end{Cor}

\begin{example}\label{exa-3-11}
We continue Example~\ref{ex-3-7}. We calculate
\[
\begin{array}{c|cccc}
p & i_p & \bk^{(p)} & r_p & r_p+k_{i_p}^{(p)}=\rho_p\\\hline 
0 &     & (0,0,0) \\
1 &  3  & (0,0,-2)  & 4   & 2  \\
2 &  1  & (-4,0,-2) & 4   & 0  \\
3 &  2  & (-4,-4,-2)  & 2  &-2\\
4 &  3  & (-4,-4,-4)  &-4  &-8\\
5 &  1  & (-8,-4,-4)  &-4  &-12\\
6 &  2  & (-8,-8,-4)  &-6  &-14\\
7 &  3  & (-8,-8,-6)  &-12 &-18\\
8 &  1  & (-12,-8,-6) &-12 &-24\\
9 &  2  & (-12,-12,-6)&-14 &-26\\
\end{array}
\]
Using this table, we can calculate the degrees of the oblique arrows in Figure~\ref{fig:1-2}.
For instance:
\begin{eqnarray*}{}
\deg\left(\alp_{32}^{(2)}\right) &=& b_{32}+k_3^{(2)}-k_{2}^{(3)}=0,\\
\deg\left(\alp_{12}^{(2)}\right) &=& b_{12}+k_1^{(2)}-k_{2}^{(3)} =+2,\\
\deg\left(\alp_{23}^{(-12)}\right) &=&b_{23}+k_2^{(6)}-k_{3}^{(6)} =-4,\\
\deg\left(\bet_{23}^{(-11)}\right) &=& k_2^{(6)}-k_{3}^{(7)} =-2,\\
\deg\left(\alp_{21}^{(-12)}\right) &=&b_{21}+k_2^{(7)}-k_{1}^{(8)} =+2.\\
\end{eqnarray*}
\end{example}

\subsubsection{Twisted braid group action}\label{tbga}

Let $z$ be a formal variable. For $i\in I$, define the endomorphism 
$t_i\in \End_{\Z[z^{\pm 1}]}\left(\Z[z^{\pm 1}]^I\right)$
by the following formulas for its matrix entries:
\begin{equation}
    (t_i)_{jk}=
\left\{    
\begin{array}{ll}
 \del_{jk} &\mbox{if } k\neq i,\\
-z^{-2d_i}  &\mbox{if } j=k=i,\\
 z^{-d_i}   &\mbox{if } k=i \mbox{ and } c_{ij}=-1,\\
 1+ z^{-2}   &\mbox{if } k=i \mbox{ and } c_{ij}=-2,\\
z+z^{-1}+z^{-3} &\mbox{if } k=i \mbox{ and } c_{ij}=-3,\\
 0          &\mbox{if } k=i \mbox{ and } c_{ij}=0,
\end{array} 
\right.
\qquad (j,k\in I).
\label{eq:t'i}\\
\end{equation}
It is well known that the $t_i$'s give a representation of the braid group attached to $W$
(see for instance \cite{FH3}; in type $A$, this is a version of the classical Burau representation). 

\begin{example}\label{exa3-12}
In type $B_3$ we have $d_1 = d_2 = 2$, $d_3 = 1$, and
\[
t_1=
\left(
\begin{array}{ccc}
 -z^{-4} & 0 & 0\\
 z^{-2}& 1 & 0\\
 0       & 0 & 1\\
\end{array}
\right),
 \quad
t_2=
\left(
\begin{array}{ccc}
 1 & z^{-2} & 0\\
 0& -z^{-4} & 0\\
 0       & z^{-2} & 1\\
\end{array}
\right),
\quad
t_3=
\left(
\begin{array}{ccc}
 1 & 0 & 0\\
 0 & 1 & 1 + z^{-2}\\
 0 & 0 & -z^{-2}\\
\end{array}
\right).
\]
\end{example}

Using the twisting datum $(\bk^{(p)})_{1\le p\le l}$ of \S~\ref{sub-sec3-2-1}, we now form the following family of diagonal matrices
\[
Z_p:=\diag\left(z^{k_i^{(p)}} \mid i\in I\right), \qquad (1\le p \le l),
\]
and defined the \emph{twisted braid group generators}
\[
M_p:= Z_{p-1}t_{i_p} Z_p^{-1},\qquad  (1\le p \le l).
\]

\begin{example}
We continue Example~\ref{exa3-12}. Using the table of Example~\ref{exa-3-11}, we calculate for example
\[
M_3= 
\left(
\begin{array}{ccc}
 1 & z^{-2} & 0\\
 0& - 1 & 0\\
 0       & 1 & 1\\
\end{array}
\right),
 \quad
M_7=
\left(
\begin{array}{ccc}
 1 & 0 & 0\\
 0 & 1 & z^{-2}+z^{-4}\\
 0 & 0 & -1\\
\end{array}
\right),
\quad
M_8=
\left(
\begin{array}{ccc}
 -1 & 0 & 0\\
 z^2 & 1 & 0\\
 0 & 0 & 1\\
\end{array}
\right).
\]
\end{example}

\begin{Prop}\label{prp:Msgn-coh}
For $0\le a < b \le l$, define 
\[
M_{ab} := M_{a+1} M_{a+2} \cdots M_b.
\]
Then the rows of $M_{ab}$ 
are sign-coherent, in the sense that all coefficients of all Laurent polynomials in a given 
row of $M_{ab}$ have the same sign. Moreover, if $a \ge 1$ these signs are all positive on row~$i_a$.
\end{Prop}

\begin{proof}
Clearly, we have
\[
M_{ab}=Z_a\, t_{i_{a+1}} t_{i_{a+2}}\cdots t_{i_b}\, Z_b^{-1}.
\]
Since $Z_a$ and $Z_b^{-1}$ are diagonal matrices with positive diagonal entries, it is sufficient to verify the claim for 
$T_{ab}:=t_{i_{a+1}}t_{i_{a+2}}\cdots t_{i_b}$. 
Now, $(i_{a+1},i_{a+2},\ldots, i_b)$ is a reduced expression for
$w_{a,b}:=s_{i_{a+1}}s_{i_{a+2}}\cdots s_{i_b}\in W$. So, this follows
from Theorem \ref{scoherence} below (after identification of the additive and multiplicative notations as explained in Section \ref{scs}). \cqfd
\end{proof}

The following proposition describes the twisted braid group generators $M_p$ in terms of the quiver $\G_{\bi,\br}$.

\begin{Prop}\label{prp:2.7}
The entries of $M_p \ (1\le p \le l)$ are given by
\begin{equation} \label{eq:MpQ}
(M_p)_{jk}=
\left\{
\begin{array}{ll}
  \del_{jk} &\mbox{if } k\neq i_p\\[2mm]
-\del_{jk}+ \displaystyle\sum z^{\deg(\gam)} & \mbox{if } k=i_p,
\end{array}
\right.
\qquad (j,k\in I),
\end{equation}
where the sum runs through all arrows $\gam$ with source $(i_p,\rho_p)$ and target of the form $(j,r)\in V$.
\end{Prop}

\begin{proof}
Let us denote by $\widetilde{M}_p$ the matrix that is described by the right-hand side of~(\ref{eq:MpQ}). 
By Corollary~\ref{cor:sgam}~(a) the set of arrows $\gamma$ involved in the sum is equal to 
\[
 \left\{\alp_{ji_p}^{(r_p)}\mid b_{i_pj}<0\right\}\ \bigcup\ 
\left\{\bet_{ji_p}^{(r_p+1)}\mid c_{i_pj}=-2\right\}  \bigcup\ 
\left\{\bet_{ji_p}^{(r_p)}, \bet_{ji_p}^{(r_p+2)}   \mid c_{i_pj}=-3\right\}.
\qquad (*)
\]
It follows that 
\[
(\widetilde{M}_p)_{ji_p}=
\left\{
\begin{array}{ll}
-1 &\mbox{if } c_{i_pj} = 2,\\[2mm]
z^{\deg\left(\alp_{ji_p}^{(r_p)}\right)} & \mbox{if } c_{i_pj} = -1,\\[2mm]
z^{\deg\left(\alp_{ji_p}^{(r_p)}\right)}(1+z^2) & \mbox{if } c_{i_pj} = -2,\\[2mm]
 z^{\deg\left(\alp_{ji_p}^{(r_p)}\right)}(1+z^2+z^4) & \mbox{if } c_{i_pj} = -3,\\[2mm]
0 &\mbox{if } c_{i_pj} = 0.
\end{array}
\right. 
\]
Here, the case $c_{i_pj}=2$ (\emph{i.e.} $j=i_p$) follows from the fact that the set of arrows described in (*)
contains no vertical arrows. 
The case $c_{i_pj}=-1$ follows clearly from (*).
The case $c_{i_pj}=-2$ follows again from (*), together with Equations~(\ref{eq:2.2a}) and (\ref{eq:2.2b}) of Lemma~\ref{lem-3-9}.

Now, again by (\ref{eq:2.2a}), we see that
\[
\deg\left(\alp_{ji_p}^{(r_p)}\right)= b_{ji_p}+ k_j^{(p-1)}- k_j^{(p)}.
\]
In view of the definition $M_p=Z_{p-1}t_{i_p}Z_p^{-1}$
we conclude that indeed $M_p=\widetilde{M}_p$.
\cqfd
\end{proof}

\subsection{$G$-matrices}

In what follows it will be convenient to identify the free abelian
group $\Z^{(V)}$ with an abelian subgroup of $\Z[z^{\pm1}]^I$ via
the injective homomorphism
\[
\Z^{(V)}\hookrightarrow \Z[z^{\pm 1}]^I,\quad \be_{(i,k)}\mapsto z^k\be_i.
\]
Note, that under this identification the submodule $\Z^{(V)}$ is stable under the action of the twisted braid group generators $M_p\ (1\le p \le l)$. 

\smallskip
We are going to define a sequence of \emph{G-matrices} $G^{(0)}, G^{(1)}, \ldots, G^{(l)}\in \Z^{(V)\times (V)}$.
The columns of these matrices, which will later be identified with certain $g$-vectors of our initial  cluster variables, are 
defined as follows.
Recall the partition $V = V_0\cup V_1\cup\cdots\cup V_l$ of \S\ref{sub-sec-311}.
For $(i,k)\in V$ and $0\le m \le l$, we put: 
\begin{equation} \label{eq:gvec-M}
\bg_{(i,k)}^{(m)}:=
\left\{
\begin{array}{ll}
    M_{mp}(z^k\be_{i_p}) &\mbox{if } (i,k)\in V_p\mbox{ for some } p > m, \\[2mm]
    z^k\be_i &\mbox{otherwise}.
\end{array}
\right.
\end{equation}
We then define $G^{(m)}$ as the matrix with column vectors $\bg_{(i,k)}^{(m)}\quad  ((i,k)\in V)$.
In particular, $G^{(l)}$ is the unit matrix with $\bg^{(i,k)}(l)=z^k\be_i$ for all $(i,k)\in V$.

\begin{Prop}\label{lem-3-16}
Let $1\le m \le l$. All rows of $G^{(m)}$ labelled by an element of $V$ of the form $(i_m,k)$
have nonnegative entries.
\end{Prop}
\begin{proof}
This follows immediately from Proposition~\ref{prp:Msgn-coh}. \cqfd
\end{proof}

\begin{example}
We continue our running example in type $B_3$. We have $V=I\times 2\Z$ and $l=9$. 
Recall that the values of $i_m$ and $\rho_m$ are as follows:
\[
\begin{array}{l|rrrrrrrrr}
  m   & 1 & 2 & 3 & 4 & 5 & 6 & 7 & 8 & 9  \\\hline
  i_m & 3 & 1 & 2 & 3 & 1 & 2 & 3 & 1 & 2\\
\rho_m& 2 & 0 &-2 &-8 &-12&-14&-18&-24&-26
\end{array}
\]
Also, $m^+=m+3$ for $1\leq m\leq 6$ and $m^+=\infty$ for $7\leq m\leq 9$. We deduce that
\[
\begin{array}{lll}
V_1 =3\times\{2,0,\ldots,-6\}, & V_2=1\times\{0,-2,\ldots,-10\},  & V_3=2\times\{-2,-4,\ldots,-12\},\\
V_4 =3\times\{-8,-10,\ldots,\!-16\}, & V_5=1\times\{-12,-14,\ldots,\!-22\},&
V_6=2\times\{-14,-16,\ldots,\!-24\},\\
V_7=3\times (2\Z)_{\leq-18}, & V_8=1\times (2\Z)_{\leq -24}, &
V_9= 2\times(2\Z)_{\leq -26}.
\end{array}
\]
The matrices $G(m)$ are readily calculated from the 
vectors $M_{mp}\be_{i_p}\ (m<p\leq l)$, which are displayed in the following table,
 where  the abbreviation $p:=1+z^2$ is used. 
\[
\begin{array}{c|ccccccccc}
p\backslash m & 0 &1 & 2 & 3 & 4 & 5 & 6 & 7 & 8  \\\hline \\
1 &
\begin{array}{c}0\\p\\-1 \end{array}&
\\ \\
2 &
\begin{array}{c}-1\\z^2\\0 \end{array}&
\begin{array}{c}-1\\z^2\\\blue{0} \end{array}&
\\ \\
3 &
\begin{array}{c}-z^{-2}\\p\\-1 \end{array}&
\begin{array}{c}-z^{-2}\\0\\\blue{1} \end{array}&
\begin{array}{c} \blue{z^{-2}}\\-1\\1\end{array}&
\\ \\
4 &
\begin{array}{c}-z^{-4}p\\z^{-2}p\\ -z^{-2}\end{array}&
\begin{array}{c}-z^{-4}p\\0\\  \blue{z^{-2}}\end{array}&
\begin{array}{c}\blue{z^{-4}p} \\-z^{-2}p\\z^{-2}\end{array}&
\begin{array}{c}0\\\blue{z^{-2}p}\\-1 \end{array}&
\\ \\
5 &
\begin{array}{c}0\\z^4\\-z^2\end{array}&
\begin{array}{c}0\\-z^2\\\blue{z^2} \end{array}&
\begin{array}{c}\blue{0}\\-z^2\\z^2 \end{array}&
\begin{array}{c}-1\\ \blue{z^2}\\0 \end{array}&
\begin{array}{c}-1\\z^2\\\blue{0} \end{array}&
\\ \\
6 &
\begin{array}{c} -1\\z^2\\-1 \end{array}&
\begin{array}{c} -1\\-1\\\blue{1}\end{array}&
\begin{array}{c} \blue{1}\\-p\\1 \end{array}&
\begin{array}{c} -z^{-2}\\ \blue{p}\\-z^2\end{array}&
\begin{array}{c} -z^{-2}\\0\\\blue{z^2}\end{array}&
\begin{array}{c} \blue{z^{-2}}\\-1\\z^2\end{array}&
\\ \\
7 & 
\begin{array}{c}0\\ 0\\ -z^{-4} \end{array}& 
\begin{array}{c} 0\\ -z^{-4}p\\ \blue{z^{-4}}  \end{array}& 
\begin{array}{c}\blue{0}\\-z^{-4}p\\z^{-4} \end{array}& 
\begin{array}{c}-z^{-6}p\\\blue{z^{-4}p}\\-z^{-2} \end{array}& 
\begin{array}{c} -z^{-6}p\\ 0\\ \blue{z^{-2}}\end{array}& 
\begin{array}{c} \blue{z^{-6}p}\\-z^{-4}p\\z^{-2}\end{array}&
\begin{array}{c} 0\\\blue{z^{-4}p}\\ -1
\end{array}
\\ \\
8 & 
\begin{array}{c}-z^2\\0\\ 0\end{array}& 
\begin{array}{c}-z^2\\0\\\blue{0} \end{array}& 
\begin{array}{c}\blue{z^2}\\-z^4\\ 0\end{array}& 
\begin{array}{c}0\\\blue{z^4}\\-z^4 \end{array}& 
\begin{array}{c}0\\-z^2\\\blue{z^4} \end{array}& 
\begin{array}{c}\blue{0}\\-z^2\\z^4 \end{array}&
\begin{array}{c}-1\\\blue{z^2}\\-0\end{array}&
\begin{array}{c}-1\\z^2\\\blue{0}\end{array}
\\ \\
9 & 
\begin{array}{c}0\\-z^2\\ 0\end{array}& 
\begin{array}{c}0\\-z^2\\\blue{0} \end{array}& 
\begin{array}{c}\blue{0}\\-z^2\\0\end{array}& 
\begin{array}{c}-1\\\blue{z^2}\\-z^2\end{array}& 
\begin{array}{c}-1\\-1\\\blue{z^2}\end{array}&
\begin{array}{c} \blue{1}\\-z^{-2}p\\z^2\end{array}&
\begin{array}{c}-z^{-2}\\\blue{p}\\-z^4\end{array}&
\begin{array}{c}-z^{-2}\\0\\\blue{z^4}\end{array} &
\begin{array}{c}\blue{z^{-2}}\\-1\\z^4\end{array}\\ \\
\end{array}
\]
Note that all the blue entries have nonnegative coefficients, as asserted by Proposition~\ref{lem-3-16}. 
\end{example}

\subsection{$g$-vectors}
\label{ssec:g-vect} \label{rem:Def_g-vect}
In this section, we are going to show that the column vectors of the matrix $G^{(0)}$ are the 
stabilized $g$-vectors of the cluster variables of $\AA_{w_0}$ belonging to the initial seed 
with quiver $\G_{\bi,\br}$, with respect to the reference seed ``at infinity'' with quiver $\G$.

\begin{Rem} \label{rem:lowering}
Consider the last reflexion $s_{i_l}$ of our fixed reduced expression $\bi$ for $w_0$. 
Note, that by our construction, the  $d_{i_l}$ “lowest” green vertices of
$\G_{\bi,\br}$ are $(i_l,\rho_l), (i_l,\rho_l-2),\ldots, (i_l,\rho_l-2(d_{i_l}-1))$.
By Lemma~\ref{Lem28}~(ii), if we mutate at the 
green vertex $(i_l,\rho_l)$, only the last entry of $\br$ will
be decreased by two, say $\br(-2):=(r_1,r_2,\ldots, r_l-2)$. In formulas
\[
\mu_{(i_l,\rho_l)}(\G_{\bi,\br})=\G_{\bi,\br(-2)}.
\]
Moreover, relative to $\br(-2)$, the new $\rho_l$ will be 
$\rho_l-2$. Thus,
we can iterate and find
\begin{equation} \label{eq:lowering}
\mu_{(i_l,\rho_l-2k)}\circ\cdots\circ\mu_{(i_l,\rho_l-2)}\circ\mu_{(i_l,\rho_l)}(\G_{\bi,\br})=\G_{\bi,\br(-2k)}.
\end{equation}
We should also point out, that in case $d_{i_l}\in\{2,3\}$ the quiver
$\mu_{(i_l,\rho_l-2)}(\Gam_{\bi,\br})$ is \emph{not} of the form
$\Gam_{\bi,\br'}$ for any $\br'$. A similar problems occurs for $d_{i_l}=3$ and $\mu_{(i_l,\rho_l-4)}(\Gam_{\bi,\br})$. On the other hand, if $d_{i_l}\in\{2,3\}$ we have for example 
\[
\mu_{(i_l,\rho_l)}\circ\mu_{(i_l,\rho_l-2)}(\Gam_{\bi,\br})=
\mu_{(i_l,\rho_l-2)}\circ\mu_{(i_l,\rho_l)}(\Gam_{\bi,\br})=
\Gam_{\bi,\br(-4)}.
\]
This will not be a problem for our goals in the present paper. However, we plan to address this kind of issues in a forthcoming publication by
introducing a larger class of quivers.
\end{Rem}

The infinite sequence of mutations
that we will use to go from $\G_{\bi,\br}$ to $\G$ is different from the one we used in \cite{GHL}. 
 Based on~\eqref{eq:lowering}  we will get, for $k\ra\infty$, a  sequence of mutations transforming $\G_{\bi,\br}$ into a sequence of quivers whose lowest 
$d_{i_l}$ pairs of green and red vertices are translated down to  minus infinity, so that in the limit we
will obtain the quiver 
\[
\G^{(l-1)} := \G_{(i_1,\ldots,i_{l-1}),(r_1,\ldots,r_{l-1})}.
\]
We can now repeat the same procedure with the quiver $\G^{(l-1)}$ and obtain a second infinite sequence
of mutations transforming $\G^{(l-1)}$ into
\[
\G^{(l-2)} := \G_{(i_1,\ldots,i_{l-2}),(r_1,\ldots,r_{l-2})}.
\]
And so on. In this way we get an infinite sequence of mutations going from $\G^{(l)} := \G_{\bi,\br}$
to $\G^{(0)} := \G$, via the intermediate steps 
\[
\G^{(m)} := \G_{(i_1,\ldots,i_{m}),(r_1,\ldots,r_{m})} \ (1\le m \le l-1).
\]
What we are going to show is that the column vectors of the matrix $G(m)$ are in fact the 
stabilized $g$-vectors of the cluster variables of the initial seed with quiver $\G_{\bi,\br}$
of $\AA_{w_0}$, with respect to the reference seed with quiver $\Gamma^{(m)}$.
In particular, taking $m=0$, we will get the desired result.

Following~\cite[(2.2)]{CA4} we write the FZ-mutation rule for
a matrix ${B} = (b_{vt}) \in \Z^{(V)\times (V)}$ in direction $s\in V$ as
\[
\left(\mu_s(B)\right)_{vt} =
\left\{
\begin{array}{ll}
    -b_{vt} &\mbox{if } v=s \mbox{ or } t=s,\\[2mm]
     b_{vt}+\sgn(b_{st})[b_{vs}b_{st}]_+ &\mbox{otherwise,}
\end{array}
\right.
\qquad (v,t\in V).
\]
Since we consider only skew-symmetric matrices
$B\in \Z^{(V)\times (V)}$, we may codify them as in \cite[(1.4)]{DWZ2} by a quiver $\G(B)$ 
with vertex set $V$ and
$b_{ts}$ many arrows $\gam : s \to t$ for $b_{ts}>0$. And conversely, every quiver 
with vertex set $V$ can be converted using this rule into a skew-symmetric matrix in $\Z^{(V)\times (V)}$.

According to \cite[(1.3) \& Thm.~1.7]{DWZ2}, the transformation of $g$-vectors under mutation of the reference seed
with matrix $B$ is as follows.
For $s\in V$, and a cluster variable $x$ with $g$-vector $\bg=(g_v)_{v\in V}\in\Z^{(V)}$, the new $g$-vector 
$\bg'$ of $x$ after a mutation of the reference seed in direction $s$ is given by
\begin{equation} \label{eq:g-transf}
g'_v=
\left\{
\begin{array}{ll}
    -g_v &\mbox{if } v=s,\\[2mm]
     g_v+[b_{vs}]_+g_s -b_{vs}\min(g_s, 0) &\mbox{if } v\neq s. 
\end{array} 
\right.
\end{equation}
Note that $[b_{vs}]_+g_s -b_{vs}\min(g_s, 0)=\sgn(g_s)[b_{vs}g_s]_+$.
Thus, a $g$-vector transforms like an additional column attached to $B$.
In terms of quivers, we see that if $\sgn(g_s) > 0$, then $g'_v$ is obtained by adding to $g_v$
the number $g_s \times |\{\gam : s \to v\}|$.

For $1\le m\le l$, let us denote by $\vec{\mu}_{(i_m,\rho_m-2\Z_{\ge 0})}$ the infinite sequence of mutations
\[
\vec{\mu}_{(i_m,\rho_m-2\Z_{\ge 0})} := \cdots\circ\mu_{(i_m,\rho_m-4)}\circ\mu_{(i_m,\rho_m-2)}\circ\mu_{(i_m,\rho_m)}
\]
transforming $\G^{(m)}$ into $\G^{(m-1)}$, as explained above.
We also denote by $B^{(m)}$ the skew-symmetric matrix corresponding to $\G^{(m)}$.

We can now state and prove our key result.
\begin{Prop} \label{prp:key}
Let $1\le m \le l$. We have
\[
 \vec{\mu}_{(i_m,\rho_m-2\Z_{\ge 0})}(B^{(m)},G^{(m)}) = (B^{(m-1)},G^{(m-1)}),
\]
where $(B^{(m)},G(m))$ denotes the matrix obtained by adding to $B^{(m)}$ the column vectors of $G(m)$.
\end{Prop}

The proof of Proposition~\ref{prp:key} will use the following lemma.
\begin{Lem}\label{lem:key-est}
For $1\le m < p \le l$, let $\del(m,p)$ denote the degree of the Laurent polynomial 
\[
\left(M_{mp}\be_{i_p}\right)_{i_m} \in \Z[z^{\pm 1}].
\]
(If this polynomial is equal to 0 we put $\del(m,p) := -\infty$.)
Then we have
\[
    \del(m,p)+\rho_p\leq \rho_m.
\]
In particular, if $(i,k)\in V_{p+1}\cup V_{p+2}\cup\cdots\cup V_l$,
and $\left(\bg_{(i,k)}^{(m)}\right)_{(i_m,r)}\neq 0$, then $r\leq \rho_m$.
Moreover, in this case $\left(\bg_{(i,k)}^{(m)}\right)_{(i_m,r)}>0$. 
\end{Lem}

\begin{proof}
Recall that $\rho_m=r_m+ k_{i_m}^{(m)}$. On the other hand,
\[
    \del(m,p) =\deg(\be_{i_m}^T M_{m+1}\cdots M_m \be_{i_p})
    =\deg(\be_{i_m}^T Z_m \,t_{i_{m+1}}\cdots t_{i_p} Z_p^{-1}\be_{i_p}).
\]
Since for every $i\in I$ all entries of the matrix $t_i$ belong to $\Z[z^{-1}]$, we have
\[
\deg(\be_{i_m}^T t_{i_{m+1}}\cdots t_{i_p}\be_{i_p})\leq 0.
\]
Thus, in view of the definition of $Z_m$ and $Z_p^{-1}$ we find
\[
\del(m,p)+\rho_p \leq (k_{i_m}^{(m)}-k_{i_p}^{(p)})+(k_{i_p}^{(p)}+r_p)
    = k_{i_m}^{(m)}+r_p\leq k_{i_m}^{(m)}+r_m=\rho_m.
\]
The final statements follow from the definition of $\bg_{(i,k)}^{(m)}$, and from 
Proposition~\ref{lem-3-16}.
\cqfd
\end{proof}

\noindent
\emph{Proof of Proposition~\ref{prp:key}} ---\ 
Let $\bi'=(i_1,i_2,\ldots,i_m)$ and $\br'=(r_1,r_2,\ldots, r_m)$.
Then $\G^{(m)}=\G_{\bi',\br'}$ and
\[
\mu_{(i_m,\rho_m-2(k-1))}\circ\cdots\circ \mu_{(i_m,\rho_m)}(\G^{(m)})=
\G_{\bi',\br'[k]},\qquad (k\ge 1),
\]
where $\br'[k] := (r_1,\ldots, r_{m-1}, r_m - 2k)$.
The green vertex $(i_m,\rho_m-2k)$ of $\G_{\bi',\br'(k)}$ is a sink for
the vertical arrows incident to it, and a source for the
oblique arrows incident to it. 

Moreover, if we denote by 
$\gam_1, \ldots, \gam_s$ the arrows of $\G_{\bi',\br'}$ with source
$(i_m, \rho_m)$, and  if we put $t(\gam_h) = (j_h,t_h)\ (1\le h \le s)$,
then we also have $s$ arrows $\gam'_1, \ldots, \gam'_s$ of $\G_{\bi',\br'[k]}$ with source 
$(i_m, \rho_m-2k)$, and we have $t(\gam'_h) = (j_h,t_h-2k)\ (1\le h \le s)$.
Note that $b_{i_mj_h}<0$ for all $1\le h \le s$.

Next, for $k\in\Z_{\geq 0}$, define  $\Z$-linear endomorphisms $M_m^{(k)} : \Z^{(V)}\ra\Z^{(V)}$ by:
\[
M_m^{(k)}(\be_v) := 
\left\{
\begin{array}{ll}
    M_m\left(z^{\rho_m-2k}\be_{i_m}\right) &\mbox{if } v=(i_m,\rho_m-2k),\\[2mm]
    \be_v  &\mbox{otherwise}.
\end{array}
\right.
\]
For $\bg \in \Z^{(V)}$, we have:
\[
\left(M_m^{(k)}(\bg)\right)_{(i_m,\rho_m-2h)} = 
\left\{
\begin{array}{ll}
 -g_{(i_m,\rho_m-2k)}    &\mbox{if } h=k,\\[2mm]
 g_{(i_m,\rho_m-2h)}&\mbox{otherwise}.
\end{array}
\right.
\]
It follows that the operators $M_m^{(0)}, M_m^{(1)},\ldots $ pairwise commute. 
Now, let $\bg\in\Z^{(V)}$ be a $g$-vector of $\G_{\bi',\br'(k)}$ with
component $g_{i_m,\rho_m-2k}\geq 0$. 
Using Proposition~\ref{prp:2.7} and the mutation rule for $g$-vectors recalled above, we obtain
\[
\mu_{(i_m,\rho_m-2k)}\left(\G_{\bi',\br'[k]},\bg\right)=\left(\G_{\bi',\br'[k+1]}, {M}_m^{(k)}(\bg)\right).
\]
By Lemma~\ref{lem:key-est}, if $\bg=(g_v)_{v\in V}\in\Z^{(V)}$ is a g-vector for $\G_{\bi',\br'}=\G^{(m)}$ then
\begin{equation} \label{eq:gvec+}
g_{(i_m,r)}=
\left\{
\begin{array}{ll}
    0      &\mbox{if } r>\rho_m,\\
    \geq 0 &\mbox{if } r\leq\rho_m. 
\end{array}
\right.
\end{equation}
It follows that 
\[
\mu_{(i_m,\rho_m-2k)}\circ\cdots\circ\mu_{(i_m,\rho_m-2)}\circ
\mu_{(i_m,\rho_m)}\left(\G_{\bi',\br'},\bg\right)=\left(\G_{\bi',\br'[k+1]},
\,{M}_m^{(k)}\cdots M_m^{(1)}{M}_m^{(0)}(\bg)\right).
\]
Since $\bg\in\Z^{(V)}$ has by definition finite support, for $k\gg 0$ we have:
\[
M_m^{(k)}\cdots M_m^{(1)}{M}_m^{(0)}(\bg)=M_m(\bg).
\]
By Lemma~\ref{lem:key-est} and Proposition~\ref{lem-3-16} all the vectors 
\[
M_{m,p}(z^l\be_{i_p})\quad\mbox{with}\quad m<p\leq l\mbox{ and } (i_p,k)\in V_p
\]
also fulfill condition~(\ref{eq:gvec+}). This implies our claim.
\cqfd

We can now conclude by descending induction on $1\le m\le l$. 
Since $G^{(l)}$ is equal to the unit matrix, it can be seen as the matrix with column vectors 
the $g$-vectors of the cluster variables of the initial seed $\G_{\bi,\br} = \G^{(l)}$,
with reference to itself. Then by induction, Proposition~\ref{prp:key} implies that for every 
$0\le m\le l$, the column vectors of the matrix $G^{(m)}$ are the stabilized $g$-vectors of the initial cluster variables $z_{(i,k)}$ (with $(i,k)\in V$) associated to
$\G_{\bi,\br}$, with respect to $\G^{(m)}$. The desired result is obtained for $m=0$.

\begin{Rem}\label{comps1} For later use, let us unpack the relevant definitions for the construction of the limit g-vectors 
$\bg_{(i,r)}^{\infty}=\bb_{(i,r);(\bi,\br)}^\infty:=\bg_{(i,r)}^{(0)}$ for $(i,r)\in V$, which
depends on the reduced expression $\bi=(i_1,i_2,\ldots, i_l)$ and the (weakly decreasing)  sequence $\br=(r_1, r_2,..,r_l)\in\Z^l$,
which is adapted to~$\bi$.

Let $L:=\{1,2,\ldots, l\}$ and recall that in Section~\ref{sub-sec-311} we defined the function
\[
\sig\df V\ra\Z,  (i,r)\mapsto 2d_i\cdot\abs{\{j\in L\mid i_j=i_r \text{ and } r_j\geq r\}}.
\]
Moreover, we set $s_0:=0$, and for $m\in L$ we define 
\begin{align*}
s_m &:=\sig(i_m, r_m)=2d_{i_m}\cdot\abs{\{1\leq j\leq m\mid i_j=i_m\}},\\
    \rho_m &:= r_m-s_m,\\ 
    \rho_{\infty}&:=-\infty,\\
    m^+ &:= \begin{cases}
        \min\{j\in L\mid j> m \text{ and } i_j= i_m\} &\text{if this set is nonempty},\\
        \infty &\text{else,}
    \end{cases}\\
    V_m &:= \{(i_m,r)\in V\mid \rho_{m^+}< r\leq \rho_m\},\\
    V_0 &:= V\setminus (\cup_{m\in L} V_m).
\end{align*}
It follows from our constructions that
for $0\leq m\leq l$ and $(i,r)\in V_m$ we have $i=i_m$. 
Since the diagonal entry  $(i_m,i_m)$ of $Z_m^{-1}$ is $z^{s_m}$ we find
\begin{equation} \label{eq:lgect}
\bg_{(i,r)}^{\infty}= t_{i_1} t_{i_2}\cdots t_{i_m}(\be_{i_m}) z^{r+s_m}
=t_{i_1}t_{i_2}\cdots t_{i_m}(\be_{i, r+s_m}),
\end{equation}
where we recall that $\be_{i,r}:=\be_i\, z^r$, and that the operators $t_i$ are $\Z[z^{\pm 1}]$-linear. This formula is very similar to the one from Proposition~4.19 in~\cite[Sec.4.3.4]{GHL} for the limit g-vectors  $\bg_{(i_p,\rho_p)}^\infty$ corresponding to the green vertices $(i_p,\rho_p)$ with $p\in L$.  
However, in~\cite{GHL} we were dealing only with simply laced Cartan matrices, reduced expressions which were adapted to some orientation of the corresponding Dynkin diagram, and (implicitly)
sequences $\br$ of minimal descent.
Note that in \emph{loc.\@ cit.\@} we used the operator $?[s]$ (defined in~\cite[Sec.~4.3.3]{GHL}) which in our notation here is identified with the multiplication by $z^{2s}$. Next, for $i\in I$, let  $m_i:=\max\{r_j\mid i_j=i\}$. Then in \emph{loc.\@ cit.\@} we used the “highest red vertex”
$(i_p, m_{i_p})$ as reference point rather than the (green) vertex $(i_p,\rho_p)$ itself.  For these reasons, the shift
parameter $s_m$ there was defined differently. Finally, note that in the  
formula~\cite[(12)]{GHL}, which is valid for all vertices $(i,r)\in V$, the shift parameter $s$ was not explicitly specified. 
\end{Rem}

\subsection{A knitting algorithm}

In this section we present a simple algorithm (similar to the classical knitting algorithm 
for dimension vectors in Aus\-lan\-der-Reiten combinatorics) for calculating the stabilized $g$-vectors of the cluster variables of the initial seed 
of $\AA_{w_0}$ with quiver $\G_{\bi,\br}$. 

As explained above, these stabilized $g$-vectors are the column vectors
\[
\bg_{(i,k)}:= \bg_{(i,k)}^{(0)},\qquad ((i,k)\in V),
\]
of the matrix $G^{(0)}$, and this matrix can be calculated by induction using Proposition~\ref{prp:key} (we omit in this section for typographic reasons the superindex $\infty$).
For its column vectors, this boils down to the following simple recipe.

\begin{Rem} \label{rem:vert-Vk}
Recall that for $1\le p \le l$ and
$0\leq k\leq d_{i_p}-1$, we have a vertical down arrow $\beta(p,k):= \bet_{i_pi_p}^{(r_p-d_{i_p}-2k)}\in\G_{\bi,\br}$ with 
\[
t(\bet(p))=(i_p, \rho_p-2k)\in V_p, \qquad 
s(\bet(p))=(i_p,\rho_p+2d_{i_p}-2k)\in V_{m^-(i_p,p)}.
\]
All the remaining vertical arrows of $\G_{\bi,\br}$ are up arrows of the form 
$\alpha_{ii}^{(r)}$ for some 
 $ (i,r)\in V$. Note that 
if $s\left(\alp_{ii}^{(r)}\right) \in V_p$ for some $0\le p \le l$, then
we also have $t\left(\alp_{ii}^{(r)}\right)\in V_p$. Indeed, by Corollary~\ref{cor:sgam}~(a), the only arrows with source $(i_p, \rho_p)$ 
are oblique arrows. Moreover, the same is true for 
$(i_p, \rho_p -2)$ if $d_{i_p}=2$, by Corollary~\ref{cor:sgam}~(b).
\end{Rem}
 
\begin{Prop} \label{prp:gknit}
The $g$-vectors $\bg_{(i,k)}$ are determined by the following rules:
\begin{itemize}
 \item[(a)] For every $(i,k)\in V_0$, we have $\bg_{(i,k)} = \be_{(i,k)} := z^k\be_i$.
 \item[(b)]
 For $i\in I$ and for every vertical up arrow $\alpha := \alpha_{ii}^{(r)}$ we have 
 $\bg_{s(\alpha)} = z^{-2d_i}\bg_{t(\alpha)}.$
 \item[(c)] For $1\le p \le l$  and  $0\leq k\leq d_{i_p}-1$, we have
 \begin{equation}\label{eq-14}
 \bg_{t(\bet(p,k))} =\sum_\gam \bg_{t(\gam)} - z^{-2d_{i_p}}\bg_{s(\bet(p,k))},
 \end{equation}
where the sum is over all arrows $\gamma$ with source $t(\beta(p,k))$. 
\end{itemize}
\end{Prop}
\begin{proof}
Statement (a) is obvious.  

{ For (b), recall from Section~\ref{ssec:g-vect} that the stabilized $g$-vectors are the column vectors of the matrix $G^{(0)}$.  
In view of Remark~\ref{rem:vert-Vk},
we have $\left\{s\left(\alp_{ii}^{(r)}\right),\, t\left(\alp_{ii}^{(r)}\right)\right\}\subset V_p$ for some $p$.
Thus, our claim follows from the definition of $G^{(0)}$ in Equation~(\ref{eq:gvec-M}).}

Let us check (c).
Recall from  Proposition~\ref{prop:Gam-basic} and
from Corollary~\ref{cor:sgam}, that for each arrow $\gam$ of $\G_{\bi,\br}$ with $s(\gam)=t(\bet(p,k))$ we have
$t(\gam)\in V_{m^-(j(\gam),p)}$ for some $j(\gam)\in I\setminus\{i_p\}$. Thus, writing $m(\gam):=m^-(j(\gam),p)$, we have by Corollary~\ref{cor:sgam} and Lemma~\ref{lem-3-9} 
\[
 \bg_{t(\gam)}= M_{0,m(\gam)}\left(z^{k(\gam)}\be_{i_{m(\gam)}}\right)
 = M_{0,p-1}\left(z^{k(\gam)}\be_{i_{m(\gam)}}\right),
\]
where
\[
 k(\gam) =
 \left\{
 \begin{array}{lll}
    b_{ji_p}+r_p+k_j^{(p-1)} &\mbox{if } 
    \gam=\alp_{ji_p}^{(r_p)} &\mbox{and } c_{i_pj}\leq -1,\\[2mm]
    r_p\ +k_j^{(p-1)} &\mbox{if } 
    \gam=\bet_{ji_p}^{(r_p+1)} &\mbox{and } c_{i_pj}=-2,\\[2mm]
   k-1+r_p +k_j^{(p-1)} &\mbox{if } 
    \gam=\bet_{ji_p}^{(r_p+h)} &\mbox{and } c_{i_pj}=-3,\ h\in\{0,2\}.\\[2mm]    
\end{array}
\right.
\]
Similarly,
\[
z^{-2d_{i_p}}\bg_{s(\bet(p,k))}= M_{0,p-1}\left(z^{k'(p)}\be_{i_p}\right),
\]
where $k'(p):=r_p+k_{i_p}^{(p)}$.
Thus, the right-hand side of Equation~(\ref{eq-14}) becomes
\begin{eqnarray*}
 M_{0,p-1}\left(
\sum_{\gam}z^{k(\gam)}\be_{i_{m(\gam)}} - z^{k'(p)}\be_{i_p}z^{k'(p)} \right)
&=&
  M_{0,p-1}\left(z^{k'(p)}\left(
\sum_{\gam}z^{\deg(\gam)}\be_{i_{m(\gam)}} - \be_{i_p} \right)\right)
\\[2mm]
&=& 
M_{0,p}\left(z^{\rho_p}\be_{i_p})\right)
\\[2mm]
&=& 
\bg_{(i_p,\rho_p)},
\end{eqnarray*}
where the first equality follows from Corollary~\ref{Cor3-10}, and the second from Proposition~\ref{prp:2.7}.
\cqfd
\end{proof}

To calculate the stabilized $g$-vectors, we can now proceed as follows. By (a), in the upper
half-infinite part of $\G_{\bi,\br}$ with vertex set $V_0$, the $g$-vectors are trivial. We then go down 
and calculate successively the $g$-vectors of $V_1$, \ldots , $V_l$ using at each step the relevant rule 
(b) or (c). 

\begin{example}
We continue our running example in type $B_3$. 
The first non trivial $g$-vector is $\bg_{(3,2)}$. By (c) we have:
\[
\bg_{(3,2)} = \bg_{(2,4)}+\bg_{(2,2)}-z^{-2}\bg_{(3,4)} 
= \be_{(2,4)}+\be_{(2,2)}-z^{-2}\be_{(3,4)}
= \be_{(2,4)}+\be_{(2,2)} -\be_{(3,2)}.
\]

Then, using (b) we deduce the values of $\bg_v$ for the remaining vertices $v$ of $V_1$, 
namely,
\[
\bg_{(3,0)} =  \be_{(2,2)}+\be_{(2,0)} -\be_{(3,0)},
\quad \ldots, \quad 
\bg_{(3,-6)} =  \be_{(2,-4)}+\be_{(2,-6)} -\be_{(3,-6)}.
\]
After that we go to $V_2$, whose higher vertices are $(1,0)$ and $(1,-2)$ (here we have $d_{i_2} = 2$).
By (c) we have
\[
\bg_{(1,0)} = \bg_{(2,2)}-z^{-4}\bg_{(1,4)} 
= \be_{(2,2)}-z^{-4}\be_{(1,4)} 
=  \be_{(2,2)} -\be_{(1,0)},
\]
and similarly
\[
\bg_{(1,-2)} = \bg_{(2,0)}-z^{-4}\bg_{(1,2)} 
= \be_{(2,0)}-z^{-4}\be_{(1,2)} 
= \be_{(2,0)} -\be_{(1,-2)}. 
\]
The remaining $\bg^v$ for $v\in V_2$ are again readily calculated by means of (b).
Then we go to $V_3$, whose highest vertices are $(2,-2)$ and $(2,-4)$. By (c) we have
\begin{eqnarray*}
\bg_{(2,-2)} &=& \bg_{(3,-2)}+\bg_{(1,-4)}-z^{-4}\bg_{(2,2)}\\ 
& =& (\be_{(2,0)}+\be_{(2,-2)}-\be_{(3,-2)}) + (\be_{(2,-2)}-\be_{(1,-4)}) - \be_{(2,-2)} \\
& = & \be_{(2,0)}+\be_{(2,-2)}-\be_{(3,-2)} -\be_{(1,-4)},
\end{eqnarray*}
and similarly we find 
\[
\bg_{(2,-4)} = \be_{(2,-2)}+\be_{(2,-4)}-\be_{(3,-4)} -\be_{(1,-6)}. 
\]
The remaining $\bg_v$ for $v\in V_3$, again, are immediately given by (b).
And so on for $V_4, \ldots, V_9$. We leave it to the reader to finish the calculation of this
example.
\end{example}


\section{Mutation equivalence of the  initial quivers}\label{mutequiv}

In this Section, we finish the proof of Theorem \ref{Prop2-10}. Recall the basic quiver $\Gam=\Gam(C)$ from  Section~\ref{subsec-basic-quiver}, as well as the fine description of the modified quivers
$\Gam_{\bi,\br}$ from Section~\ref{subsec-fine}.
The aim of this section is to show that the mutation class of 
$\Gam_{\bi,\br}$ does \emph{not} depend on the choice of the reduced
expression $\bi$. More precisely, we have the following result:

\begin{Thm} \label{thm:mut_equivQ}
Let $C\in \Z^{I\times I}$ be a connected Cartan matrix of finite type, and $\bi, \bi'$ two reduced expression of an element $w$ in the corresponding Weyl group
$W(C)$, with length $l(w)=l$. If $\br\in\Z^l$ (\resp $\br'\in\Z^l$) is a weakly decreasing sequence adapted to $\bi$ (\resp to $\bi'$), then there exists a finite mutation sequence $\mu_\bullet$, such that $\Gam(C)_{\bi,\br'}$ is isomorphic to $\mu_\bullet(\Gam(C)_{\bi,\br})$.  
\end{Thm}

We outline here the main steps of the proof. The details
are worked out in the following subsections.

If $\bi$ and $\bi'$ are in the same commutation class, it is  an easy exercise  to find a mutation sequence $\mu_\bullet$ such that in fact 
$\Gam(C)_{\bi',\br'}=\mu_\bullet(\Gam(C)_{\bi,\br})$.  

Thus, it is sufficient to show our claim for the case when $\bi$ and $\bi'$ are related
by a single braid relation and moreover $\br$ and $\br'$ are given by specific (admissible) choices. If $C$ is \emph{not} of type $\sfG_2$, this is covered by
Proposition~\ref{prp:A2-braid} and Proposition~\ref{prp:B2-braid} below.

If $C$ is of type~$\sfG_2$, only the longest element $w_0$ has two different reduced expressions, namely $\bi:=(2,1,2,1,2,1)$ and $\bi':=(1,2,1,2,1,2)$. Thus, in this case it is sufficient to relate two quivers $\Gam(\sfG_2)_{\bi,\br}$
and $\Gam(\sfG_2)_{\bi',\br'}$, for specific (admissible) choices $\br,\br'\in\Z^6$ by
an explicit sequence of  mutations. This is done in Lemma~\ref{lem:G2-key}  of Appendix~\ref{App:G2-braid}, where we connect $\Gam(\sfG_2)_{(2,1,2,1,2,1),(26,23,16,13,6,3)}$ and
$\Gam(\sfG_2)_{(1,2,1,2,1,2),(35,28,25,18,15,8)}$ by an explicit sequence of 163 mutations. 

For the rest of this section we will assume that 
$C=(c_{ij})_{i,j\in\Z}$ is a connected Cartan matrix of finite type with minimal symmetrizer $D=\diag((d_i)_{i\in I})$ and $B:=D\cdot B$. 
Moreover, we will assume that $\bi=(i_1,\ldots,i_l)$ is a reduced expression for some $w\in W(C)$,
and that $\br=(r_1,\ldots, r_l)\in\Z^l$ is adapted to $\bi$.

\subsection{The $\sfA_2$-braid relation}
Suppose that for some $k$ the parabolic subgroup $W(C)_{\{i_k,i_{k+1}\}}$ is isomorphic to $W(\sfA_2)$, and
that $i_{k+2}=i_k$.
Then  we may consider the sequence $\bi'=(i'_1, \ldots, i'_l)$, 
which is obtained from $\bi$  by the $\sfA_2$-braid move at $k$,
i.e.  $(i'_k, i'_{k+1}, i'_{k+2})=(i_{k+1},i_k, i_{k+1})$
whilst all other letters remain unchanged. 
Then, $\bi'$ is also a reduced expression for the same element $w\in W(C)$.
The following proposition implies that in this situation  $\Gam_{\bi',\br'}$ is mutation equivalent to $\Gam_{\bi,\br}$.
A similar statement was previously proved in~\cite[Prop.~3.4]{GHL} for $C$ symmetric.  The general case is slightly more complicated, and in fact, several cases have to be distinguished. 
For this reason we give here a complete proof.

\begin{Prop} \label{prp:A2-braid}
Consider a reduced expression
$\bi=(i_1,\ldots, i_l)$ of some
element $w\in W(C)$, and  $\br=(r_1,\ldots, r_l)\in\Z^l$  adapted to $\bi$.

Suppose that for some $1\leq k\leq l-2$ we have
\begin{itemize}
\item $i_{k+2}=i_k$,
\item $c_{i_{k+1},i_k}=-1=c_{i_k, i_{k-1}}$ (and thus $d:=d_{i_k}=d_{i_{k+1}}\in\{1,2\}$),
\item $(r_k, r_{k+1}, r_{k+2})=(r_k,r_k+b_{i_k,i_{k+1}}, r_k+2b_{i_k, i_{k+1}})$ and
\item $r_{k-1}+b_{i_{k-1},i_{k+1}}\geq r_k+d_{i_k}$ where we agree that $r_0=\infty$.
\end{itemize}
Then we have a quiver isomorphism
\[
\Gam_{\bi',\br'}\cong\begin{cases}
    \mu_{(i_{k+1}, \rho_{k+1}+2d)}\circ\mu_{(i_k,\rho_k)}\circ\mu_{(i_k,\rho_{k}-d)}(\Gam_{\bi,\br}) &\text{if }d=1,\\
    (\mu_{(i_{k+1}, \rho_{k+1}+2d)}\circ\mu_{(i_{k+1}, \rho_{k+1}+d)})
    \circ(\mu_{(i_k,\rho_k)}\circ\mu_{(i_k,\rho_k-d)})\\
    \hspace{4.5cm}\circ(\mu_{(i_k,\rho_{k}-d)} \circ\mu_{(i_k,\rho_{k}-2d)})(\Gam_{\bi,\br})  &\text{if }d=2,
\end{cases}
\]
where
$\bi'$ is obtained from $\bi$ by replacing the subsequence $(i_k,i_{k+1},i_{k+2})$ by
$(i_{k+1},i_k,i_{k+1})$ and $\br'$ is obtained from $\br$ by replacing the subsequence
$(r_k,r_{k+1},r_{k+2})$ by $(r_k+d_{i_k}, r_k, r_{k+1})$.
In particular, the reduced expression $\bi'$ is obtained from $\bi$ by a braid move of type $\sfA_2$.
\end{Prop}

\begin{proof}
Let us consider the Dynkin diagram $X$ corresponding to $C$. 
By our hypothesis, $X$ has a simple edge $i_k-i_{k+1}$. In particular, $X$ is not of type $\sfG_2$, and it has at most one edge of weight 2.  Thus, it is sufficient to verify our claim
for the following "local" configurations: 
\[\xymatrix{
(A) &h\ar@{-}[r]& i_k\ar@{-}[r] &i_{k+1}\ar@{-}[r] &j\\
(B) &h\ar@{-}[r]& i_k\ar@{-}[r] &i_{k+1}\ar@{=}[r]|{>} &j\\
(C) &h\ar@{-}[r]& i_k\ar@{-}[r] &i_{k+1}\ar@{=}[r]|{<} &j\\
(D) &h\ar@{-}[r]& i_k\ar@{-}[r] &i_{k+1}\ar@{-}[r]\ar@{-}[d] &j\\
    &           &               &l
}\]
In fact, the case by case analysis below shows that the vertical arrows in the boundary
columns corresponding to $h,j$ (and $l$) are not affected by the above-mentioned mutation sequence.  In each case, for typographical reasons, we choose specific labels for the vertices $i_k$ and 
the spectral parameters $r_k$.

(A) This case occurs, with $d=1$, for all simply laced types, when none of the vertices $\{i_k,i_{k+1}\}$ is a branching point of $X$. It also occurs in type $\sfC_n$ if 
$\{i_k,i_{k+1}\}\cap\{n-1,n\}=\emptyset$. 
This case was already discussed in~\cite[Prp.~3.4]{GHL}. For the convenience of the reader we verify this case in Figure~\ref{fig:A2-A4} with a little more detail.
\begin{figure}[ht]
    \centering
   \scalebox{0.5}{\begin{tikzcd}[cramped,sep=tiny]
	&& {(2,14)} &&&& {(4,14)} &&&&&&&&&&&&&&&&&&&&&&& {(2,14)} &&&& {(4,14)} \\
	{(1,13)} &&&& {(3,13)} &&&&&&& {(2,14)} &&&& {(4,14)} &&&&& {(2,14)} &&&& {(4,14)} &&& {(1,13)} &&&& {(3,13)} \\
	&& {(2,12)} &&&& {(4,12)} &&& {(1,13)} &&&& {(3,13)} &&&&& {(1,13)} &&&& {(3,13)} &&&&&&& {(2,12)} &&&& {(4,12)} \\
	{(1,11)} &&&& {(3,11)} &&&&&&& {(2,12)} &&&& {(4,12)} &&&&& {(2,12)} &&&& {(4,12)} &&&&&&& \textcolor{rgb,255:red,255;green,0;blue,4}{{(3,11)}} \\
	&& \textcolor{rgb,255:red,255;green,0;blue,4}{{(2,10)}} &&&&&&& {(1,11)} &&&& {(3,11)} &&&&& {(1,11)} &&&& {(3,11)} &&&&& {(1,11)} \\
	&&&&&& {(4,10)} &&&&& {(2,10)} &&&& {(4,10)} &&&&& {(2,10)} &&&& {(4,10)} &&&&&&& \textcolor{rgb,255:red,0;green,235;blue,16}{{(3,9)}} \\
	&& \textcolor{rgb,255:red,0;green,235;blue,16}{{(2,8)}} &&&&&&&&&&& {(3,9)} &&&&&&&&& {\mathbf{(3,9)}} &&&&&&& \textcolor{rgb,255:red,255;green,0;blue,4}{{(2,10)}} \\
	&&&& \textcolor{rgb,255:red,255;green,0;blue,4}{{(3,9)}} &&&&&&& {\mathbf{(2,8)}} &&&&&&&&&&&&&&&&&&&&&& {(4,10)} \\
	{(1,9)} &&&&&&&&&&&&&&&&&&&&&& {(2,8)} &&&&&&& \textcolor{rgb,255:red,0;green,235;blue,16}{{(2,6)}} \\
	&&&& \textcolor{rgb,255:red,0;green,235;blue,16}{{(3,7)}} &&&&&&& {(2,6)} &&&&&&&&& {(2,6)} &&&&&&&&&&& \textcolor{rgb,255:red,255;green,0;blue,4}{{(2,8)}} \\
	&& \textcolor{rgb,255:red,255;green,0;blue,4}{{\mathbf{(2,6)}}} &&&&&&& {(1,9)} &&&& {(3,7)} &&&&& {(1,9)} &&&& {(3,7)} &&&&& {(1,9)} \\
	&&&&&& {(4,8)} &&&&& {(2,4)} &&&& {(4,8)} &&&&& {(2,4)} &&&& {(4,8)} &&&&&&& \textcolor{rgb,255:red,0;green,235;blue,16}{{(3,7)}} \\
	&& \textcolor{rgb,255:red,0;green,235;blue,16}{{(2,4)}} &&&&&&& {(1,7)} &&&& {(3,5)} &&&&& {(1,7)} &&&& {(3,5)} &&&&&&& {(2,4)} &&&& {(4,8)} \\
	{(1,7)} &&&& {(3,5)} &&&&&&& {(2,2)} &&&& {(4,6)} &&&&& {(2,2)} &&&& {(4,6)} &&& {(1,7)} &&&& {(3,5)} \\
	&& {(2,2)} &&&& {(4,6)} &&& {(1,5)} &&&& {(3,3)} &&&&& {(1,5)} &&&& {(3,3)} &&&&&&& {(2,2)} &&&& {(4,6)} \\
	{(1,5)} &&&& {(3,3)} &&&&&&& {(2,0)} &&&& {(4,4)} &&&&& {(2,0)} &&&& {(4,4)} &&& {(1,5)} &&&& {(3,3)} \\
	&& {(2,0)} &&&& {(4,4)} &&&&&&&&&&&&&&&&&&&&&&& {(2,0)} &&&& {(4,4)}
	\arrow[from=1-3, to=2-1]
	\arrow[from=1-3, to=2-5]
	\arrow[from=1-7, to=2-5]
	\arrow[from=1-30, to=2-28]
	\arrow[from=1-30, to=2-32]
	\arrow[from=1-34, to=2-32]
	\arrow[from=2-1, to=3-3]
	\arrow[from=2-5, to=3-3]
	\arrow[from=2-5, to=3-7]
	\arrow[from=2-12, to=3-10]
	\arrow[from=2-12, to=3-14]
	\arrow[from=2-16, to=3-14]
	\arrow[from=2-21, to=3-19]
	\arrow[from=2-21, to=3-23]
	\arrow[from=2-25, to=3-23]
	\arrow[from=2-28, to=3-30]
	\arrow[from=2-32, to=3-30]
	\arrow[from=2-32, to=3-34]
	\arrow[from=3-3, to=1-3]
	\arrow[from=3-3, to=4-1]
	\arrow[from=3-3, to=4-5]
	\arrow[from=3-7, to=1-7]
	\arrow[from=3-7, to=4-5]
	\arrow[from=3-10, to=4-12]
	\arrow[from=3-14, to=4-12]
	\arrow[from=3-14, to=4-16]
	\arrow[from=3-19, to=4-21]
	\arrow[from=3-23, to=4-21]
	\arrow[from=3-23, to=4-25]
	\arrow[from=3-30, to=1-30]
	\arrow[from=3-30, to=4-32]
	\arrow[from=3-30, to=5-28]
	\arrow[from=3-34, to=1-34]
	\arrow[from=3-34, to=4-32]
	\arrow[from=4-1, to=2-1]
	\arrow[from=4-1, to=5-3]
	\arrow[from=4-5, to=2-5]
	\arrow[from=4-5, to=5-3]
	\arrow[from=4-5, to=6-7]
	\arrow[from=4-12, to=2-12]
	\arrow[from=4-12, to=5-10]
	\arrow[from=4-12, to=5-14]
	\arrow[from=4-16, to=2-16]
	\arrow[from=4-16, to=5-14]
	\arrow[from=4-21, to=2-21]
	\arrow[from=4-21, to=5-19]
	\arrow[from=4-21, to=5-23]
	\arrow[from=4-25, to=2-25]
	\arrow[from=4-25, to=5-23]
	\arrow[from=4-32, to=2-32]
	\arrow[color={rgb,255:red,0;green,4;blue,255}, from=4-32, to=6-32]
	\arrow[from=5-3, to=3-3]
	\arrow[color={rgb,255:red,0;green,4;blue,255}, from=5-3, to=7-3]
	\arrow[from=5-10, to=3-10]
	\arrow[from=5-10, to=6-12]
	\arrow[from=5-14, to=3-14]
	\arrow[from=5-14, to=6-12]
	\arrow[from=5-14, to=6-16]
	\arrow[from=5-19, to=3-19]
	\arrow[from=5-19, to=6-21]
	\arrow[from=5-23, to=3-23]
	\arrow[from=5-23, to=6-21]
	\arrow[from=5-23, to=6-25]
	\arrow[from=5-28, to=2-28]
	\arrow[from=5-28, to=7-30]
	\arrow[from=6-7, to=3-7]
	\arrow[from=6-7, to=8-5]
	\arrow[from=6-12, to=4-12]
	\arrow[from=6-12, to=8-12]
	\arrow[from=6-16, to=4-16]
	\arrow[from=6-16, to=7-14]
	\arrow[from=6-21, to=4-21]
	\arrow[from=6-21, to=7-23]
	\arrow[from=6-21, to=10-21]
	\arrow[from=6-25, to=4-25]
	\arrow[from=6-25, to=7-23]
	\arrow[from=6-32, to=7-30]
	\arrow[from=6-32, to=8-34]
	\arrow[from=7-3, to=8-5]
	\arrow[from=7-3, to=9-1]
	\arrow[from=7-14, to=5-14]
	\arrow[from=7-14, to=11-14]
	\arrow[from=7-23, to=5-23]
	\arrow[from=7-23, to=9-23]
	\arrow[from=7-30, to=3-30]
	\arrow[color={rgb,255:red,0;green,4;blue,255}, from=7-30, to=9-30]
	\arrow[from=8-5, to=4-5]
	\arrow[color={rgb,255:red,0;green,4;blue,255}, from=8-5, to=10-5]
	\arrow[from=8-12, to=7-14]
	\arrow[from=8-12, to=10-12]
	\arrow[from=8-34, to=3-34]
	\arrow[from=8-34, to=10-32]
	\arrow[from=9-1, to=4-1]
	\arrow[from=9-1, to=11-3]
	\arrow[from=9-23, to=6-21]
	\arrow[from=9-23, to=11-23]
	\arrow[from=9-30, to=10-32]
	\arrow[from=9-30, to=11-28]
	\arrow[from=10-5, to=11-3]
	\arrow[from=10-5, to=12-7]
	\arrow[from=10-12, to=11-10]
	\arrow[from=10-12, to=11-14]
	\arrow[from=10-21, to=9-23]
	\arrow[from=10-21, to=11-19]
	\arrow[from=10-32, to=6-32]
	\arrow[color={rgb,255:red,0;green,4;blue,255}, from=10-32, to=12-32]
	\arrow[from=11-3, to=7-3]
	\arrow[color={rgb,255:red,0;green,4;blue,255}, from=11-3, to=13-3]
	\arrow[from=11-10, to=5-10]
	\arrow[from=11-10, to=12-12]
	\arrow[from=11-14, to=8-12]
	\arrow[from=11-14, to=12-12]
	\arrow[from=11-14, to=12-16]
	\arrow[from=11-19, to=5-19]
	\arrow[from=11-19, to=12-21]
	\arrow[from=11-23, to=12-21]
	\arrow[from=11-23, to=12-25]
	\arrow[from=11-28, to=5-28]
	\arrow[from=11-28, to=13-30]
	\arrow[from=12-7, to=6-7]
	\arrow[from=12-7, to=14-5]
	\arrow[from=12-12, to=10-12]
	\arrow[from=12-12, to=13-10]
	\arrow[from=12-12, to=13-14]
	\arrow[from=12-16, to=6-16]
	\arrow[from=12-16, to=13-14]
	\arrow[from=12-21, to=10-21]
	\arrow[from=12-21, to=13-19]
	\arrow[from=12-21, to=13-23]
	\arrow[from=12-25, to=6-25]
	\arrow[from=12-25, to=13-23]
	\arrow[from=12-32, to=13-30]
	\arrow[from=12-32, to=13-34]
	\arrow[from=13-3, to=14-1]
	\arrow[from=13-3, to=14-5]
	\arrow[from=13-10, to=11-10]
	\arrow[from=13-10, to=14-12]
	\arrow[from=13-14, to=11-14]
	\arrow[from=13-14, to=14-12]
	\arrow[from=13-14, to=14-16]
	\arrow[from=13-19, to=11-19]
	\arrow[from=13-19, to=14-21]
	\arrow[from=13-23, to=11-23]
	\arrow[from=13-23, to=14-21]
	\arrow[from=13-23, to=14-25]
	\arrow[from=13-30, to=9-30]
	\arrow[from=13-30, to=14-28]
	\arrow[from=13-30, to=14-32]
	\arrow[from=13-34, to=8-34]
	\arrow[from=13-34, to=14-32]
	\arrow[from=14-1, to=9-1]
	\arrow[from=14-1, to=15-3]
	\arrow[from=14-5, to=10-5]
	\arrow[from=14-5, to=15-3]
	\arrow[from=14-5, to=15-7]
	\arrow[from=14-12, to=12-12]
	\arrow[from=14-12, to=15-10]
	\arrow[from=14-12, to=15-14]
	\arrow[from=14-16, to=12-16]
	\arrow[from=14-16, to=15-14]
	\arrow[from=14-21, to=12-21]
	\arrow[from=14-21, to=15-19]
	\arrow[from=14-21, to=15-23]
	\arrow[from=14-25, to=12-25]
	\arrow[from=14-25, to=15-23]
	\arrow[from=14-28, to=11-28]
	\arrow[from=14-28, to=15-30]
	\arrow[from=14-32, to=12-32]
	\arrow[from=14-32, to=15-30]
	\arrow[from=14-32, to=15-34]
	\arrow[from=15-3, to=13-3]
	\arrow[from=15-3, to=16-1]
	\arrow[from=15-3, to=16-5]
	\arrow[from=15-7, to=12-7]
	\arrow[from=15-7, to=16-5]
	\arrow[from=15-10, to=13-10]
	\arrow[from=15-10, to=16-12]
	\arrow[from=15-14, to=13-14]
	\arrow[from=15-14, to=16-12]
	\arrow[from=15-14, to=16-16]
	\arrow[from=15-19, to=13-19]
	\arrow[from=15-19, to=16-21]
	\arrow[from=15-23, to=13-23]
	\arrow[from=15-23, to=16-21]
	\arrow[from=15-23, to=16-25]
	\arrow[from=15-30, to=13-30]
	\arrow[from=15-30, to=16-28]
	\arrow[from=15-30, to=16-32]
	\arrow[from=15-34, to=13-34]
	\arrow[from=15-34, to=16-32]
	\arrow[from=16-1, to=14-1]
	\arrow[from=16-1, to=17-3]
	\arrow[from=16-5, to=14-5]
	\arrow[from=16-5, to=17-3]
	\arrow[from=16-5, to=17-7]
	\arrow[from=16-12, to=14-12]
	\arrow[from=16-16, to=14-16]
	\arrow[from=16-21, to=14-21]
	\arrow[from=16-25, to=14-25]
	\arrow[from=16-28, to=14-28]
	\arrow[from=16-28, to=17-30]
	\arrow[from=16-32, to=14-32]
	\arrow[from=16-32, to=17-30]
	\arrow[from=16-32, to=17-34]
	\arrow[from=17-3, to=15-3]
	\arrow[from=17-7, to=15-7]
	\arrow[from=17-30, to=15-30]
	\arrow[from=17-34, to=15-34]
\end{tikzcd}
}
    \caption{$\Gam(\sfA_4)_{(2,3,2),(10,9,8)}\xlrarrow{\mu_{(2,6)}}\bullet\xlrarrow{\mu_{(2,8)}}\bullet\xlrarrow{\mu_{(3,9)}}\bullet\cong\Gam(\sfA_4)_{(3,2,3),(11,10,9)}$}
    \label{fig:A2-A4}
\end{figure}
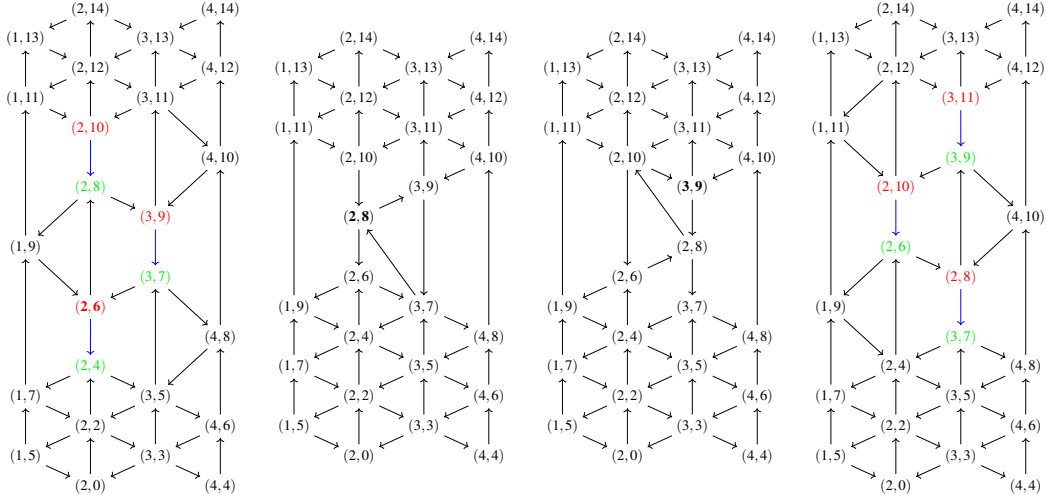
The situation $d=2$ can only occur in type $\sfB_n$ with 
$\{i_k,i_{k+1}\}\cap\{n-1,n\}=\emptyset$. 
This case is very similar to the case (B) below. 

(B) This situation can only occur if $X$ is of type $\sfB_n$ with $\{i_k,i_{k+1}\}=\{n-2,n-1\}$, or in type $\sfF_4$ with $\{i_k,i_{k+1}\}=\{1,2\}$. 
Note that in both cases we have $d=2$.
In view of the behavior of our mutation sequence in Figure~\ref{fig:A2-A4}, we can omit the vertex $h=n-3$ for our analysis in Figure~\ref{fig:A2-B3}. 
\begin{figure}
    \centering
   \scalebox{0.45}{\input{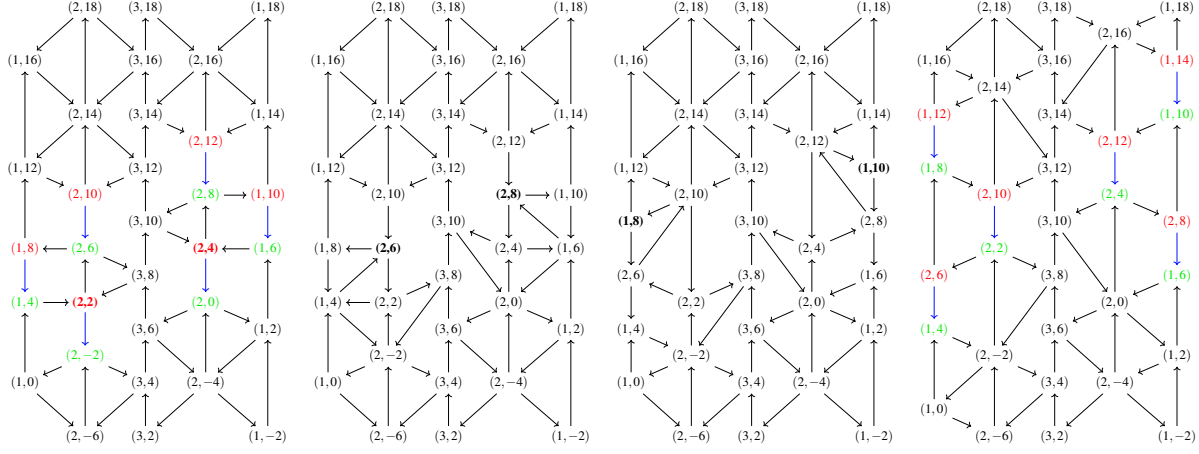}}
    \caption{$\!\Gam(\sfB_3)_{(2,1,2),(12,10,8)}\!\xlrarrow{\mu_{(2,2)}\mu_{(2,4)}}\!\bullet\!\xlrarrow{\mu_{(2,6)}\mu_{(2,8)}}\!\bullet\!\xlrarrow{\mu_{(1,8)}\mu_{(1,10)}}\!\bullet\cong\Gam(\sfB_3)_{(1,2,1),(14,12,10)}$}
    \label{fig:A2-B3}
\end{figure}

(C) Note that this configuration can only occur in a Diagram of type $\sfC_n$ with
$\{i_k,i_{k+1}\}=\{n-2,n-1\}$, or in a Diagram of type $\sfF_4$ with $\{i_k,i_{k+1}\}=\{3,4\}$.
We verify in Figure~\ref{fig:A2-C3} that in this configuration our 3-step mutation sequence has
the desired effect. Again, in view of Figure~\ref{fig:A2-A4} we may omit the vertex $h$.
\begin{figure}
\scalebox{0.5}{\begin{tikzcd}[cramped,sep=tiny]
	&& {(2,12)} &&&&&&&&& {(2,12)} &&&&&&&&& {(2,12)} &&&&&&&&& {(2,12)} &&&& \\
	&&&& {(3,11)} && {(1,12)} &&&&&&& {(3,11)} && {(1,12)} &&&&&&& {(3,11)} && {(1,12)} &&&&&&& {(3,11)} && {(1,12)} \\
	{(1,10)} && {(2,10)} &&&&&&&&& {(2,10)} &&&&&&&&& {(2,10)} &&&&&&&&& {(2,10)} \\
	&&&& {(3,9)} &&&&& {(1,10)} &&&& {(3,9)} &&&&& {(1,10)} &&&& {(3,9)} &&&&& {(1,10)} &&&& {(3,9)} \\
	&& {(2,8)} &&&&&&&&& {(2,8)} &&&&&&&&& {(2,8)} &&&&&&&&& \textcolor{rgb,255:red,255;green,0;blue,8}{{(2,8)}} \\
	&&&& \textcolor{rgb,255:red,255;green,0;blue,8}{{(3,7)}} && {(1,8)} &&&&&&& {(3,7)} && {(1,8)} &&&&&&&&&&&&&&&&&& {(1,8)} \\
	&&&&&&&&&&&&&&&&&&&&&& {(3,7)} && {(1,8)} &&&&& \textcolor{rgb,255:red,5;green,228;blue,1}{{(2,6)}} \\
	&&&& \textcolor{rgb,255:red,5;green,228;blue,1}{{(3,5)}} &&&&&&&&& {\mathbf{(3,5)}} &&&&&&& {\mathbf{(2,6)}} &&&&&&&&&&& \textcolor{rgb,255:red,255;green,0;blue,8}{{(3,7)}} \\
	&& \textcolor{rgb,255:red,255;green,0;blue,8}{{(2,6)}} &&&&&&& {(1,6)} && {(2,6)} &&&&&&& {(1,6)} \\
	{(1,6)} &&&&&&&&&&&&&&&&&&&& {(3,5)} &&&&&&&&&&& \textcolor{rgb,255:red,5;green,228;blue,1}{{(3,3)}} \\
	&& \textcolor{rgb,255:red,5;green,228;blue,1}{{(2,4)}} &&&&&&&&& {(2,4)} &&&&&&&&&&& {(3,3)} &&&&&&& \textcolor{rgb,255:red,255;green,0;blue,8}{{(3,5)}} \\
	&&&& \textcolor{rgb,255:red,255;green,0;blue,8}{{\mathbf{(3,3)}}} && {(1,4)} &&&&&&& {(3,3)} && {(1,4)} &&&&& {(2,4)} &&&&&&& {(1,6)} \\
	&&&&&&&&&&&&&&&&&&&&&& {(3,1)} && {(1,4)} &&&&& \textcolor{rgb,255:red,5;green,228;blue,1}{{(2,4)}} \\
	&&&& \textcolor{rgb,255:red,5;green,228;blue,1}{{(3,1)}} &&&&&&&&& {(3,1)} &&&&&&& {(2,0)} &&&&&&&&&&& {(3,1)} && {(1,4)} \\
	&& {(2,0)} &&&&&&&&& {(2,0)} &&&&&&&&&&& {(3,-1)} &&&&&&& {(2,0)} \\
	&&&& {(3,-1)} &&&&&&&&& {(3,-1)} &&&&& {(1,2)} && {(2,-2)} &&&&&&&&&&& {(3,-1)} \\
	{(1,2)} && {(2,-2)} &&&&&&& {(1,2)} && {(2,-2)} &&&&&&&&&&&&&&&& {(1,2)} && {(2,-2)}
	\arrow[from=1-3, to=2-5]
	\arrow[from=1-3, to=3-1]
	\arrow[from=1-12, to=2-14]
	\arrow[from=1-12, to=4-10]
	\arrow[from=1-21, to=2-23]
	\arrow[from=1-21, to=4-19]
	\arrow[from=1-30, to=2-32]
	\arrow[from=1-30, to=4-28]
	\arrow[from=2-5, to=3-3]
	\arrow[from=2-7, to=3-3]
	\arrow[from=2-14, to=3-12]
	\arrow[from=2-16, to=3-12]
	\arrow[from=2-23, to=3-21]
	\arrow[from=2-25, to=3-21]
	\arrow[from=2-32, to=3-30]
	\arrow[from=2-34, to=3-30]
	\arrow[from=3-1, to=5-3]
	\arrow[from=3-3, to=1-3]
	\arrow[from=3-3, to=4-5]
	\arrow[from=3-3, to=6-7]
	\arrow[from=3-12, to=1-12]
	\arrow[from=3-12, to=4-14]
	\arrow[from=3-12, to=6-16]
	\arrow[from=3-21, to=1-21]
	\arrow[from=3-21, to=4-23]
	\arrow[from=3-21, to=7-25]
	\arrow[from=3-30, to=1-30]
	\arrow[from=3-30, to=4-32]
	\arrow[from=3-30, to=6-34]
	\arrow[from=4-5, to=2-5]
	\arrow[from=4-5, to=5-3]
	\arrow[from=4-10, to=5-12]
	\arrow[from=4-14, to=2-14]
	\arrow[from=4-14, to=5-12]
	\arrow[from=4-19, to=5-21]
	\arrow[from=4-23, to=2-23]
	\arrow[from=4-23, to=5-21]
	\arrow[from=4-28, to=5-30]
	\arrow[from=4-32, to=2-32]
	\arrow[from=4-32, to=5-30]
	\arrow[from=5-3, to=3-3]
	\arrow[from=5-3, to=6-5]
	\arrow[from=5-3, to=10-1]
	\arrow[from=5-12, to=3-12]
	\arrow[from=5-12, to=6-14]
	\arrow[from=5-12, to=9-10]
	\arrow[from=5-21, to=3-21]
	\arrow[from=5-21, to=7-23]
	\arrow[from=5-21, to=9-19]
	\arrow[from=5-30, to=3-30]
	\arrow[draw=blue, from=5-30, to=7-30]
	\arrow[from=6-5, to=4-5]
	\arrow[draw=blue, from=6-5, to=8-5]
	\arrow[from=6-7, to=2-7]
	\arrow[from=6-7, to=9-3]
	\arrow[from=6-14, to=4-14]
	\arrow[from=6-14, to=8-14]
	\arrow[from=6-16, to=2-16]
	\arrow[from=6-16, to=9-12]
	\arrow[from=6-34, to=2-34]
	\arrow[draw=blue, from=6-34, to=5-30]
	\arrow[from=6-34, to=11-30]
	\arrow[from=7-23, to=4-23]
	\arrow[from=7-23, to=8-21]
	\arrow[from=7-23, to=11-23]
	\arrow[from=7-25, to=2-25]
	\arrow[from=7-25, to=8-21]
	\arrow[draw=blue, from=7-30, to=6-34]
	\arrow[from=7-30, to=8-32]
	\arrow[from=7-30, to=12-28]
	\arrow[from=8-5, to=9-3]
	\arrow[from=8-14, to=9-12]
	\arrow[from=8-14, to=12-14]
	\arrow[from=8-21, to=5-21]
	\arrow[from=8-21, to=10-21]
	\arrow[from=8-32, to=4-32]
	\arrow[draw=blue, from=8-32, to=10-32]
	\arrow[from=9-3, to=5-3]
	\arrow[draw=blue, from=9-3, to=11-3]
	\arrow[from=9-10, to=4-10]
	\arrow[from=9-10, to=9-12]
	\arrow[from=9-10, to=15-12]
	\arrow[from=9-12, to=5-12]
	\arrow[from=9-12, to=11-12]
	\arrow[from=9-19, to=4-19]
	\arrow[from=9-19, to=8-21]
	\arrow[from=9-19, to=14-21]
	\arrow[from=10-1, to=3-1]
	\arrow[draw=blue, from=10-1, to=9-3]
	\arrow[from=10-1, to=15-3]
	\arrow[from=10-21, to=7-23]
	\arrow[from=10-21, to=12-21]
	\arrow[from=10-32, to=11-30]
	\arrow[draw=blue, from=11-3, to=10-1]
	\arrow[from=11-3, to=12-5]
	\arrow[from=11-3, to=12-7]
	\arrow[from=11-12, to=8-14]
	\arrow[from=11-12, to=9-10]
	\arrow[from=11-12, to=12-16]
	\arrow[from=11-12, to=14-14]
	\arrow[from=11-23, to=10-21]
	\arrow[from=11-30, to=7-30]
	\arrow[draw=blue, from=11-30, to=13-30]
	\arrow[from=12-5, to=8-5]
	\arrow[draw=blue, from=12-5, to=14-5]
	\arrow[from=12-7, to=6-7]
	\arrow[from=12-7, to=17-3]
	\arrow[from=12-14, to=11-12]
	\arrow[from=12-16, to=6-16]
	\arrow[from=12-16, to=17-12]
	\arrow[from=12-21, to=9-19]
	\arrow[from=12-21, to=13-23]
	\arrow[from=12-21, to=13-25]
	\arrow[from=12-28, to=4-28]
	\arrow[draw=blue, from=12-28, to=11-30]
	\arrow[from=12-28, to=15-30]
	\arrow[from=13-23, to=11-23]
	\arrow[from=13-23, to=14-21]
	\arrow[from=13-25, to=7-25]
	\arrow[from=13-25, to=16-21]
	\arrow[draw=blue, from=13-30, to=12-28]
	\arrow[from=13-30, to=14-32]
	\arrow[from=13-30, to=14-34]
	\arrow[from=14-5, to=15-3]
	\arrow[from=14-14, to=12-14]
	\arrow[from=14-14, to=15-12]
	\arrow[from=14-21, to=12-21]
	\arrow[from=14-21, to=15-23]
	\arrow[from=14-21, to=16-19]
	\arrow[from=14-32, to=10-32]
	\arrow[from=14-32, to=15-30]
	\arrow[from=14-34, to=6-34]
	\arrow[from=14-34, to=17-30]
	\arrow[from=15-3, to=11-3]
	\arrow[from=15-3, to=16-5]
	\arrow[from=15-3, to=17-1]
	\arrow[from=15-12, to=11-12]
	\arrow[from=15-12, to=16-14]
	\arrow[from=15-12, to=17-10]
	\arrow[from=15-23, to=13-23]
	\arrow[from=15-23, to=16-21]
	\arrow[from=15-30, to=13-30]
	\arrow[from=15-30, to=16-32]
	\arrow[from=15-30, to=17-28]
	\arrow[from=16-5, to=14-5]
	\arrow[from=16-5, to=17-3]
	\arrow[from=16-14, to=14-14]
	\arrow[from=16-14, to=17-12]
	\arrow[from=16-19, to=9-19]
	\arrow[from=16-21, to=14-21]
	\arrow[from=16-32, to=14-32]
	\arrow[from=16-32, to=17-30]
	\arrow[from=17-1, to=10-1]
	\arrow[from=17-3, to=15-3]
	\arrow[from=17-10, to=9-10]
	\arrow[from=17-12, to=15-12]
	\arrow[from=17-28, to=12-28]
	\arrow[from=17-30, to=15-30]
\end{tikzcd}
}
\caption{$\Gam(\sfC_3)_{(3,2,3),(7,6,5)}\xlrarrow{\mu_{(3,3)}}\bullet\xlrarrow{\mu_{(3,5)}}\bullet\xlrarrow{\mu_{(2,6)}}
\bullet\cong\Gam(\sfC_3)_{(2,3,2),(8,7,6)}$}
\label{fig:A2-C3}
\end{figure}

(D) Note, that this configuration, with a branching point, can only occur (with $d=1$) in the simply laced cases.  The verification is very similar to type (A) above, and we omit it. 
\end{proof}

\subsection{The  $\sfB_2$-braid relation}
\label{subsec-5.2}
Suppose now that the parabolic subgroup $W(C)_{\{i_k,i_{k+1}\}}$ is isomorphic to $W(\sfB_2)$, and
$(i_{k+2},i_{k+3})=(i_k,i_{k+1})$. Then we may consider the reduced expression $\bi'$, which is obtained
 from $\bi$ by the $\sfB_2$-braid move at $k$ with 
 $(i'_k, i'_{k+1},i'_{k+2},i'_{k+3})=(i_{k+1},i_{k},i_{k+1},i_k)$ and all other letters unchanged. 
 Our aim is here to connect, for an adequately modified $\br'$, the quiver $\Gam_{\bi',\br'}$ to $\Gam_{\bi,\br}$ by an explicit sequence of mutations. More precisely:

\begin{Prop} \label{prp:B2-braid}
Consider a reduced expression $\bi=(i_1,\ldots, i_l)$ of some
element $w\in W(C)$, and  $\br=(r_1,\cdots, r_l)\in\Z^l$  adapted to $\bi$.
\begin{itemize}
\item[(a)] 
Suppose that for some $1\leq k\leq l-3$  we  have
\begin{itemize}
\item $i_k=i_{k+2}$ and $i_{k+1}=i_{k+3}$ ,
\item $c_{i_{k+1},i_k}=-2$,
\item $(r_k,r_{k+1},r_{k+2},r_{k+3})=(r_k,r_k-4,r_k-6,r_k-10)$ and 
\item $r_{k-1}+b_{i_{k-1},i_{k+1}}\geq r_k+6$ where we agree that $r_0=\infty$.
\end{itemize}
Then we have a quiver isomorphism $\Gam_{\bi',\br'}\cong\mu_\bullet(\Gam_{\bi,\br})$, where  $\mu_\bullet$ is a sequence of 23 quiver mutations,  $\bi'$ is obtained from $\bi$
by replacing $(i_k, i_{k+1},i_{k+2},i_{k+3})$ by $(i_{k+1},i_k,i_{k+3},i_{k+2})$
and $\br'$ is obtained from $\br$ by replacing $(r_k,r_{k+1},r_{k+2},r_{k+3})$ by
$(r_p+6, r_{p+1}+8, r_{p+2}+6, r_{p+3}+8)$. In particular, the reduced expression $\bi'$ is obtained from $\bi$ by a braid move of type $\sfB_2$.
\item[(b)]
Suppose that  for some $4\leq k\leq l$ we have
\begin{itemize}
\item $i_{k-2}=i_k$ and $i_{k-3}=i_{k-1}$,
\item $c_{i_{k-1},i_k}=-2$,
\item $(r_{k-3},r_{k-2},r_{k-1},r_{k})=(r_k+8,r_k+6,r_k+2,r_k)$ and 
\item $r_{k}+b_{i_{k},i_{k+1}}\geq r_{k+1}+8$ where we agree that $r_{l+1}=-\infty$.
\end{itemize}
Then we have a quiver isomorphism $\Gam_{\bi',\br'}\cong\mu'_\bullet(\Gam_{\bi,\br})$, where  $\mu'_\bullet$ is a sequence of 23 quiver mutations,
$\bi'$ is obtained from $\bi$
by replacing $(i_{k-3}, i_{k-2},i_{k-1},i_{k})$ by $(i_{k-2},i_{k-3},i_{k},i_{k-1})$
and $\br'$ is obtained from $\br$ by replacing $(r_{k-3},r_{k-2},r_{k-1},r_{k})$ by
$(r_{k-3}-6, r_{k-2}-8, r_{p+2}-6, r_{k}-8)$.
In particular, the reduced expression $\bi'$ is obtained from $\bi$ by a braid move of type $\sfB_2$.
\end{itemize}
\end{Prop}

\begin{proof}
It is sufficient to show part~(a), since part~(b) follows from part~(a) by applying the same mutation sequence, up to permutation and shift of spectral parameters, in opposite order. 

For~(a)  we distinguish several cases, where we note that, under the given hypothesis,
$C$ is either of type $\sfF_4$, of type $\sfB_n$ with $n\geq 2$, or of type $\sfC_n$ with
$n\geq 3$. 

If $C$ is of type $\sfF_4$, this follows readily from  Lemma~\ref{Lem:B2-F4-braid} below, up to a shift of spectral parameters. Moreover, condition on $r_{k-1}$ ensures that $\br'$ is indeed adapted to $\bi'$.

If $C$ is of type $\sfB_2$, under the given hypothesis $w=w_0$, and we can apply Lemma~\ref{Lem:B2-F4-braid} directly up to relabelling, after removing from 
$\Gam(\sfF_4)_{(2,3,2,3),(8,4,3,-3)}$ the $1$- and $4$-columns. 

Now, let $C$ be of type $\sfB_n$ with $n\geq 3$. After removing from 
$\Gam(\sfF_4)_{(2,3,2,3),(8,4,2,-2)}$ the
$1$-column, we obtain $\Gam(\sfB_3)_{(2,3,2,3),(8,4,2,-2)}$ and the same mutation sequence from Lemma~\ref{Lem:B2-F4-braid} works.  After replacing the first components of the vertices of this quiver by $i\mapsto n+i-3$ for $i=1,2,3$ and an adequate shift
of the spectral parameters, we can identify the non-periodic part of this new quiver with a convex subquiver of $\Gam(C)_{\bi,\br}$. Since, according to Lemma~\ref{Lem:B2-F4-braid}, the vertical arrows in the
$1$-column of $\Gam(\sfB_3)_{(2,3,2,3),(8,4,2,-2)}$ are not affected by our mutation sequence, we conclude that our mutation sequence (up to relabelling)  indeed transforms
$\Gam_{\bi,\br}$ into a quiver which is isomorphic to $\Gam_{\bi',\br'}$ in this case.

Finally, if $C$ is of type $\sfC_n$ with $n\geq 3$, we proceed similarly to the case
$\sfB_n$ with $n\geq 3$:  In this case we remove from $\Gam(\sfF_4)_{(2,3,2,3),(8,4,2,-2)}$ the two columns with vertices of the form $(1,r)$ in order to obtain
a quiver which is (up to relabelling) of the form $\Gam(\sfC_3)_{(3,2,3,2),(8,4,2,-2)}$
(note, that in a Dynkin Diagram of type $\sfC_3$ the unique long simple root corresponds to the vertex $3$).  Then we can argue as in the case $\sfB_n$. 
\end{proof}

\subsection{The $\sfB_2$-braid relation in Dynkin type $\sfF_4$} \label{ssec:B2-F4}
Let us consider the Dynkin diagram of type $F_4$ and its unfolding $\widetilde{\sfF}_4$:
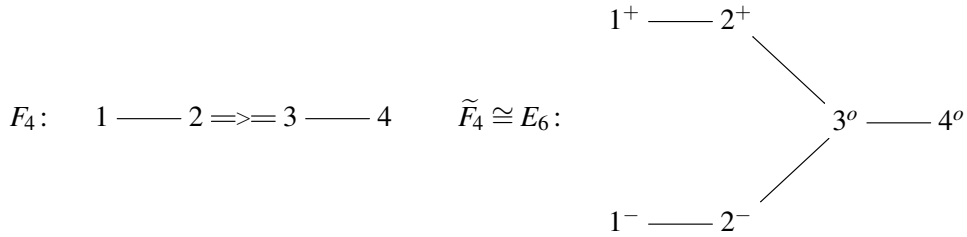
\[
\sfF_4\df\quad \xymatrix{1\ar@{-}[r] &2 \ar@{=}[r]|{>} & 3 \ar@{-}[r] &4}
\qquad\widetilde{\sfF}_4\cong\sfE_6\df\quad
\vcenter{\xymatrix{
1^{+}\ar@{-}[r] & 2^{+}\ar@{-}[rd]\\
&& 3^{o}\ar@{-}[r] &4^{o}\\
1^{-}\ar@{-}[r] & 2^{-}\ar@{-}[ru]
}}
\]
together with the projection $?'\df\{1^{\pm}, 2^{\pm}, 3^{o}, 4^{o}\}\ra \{1,2,3,4\}, i^{*}\mapsto i$. In this section we agree that
$I=\{1,2,3,4\}$ and $C\in\Z^{I\times I}$ is the Cartan matrix corresponding
to our (standard) labeling of the Dynkin Diagram $\sfF_4$.  In particular,
we have $c_{3,2}=-2$ and $D=\diag(2,2,1,1)$ is the minimal left symmetrizer
of $C$. As usual $B:=D\cdot C$. 

Moreover, we may choose the basic quiver $\Gam=\Gam(\sfF_4)$ such that $V=(\{1,2,3\}\times 2\Z)\cup (\{4\}\times (2\Z+1)\Z)$. 

\begin{Lem} \label{Lem:B2-F4-braid}
With the above conventions, consider the quiver $\Gam^{(0)}:=\Gam(\sfF_4)_{(2,3,2,3),(8,4,2,-2)}$,
and the following sequence of recursively defined quivers:
\begin{align*}
\Gam^{(1)} &:=  \mu_{(2,8)}(\Gam^{(0)}), &
\Gam^{(2)} &:=  \mu_{(3,6)}\circ\mu_{(3,4)}\circ\mu_{(2,2)}\circ\mu_{(2,-2)}(\Gam^{(1)}),\\
\Gam^{(3)} &:=  \mu_{(3,2)}(\Gam^{(2)}), &
\Gam^{(4)} &:=  \mu_{(2,-4)}\circ\mu_{(3,0)}\circ\mu_{(3,-2)}\circ\mu_{(3,-4)}(\Gam^{(3)}),\\
\Gam^{(5)} &:=  \mu_{(2,0)}\circ\mu_{(3,8))}(\Gam^{(4)}), &
\Gam^{(6)} &:=  \mu_{(3,12)}\circ\mu_{(3,10)}\circ\mu_{(2,8)}\circ\mu_{(2,4)}(\Gam^{(5)}),\\
\Gam^{(7)} &:=  \mu_{(3,2)}\circ\mu_{(3,8)}(\Gam^{(6)}), &
\Gam^{(8)} &:=  \mu_{(2,6)}\circ\mu_{(3,6)}\circ\mu_{(3,4))}\circ\mu_{(2,2)}(\Gam^{(7)}),\\
\Gam^{(9)} &:=  \mu_{(3,0)}(\Gam^{(8)}).
\end{align*}
Then $\Gam^{(9)}\cong\Gam(\sfF_4)_{(3,2,3,2),(14,12,8,6)}$.  
The relevant parts of those two quivers are displayed in Figure~\ref{fig:6-F4-braid}.
\begin{figure}
    {\scalebox{0.6}{\begin{tikzcd}[cramped,sep=small]
	&& {(2,16)} && {(3,16)} &&& \textcolor{rgb,255:red,128;green,128;blue,128}{{(1,16)}} \\
	\textcolor{rgb,255:red,128;green,128;blue,128}{{(1,14)}} & \textcolor{rgb,255:red,128;green,128;blue,128}{{(1,15)}} \\
	&&&& {(3,14)} && {(2,14)} \\
	& \textcolor{rgb,255:red,128;green,128;blue,128}{{(1,13)}} \\
	&& {(2,12)} && {(3,12)} &&& \textcolor{rgb,255:red,128;green,128;blue,128}{{(1,12)}} \\
	\textcolor{rgb,255:red,128;green,128;blue,128}{{(1,10)}} & \textcolor{rgb,255:red,128;green,128;blue,128}{{(4,11)}} \\
	&&&& {(3,10)} && {(2,10)} \\
	& \textcolor{rgb,255:red,128;green,128;blue,128}{{(4,9)}} & \textcolor{rgb,255:red,255;green,0;blue,17}{{(2,8)}} &&&&& \textcolor{rgb,255:red,128;green,128;blue,128}{{(1,8)}} \\
	&&&& {(3,8)} \\
	& \textcolor{rgb,255:red,128;green,128;blue,128}{{(4,7)}} & \textcolor{rgb,255:red,0;green,255;blue,17}{{(2,4)}} &&&& \textcolor{rgb,255:red,255;green,0;blue,17}{{(2,6)}} \\
	&&&& {(3,6)} \\
	\textcolor{rgb,255:red,128;green,128;blue,128}{{(1,6)}} & \textcolor{rgb,255:red,128;green,128;blue,128}{{(4,5)}} &&&&& \textcolor{rgb,255:red,0;green,255;blue,17}{{(2,2)}} \\
	&&&& \textcolor{rgb,255:red,255;green,0;blue,17}{{(3,4)}} \\
	&& {(2,0)} &&&&& \textcolor{rgb,255:red,128;green,128;blue,128}{{(1,4)}} \\
	&&&& \textcolor{rgb,255:red,0;green,255;blue,17}{{(3,2)}} \\
	\textcolor{rgb,255:red,128;green,128;blue,128}{{(1,2)}} & \textcolor{rgb,255:red,128;green,128;blue,128}{{(4,3)}} &&&&& \textcolor{rgb,255:red,255;green,0;blue,17}{{(2,-2)}} \\
	&&&& {(3,0)} \\
	& \textcolor{rgb,255:red,128;green,128;blue,128}{{(4,1)}} & \textcolor{rgb,255:red,255;green,0;blue,17}{{(2,-4)}} &&&& \textcolor{rgb,255:red,0;green,255;blue,17}{{(2,-6)}} \\
	&&&& {(3,-2)} \\
	& \textcolor{rgb,255:red,128;green,128;blue,128}{{(4,-1)}} & \textcolor{rgb,255:red,0;green,255;blue,17}{{(2,-8)}} &&&&& \textcolor{rgb,255:red,128;green,128;blue,128}{{(1,0)}} \\
	&&&& \textcolor{rgb,255:red,255;green,0;blue,17}{{(3,-4)}} \\
	\textcolor{rgb,255:red,128;green,128;blue,128}{{(1,-2)}} &&&&&& {(2,-10)} \\
	&&&& \textcolor{rgb,255:red,0;green,255;blue,17}{{(3,-6)}} \\
	& \textcolor{rgb,255:red,128;green,128;blue,128}{{(4,-3)}} & {(2,-12)} \\
	&&&& {(3,-8)} &&& \textcolor{rgb,255:red,128;green,128;blue,128}{{(1,-4)}}
	\arrow[draw={rgb,255:red,128;green,128;blue,128}, from=1-3, to=2-1]
	\arrow[from=1-3, to=3-5]
	\arrow[color={rgb,255:red,128;green,128;blue,128}, from=1-5, to=2-2]
	\arrow[from=1-5, to=3-7]
	\arrow[draw={rgb,255:red,128;green,128;blue,128}, from=1-8, to=3-7]
	\arrow[color={rgb,255:red,128;green,128;blue,128}, from=2-1, to=5-3]
	\arrow[draw={rgb,255:red,128;green,128;blue,128}, from=2-2, to=3-5]
	\arrow[from=3-5, to=1-5]
	\arrow[draw={rgb,255:red,128;green,128;blue,128}, from=3-5, to=4-2]
	\arrow[from=3-5, to=5-3]
	\arrow[from=3-7, to=5-5]
	\arrow[draw={rgb,255:red,128;green,128;blue,128}, from=3-7, to=5-8]
	\arrow[draw={rgb,255:red,128;green,128;blue,128}, from=4-2, to=2-2]
	\arrow[draw={rgb,255:red,128;green,128;blue,128}, from=4-2, to=5-5]
	\arrow[from=5-3, to=1-3]
	\arrow[draw={rgb,255:red,128;green,128;blue,128}, from=5-3, to=6-1]
	\arrow[from=5-3, to=7-5]
	\arrow[from=5-5, to=3-5]
	\arrow[draw={rgb,255:red,128;green,128;blue,128}, from=5-5, to=6-2]
	\arrow[from=5-5, to=7-7]
	\arrow[draw={rgb,255:red,128;green,128;blue,128}, from=5-8, to=1-8]
	\arrow[draw={rgb,255:red,128;green,128;blue,128}, from=5-8, to=7-7]
	\arrow[draw={rgb,255:red,128;green,128;blue,128}, from=6-1, to=2-1]
	\arrow[draw={rgb,255:red,128;green,128;blue,128}, from=6-1, to=8-3]
	\arrow[draw={rgb,255:red,128;green,128;blue,128}, from=6-2, to=4-2]
	\arrow[draw={rgb,255:red,128;green,128;blue,128}, from=6-2, to=7-5]
	\arrow[from=7-5, to=5-5]
	\arrow[draw={rgb,255:red,128;green,128;blue,128}, from=7-5, to=8-2]
	\arrow[from=7-5, to=8-3]
	\arrow[from=7-7, to=3-7]
	\arrow[draw={rgb,255:red,128;green,128;blue,128}, from=7-7, to=8-8]
	\arrow[from=7-7, to=9-5]
	\arrow[draw={rgb,255:red,128;green,128;blue,128}, from=8-2, to=6-2]
	\arrow[draw={rgb,255:red,128;green,128;blue,128}, from=8-2, to=9-5]
	\arrow[from=8-3, to=5-3]
	\arrow[draw={rgb,255:red,0;green,8;blue,255}, from=8-3, to=10-3]
	\arrow[draw={rgb,255:red,128;green,128;blue,128}, from=8-8, to=5-8]
	\arrow[draw={rgb,255:red,128;green,128;blue,128}, from=8-8, to=10-7]
	\arrow[from=9-5, to=7-5]
	\arrow[draw={rgb,255:red,128;green,128;blue,128}, from=9-5, to=10-2]
	\arrow[from=9-5, to=10-7]
	\arrow[draw={rgb,255:red,128;green,128;blue,128}, from=10-2, to=8-2]
	\arrow[draw={rgb,255:red,128;green,128;blue,128}, from=10-2, to=11-5]
	\arrow[from=10-3, to=11-5]
	\arrow[draw={rgb,255:red,128;green,128;blue,128}, from=10-3, to=12-1]
	\arrow[from=10-7, to=7-7]
	\arrow[draw={rgb,255:red,0;green,8;blue,255}, from=10-7, to=12-7]
	\arrow[from=11-5, to=9-5]
	\arrow[draw={rgb,255:red,128;green,128;blue,128}, from=11-5, to=12-2]
	\arrow[from=11-5, to=14-3]
	\arrow[draw={rgb,255:red,128;green,128;blue,128}, from=12-1, to=6-1]
	\arrow[draw={rgb,255:red,128;green,128;blue,128}, from=12-1, to=14-3]
	\arrow[draw={rgb,255:red,128;green,128;blue,128}, from=12-2, to=10-2]
	\arrow[draw={rgb,255:red,128;green,128;blue,128}, from=12-2, to=13-5]
	\arrow[from=12-7, to=13-5]
	\arrow[draw={rgb,255:red,128;green,128;blue,128}, from=12-7, to=14-8]
	\arrow[from=13-5, to=11-5]
	\arrow[draw={rgb,255:red,0;green,8;blue,255}, from=13-5, to=15-5]
	\arrow[from=14-3, to=10-3]
	\arrow[from=14-3, to=13-5]
	\arrow[draw={rgb,255:red,128;green,128;blue,128}, from=14-3, to=16-1]
	\arrow[from=14-3, to=17-5]
	\arrow[draw={rgb,255:red,128;green,128;blue,128}, from=14-8, to=8-8]
	\arrow[draw={rgb,255:red,128;green,128;blue,128}, from=14-8, to=16-7]
	\arrow[from=15-5, to=14-3]
	\arrow[draw={rgb,255:red,128;green,128;blue,128}, from=15-5, to=16-2]
	\arrow[from=15-5, to=16-7]
	\arrow[draw={rgb,255:red,128;green,128;blue,128}, from=16-1, to=12-1]
	\arrow[draw={rgb,255:red,128;green,128;blue,128}, from=16-1, to=18-3]
	\arrow[draw={rgb,255:red,128;green,128;blue,128}, from=16-2, to=12-2]
	\arrow[draw={rgb,255:red,128;green,128;blue,128}, from=16-2, to=17-5]
	\arrow[from=16-7, to=12-7]
	\arrow[draw={rgb,255:red,0;green,8;blue,255}, from=16-7, to=18-7]
	\arrow[from=17-5, to=15-5]
	\arrow[draw={rgb,255:red,128;green,128;blue,128}, from=17-5, to=18-2]
	\arrow[from=17-5, to=18-3]
	\arrow[draw={rgb,255:red,128;green,128;blue,128}, from=18-2, to=16-2]
	\arrow[draw={rgb,255:red,128;green,128;blue,128}, from=18-2, to=19-5]
	\arrow[from=18-3, to=14-3]
	\arrow[draw={rgb,255:red,0;green,8;blue,255}, from=18-3, to=20-3]
	\arrow[from=18-7, to=19-5]
	\arrow[draw={rgb,255:red,128;green,128;blue,128}, from=18-7, to=20-8]
	\arrow[from=19-5, to=17-5]
	\arrow[draw={rgb,255:red,128;green,128;blue,128}, from=19-5, to=20-2]
	\arrow[from=19-5, to=22-7]
	\arrow[draw={rgb,255:red,128;green,128;blue,128}, from=20-2, to=18-2]
	\arrow[draw={rgb,255:red,128;green,128;blue,128}, from=20-2, to=21-5]
	\arrow[from=20-3, to=21-5]
	\arrow[draw={rgb,255:red,128;green,128;blue,128}, from=20-3, to=22-1]
	\arrow[draw={rgb,255:red,128;green,128;blue,128}, from=20-8, to=14-8]
	\arrow[draw={rgb,255:red,128;green,128;blue,128}, from=20-8, to=22-7]
	\arrow[from=21-5, to=19-5]
	\arrow[draw={rgb,255:red,0;green,8;blue,255}, from=21-5, to=23-5]
	\arrow[draw={rgb,255:red,128;green,128;blue,128}, from=22-1, to=16-1]
	\arrow[draw={rgb,255:red,128;green,128;blue,128}, from=22-1, to=24-3]
	\arrow[from=22-7, to=18-7]
	\arrow[draw={rgb,255:red,0;green,8;blue,255}, from=22-7, to=21-5]
	\arrow[from=22-7, to=25-5]
	\arrow[draw={rgb,255:red,128;green,128;blue,128}, from=22-7, to=25-8]
	\arrow[draw={rgb,255:red,0;green,8;blue,255}, from=23-5, to=22-7]
	\arrow[draw={rgb,255:red,128;green,128;blue,128}, from=23-5, to=24-2]
	\arrow[from=23-5, to=24-3]
	\arrow[draw={rgb,255:red,128;green,128;blue,128}, from=24-2, to=20-2]
	\arrow[draw={rgb,255:red,128;green,128;blue,128}, from=24-2, to=25-5]
	\arrow[from=24-3, to=20-3]
	\arrow[from=25-5, to=23-5]
	\arrow[draw={rgb,255:red,128;green,128;blue,128}, from=25-8, to=20-8]
\end{tikzcd}
}\quad
    \scalebox{0.6}{\begin{tikzcd}[cramped,sep=small]
	&&&& {(3,16)} \\
	& \textcolor{rgb,255:red,128;green,128;blue,128}{{(4,15)}} & {(2,16)} &&&&& \textcolor{rgb,255:red,128;green,128;blue,128}{{(1,16)}} \\
	&&&& \textcolor{rgb,255:red,250;green,0;blue,12}{{(3,14)}} \\
	\textcolor{rgb,255:red,128;green,128;blue,128}{{(1,14)}} &&&&&& {(2,14)} \\
	&&&& \textcolor{rgb,255:red,21;green,255;blue,0}{{(3,12)}} \\
	& \textcolor{rgb,255:red,128;green,128;blue,128}{{(4,13)}} & \textcolor{rgb,255:red,250;green,0;blue,12}{{(2,12)}} &&&&& \textcolor{rgb,255:red,128;green,128;blue,128}{{(1,12)}} \\
	&&&& {(3,10)} \\
	& \textcolor{rgb,255:red,128;green,128;blue,128}{{(4,11)}} & \textcolor{rgb,255:red,21;green,255;blue,0}{{(2,4)}} &&&& \textcolor{rgb,255:red,250;green,0;blue,12}{{(2,10)}} \\
	&&&& {\mathbf{(2,8)}} \\
	\textcolor{rgb,255:red,128;green,128;blue,128}{{(1,10)}} & \textcolor{rgb,255:red,128;green,128;blue,128}{{(4,9)}} &&&&& \textcolor{rgb,255:red,21;green,255;blue,0}{{(2,6)}} \\
	&&&& \textcolor{rgb,255:red,250;green,0;blue,12}{{(3,8)}} \\
	&& {(2,0)} &&&&& \textcolor{rgb,255:red,128;green,128;blue,128}{{(1,8)}} \\
	&&&& \textcolor{rgb,255:red,21;green,255;blue,0}{{(3,4)}} \\
	\textcolor{rgb,255:red,128;green,128;blue,128}{{(1,6)}} & \textcolor{rgb,255:red,128;green,128;blue,128}{{(4,7)}} &&&&& \textcolor{rgb,255:red,250;green,0;blue,12}{{\mathbf{(3,6)}}} \\
	&&&& {\mathbf{(2,2)}} \\
	& \textcolor{rgb,255:red,128;green,128;blue,128}{{(4,5)}} & \textcolor{rgb,255:red,250;green,0;blue,12}{{(2,-4)}} &&&& \textcolor{rgb,255:red,21;green,255;blue,0}{{(2,-2)}} \\
	&&&& {(3,2)} &&& \textcolor{rgb,255:red,128;green,128;blue,128}{{(1,4)}} \\
	& \textcolor{rgb,255:red,128;green,128;blue,128}{{(4,3)}} & \textcolor{rgb,255:red,21;green,255;blue,0}{{\mathbf{(3,0)}}} \\
	&&&& {(3,-2)} && {(2,-6)} \\
	\textcolor{rgb,255:red,128;green,128;blue,128}{{(1,2)}} & \textcolor{rgb,255:red,128;green,128;blue,128}{{(4,1)}} \\
	&& {(2,-8)} && {(3,-4)} &&& \textcolor{rgb,255:red,128;green,128;blue,128}{{(1,0)}} \\
	& \textcolor{rgb,255:red,128;green,128;blue,128}{{(4,-1)}} \\
	\textcolor{rgb,255:red,128;green,128;blue,128}{{(1,-2)}} &&&& {(3,-6)} && {(2,-10)} \\
	& \textcolor{rgb,255:red,128;green,128;blue,128}{{(4,-3)}} \\
	&& {(2,-12)} && {(3,-8)} &&& \textcolor{rgb,255:red,128;green,128;blue,128}{{(1,-4)}}
	\arrow[draw={rgb,255:red,128;green,128;blue,128}, from=1-5, to=2-2]
	\arrow[from=1-5, to=4-7]
	\arrow[draw={rgb,255:red,128;green,128;blue,128}, from=2-2, to=3-5]
	\arrow[from=2-3, to=3-5]
	\arrow[draw={rgb,255:red,128;green,128;blue,128}, from=2-3, to=4-1]
	\arrow[draw={rgb,255:red,128;green,128;blue,128}, from=2-8, to=4-7]
	\arrow[from=3-5, to=1-5]
	\arrow[draw={rgb,255:red,0;green,4;blue,255}, from=3-5, to=5-5]
	\arrow[draw={rgb,255:red,128;green,128;blue,128}, from=4-1, to=6-3]
	\arrow[from=4-7, to=3-5]
	\arrow[draw={rgb,255:red,128;green,128;blue,128}, from=4-7, to=6-8]
	\arrow[from=4-7, to=7-5]
	\arrow[from=5-5, to=4-7]
	\arrow[draw={rgb,255:red,128;green,128;blue,128}, from=5-5, to=6-2]
	\arrow[from=5-5, to=6-3]
	\arrow[draw={rgb,255:red,128;green,128;blue,128}, from=6-2, to=2-2]
	\arrow[draw={rgb,255:red,128;green,128;blue,128}, from=6-2, to=7-5]
	\arrow[from=6-3, to=2-3]
	\arrow[draw={rgb,255:red,0;green,4;blue,255}, from=6-3, to=8-3]
	\arrow[draw={rgb,255:red,128;green,128;blue,128}, from=6-8, to=2-8]
	\arrow[draw={rgb,255:red,128;green,128;blue,128}, from=6-8, to=8-7]
	\arrow[from=7-5, to=5-5]
	\arrow[draw={rgb,255:red,128;green,128;blue,128}, from=7-5, to=8-2]
	\arrow[from=7-5, to=8-7]
	\arrow[draw={rgb,255:red,128;green,128;blue,128}, from=8-2, to=6-2]
	\arrow[draw={rgb,255:red,128;green,128;blue,128}, from=8-2, to=9-5]
	\arrow[from=8-3, to=9-5]
	\arrow[draw={rgb,255:red,128;green,128;blue,128}, from=8-3, to=10-1]
	\arrow[from=8-7, to=4-7]
	\arrow[draw={rgb,255:red,0;green,4;blue,255}, from=8-7, to=10-7]
	\arrow[from=9-5, to=7-5]
	\arrow[draw={rgb,255:red,128;green,128;blue,128}, from=9-5, to=10-2]
	\arrow[from=9-5, to=12-3]
	\arrow[draw={rgb,255:red,128;green,128;blue,128}, from=10-1, to=4-1]
	\arrow[draw={rgb,255:red,128;green,128;blue,128}, from=10-1, to=12-3]
	\arrow[draw={rgb,255:red,128;green,128;blue,128}, from=10-2, to=8-2]
	\arrow[draw={rgb,255:red,128;green,128;blue,128}, from=10-2, to=11-5]
	\arrow[from=10-7, to=11-5]
	\arrow[draw={rgb,255:red,128;green,128;blue,128}, from=10-7, to=12-8]
	\arrow[from=11-5, to=9-5]
	\arrow[draw={rgb,255:red,0;green,4;blue,255}, from=11-5, to=13-5]
	\arrow[from=12-3, to=8-3]
	\arrow[from=12-3, to=11-5]
	\arrow[from=12-3, to=13-5]
	\arrow[draw={rgb,255:red,128;green,128;blue,128}, from=12-3, to=14-1]
	\arrow[from=12-3, to=15-5]
	\arrow[draw={rgb,255:red,128;green,128;blue,128}, from=12-8, to=6-8]
	\arrow[draw={rgb,255:red,128;green,128;blue,128}, from=12-8, to=14-7]
	\arrow[draw={rgb,255:red,128;green,128;blue,128}, from=13-5, to=14-2]
	\arrow[from=13-5, to=14-7]
	\arrow[draw={rgb,255:red,128;green,128;blue,128}, from=14-1, to=10-1]
	\arrow[draw={rgb,255:red,128;green,128;blue,128}, from=14-1, to=16-3]
	\arrow[draw={rgb,255:red,128;green,128;blue,128}, from=14-2, to=10-2]
	\arrow[draw={rgb,255:red,128;green,128;blue,128}, from=14-2, to=15-5]
	\arrow[from=14-7, to=10-7]
	\arrow[draw={rgb,255:red,0;green,4;blue,255}, from=14-7, to=16-7]
	\arrow[from=15-5, to=13-5]
	\arrow[draw={rgb,255:red,128;green,128;blue,128}, from=15-5, to=16-2]
	\arrow[from=15-5, to=16-3]
	\arrow[draw={rgb,255:red,128;green,128;blue,128}, from=16-2, to=14-2]
	\arrow[draw={rgb,255:red,128;green,128;blue,128}, from=16-2, to=17-5]
	\arrow[from=16-3, to=12-3]
	\arrow[draw={rgb,255:red,0;green,4;blue,255}, from=16-3, to=18-3]
	\arrow[from=16-7, to=17-5]
	\arrow[draw={rgb,255:red,128;green,128;blue,128}, from=16-7, to=17-8]
	\arrow[from=17-5, to=15-5]
	\arrow[draw={rgb,255:red,128;green,128;blue,128}, from=17-5, to=18-2]
	\arrow[from=17-5, to=19-7]
	\arrow[draw={rgb,255:red,128;green,128;blue,128}, from=17-8, to=12-8]
	\arrow[draw={rgb,255:red,128;green,128;blue,128}, from=17-8, to=19-7]
	\arrow[draw={rgb,255:red,128;green,128;blue,128}, from=18-2, to=16-2]
	\arrow[draw={rgb,255:red,128;green,128;blue,128}, from=18-2, to=19-5]
	\arrow[from=18-3, to=19-5]
	\arrow[draw={rgb,255:red,128;green,128;blue,128}, from=18-3, to=20-1]
	\arrow[from=19-5, to=17-5]
	\arrow[draw={rgb,255:red,128;green,128;blue,128}, from=19-5, to=20-2]
	\arrow[from=19-5, to=21-3]
	\arrow[from=19-7, to=16-7]
	\arrow[from=19-7, to=21-5]
	\arrow[draw={rgb,255:red,128;green,128;blue,128}, from=19-7, to=21-8]
	\arrow[draw={rgb,255:red,128;green,128;blue,128}, from=20-1, to=14-1]
	\arrow[draw={rgb,255:red,128;green,128;blue,128}, from=20-1, to=21-3]
	\arrow[draw={rgb,255:red,128;green,128;blue,128}, from=20-2, to=18-2]
	\arrow[draw={rgb,255:red,128;green,128;blue,128}, from=20-2, to=21-5]
	\arrow[from=21-3, to=18-3]
	\arrow[draw={rgb,255:red,128;green,128;blue,128}, from=21-3, to=23-1]
	\arrow[from=21-3, to=23-5]
	\arrow[from=21-5, to=19-5]
	\arrow[draw={rgb,255:red,128;green,128;blue,128}, from=21-5, to=22-2]
	\arrow[from=21-5, to=23-7]
	\arrow[draw={rgb,255:red,128;green,128;blue,128}, from=21-8, to=17-8]
	\arrow[draw={rgb,255:red,128;green,128;blue,128}, from=21-8, to=23-7]
	\arrow[draw={rgb,255:red,128;green,128;blue,128}, from=22-2, to=20-2]
	\arrow[draw={rgb,255:red,128;green,128;blue,128}, from=22-2, to=23-5]
	\arrow[draw={rgb,255:red,128;green,128;blue,128}, from=23-1, to=20-1]
	\arrow[color={rgb,255:red,128;green,128;blue,128}, from=23-1, to=25-3]
	\arrow[from=23-5, to=21-5]
	\arrow[color={rgb,255:red,128;green,128;blue,128}, from=23-5, to=24-2]
	\arrow[from=23-5, to=25-3]
	\arrow[from=23-7, to=19-7]
	\arrow[from=23-7, to=25-5]
	\arrow[color={rgb,255:red,128;green,128;blue,128}, from=23-7, to=25-8]
	\arrow[color={rgb,255:red,128;green,128;blue,128}, from=24-2, to=22-2]
	\arrow[color={rgb,255:red,128;green,128;blue,128}, from=24-2, to=25-5]
	\arrow[from=25-3, to=21-3]
	\arrow[from=25-5, to=23-5]
	\arrow[color={rgb,255:red,128;green,128;blue,128}, from=25-8, to=21-8]
\end{tikzcd}
}}
    \caption{$\Gam(\sfF_4)_{(2323), (8,4,2,-2)}=\Gam^{(0)}$ and $\Gam^{(9)}\cong\Gam_{(3232), (14,12,8,6)}$ (\emph{mutated})}
    \label{fig:6-F4-braid}
\end{figure}
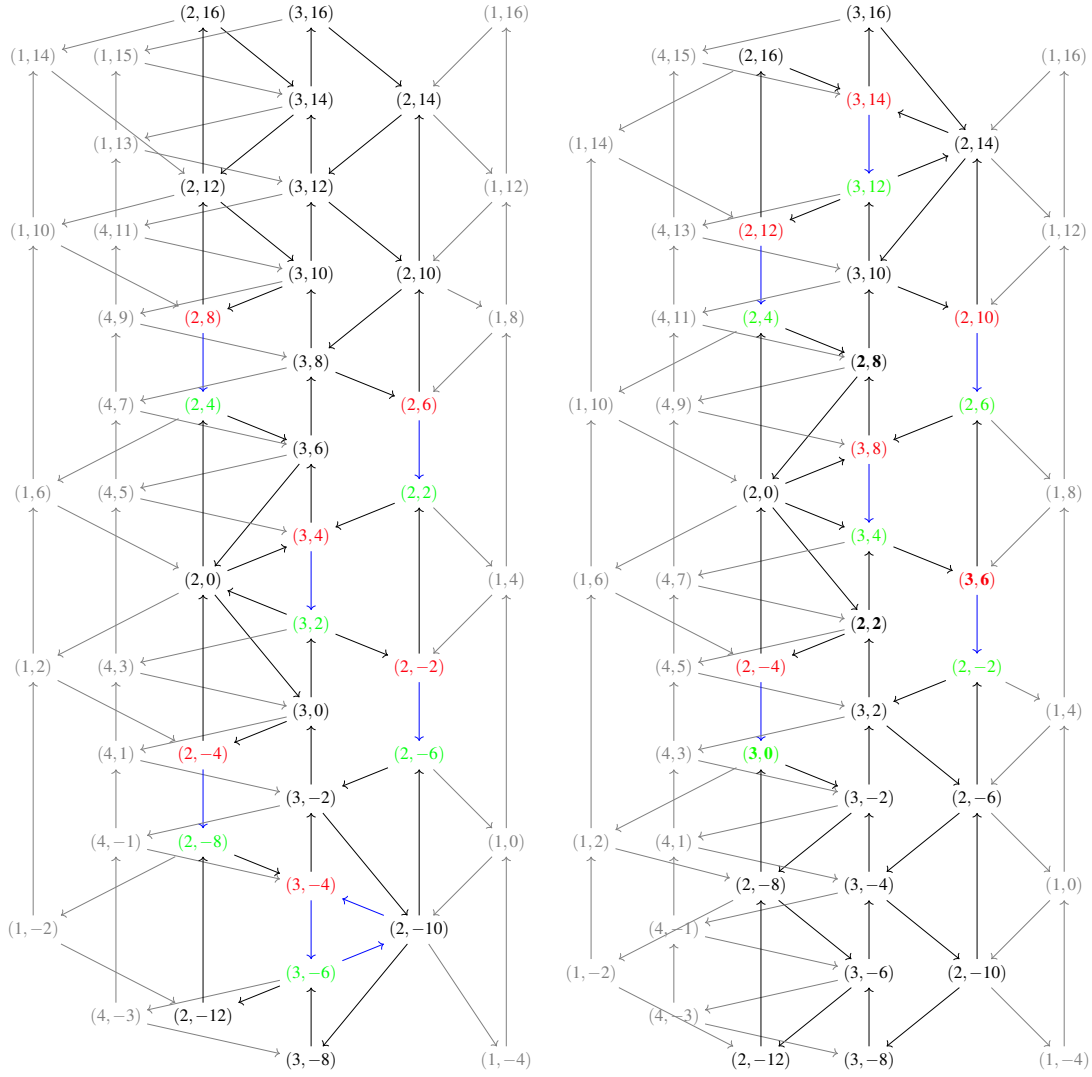
Moreover, under this mutation sequence, at each step 
the vertical (up) arrows in the $1$- and $4$-columns are never changed.
\end{Lem}

\begin{proof}
 The relevant parts of the quivers $\Gam^{(1)}$ to $\Gam^{(4)}$ are displayed in 
 Figure~\ref{fig:B2-mutA}
 \begin{sidewaysfigure}
   \begin{minipage}[t][8cm][b]{27cm}  
  \scalebox{0.55}{\input{Fig6-B2-F4-mutA}}  
 \end{minipage}
\caption{Proof of Lemma~\ref{Lem:B2-F4-braid} - Quivers $\Gam^{(1)} - \Gam^{(4)}$} \label{fig:B2-mutA}
 \end{sidewaysfigure}

 The relevant parts of the quivers $\Gam^{(5)}$ to $\Gam^{(8)}$ are displayed in 
 Figure~\ref{fig:B2-mutB}
 \begin{sidewaysfigure}
   \begin{minipage}[t][8cm][b]{27cm}  
  \scalebox{0.55}{\input{Fig6-B2-F4-mutB}}  
 \end{minipage}
\caption{Proof of Lemma~\ref{Lem:B2-F4-braid} - Quivers $\Gam^{(5)} - \Gam^{(8)}$} \label{fig:B2-mutB}
 \end{sidewaysfigure}
\end{proof}

\begin{Rem} 
(1) Note that $\Gam^{(9)}$ is only \emph{isomorphic}, but not equal
to $\Gam(\sfF_4)_{(3,2,3,2),(14,12,8,6)}$, since under the mutation procedure four vertices, namely $(2,8), (2,2),(3,6)$ and $(3,0)$ “moved" to another column.  Note also that during the mutation procedure several oblique arrows which
connect vertices in columns $1$ and $4$ with vertices in columns $2$ and $3$ change in a nontrivial way. 

(2) The quivers $\Gam^{(1)} - \Gam^{(8)}$ may be thought of as part of a larger
class of quivers associated to reduced expressions for elements of the above-mentioned
unfolding $\tsfF_4\cong\sfE_6$. Moreover, each of the four-step mutations
from $\Gam^{(2k-1)}$ to $\Gam^{(2k)}$ for $k=1,2,3,4$ are related to an 
$\sfA_2$-braid move in the unfolded Weyl group. For example, the four-step mutation sequence connecting 
the quivers 
  $\Gam^{(3)}\cong\tGam_{(2^+,3^o,2^-,3^o,2^+,3^o),(12,8,6,2,0,-2)}$
  and 
  $\Gam^{(4)}\cong\tGam_{(2^+,3^o,2^-,2^+,3^o,2^+).(12,8,6,4,2,\ 0)}$
correspond to an $A_2$-braid move $ (3^o,2^+,3^o)\leftrightarrow (2^+,3^o,2^+)$. In fact, our  strategy to find the above-mentioned 23-step mutation sequence was to first connect two adequate reduced expressions for the longest element 
$w_0$ in the parabolic subgroup  $W(\sfA_3)\cong W_{\{2^+,3^o.2^-\}} W(\tsfF_4)$ by
braid and commutation relations. 

We will explore this kind of observations more systematically in a forthcoming publication.
\end{Rem}

\section{Stabilized g-vectors of cluster variables}\label{stabg}

 The aim of this section is to show that \emph{every} cluster variable $x$ of $\mathcal{A}_w = \mathcal{A}(\Gam_{\bi,\br}, \bz)$ has a stabilized $g$-vector $\bg^\infty(x)$.
 Using this new notation we have in particular $\bg^\infty(z_{(i,r)})=\bg^\infty_{(i,r)}$ for the cluster variables $z_{(i,r)}$ of the initial seed  $\SS_{\bi,\br}$ attached to the quiver $\Gam_{\bi,\br}$.
Moreover, the stabilized 
$g$-vector of $x$ is intrinsic and does not depend on the choice
of the initial  seed $\SS_{\bi,\br}$ used to calculate it,
as long as we work with compatible seeds, see Section~\ref{ssec:compseeds} below.
In particular, $\bg^\infty(x)$ is independent of the choice of reduced expression $\bi$ of $w$. 

\subsection{Calculation of stabilized g-vectors with respect to a standard seed}\label{ssec:csgvect}

Let us start with a reminder on ordinary g-vectors.

\begin{Rem} \label{rem:wdef-gvect}
Recall that g-vectors were introduced in~\cite[Sec.~6]{CA4} as (multi-) degrees  
$\deg(X^{B^0,t_0}_{l,t})$  of the cluster variables in the principal coefficient cluster algebras 
$\AA_\bullet(B^0,\bz)$.  Now, for a skew-symmetric cluster algebra the cluster complex does not depend on the choice of coefficients, see~\cite[Thm.~4.7]{CKLP13}.  
Thus, for a skew-symmetric cluster algebra with trivial coefficients, the g-vector $\bg^\SS(x)$ of a cluster variable $x$ with respect to a seed $\SS$ is well-defined. In particular, it does not depend on the
relative position of $x$ and $\SS$ in the cluster pattern. 
It is not difficult to see that this is also valid for our skew-symmetric cluster algebras of infinite rank. 
\end{Rem}

Let $\bi\in I^l$ be a reduced expression for $w\in W$ with $l=l(w)\geq 1$ and $\br\in\Z^l$ be adapted to 
$\bi$. For $k\in\Z_{\geq 0}$ we set
\[
D_k:=\{\bd=(d_1,\ldots, d_l)\in (2\Z)^l\mid k\leq d_1\leq\cdots \leq d_l\}.
\]
Moreover,  for $k$ even, we set $k\bll:=(k,\ldots, k)\in D_k$. Given $\bd\in D_0$ we define  the sequences of mutations
\begin{align*}
    \mu^{(d_p)}_{(i_p,\rho_p)} &:= \mu_{(i_p,\rho_p-d_p+2)}\circ\cdots\circ \mu_{(i_p,\rho_p-2)}\circ\mu_{(i_p,\rho_p)}\quad\text{and}\\
    \mu_{\bd,\br} &:= \mu^{(d_1)}_{(i_1,\rho_1)}\circ\mu^{( d_2})_{(i_2,\rho_2)}\circ\cdots\circ\mu^{(d_l)}_{(i_l,\rho_l)}. 
\end{align*}
With this notation, we obtain as a slight generalization of Remark~\ref{rem:lowering} 
\[
\mu_{\bd,\br}(\Gam_{\bi,\br})=\Gam_{\bi,\br-\bd}.
\]
Similarly, if $\bv=((j_m, s_m),\ldots (j_2,s_2), (j_1,s_1))\in V^m$ is a sequence of vertices, we set
\begin{align}
\mu_\bv &:= \mu_{(j_m,s_m)}\circ\cdots\circ\mu_{(j_2,s_2)}\circ\mu_{(j_1,s_1)},\quad\text{and} \label{eq:muv}\\
\mu_{\bv-k\bll} &:= \mu_{(j_m,s_m-k)}\circ\cdots\circ\mu_{(j_2,s_2-k)}\circ\mu_{(j_1,s_1-k)}.
\end{align}
Suppose that we have $\mu_\bv(\Gam_{\bi,\br})\cong \Gam_{\bi',\br'}$ for some
reduced expression $\bi'$ of $w$ and $\br'$ adapted to $\bi$. Then we also have
\[
\mu_{\bv -k\bll}(\Gam_{\bi,\br- k\bll})\cong \Gam_{\bi',\br'-k\bll}.
\]
Now, in the cluster algebra $\AA(\Gam_{\bi,\br},\bz)$ with $\bz=(z_{(j,s)})_{(j,s)\in V}$ 
consider the
initial seed $\SS_{\bi,\br}=(\Gam_{\bi,\br},\bz)$, as well as the mutated seeds  
\[
\SS_{\bi,\br-\bd}:= \mu_{\bd,\br}(\SS_{\bi,\br})=(\Gam_{\bi,\br-\bd},\bz^{(-\bd)}),
\qquad (\bd\in D_0).
\]
Note that the mutation sequence $\mu_{\bd,\br}$, when applied to $\SS_{\bi,\br}$, at the level of cluster variables, only involves QQ-relations.  Recall the definition of the stabilized g-vectors $\bg_{(j,s)}^\infty$ from  Remark~\ref{comps1}.

\begin{Lem}
Let $(j,s)\in V$, then we have, with the above notations, the following:
\begin{itemize}
\item[(a)]
There exists $d=d(j,s)\in 2\Z_{\geq 0}$ such that for all $\bd\in D_d$
\[
\bg^\infty_{(j,s)}=\bg^{\SS_{\bi,\br-\bd}}(z_{(j,s)}).
\]
\item[(b)]
More generally, for each $k\in 2\Z_{\geq 0}$ we have
\[
z^{-k}\bg^\infty_{(j,s)}=\bg^{\SS_{\bi,\br-\bd-k\bll}}(z^{(-k\bll)}_{(j,s-k)}), \qquad (\bd\in D_d).
\]
In particular, for $(j,s)\in V_0=V_0(\bi,\br)$ we have $\be_{(j,s-k)}=\bg^{\SS_{\bi,\br-\bd-k\bll}}(z_{(j,s-k)}^{(-k\bll)})$.
\end{itemize}
\end{Lem}

\begin{proof}
(a) This follows, in view of the definition of $\bg^\infty_{(j,s)}$, directly from the proof of Proposition~\ref{prp:key}. 

(b) We observe that the assignation $z_{{j,s}}\mapsto z_{(j,s-k)}^{(-k\bll)}$ for all $(j,s)\in V$
induces a cluster algebra automorphism of $\AA(\Gam_{\bi,\br},\bz)$ which sends $\SS_{\bi,\br}$ to $\SS_{\bi,\br-k\bll}$, and thus sends $\SS_{\bi,\br-\bd}$ to $\SS_{\bi,\br-\bd-k\bll}$ for all $\bd\in D_0$. Using this
this observation, we see that (b) follows from (a). 
\cqfd
\end{proof}

\begin{Lem} \label{lem:lgvect1}
Let $x\in\AA(\Gam_{\bi,\br},\bz)$ be a cluster variable.  Then there exists $k=k(x)\in 2\Z_{\geq 0}$ such that
$\bg^\infty(x):=\bg^{\SS_{\bi,\br-\bd}}(x)$ for all $\bd\in D_k$.
\end{Lem}

\begin{proof}
We may assume that there exists a sequence of vertices $\bv=((j_m,s_m),\ldots, (j_2,s_2))$  in $V$ such that 
$x=z^{(m)}_{(j_m,s_m)}$ for some $(j,s)\in V$ and $\mu_\bv(\SS_{\bi,\br})=:(\Gam^{(m)},\bz^{(m)})$. 

Let $\bv'=((j_{m-1},s_{m-1}),\ldots. (j_2,s_2),(j_1,s_1))$ and $\mu_{\bv'}(\SS_{\bi,\br})=:(\Gam^{(m-1)},\bz^{(m-1)})$. By induction, we may assume that our claim holds for all cluster variables $z^{(m-1)}_{(j,s)}$
with $(j,s)\in V$. Set
\[
\de(\bd):= 
\sum_{\substack{\alp\in \Gam^{(m-1)}_1:\\ s(\alp)=(j_m,s_m)}} \bg^{\SS_{\bi,\br-\bd}}(x^{(m-1)}_{t(\alp)}) 
\;\;-\!\!
\sum_{\substack{\bet\in \Gam^{(m-1)}_1:\\ t(\bet)=(j_m,s_m)}} \bg^{\SS_{\bi,\br-\bd}}(x^{(m-1)}_{s(\bet)})
\]
In view of the mutation rule for g-vectors~\cite[Prp.~6.6]{CA4} and the sign-coherence of c-vectors 
(which follows by~\cite{DWZ2} from~\cite[Prp.~5.6(iii)]{CA4}) we have 
\begin{equation} \label{eq:gvmut}
\bg^{\SS_{\bi,\br-\bd}}(x^{(m)}_{(j_m,s_m)})= -\bg^{\SS_{\bi,\br-\bd}}(x^{(m-1)}_{(j_m,s_m)}) +
\begin{cases}
\sum_{\alp\in \Gam^{(m-1)}_1: s(\alp)=(j_m,s_m)} \bg^{\SS_{\bi,\br-\bd}}(x^{(m-1)}_{t(\alp)}) &
\text{if } \del(\bd)<0,\\
\sum_{\bet\in \Gam^{(m-1)}_1: t(\bet)=(j_m,s_m)} \bg^{\SS_{\bi,\br-\bd}}(x^{(m-1)}_{s(\bet)}) &
\text{if } \de(\bd)>0. 
\end{cases}
\end{equation}
Here, we say that 
\begin{align*}
\del(\bd)>0 &\text{ if } \del(\bd)\in\operatorname{span}_{\Z_{\geq 0}}\{\mathbf{b}^{(\bd)} _{(j,s)}\mid_{(j,s)\in V}\},\text{ and}\\
\del(\bd)<0 &\text{ if } \del(\bd)\in\operatorname{span}_{\Z_{\le 0}}\{\mathbf{b}^{(\bd)}_{(j,s)}\mid_{(j,s)\in V}\},
\end{align*}
where $\mathrm{b}^{(\bd)}_{(j,s)}\in\Z^{(V)}$ is the $(j,s)$-column of the exchange matrix 
$B(\Gam_{\bi,\br-\bd})\in \Z^{(V\times V)}$. 
Now, by induction hypothesis, we can choose $k\in 2\Z_{\geq 0}$ such that for all $\bd\in D_k$ the
summands of $\de(\bd)$ as well as  $\bg^{\SS_{\bi,\br-\bd}}(x^{(m-1)}_{(j_m,s_m)})$  do not depend on $\bd$. 
If necessary, we can increase $k$ further so that the condition on $\del(\bd)$ in equation~\eqref{eq:gvmut}
also becomes independent of $\bd$. 
\cqfd    
\end{proof}

\subsection{Compatible seeds} \label{ssec:compseeds}

Recall the lacing number $t$ defined in Section \ref{ssecCartan}.

\begin{Rem}
Let $\Gam(\bv):=\mu_\bv(\Gam_{\bi,\br})$ for a mutation sequence
$\mu_\bv$ as in~\eqref{eq:muv}, where $\mu_\bv$ is composed of a sequence of braid moves
or shifts as in Section~\ref{mutequiv}. Then, there exists a sequence of mutations
\[
\mu^-_{\Gam(\bv)}= \mu_{w_m}\circ\cdots\circ\mu_{w_2}\circ\mu_{w_1} 
\]
with the following properties:
\begin{itemize}
\item[(1)] 
For $\Gam(\bv,k):=\mu_{w_k}\circ\cdots\mu_{w_1}(\Gam(\bv))$ the vertex $w_{k+1}\in V$ is a sink for the vertical arrows and a source for the oblique arrows
($k=0,1,\ldots,m$ and $\Gam(\bv,0):=\Gam(\bv)$).
\item[(2)]
$\Gam^-(\bv)=\Gam(\bv,m)$ has an arrow $(i,r)\ra(j,s)$ if and only if $\Gam(\bv)$ has an arrow $(i,r+t)\ra (j,s+t)$.
\end{itemize}
For example, with the notation of Section~\ref{ssec:csgvect}, we have 
$\Gam^-_{\bi,\br}=\Gam_{\bi,\br-2d\bll}$ and $\mu^-_{\Gam_{\bi,\br}}=\mu_{2d\bll,\br}$.
On the other hand, for the quiver $\Gam^{(4)}$ from Figure~\ref{fig:B2-mutA} we have
$t=2$, and $\mu_{\Gam^{(4)}}^-$ is determined by the sequence of vertices
\[
\bw=((2,-8),(3,-4),(3,-6),(2,-4),(2,-2),(3,-4),(3,-2),(2,-8),(3,-8),(3,-6)).\]
Similarly, we can define a mutation sequence $\mu^+_{\Gam(\bv)}$ that shifts all arrows
of $\Gam(\bv)$ up by degree $2t$ in such a way that $\mu^+_{\Gam^-(\bv)}=\mu_{\bw'}$
with $\bw'=(w_1,w_2,\ldots, w_m)$ if $\mu^-_{\Gam(\bv)}=\mu_\bw$ for $\bw=(w_m,\ldots, w_2,w_1)$.

\end{Rem}

\begin{Prop} \label{prop:intertw}
Consider the cluster algebra $\AA(\Gam_{\bi,\br},\bz)$ from Section~\ref{ssec:csgvect} 
with initial seed $\SS_{\bi,\br}=(\Gam_{\bi,\br},\bz)$.
Let $\mu_\bv$ be a mutation sequence of the type considered in Section~\ref{mutequiv} such that $\Gam(\bv):=\mu_\bv(\Gam_{\bi,\br})\cong \Gam_{\bi',\br'}$ and set 
$\SS_{\bi',\br'}:=\mu_\bv(\SS_{\bi,\br})$, then we have
\[
\mu^-_{\Gam(\bv)}(\SS_{\bi',\br'})\cong \mu_{\bv-2d\bll}(\SS_{\bi,\br-2t\bll}), 
\]
where we recall that $\SS_{\bi,\br-2t\bll}:=\mu^-_{\Gam_{\bi,\br}}(\SS_{\bi,\br})$.
\end{Prop}
Note, that the quivers $\mu^-_{\Gam(\bv)}\circ\mu_\bv(\Gam_{\bi,\br})$ and
$\mu_{\bv-2t\bll}\circ\mu_{2t\bll,\br}(\Gam_{\bi,\br})$ are in general only
isomorphic, but not equal. Typically, they are related by a (finite cyclic)
permutation of $V$.

\begin{proof}
Recalling that $\SS_{\bi,\br-\bd}=\mu_{\bd,\br}(\SS_{\bi,\br})$ it is easy to see that
\[
\mu^-_{\Gam_{\ii,\br-\be}}(\SS_{\bi,\br-\be})=\mu_{\be,\br-2t\bll}(\SS_{\bi,\br-2t\bll})
\quad\text{for all } \be\in D(0).
\]
Thus, by induction, it is enough to show the claim for the case where $\bi$ and $\bi'$ are related by a braid move. In other words, following the case by case analysis in Section~\ref{mutequiv}, we have to show for an abstract seed $\SS'_{\bi,\br}=(\Gam_{\bi,\br},\bx)$ corresponding to one of the quivers
$\Gam_{\bi,\br}$ considered in Proposition~\ref{prp:A2-braid}, Proposition~\ref{prp:B2-braid} or Appendix~\ref{App:G2-braid}, together
with the corresponding (explicit) mutation sequence $\mu_\bv$ (mentioned in
Proposition~\ref{prp:A2-braid}, Lemma~\ref{Lem:B2-F4-braid},  and Lemma~\ref{lem:G2-key}), that
\begin{equation} \label{eq:ExchangeCycle}
\mu^-_{\mu_\bv(\Gam_{\bi,\br})}\circ\mu_\bv(\SS'_{\bi,\br})\cong 
\mu_{\bv-t\bll}\circ\mu^-_{\Gam_{\bi,\br}}(\SS'_{\bi,\br}).
\end{equation}
Note, that for the braid relations of type $A_2$ and $B_2$, similar to the discussion in Section~\ref{mutequiv}, we may assume that $\bi=(i_1,i_2,i_1)$
resp.\@ $\bi=(i_1,i_2,i_1, i_2)$ embedded into a small list of Cartan matrices
of rank 4 or 5. We will deal with those cases in the next few paragraphs separately.

\medskip
\textbf{Braid relation of type $A_2$}. We deal only with the case
$\Gam_{\bi,\br}=\Gam(A_4)_{(2,3,2),(10,9,8)}$, as illustrated in Figure~\ref{fig:A2-A4}.  
In this situation, $\bv=((3,9),(2,8),(2,6))$,
$\mu^-_{\Gam(\bv)}=\mu_{(3,9)}\circ\mu_{(2,6)}\circ\mu_{(3,7)}$ and $t=1$.
Thus,  Equation~\ref{eq:ExchangeCycle} becomes in this case explicitly
\begin{multline*}    
(\mu_{(3,9)}\mu_{(2,6)}\mu_{(3,7)})\circ(\mu_{(3,9)}\mu_{(2,8)}\mu_{(2,6)})(\SS'_{(2,3,2),(10,9,8)})\\\cong
(\mu_{(3,7)}\mu_{(2,6)}\mu_{(2,4)})\circ(\mu_{(2,8)}\mu_{(3,7)}\mu_{(2,4)})(\SS'_{(2,3,2),(10,9,8)})
\end{multline*}
In order to verify this, we perform the corresponding seed mutations explicitly:
\begin{align*}
x_{(2,6)} = &(x_{(2,4)} x_{(2,8)} + x_{(3,7)} x_{(1,9)})/x_{(2,6)};\\
x^+_{(2,8)} =&(x_{(2,4)} x_{(2,8)} x_{(3,9)} + x_{(2,6)} x_{(3,7)} x_{(2,10)} + x_{(3,7)} x_{(3,9)} x_{(1,9)})/(x_{(2,6)} x_{(2,8)});\\
x_{(3,9)} =& (x_{(2,4)} x_{(2,8)} x_{(3,9)} + x_{(2,6)} x_{(3,7)} x_{(2,10)}) x_{(3,11)} + x_{(3,7)} x_{(3,9)} x_{(3,11)} x_{(1,9)} \\ 
&+x_{(2,6)} x_{(2,8)} x_{(2,10)} x_{(4,10)})/(x_{(2,6)} x_{(2,8)} x_{(3,9)});\\
x^+_{(3,7)} =&(x_{(2,4)} x_{(3,5)} x_{(2,8)} x_{(3,9)} + x_{(3,5)} x_{(2,6)} x_{(3,7)} x_{(2,10)} + x_{(2,4)} x_{(2,6)} x_{(2,8)} x_{(4,8)} \\
&+ x_{(3,5)} x_{(3,7)} x_{(3,9)} x_{(1,9)})/(x_{(2,6)} x_{(3,7)} x_{(2,8)});\\
x^+_{(2,6)}=&(x_{(2,6)} x_{(2,10)} + x_{(3,9)} x_{(1,9)})/x_{(2,8)};\\
x^+_{(3,9)}=&x_{(3,9)};\\
\end{align*}
and similarly:
\begin{align*}
x_{(2,4)}=&(x_{(2,2)} x_{(2,6)} + x_{(3,5)} x_{(1,7)})/x_{(2,4)};\\
x_{(3,7)}=&(x_{(3,5)} x_{(3,9)} + x_{(2,6)} x_{(4,8)})/x_{(3,7)};\\
x^-_{(2,8)}=&(x_{(2,6)} x_{(2,10)} + x_{(3,9)} x_{(1,9)})/x_{(2,8)};\\
x^-_{(2,4)}=&x_{(2,4)};\\
x^-_{(2,6)}=&(x_{(2,4)} x_{(3,5)} x_{(2,8)} x_{(3,9)} + x_{(3,5)} x_{(2,6)} x_{(3,7)} x_{(2,10)} + x_{(2,4)} x_{(2,6)} x_{(2,8)} x_{(4,8)}\\ 
&+ x_{(3,5)} x_{(3,7)} x_{(3,9)} x_{(1,9)})/(x_{(2,6)} x_{(3,7)} x_{(2,8)});\\
x^-_{(3,7)}=&(x_{(2,4)} x_{(2,8)} x_{(3,9)} + x_{(2,6)} x_{(3,7)} x_{(2,10)} + x_{(3,7)} x^-_{(3,9)} x_{(1,9)})/(x_{(2,6)} x_{(2,8)});
\end{align*}
In this case the quivers $\mu^-_{\Gam(\bv)}\circ\mu_\bv(\Gam_{\bi,\br})$ and
$\mu_{\bv-2t\bll}\circ\mu_{2t\bll,\br}(\Gam_{\bi,\br})$ are related by the cyclic permutation $((2,6), (2,8), (3,7)$.
Accordingly, we have
\[
x^+_{(2,6)} =  x^-_{(2,8)};\quad
x^+_{(3,7)} =  x^-_{(2,6)};\quad
x^+_{(2,8)} =  x^-_{(3,7)};
\]
The other cases, $B,C$ and $D$ are quite similar. 

\medskip
\textbf{Braid relation of type $B_2$.}
In this case it is sufficient to show, that the following isomorphisms hold for the
“unfolded $A_2$-type braid-move” mutations in Lemma~\ref{Lem:B2-F4-braid}, namely
\begin{align*}
\mu^-_{\Gam^{(2)}}\circ(\mu_{(3,6)}\mu_{(3,4)}\mu_{(2,2)}\mu_{(2,-2)})(\Gam^{(1)},\bx^{(1)}) 
&\cong (\mu_{(3,2)}\mu_{(3,0)}\mu_{(2,-2)}\mu_{(2,-4)})  \circ\mu^-_{\Gam^{(1)}} (\Gam^{(1)},\bx^{(1)}), \\
\mu^-_{\Gam^{(4)}}\circ(\mu_{(2,-4)}\mu_{(3,0)}\mu_{(3,-2)}\mu_{(3,-4)})(\Gam^{(3)},\bx^{(3)}) 
&\cong  (\mu_{(2,-8)}\mu_{(3,-4)}\mu_{(3,-6)}\mu_{(3,-8)})  \circ\mu^-_{\Gam^{(3)}} (\Gam^{(3)},\bx^{(3)}),  \\
\mu^-_{\Gam^{(6)}}\circ(\mu_{(3,12)}\mu_{(3,10)}\mu_{(2,8)}\mu_{(2,4)})(\Gam^{(5)},\bx^{(5)}) &\cong  (\mu_{(3,8)}\mu_{(3,6)}\mu_{(2,4)}\mu_{(2,0)}) \circ\mu^-_{\Gam^{(5)}} (\Gam^{(5)},\bx^{(5)}),  \\
\mu^-_{\Gam^{(8)}}\circ(\mu_{(2,6)}\mu_{(3,6)}\mu_{(3,4))}\mu_{(2,2)})(\Gam^{(7)},\bx^{(7)}) &\cong  (\mu_{(2,-2)}\mu_{(2,2)}\mu_{(3,2))}\mu_{(3,-2)}) \circ\mu^-_{\Gam^{(7)}} (\Gam^{(7)},\bx^{(7)}), \\
\end{align*}
where we recall that the relevant parts of the quiver $\Gam^{(i)}$ for $1\leq i\leq 8$ are displayed in
Figure~\ref{fig:B2-mutA} and in Figure~\ref{fig:B2-mutB}. 
Note that for these isomorphisms of seeds have all essentially the same structure. So, it comes down to show, for $\bv=((3,6),(3,4),(2,2),(2,-2))$ and $t=2$ the following  isomorphism 
\[
    \mu^-_{\mu_\bv(\Gam^{*})}\circ\mu_\bv(\Gam^{(*)},\bx)\cong \mu_{\bv-4\bll}\circ\mu^-_{\Gam^{(*)}}(\Gam^{(*)},\bx)
\]
for the quiver $\Gam^{(*)}$ which is displayed in Figure~\ref{fig:B2-F4-Abraid}.
\begin{figure}
    \centering
    \scalebox{0.5}{\begin{tikzcd}[cramped,sep=scriptsize]
	{(1,10)} &&& {(3,10)} &&& {(2,10)} &&&&& {(2,8)} &&& {(4,7)} &&& {(1,8)} \\
	&&&& {(4,9)} &&&&&&&&& {(3,6)} &&& {(2,6)} \\
	& {(2,8)} && {(3,8)} &&&& {(1,8)} &&& {(1,6)} &&&& {(4,5)} \\
	&&&& {(4,7)} && \textcolor{rgb,255:red,255;green,0;blue,13}{{(2,6)}} &&&&&&& \textcolor{rgb,255:red,255;green,0;blue,13}{{(3,4)}} \\
	{(1,6)} &&& {(3,6)} &&&&&&&& {(2,4)} \\
	&&&& {(4,5)} && \textcolor{rgb,255:red,14;green,204;blue,0}{{(2,2)}} &&&&&&& \textcolor{rgb,255:red,14;green,204;blue,0}{{(2,2)}} &&&& {(1,4)} \\
	&&& \textcolor{rgb,255:red,255;green,0;blue,13}{{(3,4)}} &&&&&&&&&&& {(4,3)} && \textcolor{rgb,255:red,255;green,0;blue,13}{{(2,-2)}} \\
	& {(2,4)} &&&&&& {(1,4)} &&& {(1,2)} &&& {(3,2)} \\
	&&& \textcolor{rgb,255:red,14;green,204;blue,0}{{(3,2)}} &&&&&&&&&&& {(4,1)} && \textcolor{rgb,255:red,14;green,204;blue,0}{{(2,-6)}} \\
	&&&& {(4,3)} && \textcolor{rgb,255:red,255;green,0;blue,13}{{(2,-2)}} &&&&&&& \textcolor{rgb,255:red,255;green,0;blue,13}{{(3,0)}} &&&& {(1,0)} \\
	{(1,2)} &&& {(3,0)} &&&&&&&& {(2,0)} \\
	&&&& {(4,1)} && \textcolor{rgb,255:red,14;green,204;blue,0}{{(2,-6)}} &&&&&&& \textcolor{rgb,255:red,14;green,204;blue,0}{{(3,-2)}} \\
	& {(2,0)} && {(3,-2)} &&&& {(1,0)} &&& {(1,-2)} &&&& {(4,-1)} \\
	&&&& {(4,-1)} &&&&&&&&& {(3,-4)} &&& {(2,-10)} \\
	{(1,-2)} &&& {(3,-4)} &&& {(2,-10)} &&&&& {(2,-4)} &&& {(4,-3)} &&& {(1,-4)} \\
	&&&&&&&& {} &&&&& {(3,-6)}
	\arrow[from=1-1, to=3-2]
	\arrow[from=1-4, to=2-5]
	\arrow[from=1-4, to=3-2]
	\arrow[from=1-7, to=3-4]
	\arrow[from=1-7, to=3-8]
	\arrow[from=1-12, to=2-14]
	\arrow[from=1-12, to=3-11]
	\arrow[from=1-15, to=2-14]
	\arrow[from=1-18, to=2-17]
	\arrow[from=2-5, to=3-4]
	\arrow[from=2-14, to=3-15]
	\arrow[from=2-14, to=5-12]
	\arrow[from=2-17, to=4-14]
	\arrow[from=2-17, to=6-18]
	\arrow[from=3-2, to=5-1]
	\arrow[from=3-2, to=5-4]
	\arrow[from=3-4, to=1-4]
	\arrow[from=3-4, to=4-5]
	\arrow[from=3-4, to=4-7]
	\arrow[from=3-8, to=4-7]
	\arrow[from=3-11, to=5-12]
	\arrow[from=3-15, to=1-15]
	\arrow[from=3-15, to=4-14]
	\arrow[from=4-5, to=2-5]
	\arrow[from=4-5, to=5-4]
	\arrow[from=4-7, to=1-7]
	\arrow[draw={rgb,255:red,4;green,8;blue,251}, from=4-7, to=6-7]
	\arrow[from=4-14, to=2-14]
	\arrow[draw={rgb,255:red,4;green,8;blue,251}, from=4-14, to=6-14]
	\arrow[from=5-1, to=1-1]
	\arrow[from=5-1, to=8-2]
	\arrow[from=5-4, to=3-4]
	\arrow[from=5-4, to=6-5]
	\arrow[from=5-4, to=8-2]
	\arrow[from=5-12, to=1-12]
	\arrow[draw={rgb,255:red,4;green,8;blue,251}, from=5-12, to=4-14]
	\arrow[from=5-12, to=8-11]
	\arrow[from=5-12, to=8-14]
	\arrow[from=6-5, to=4-5]
	\arrow[from=6-5, to=7-4]
	\arrow[from=6-7, to=7-4]
	\arrow[from=6-7, to=8-8]
	\arrow[draw={rgb,255:red,4;green,8;blue,251}, from=6-14, to=5-12]
	\arrow[from=6-14, to=7-15]
	\arrow[from=6-14, to=7-17]
	\arrow[from=6-18, to=1-18]
	\arrow[from=6-18, to=7-17]
	\arrow[from=7-4, to=5-4]
	\arrow[draw={rgb,255:red,4;green,8;blue,251}, from=7-4, to=9-4]
	\arrow[from=7-15, to=3-15]
	\arrow[from=7-15, to=8-14]
	\arrow[from=7-17, to=2-17]
	\arrow[draw={rgb,255:red,4;green,8;blue,251}, from=7-17, to=9-17]
	\arrow[from=8-2, to=3-2]
	\arrow[draw={rgb,255:red,4;green,8;blue,251}, from=8-2, to=7-4]
	\arrow[from=8-2, to=11-1]
	\arrow[from=8-2, to=11-4]
	\arrow[from=8-8, to=3-8]
	\arrow[from=8-8, to=10-7]
	\arrow[from=8-11, to=3-11]
	\arrow[from=8-11, to=11-12]
	\arrow[from=8-14, to=6-14]
	\arrow[from=8-14, to=9-15]
	\arrow[from=8-14, to=11-12]
	\arrow[draw={rgb,255:red,4;green,8;blue,251}, from=9-4, to=8-2]
	\arrow[from=9-4, to=10-5]
	\arrow[from=9-4, to=10-7]
	\arrow[from=9-15, to=7-15]
	\arrow[from=9-15, to=10-14]
	\arrow[from=9-17, to=10-14]
	\arrow[from=9-17, to=10-18]
	\arrow[from=10-5, to=6-5]
	\arrow[from=10-5, to=11-4]
	\arrow[from=10-7, to=6-7]
	\arrow[draw={rgb,255:red,4;green,8;blue,251}, from=10-7, to=12-7]
	\arrow[from=10-14, to=8-14]
	\arrow[draw={rgb,255:red,4;green,8;blue,251}, from=10-14, to=12-14]
	\arrow[from=10-18, to=6-18]
	\arrow[from=10-18, to=14-17]
	\arrow[from=11-1, to=5-1]
	\arrow[from=11-1, to=13-2]
	\arrow[from=11-4, to=9-4]
	\arrow[from=11-4, to=12-5]
	\arrow[from=11-4, to=13-2]
	\arrow[from=11-12, to=5-12]
	\arrow[draw={rgb,255:red,4;green,8;blue,251}, from=11-12, to=10-14]
	\arrow[from=11-12, to=13-11]
	\arrow[from=11-12, to=14-14]
	\arrow[from=12-5, to=10-5]
	\arrow[from=12-5, to=13-4]
	\arrow[from=12-7, to=13-4]
	\arrow[from=12-7, to=13-8]
	\arrow[draw={rgb,255:red,4;green,8;blue,251}, from=12-14, to=11-12]
	\arrow[from=12-14, to=13-15]
	\arrow[from=12-14, to=14-17]
	\arrow[from=13-2, to=8-2]
	\arrow[from=13-2, to=15-1]
	\arrow[from=13-2, to=15-4]
	\arrow[from=13-4, to=11-4]
	\arrow[from=13-4, to=14-5]
	\arrow[from=13-4, to=15-7]
	\arrow[from=13-8, to=8-8]
	\arrow[from=13-8, to=15-7]
	\arrow[from=13-11, to=8-11]
	\arrow[from=13-11, to=15-12]
	\arrow[from=13-15, to=9-15]
	\arrow[from=13-15, to=14-14]
	\arrow[from=14-5, to=12-5]
	\arrow[from=14-5, to=15-4]
	\arrow[from=14-14, to=12-14]
	\arrow[from=14-14, to=15-12]
	\arrow[from=14-14, to=15-15]
	\arrow[from=14-17, to=9-17]
	\arrow[from=14-17, to=15-18]
	\arrow[from=14-17, to=16-14]
	\arrow[from=15-1, to=11-1]
	\arrow[from=15-4, to=13-4]
	\arrow[from=15-7, to=12-7]
	\arrow[from=15-12, to=11-12]
	\arrow[from=15-15, to=13-15]
	\arrow[from=15-15, to=16-14]
	\arrow[from=15-18, to=10-18]
	\arrow[from=16-14, to=14-14]
\end{tikzcd}
}
    \caption{Quivers $\Gam^{(*)}$ and $\mu^-_{\mu_\bv(\Gam^{(*)})}\circ\mu_\bv(\Gam^{(*)})$}
    \label{fig:B2-F4-Abraid}
\end{figure}
In this case we have 
\[
\mu^-_{\Gam^{(*)}}=\mu_{(2,2))}\mu_{(3,0)}\mu_{(3,2)}\mu_{(2,-6)}\quad\text{and}\quad
\mu^-_{\mu_\bv(\Gam^{*})}= \mu_{((3,4)}\mu_{(3,6)}\mu_{(2,-2)}\mu_{(3,0)}\mu_{(3,2)}.
\]
We carry out the corresponding sequences of seed mutations:
\begin{align*}
x'_{(2,-2)} = &(x_{(2,-6)} x_{(2,2)} + x_{(3,2)} x_{(1,4)})/x_{(2,-2)};\\
x^+_{(2,2)} = &(x_{(2,-6)} x_{(3,4)} x_{(2,2)} + x_{(3,2)} x_{(2,-2)} x_{(2,6)} + x_{(3,2)} x_{(3,4)} x_{(1,4)})/(x_{(2,-2)} x_{(2,2)});\\
x'_{(3,4)}  = &((x_{(2,-6)} x_{(3,4)} x_{(3,6)} x_{(2,2)} + x_{(3,2)} x_{(2,-2)} x_{(3,6)} x_{(2,6)} + x_{(3,2)} x_{(3,4)} x_{(3,6)} x_{(1,4)} \\ &+x_{(2,-2)} x_{(2,4)} x_{(2,2)} x_{(2,6)} x_{4,5})/(x_{(2,-2)} x_{(3,4)} x_{(2,2)});\\
x'_{(3,6)}  = &(x_{(2,-6)} x_{(3,4)} x_{(3,6)} x_{(2,2)} x_{(3,8)} + x_{(3,2)} x_{(2,-2)} x_{(3,6)} x_{(3,8)} x_{(2,6)}\\ & + x_{(3,2)} x_{(3,4)} x_{(3,6)} x_{(3,8)} x_{(1,4)} + x_{(2,-2)} x_{(2,4)} x_{(2,2)} x_{(3,8)} x_{(2,6)} x_{4,5}\\ &+ x_{(2,-2)} x_{(3,4)} x_{(2,2)} x_{(2,8)} x_{(2,6)} x_{(4,7)})/(x_{(2,-2)} x_{(3,4)} x_{(3,6)} x_{(2,2)});\\
x^+_{(3,2)}  := &(x_{(2,-6)} x_{(3,0)} x_{(3,4)} x_{(2,2)} + x_{(3,0)} x_{(3,2)} x_{(2,-2)} x_{(2,6)} + x_{(3,0)} x_{(3,2)} x_{(3,4)} x_{(1,4)}\\& + x_{(2,-6)} x_{(2,-2)} x_{(2,4)} x_{(2,2)} x_{(4,3)})/(x_{(3,2)} x_{(2,-2)} x_{(2,2)});\\
x+_{(3,0)}   = &(x_{(3,-2)} x_{(2,-6)} x_{(3,0)} x_{(3,4)} x_{(2,2)} + x_{(3,-2)} x_{(3,0)} x_{(3,2)} x_{(2,-2)} x_{(2,6)} \\ &+ x_{(2,-6)} x_{(2,0)} x_{(3,2)} x_{(2,-2)} x_{(2,2)} x_{(4,1)} + x_{(3,-2)} x_{(3,0)} x_{(3,2)} x_{(3,4)} x_{(1,4)}\\& + x_{(3,-2)} x_{(2,-6)} x_{(2,-2)} x_{(2,4)} x_{(2,2)} x_{(4,3)})/(x_{(3,0)} x_{(3,2)} x_{(2,-2)} x_{(2,2)});\\
x^+_{(2,-2)} := &(x_{(2,-2)} x_{(2,6)} + x_{(3,4)} x_{(1,4)})/x_{(2,2)};\\
x^+_{(3,6)}  := &(x_{(3,6)};\\
x^+_{(3,4)}  := &x_{(3,4)};\\   
\end{align*}
and
\begin{align*}
    x'_{(2,-6)}  = &(x_{(2,-10)} x_{(2,-2)} + x_{(3,-2)} x_{(1,0)})/x_{(2,-6)};\\
x'_{(3,2)}   = &(x_{(3,0)} x_{(3,4)} + x_{(2,-2)} x_{(2,4)} x_{(4,3)})/x_{(3,2)};\\
x'_{(3,0)}   = &(x_{(3,-2)} x_{(3,0)} x_{(3,4)} + x_{(2,0)} x_{(3,2)} x_{(2,-2)} x_{(4,1)} + x_{(3,-2)} x_{(2,-2)} x_{(2,4)} x_{(4,3)})/(x_{(3,0)} x_{(3,2)});\\
x^-_{(2,2)}  = &(x_{(2,-2)} x_{(2,6)} + x_{(3,4)} x_{(1,4)})/x_{(2,2)};\\
x^-_{(2,-6)} = &x_{(2,-6)};\\
x^-_{(2,-2)} = &(x_{(3,-2)} x_{(2,-6)} x_{(3,0)} x_{(3,4)} x_{(2,2)} + x_{(3,-2)} x_{(3,0)} x_{(3,2)} x_{(2,-2)} x_{(2,6)} \\& + x_{(2,-6)} x_{(2,0)} x_{(3,2)} x_{(2,-2)} x_{(2,2)} x_{(4,1)} + x_{(3,-2)} x_{(3,0)} x_{(3,2)} x_{(3,4)} x_{(1,4)} \\&+ x_{(3,-2)} x_{(2,-6)} x_{(2,-2)} x_{(2,4)} x_{(2,2)} x_{(4,3)})/(x_{(3,0)} x_{(3,2)} x_{(2,-2)} x_{(2,2)});\\
x^-_{(3,0)}  = &(x_{(2,-6)} x_{(3,0)} x_{(3,4)} x_{(2,2)} + x_{(3,0)} x_{(3,2)} x_{(2,-2)} x_{(2,6)} + x_{(3,0)} x_{(3,2)} x_{(3,4)} x_{(1,4)}\\& + x_{(2,-6)} x_{(2,-2)} x_{(2,4)} x_{(2,2)} x_{(4,3)})/(x_{(3,2)} x_{(2,-2)} x_{(2,2)});\\
x^-_{(3,2)} := &(x_{(2,-6)} x_{(3,4)} x_{(2,2)} + x_{(3,2)} x_{(2,-2)} x_{(2,6)} + x_{(3,2)} x_{(3,4)} x_{(1,4)})/(x_{(2,-2)} x_{(2,2)});
\end{align*}
Now, we  observe that the isomorphism between the quivers
$\mu^-_{\mu_\bv(\Gam^{*})}\circ\mu_\bv(\Gam^{(*)},\bx)$ and 
$ \mu_{\bv-4\bll}\circ\mu^-_{\Gam^{(*)}}(\Gam^{(*)},\bx)$ is given by the cyclic permutations of
vertices $((2,2), (3,2),(3,0),(2,-2))$. Accordingly, we find from the above calculations that
\[
x^+_{(2,2)}=x^-_{(3,2)},\quad x^+_{(3,2)}=x^-_{(3,0)},\quad x^+_{(3,0)}=x^-_{(2,-2)},\quad
x^+_{(2,-2)}=x^-_{(2,2)}.
\]
for all remaining vertices we have $x^+_{(i,r)}=x_{(i,r)}=x^-_{(i,r)}$. Thus, the two seeds are indeed isomorphic.

\medskip
{\bf Braid relation of type $G_2$.} This can be dealt with in a way similar to the relation of type~$B_2$.
It is in fact slightly easier because no embedding in a bigger diagram is needed. We leave the details to the reader. 
\cqfd
\end{proof}

\subsection{Stabilized g-vectors do not depend on the reduced expression}
\begin{Def}\label{comps}
Let $\mu_\bv$ be a sequence of mutations of the type considered in Section~\ref{mutequiv} such that
$\mu_\bv(\Gam_{\bi,\br})\cong\Gam_{\bi',\br'}$. Thus, $\bv$ is composed of mutation sequences which correspond either to shift moves, or to braid moves in the unfolded Weyl group. Then, we say that the seed $\SS_{\bi',\br'}:=\mu_\bv(\SS_{\bi,\br})$ is \emph{compatible} with the initial seed 
$\SS_{\bi,\br}$.
\end{Def}

\begin{Rem}
Apriori the seed $\SS_{\bi',\br'}$ might depend on the choice of $\bv$. We will see in Section~\ref{ssec:uniqueCS} that this is not the case as long as $\bv$ is of the above mentioned type.   
\end{Rem}

\begin{Prop}\label{uniquegv}
Let $x\in\AA(\Gam_{\bi,\br},\bz)$ be a cluster variable and $\SS_{\bi',\br'}=\mu_\bv(\SS_{\bi,\br})$
a compatible seed. Then there exists $k=k(x)\in\Z_{\geq 0}$ such that 
\begin{equation} \label{eq:limgvect+}
    \bg^\infty(x):=\bg^{\SS_{\bi,\br-\bd}}(x)=\bg^{\SS_{\bi',\br'-\bd'}}(x) \quad\text{for all } \bd,\bd'\in D_k.
\end{equation}
This holds in particular for the unit cluster variables $u_{(j,s)}:= z_{(j,s)}^{(-k'\bll)}$, where
$k'\in 2\Z_{\geq 0}$ is chosen so that $(j,s+k)\in V_0(\bi,\br)$.  In this case we have 
$\bg^\infty(u_{(j,s)})=\be_{j,s}$.
\end{Prop}

\begin{proof}
By Lemma~\ref{lem:lgvect1} we find $l\in2Z_{\geq 0}$ such that $\bg^\infty(x)=\bg^{\SS_{\bi,\br-\bd}}(x)$
for all $\bd\in D_l$. Thus, we can find $k\geq l$ such that under $\mu_{\bv-k\bll}$, applied to
$\SS_{\bi,\br-k\bll}$, the vector $\bg^\infty(x)$ is not altered under the transformation rule for g-vectors~\eqref{eq:g-transf}. Thus, we have
\[
\bg^\infty(x)=\bg^{\SS_{\bi,\br-k\bll}}(x)=\bg^{\mu_{\bv-k\bll}(\SS_{\bi,\br-k\bll})}(x)=\bg^{\SS_{\bi',\br'-k\bll}}(x),
\]
where the last equality follows by Proposition~\ref{prop:intertw} and Remark~\ref{rem:wdef-gvect}.
The rest is clear. 
\cqfd
\end{proof}


\section{Rings of formal power series and Weyl group action}
\label{sec:rings}

Following \cite{FH2}, we introduce certain rings of formal power series related to the representation theory of  the quantum affine algebra $U_q(\widehat{\g})$ 
equipped with
automorphisms 
generating an action of the Weyl group $W$ of $\g$.
This was already reviewed in \cite{GHL} for the simply-laced types $A,D,E$, but certain formulas are modified for non simply-laced types and 
this is explained in this section. 
We also establish a new sign-coherence result (Theorem \ref{scoherence}) for the extended braid group action.

\subsection{Variables}

We fix $q\in\C^*$ of infinite multiplicative order. 
We denote by $\{\Psi_{i,a} \mid i\in I,a \in \C^*\}$ an infinite set of commuting variables. To every $\la\in P$ we also attach a commutative variable $[\la]$ such that 
\[
 [\la][\mu] = [\la + \mu],\quad (\la,\mu \in P).
\]
Next we introduce for $i\in I$ and $a\in\C^*$ the following Laurent monomials:
\begin{eqnarray}
 Y_{i,a} &:=& [\varpi_i]\frac{\Psi_{i,aq^{-d_i}}}{\Psi_{i,aq^{d_i}}}, 
\\\label{def-tilde-a}
A_{i,a} &:=&
Y_{i,aq^{-d_i}}Y_{i,aq^{d_i}}
\left(\prod_{j :\ c_{ji} = -1}Y_{j,a}
\prod_{j :\ c_{ji} = -2}Y_{j,aq^{-1}}Y_{j,aq}
\prod_{j :\ c_{ji} = -3}Y_{j,aq^{-2}}Y_{j,a}Y_{j,aq^2}\right)^{-1},
\\ \label{def-tilde-Psi}
\widetilde{\Psi}_{i,a}  &:=&  \Psi_{i,a}^{-1}\prod_{j :\ c_{ij}=-1}\Psi_{j,aq^{d_i}}
\prod_{j :\ c_{ij}=-2}\Psi_{j,a}\Psi_{j,aq^{2}}
\prod_{j :\ c_{ij}=-3}\Psi_{j,aq^{-1}}\Psi_{j,aq}\Psi_{j,aq^{3}}.
\end{eqnarray}

\begin{Rem} Note that the conditions on the index $j$ of the products in (\ref{def-tilde-a}) and (\ref{def-tilde-Psi}) involve matrix entries of transposed Cartan matrices, that is, Cartan matrices corresponding to Langlands dual simple finite-dimensional Lie algebras, see \cite[Section 6.2]{FH3}.
\end{Rem}

These Laurent monomials satisfy an important relation which generalizes \cite[Lemma 5.1]{GHL}. 
\begin{Lem}\label{Lem31}
For every $i\in I$ and $a\in \C^*$, we have:
\[
Y_{i,a} A_{i,aq^{-d_i}}^{-1}
=
[\varpi_i - \a_i] \frac{\widetilde{\Psi}_{i,aq^{-3d_i}}}{\widetilde{\Psi}_{i,aq^{-d_i}}}.
\]
\end{Lem}
\begin{proof}
The proof is similar to that of \cite[Lemma 4.15]{HL1} and \cite[Lemma 5.1]{GHL}. First we can check by inspection on the Cartan matrix that if $i\in I$ is such that $d_i > 1$, then $(c_{ij}<0 \Longrightarrow c_{ij} = -1)$. 
Hence, if $c_{ij} < -1$, we have $d_i = 1$. 
It then follows from the definition of $\widetilde{\Psi}_{i,a}$ that
\begin{eqnarray*}
\frac{\widetilde{\Psi}_{i,aq^{-3d_i}}}{\widetilde{\Psi}_{i,aq^{-d_i}}} &=&
\frac{\Psi_{i,aq^{-d_i}}}{\Psi_{i,aq^{-3d_i}}}
\prod_{j :\ c_{ij}=-1} \frac{\Psi_{j,aq^{-2d_i}}}{\Psi_{j,a}}
\prod_{j :\ c_{ij}=-2} \frac{\Psi_{j,aq^{-3}}}{\Psi_{j,aq}}
\prod_{j :\ c_{ij}=-3} \frac{\Psi_{j,aq^{-4}}}{\Psi_{j,aq^2}}
\\
&=& 
[-\varpi_i] Y_{i,aq^{-2d_i}}
\prod_{j :\ c_{ij}<0} \frac{\Psi_{j,aq^{-d_i+b_{ij}}}}{\Psi_{j,aq^{-d_i-b_{ij}}}}. 
\end{eqnarray*}
Now, since $2|b_{ij}| = 2d_j|c_{ji}|$, we have
\[
 \frac{\Psi_{j,aq^{-d_i+b_{ij}}}}{\Psi_{j,aq^{-d_i-b_{ij}}}} = 
 [c_{ji}\varpi_j] Y_{j,aq^{-d_i+b_{ij}+d_j}} Y_{j,aq^{-d_i+b_{ij}+3d_j}} \cdots Y_{j,aq^{-d_i-b_{ij}-d_j}},
\]
where the product on the right-hand side contains $|c_{ji}|$ variables $Y$. The desired equality follows by comparison with the definition of $A_{i,aq^{-d_i}}^{-1}$. 
\cqfd 
\end{proof}

\subsection{The ring $\Pi'$}\label{subsec-Pi}\label{sect3-2}

Let $\YY$ denote the ring of Laurent polynomials in the variables $Y_{i,a}$ over $\Z$:
\[
 \YY := \Z[Y_{i,a}^{\pm1};\ i\in I,\ a\in\C^*].
\]
Let $\M\subset \YY$ denote the multiplicative group of Laurent monomials in the variables $Y_{i,a}$. We denote
by $\omega$ the group homomorphism from $\M$ to $P$ defined by 
\[
 \omega(Y_{i,a}) := \varpi_i, \qquad (i\in I,\ a\in \C^*).
\]
In \cite[Definition 2.1]{FH2}, completions $\widetilde{\YY}_w$ of this ring were introduced for every $w\in W$. The elements of $\widetilde{\YY}_w$ are formal power series of the form
\[
 \sum_{m\in S} a_m\,m,
\]
where $S$ is any subset of $\M$ such that $w(\omega(S))$ is contained in a finite union of cones of $P$ of the form
\[
 C_\la := \la - Q_+, \qquad (\la\in P),
\]
and for any $\la \in P$, the number of $a_m\not = 0$ with $\omega(m) = \la$ is finite.

\begin{example}\label{Exa3-2}
Fix $i\in I$ and $a\in\C^*$, and consider the series:
\begin{eqnarray}
\Si_{i,a}^+ &:=& \sum_{k\ge 0}\, \prod_{0<j\le k} A_{i,aq^{-2(j-1)d_i}}^{-1} 
= 1 + A_{i,a}^{-1} + A_{i,a}^{-1}A_{i,aq^{-2d_i}}^{-1} + A_{i,a}^{-1}A_{i,aq^{-2d_i}}^{-1}A_{i,aq^{-4d_i}}^{-1} +\cdots ,
\\
\Si_{i,a}^- &:=& -\sum_{k > 0}\, \prod_{0<j\le k} A_{i,aq^{2jd_i}} 
= -\left(A_{i,aq^{2d_i}} + A_{i,aq^{2d_i}}A_{i,aq^{4d_i}} + A_{i,aq^{2d_i}}A_{i,aq^{4d_i}}A_{i,aq^{6d_i}} +\cdots\right). 
\end{eqnarray}
Then $\Si_{i,a}^+ \in \widetilde{\YY}_{w}$ for every $w\in W$ such that $w(\a_i)>0$, and 
$\Si_{i,a}^- \in \widetilde{\YY}_{w}$ for every $w\in W$ such that $w(\a_i)<0$. \cqfd
\end{example}

Following \cite[Section 3.7]{FH3}, we extend the rings $\YY$  :
\[
\YY' := \Z[\Psi_{i,a}^{\pm1}, [\la];\ i\in I,\ a\in\C^*,\ \la\in P]\ \supset\ \YY,
\]
and we set
\[
 \Pi' := \bigoplus_{w\in W} \widetilde{\mathcal{Y}}'_{w}\supset\ \Pi.
\]
where
\[\widetilde{\mathcal{Y}}'_{w} = \YY'\otimes_\YY \widetilde{\YY_w}.\]
We denote by $E'_w : \Pi' \to \widetilde{\mathcal{Y}}'_{w}$ the projection morphism  for every $w\in W$. The diagonal embedding $\YY' \to \Pi$ endows $\Pi'$ with the structure of a $\YY'$-algebra.

The following elements of $\Pi'$ introduced in \cite[Definition 2.7, Lemma 2.8]{FH2} play a key role.

\begin{Def} Let $i\in I$, $a\in \C^*$. We define the following invertible element of the ring $\Pi\subset \Pi'$ :
\[
 \Si_{i,a} := (\Si_{i,a}^w)_{w\in W} \in \Pi.
\]
where for  $w\in W$ put: 
\[
\Si_{i,a}^w := \Si_{i,a}^\pm\ \mbox{if} \ \pm w(\a_i) > 0.  
\]
\end{Def}

\subsection{The Weyl group action on $\Pi'$}

Following \cite{FH2, FH3, GHL}, for every $i\in I$ we introduce a ring endomorphism $\widetilde{\Theta}_i = (\widetilde{\Theta}_i^w)_{w\in W}$ of $\Pi$. 
We follow the presentation in \cite[Section 5.4]{GHL}.

For every $i\in I$ we introduce a ring endomorphism $\widetilde{\Theta}_i = (\widetilde{\Theta}_i^w)_{w\in W}$ of $\Pi'$. Each component $\tilde{\Theta_i}^w$ is in fact a ring homomorphism from 
$\widetilde{\mathcal{Y}}'_{w}$ to $\widetilde{\mathcal{Y}}'_{ws_i}$ defined by :
\[
\widetilde{\Theta}_i^w([\la]) = [s_i(\la)],
\qquad
 \widetilde{\Theta}_i^w(\Psi_{j,a}) = 
\left\{ 
\begin{array}{lc}
\Psi_{j,a} & \mbox{ if }j \not = i, \\[2mm]
(1-[-\a_i])\widetilde{\Psi}_{i,aq^{-2d_i}}\Si_{i,aq^{-2d_i}}^{ws_i} & \mbox{ if }j = i,
\end{array}
\right.
\]
where the Laurent monomial $\widetilde{\Psi}_{i,aq^{-2d_i}}$ is defined in Eq.~(\ref{def-tilde-Psi}). 

\begin{Rem} The restriction of $\widetilde{\Theta}_i$ to $\Pi$ is equal to the operator $\Theta_i$ of \cite{FH2}. 
We quote the following important formulas proved in \cite{FH2}:
\begin{equation}\label{form-7}
\Theta_i(A_{i,a}^{-1}) = A_{i,aq^{-2d_i}}\frac{\Si_{i,a}}{\Si_{i,aq^{-4d_i}}},\qquad
\Theta_i(\Si_{i,a}) = 1 - \Si_{i,a} = -A_{i,a}^{-1}\Si_{i,aq^{-2d_i}}, \qquad (i\in I,\a \in \C^*), 
\end{equation}
in which $A_{i,a} \in \YY$ is identified with its diagonal embedding $(A_{i,a})_{w\in W} \in \Pi'$.
\end{Rem}

The following generalizes the Weyl group action defined in \cite{FH2}. These are variations of the action defined in 
\cite[Proposition 3.17]{FH3}.

\begin{Thm}\label{waction}
These formulas combine into a well-defined continuous ring endomorphism $\widetilde{\Theta}_i'$ of $\Pi'$. The assignment $s_i \mapsto \widetilde{\Theta}_i$ defines an action of the Weyl group $W$ on the ring $\Pi'$ by ring automorphisms. 
\end{Thm}

\begin{proof} The first point is proved as for the operators $\Theta_i$ in \cite{FH2}. 
The braid relations satisfied by the $s_i$'s are of the form
\[
(s_is_j)^{m_{ij}} = e, \qquad (i\not = j),
\]
where $m_{ij}=2$ (\resp $3$, $4$, $6$) if 
$c_{ij}c_{ji} = 0$ (\resp $1$, $2$, $3$). They are proved to be satisfied by the $\Theta_i'$ in \cite[Proposition 3.17]{FH3} (see also \cite[Proposition 5.10]{GHL} 
for simply-laced types). The involutive property of the $\Theta_i'$ is obtained as in \cite[Lemma 5.8]{GHL}. Let us give the details as there are slight modifications in the non-simply laced types.
First it follows from Equation (\ref{form-7}) that
\[
 \widetilde{\Theta}_i'(\Si_{i,aq^{-2d_i}}) = \Theta_i(\Si_{i,aq^{-2d_i}}) = -A_{i,aq^{-2d_i}}^{-1}\Si_{i,aq^{-4d_i}}.
\]
On the other hand, for $j\not = i$ we have $\widetilde{\Theta}_i'(\Psi_{j,b}) = \Psi_{j,b}$ for every $b\in\C^*$, therefore
\[
 \widetilde{\Theta}_i'(\widetilde{\Psi}_{i,aq^{-2d_i}}) = \widetilde{\Theta}_i'(\Psi_{i,aq^{-2d_i}}^{-1})
 \Psi_{i,aq^{-2d_i}}\widetilde{\Psi}_{i,aq^{-2d_i}}
 = (1-[-\a_i])^{-1}\widetilde{\Psi}_{i,aq^{-4d_i}}^{-1}\Si_{i,aq^{-4d_i}}^{-1}\Psi_{i,aq^{-2d_i}}\widetilde{\Psi}_{i,aq^{-2d_i}}. 
\]
Hence, using Lemma~\ref{Lem31}, we get 
\begin{eqnarray*}
 \widetilde{\Theta}_i'(\widetilde{\Psi}_{i,aq^{-2d_i}}) &=& (1-[-\a_i])^{-1}[-\a_i+\varpi_i]\Si_{i,aq^{-4d_i}}^{-1}\Psi_{i,aq^{-2d_i}} 
 Y_{i,aq^{-d_i}}^{-1}A_{i,aq^{-2d_i}} \\
 &=& (1-[-\a_i])^{-1}[-\a_i]\Si_{i,aq^{-4d_i}}^{-1}\Psi_{i,a} A_{i,aq^{-2d_i}},
\end{eqnarray*}
and finally, 
\begin{eqnarray*}
(\widetilde{\Theta}_i')^2(\Psi_{i,a}) &=& \widetilde{\Theta}_i'(1-[-\a_i])\widetilde{\Theta}_i'(\widetilde{\Psi}_{i,aq^{-2d_i}}) \widetilde{\Theta}_i'(\Si_{i,aq^{-2d_i}})
\\
&=& (1-[\alpha_i])(1-[-\a_i])^{-1}(-[-\a_i])\,\Psi_{i,a}
\\
&=& \Psi_{i,a}.
\end{eqnarray*}
\cqfd 

\end{proof}

The automorphisms $\widetilde{\Theta}_i'$ are compatible with the action of the ordinary simple reflexions $s_i$ 
(see \cite[Lemma 5.9]{GHL}.

Using Theorem \ref{waction} we can define for every $w\in W$ a ring automorphism $\widetilde{\Theta}_w'$ of $\Pi'$ by 
\begin{equation}\label{eq-9}
 \widetilde{\Theta}_w' := \widetilde{\Theta}_{i_1}'\ldots \widetilde{\Theta}_{i_r}',
\end{equation}
where $w=s_{i_1}\ldots s_{i_r}$ denotes an arbitrary factorization of $w$.

\begin{Rem}\label{Rem3-12}
The following braid group action also extends to $\Pi'$ the action of the operators $\Theta_i$ on $\Pi$. Define an operator $\Theta_i'$ on $\Pi'$ by: 
\[
\Theta_i'([\la]) = [s_i(\la)],
\qquad
\Theta_i'(\Psi_{j,a}) = 
\left\{ 
\begin{array}{lc}
\Psi_{j,a} & \mbox{ if }j \not = i, \\[2mm]
\widetilde{\Psi}_{i,aq^{-2d_i}}\Si_{i,aq^{-2d_i}} & \mbox{ if }j = i.
\end{array}
\right.
\]
The operators $\Theta_i'$ satisfy the braid relations (see \cite[Proposition 3.17]{FH3}), but the $\Theta_i'$ are not involutions (the order of $\Theta_i'$ is not finite).
\end{Rem}

\subsection{Sign-coherence}\label{scs}
Let us recall the extended braid group action \cite[Section 3.2]{FH3} on the ring $\YY'$.
The action is generated by the ring automorphisms\footnote{We denote here by $T_i'$ the operators that are denoted by $\sigma\circ T_i'\circ \sigma$ in \cite{FH3}.} $T_i'$ $(i\in I)$ defined by $T_i'([\lambda]) = [s_i(\lambda)]$ for $\lambda \in P$ and for $a\in\mathbb{C}^*$ by 
$$
T_i'(\Psi_{j,a}) = \Psi_{j,a}\quad if\quad j\neq i
$$ 
and 
$$
T_i'(\Psi_{i,a}) = \tilde{\Psi}_{i,aq^{-2d_i}}.
$$
This can be seen as a multiplicitive notation for the operators $t_i$ defined additively in Section \ref{tbga}.
These operators $T_i'$ extend  the braid group action defined by Chari in \cite{C} on the subring $\YY$ (see the discussion in \cite[Section 3.2]{FH3}).
As these operators satisfy the braid relations, by considering a reduced expression of $w\in W$, $T_w'$ is well-defined for any $w\in W$. Then for any $i\in I$ and $a\in\mathbb{C}^*$ one defines 
\begin{equation}\label{extl}\Psi_{w(\omega_i^\vee), a} := T_w'(\Psi_{i,a}).\end{equation}
The notation is legitimate because it is established in \cite[Lemma 2.25, Section 3.2]{FH3} that $T_w'(\Psi_{i,a})$ depends only on the coweight $w(\omega_i^\vee)$ and the complex number $a$.

\begin{Thm}\label{scoherence} Fix $i,k \in I$, $w\in W$, $a\in \C^*$.
Then the exponents of the variables $\Psi_{k,b}$ occuring in the Laurent monomial $\Psi_{w(\varpi_i^\vee),a}$ have the same sign for every $b\in \mathbb{C}^*$, 
namely, the sign of $(w(\varpi_i^\vee))(\alpha_k)$.
In particular, if $(s_k w)(\varpi_i^\vee) < w(\varpi_i^\vee)$, this sign is non-negative.
\end{Thm}

There is no classical analog of this statement for classical weights, since the point here is that the sign does not depend on the spectral parameter $b\in\mathbb{C}^*$. 
Before proving this theorem
we need some preliminaries on classical weights and some computations in small ranks.

\subsubsection{Preliminaries on classical weights}

\begin{Lem}\label{sign} Assume that $\mathfrak{g}$ has rank $2$ and let $I = \{j,k\}$. Then for any $w \in W$,
if $\varpi_j^\vee$ (\resp $\varpi_k^\vee$) occurs in both $w(\varpi_j^\vee)$ and $w(\varpi_k^\vee)$, then the coefficients
with which it occurs
have the same sign.
\end{Lem}

\begin{proof} This can be treated by a direct computation, but for uniformity we give the following argument. The result is clear if $w = e$ (\resp $w = w_0$)
as we obtain the two fundamental coweights $\varpi_j^\vee$ and $\varpi_k^\vee$ (\resp the two opposite of fundamental coweights $-\varpi_{\nu(j)}^\vee$ and $-\varpi_{\nu(k)}^\vee$). Hence we can consider a reduced expression
$w= s_{i_2}\cdots s_{i_N}$ so that
$w' = s_{i_1}\cdots s_{i_N}$ and $w'' = s_{i_3}\cdots s_{i_N}$ are reduced expressions. Then the coefficient of $\varpi_{i_1}^\vee$ in $w(\varpi_i^\vee)$ is non-negative for any $i$. Similarly, the coefficient of $\varpi_{i_2}^\vee$ in $w''(\varpi_i^\vee)$ is non-negative for any $i$, and the coefficient of $\varpi_{i_2}^\vee$ in $w(\varpi_i^\vee)$ is non-positive for any $i$. \cqfd
\end{proof}

\begin{example} 
In type $A_2$, the images of $(\varpi_1^\vee,\varpi_2^\vee)$ by the Weyl group elements are:
\[
\begin{array}{lll}
e : (\varpi_1^\vee,\varpi_2^\vee) &s_1 : (\varpi_2^\vee-\varpi_1^\vee,\varpi_2^\vee) & s_2 : (\varpi_1^\vee, \varpi_1^\vee-\varpi_2^\vee) 
\\[2mm]
s_2s_1 : (-\varpi_2^\vee,\varpi_1^\vee-\varpi_2^\vee) &s_1s_2 : (\varpi_2^\vee-\varpi_1^\vee,-\varpi_1^\vee) & s_1s_2s_1 : (-\varpi_2^\vee, -\varpi_1^\vee)
\end{array}
\]
In type $B_2$, the images of $(\varpi_1^\vee,\varpi_2^\vee)$ by the Weyl group elements different from $e$ and $w_0$ are:
\[
\begin{array}{lll}
s_1 : (\varpi_2^\vee-\varpi_1^\vee,\varpi_2^\vee) & s_2 : (\varpi_1^\vee, 2\varpi_1^\vee-\varpi_2^\vee) &
s_2s_1 : (-\varpi_2^\vee+\varpi_1^\vee,2\varpi_1^\vee-\varpi_2^\vee) \\[2mm]
s_1s_2 : (\varpi_2^\vee-\varpi_1^\vee,\varpi_2^\vee-2\varpi_1^\vee) & s_1s_2s_1 : (-\varpi_1^\vee,-2\varpi_1^\vee + \varpi_2^\vee) & 
s_2s_1s_2 : (-\varpi_2^\vee+\varpi_1^\vee,-\varpi_2^\vee)
\end{array}
\]
\end{example}

\subsubsection{Rank $1$ and rank $2$}\label{reduit}

The explicit formulas for the $\ell$-weights $\Psi_{w(\varpi_i^\vee),1}$ in rank $1$ and $2$ are as follows. 
We can directly check Theorem \ref{scoherence} in these cases.

\medskip
Type $A_1$.
$$\Psi_{\varpi_1^\vee,1} = \Psi_{1,1},\quad \Psi_{-\varpi_1^\vee,1} = \Psi_{1,q^{-2}}^{-1}.$$

Type $A_1\times A_1$
$$\Psi_{\varpi_1^\vee,1} = \Psi_{1,1},\quad \Psi_{-\varpi_1^\vee,1} = \Psi_{1,q^{-2}}^{-1},\quad \Psi_{\varpi_2^\vee,1} = \Psi_{2,1}, \quad\Psi_{-\varpi_2^\vee,1} = \Psi_{2,q^{-2}}^{-1}.$$

Type $A_2$
$$\Psi_{\varpi_1^\vee,1} = \Psi_{1,1},\quad \Psi_{\varpi_2^\vee-\varpi_1^\vee,1} = \Psi_{2,q^{-1}}\Psi_{1,q^{-2}}^{-1}, \quad \Psi_{-\varpi_2^\vee,1} = \Psi_{2,q^{-3}}^{-1},$$
$$\Psi_{\varpi_2^\vee,1} = \Psi_{2,1},\quad \Psi_{\varpi_1^\vee-\varpi_2^\vee,1} = \Psi_{1,q^{-1}}\Psi_{2,q^{-2}}^{-1}, \quad\Psi_{-\varpi_1^\vee,1} = \Psi_{1,q^{-3}}^{-1}.$$

Type $B_2$

\[
\begin{array}{ccc}
 \Psi_{\varpi_1^\vee,1} = \Psi_{1,1}, & \Psi_{\varpi_2^\vee-\varpi_1^\vee,1} = \Psi_{2,q^{-2}}\Psi_{1,q^{-4}}^{-1}, &
 \Psi_{-\varpi_2^\vee + \varpi_1^\vee,1} = \Psi_{1,q^{-2}}\Psi_{2,q^{-4}}^{-1},
 \\[4mm]
 \Psi_{-\varpi_1^\vee ,1} = \Psi_{1,q^{-6}}^{-1},& \Psi_{\varpi_2^\vee,1} = \Psi_{2,1},& 
\Psi_{2\varpi_1^\vee-\varpi_2^\vee,1} = \Psi_{1,1}\Psi_{1,q^{-2}}\Psi_{2,q^{-2}}^{-1}, 
  \end{array}
\]
\[
\Psi_{-2\varpi_1^\vee+\varpi_2^\vee,1} = \Psi_{1,q^{-4}}^{-1}\Psi_{1,q^{-6}}^{-1}\Psi_{2,q^{-4}},
\qquad \Psi_{-\varpi_2^\vee,1} = \Psi_{2,q^{-6}}^{-1}.
\]

Type $G_2$
$$\Psi_{\varpi_1^\vee,1} = \Psi_{1,1},\quad \Psi_{\varpi_2^\vee-\varpi_1^\vee,1} = \Psi_{2,q^{-3}}\Psi_{1,q^{-6}}^{-1}, \quad \Psi_{-\varpi_2^\vee + 2\varpi_1^\vee,1} = \Psi_{1,q^{-2}}\Psi_{1,q^{-4}}\Psi_{2,q^{-5}}^{-1},$$ 
$$\Psi_{\varpi_2^\vee - 2\varpi_1^\vee,1} = \Psi_{1,q^{-8}}^{-1}\Psi_{1,q^{-10}}^{-1}\Psi_{2,q^{-7}},\quad 
\Psi_{-\varpi_2^\vee + \varpi_1^\vee,1} = \Psi_{2,q^{-9}}^{-1}\Psi_{1,q^{-6}}, 
\quad \Psi_{-\varpi_1^\vee,1} = \Psi_{1,q^{-12}}^{-1}.$$
$$\Psi_{\varpi_2^\vee,1} = \Psi_{2,1},\qquad \Psi_{3\varpi_1^\vee-\varpi_2^\vee,1} = \Psi_{1,q}\Psi_{1,q^{-1}}\Psi_{1,q^{-3}}\Psi_{2,q^{-2}}^{-1}, 
$$
$$
\Psi_{-3\varpi_1^\vee + 2\varpi_2^\vee,1} = \Psi_{1,q^{-5}}^{-1}\Psi_{1,q^{-7}}^{-1}\Psi_{1,q^{-9}}^{-1}\Psi_{2,q^{-4}}\Psi_{2,q^{-6}},
$$
$$
\Psi_{-2\varpi_2^\vee + 3\varpi_1^\vee,1} = \Psi_{1,q^{-3}}\Psi_{1,q^{-5}}\Psi_{1,q^{-7}}\Psi_{2,q^{-6}}^{-1}\Psi_{2,q^{-8}}^{-1}, \quad 
$$
$$
\Psi_{\varpi_2^\vee - 3\varpi_1^\vee,1} = \Psi_{1,q^{-9}}^{-1}\Psi_{1,q^{-11}}^{-1}\Psi_{1,q^{-13}}^{-1}\Psi_{2,q^{-10}}, \quad 
\Psi_{-\varpi_2^\vee,1} = \Psi_{2,q^{-12}}^{-1}.$$

\subsubsection{Proof of Theorem \ref{scoherence}}

First, it suffices to prove Theorem \ref{scoherence} for $a = 1$. 
Since the result is clear for $\mathfrak{g} = \mathfrak{sl}_2$ by \S\ref{reduit}, we may assume that $\mathfrak{g}$ has rank at least $2$.
Let $i,k\in I$ and $w\in W$. We will prove the following sign-coherence property: 
\[
\mathcal{S}_k(w(\varpi_i^\vee)):\   
\mbox{The exponents of the $\Psi_{k,b}\ (b\in\C^*)$ occurring in $\Psi_{w(\varpi_i^\vee),1}$ all have the same sign.}
\]
By \cite[Section 6.2]{FH3}, the sum of the exponents of the $\Psi_{k,b}$ occurring in $\Psi_{w(\varpi_i^\vee),1}$ is 
equal to $(w(\varpi_i^\vee))(\alpha_k)$. So, if $\mathcal{S}_k(w(\varpi_i^\vee))$ is true, this common sign is the sign of $(w(\varpi_i^\vee))(\alpha_k)$.
Let us fix $i\in I$. We will prove $\mathcal{S}_k(w(\varpi_i^\vee))$ for all $k\in I$ and $w\in W$ by induction on the height of $\varpi_i^\vee - w(\varpi_i^\vee)$. 

First $\mathcal{S}_k(\varpi_i^\vee)$ is clear for all $k\in I$, because $\Psi_{\varpi_i^\vee,1} = \Psi_{i,1}$. 
Suppose that $\mathcal{S}_k(w(\varpi_i^\vee))$ is proved for every $k\in I$, and let $j\in I$ be such that $(s_jw)(\varpi_i^\vee) < w(\varpi_i^\vee)$. We want to prove $\mathcal{S}_k((s_jw)(\varpi_i^\vee))$ for every $k\in I$, that is the sign-coherence for $\Psi_{(s_jw)(\varpi_i^\vee),1} = T_j'(\Psi_{w(\varpi_i^\vee),1})$.

If $k = j$, then $\mathcal{S}_k(w(\varpi_i^\vee))$ implies $\mathcal{S}_k(s_kw(\varpi_i^\vee))$ by definition of $T_k'$.
The same argument works if $w(\varpi_i^\vee)(\alpha_k) \geq 0$ or if $c_{kj} = 0$.

The remaining critical case is $c_{kj} < 0$ and $w(\varpi_i^\vee)(\alpha_k) < 0$. Indeed, in this case $\Psi_{w(\varpi_i^\vee),1}$ 
has factors $\Psi_{k,b}$ with negative exponents which are multiplied by factors $\Psi_{k,c}$ with positive exponents 
when $T_j'$ is applied to $\Psi_{w(\varpi_i^\vee),1}$ to obtain $\Psi_{s_jw(\varpi_i^\vee),1}$.
But, there is a Weyl group element $w''$ in the subgroup $\left\<s_j,s_k\right\>$ generated by $s_j$ and $s_k$ such that, if we write $w'= w'' w$, 
the weight $w'(\varpi_i)$ satisfies
$$\varpi_i > w'(\varpi_i) > w(\varpi_i)$$
and is dominant with respect to $j$ and $k$ (that is, the coefficients of $\varpi_j^\vee$ and $\varpi_k^\vee$ in $w'(\varpi_i)$ are non-negative). 
By the induction hypothesis for  $w'(\varpi_i^\vee)$, we have $\mathcal{S}_k(w'(\varpi_i^\vee))$ and $\mathcal{S}_j(w'(\varpi_i^\vee))$.
In particular, the exponents of the variables $\Psi_{k,c}$, $\Psi_{j,c}$ in $\Psi_{w'(\varpi_i^\vee),a}$ are non-negative. Hence, if we define $$\pi : \mathcal{Y}\rightarrow \mathbb{Z}[\Psi_{j,c}^{\pm 1}, \Psi_{k,c}^{\pm 1}]_{c\in\mathbb{C}^*}$$
by $\pi(\Psi_{j,c}) = \Psi_{j,c}$, $\pi(\Psi_{k,c}) = \Psi_{k,c}$ and $\pi(\Psi_{p,c}) = 1$ for $c\in\mathbb{C}^*$ and $p\in I\setminus \{j, k\}$, the monomial 
$$\Psi := \pi(\Psi_{w'(\varpi_i^\vee),a})$$ 
has only positive exponents.

Using $\pi$, we reduce to the rank 2 Lie algebra $\tilde{\mathfrak{g}} = \mathfrak{g}_{j,k}$ corresponding to the sub-Dynkin diagram associated to the vertices 
corresponding to $j$ and $k$. 
We regard $s_jw''$ as an element $\tilde{w}$ of the Weyl group of $\tilde{\mathfrak{g}}$. Let $\tilde{T}_j$, $\tilde{T}_k$, 
$\tilde{T}_{\tilde{w}}$ be the extended braid group operators as in \cite{FH3} but for $\tilde{\mathfrak{g}}$. We notice that 
$$\pi\circ T_j' = \tilde{T}_j\circ \pi,\quad \pi \circ T_k' = \tilde{T}_k\circ \pi.$$
Hence, we have 
$$\pi(\Psi_{ws_j(\varpi_i^\vee),1}) = \tilde{T}_{\tilde{w}}(\Psi).$$
Hence, it suffices to prove the sign-coherence for $\tilde{T}_{\tilde{w}}(\Psi)$. 
We factorize 
$$\Psi = (\Psi_{j,a_1}\Psi_{j,a_2}\cdots \Psi_{j,a_N})(\Psi_{k,b_1}\Psi_{k,b_2}\cdots \Psi_{k,b_M})$$
and so 
$$\tilde{T}_{\tilde{w}}(\Psi) =  \tilde{T}_{\tilde{w}}(\Psi_{j,a_1})\tilde{T}_{\tilde{w}}(\Psi_{j,a_2})\cdots \tilde{T}_{\tilde{w}}(\Psi_{j,a_N})\tilde{T}_{\tilde{w}}(\Psi_{k,b_1})\tilde{T}_{\tilde{w}}(\Psi_{k,b_2})\cdots \tilde{T}_{\tilde{w}}(\Psi_{k,b_M}).$$
By Section \ref{reduit}, we have the sign-coherence for each factor. 

Now, By Lemma \ref{sign}, the sign is the same for all factors. The result follows. \cqfd

\section{$QQ$-systems}
\label{sec:QQ}

Using the Weyl group action, we introduce distinguished elements $\QQ_{w(\varpi^\vee_i),\,a}$ 
of the ring $\Pi'$ satisfying a system of functional relations called the $QQ$-system.  We also establish a general uniqueness result of its set of solutions (Theorem \ref{uniqqsys}) which is valid in all completions $\widetilde{\mathcal{Y}}'_{w}$.

\subsection{Normalization factors $\chi_{w(\varpi_i)}$}

We first define some normalization factors $\chi_{w(\varpi_i)}\ (i\in I,\ w\in W)$ indexed by the extremal weights $w(\varpi_i)$ of the fundamental $\g$-modules (they are well-defined by \cite[Lemma 5.3]{FH3}).

\begin{Def}
There is a unique family $\chi_{w(\varpi_i)}$ of elements of $\Pi'$ defined by induction by the following conditions:
\begin{enumerate}
 \item[(i)] for every $i\in I$, $\chi_{\varpi_i} = 1$;
 \item[(ii)] for every $i\in I$ and $w\in W$ such that $ws_i > w$,
 \[
   \chi_{ws_i(\varpi_i)}\chi_{w(\varpi_i)} = \frac{\prod_{j\not = i} \chi_{w(\varpi_j)}^{|c_{ij}|}}{1 - [-w(\alpha_i)]}.
 \]
%
\end{enumerate}
\end{Def}

\begin{example}
In type $B_2$ we have:
\[
 \chi_{\varpi_1} = 1, \quad \chi_{s_1(\varpi_1)} = \frac{1}{1-[-\alpha_1]},\quad
 \chi_{s_2s_1(\varpi_1)} = \frac{1}{(1-[-\alpha_1-2\alpha_2])(1-[-\alpha_2])},
\]
\[
 \chi_{s_1s_2s_1(\varpi_1)} = \frac{1}{(1-[-\alpha_1-2\alpha_2])(1-[-\alpha_1])(1-[-\alpha_1-\alpha_2])},
\]
\[
 \chi_{\varpi_2} = 1, \quad \chi_{s_2(\varpi_2)} = \frac{1}{1-[-\alpha_2]},\quad
 \chi_{s_1s_2(\varpi_2)} = \frac{1}{(1-[-\alpha_1-\alpha_2])(1-[-\alpha_1])^2},
\]
\[
 \chi_{s_2s_1s_2(\varpi_2)} = \frac{1}{(1-[-\alpha_1-\alpha_2])(1-[-\alpha_2])(1-[-\alpha_1-2\alpha_2])^2}.
\]
\cqfd
\end{example}

\subsection{$\QQ$-variables}

Using the automorphisms $\widetilde{\Theta}_w$ of $\Pi'$ introduced in Eq.~(\ref{eq-9}), we can now define:
\begin{Def}
For $i\in I$, $w\in W$ and $a\in \C^*$, let
\[
 \QQ_{w(\varpi^\vee_i),\,a} := \chi_{w(\varpi_i)}\widetilde{\Theta}_w(\Psi_{i,a}).
\]
\end{Def}

$\widetilde{\Theta}_w(\Psi_{i,a})$ depends only on the weight $w(\varpi^\vee_i)$, hence the notation $\QQ_{w(\varpi^\vee_i),\,a}$.
The reason why these elements are indexed by co-weights rather than by weights will be explained later.

\begin{example}\label{Ex4-4}
 In type $B_2$ we have:
\begin{eqnarray*}
\QQ_{\varpi^\vee_1,\,a} &=& \Psi_{1,a},\\ 
\QQ_{\varpi^\vee_2,\,a} &=& \Psi_{2,a},\\
\QQ_{s_1(\varpi^\vee_1),\,a} &=& \Psi_{1,aq^{-4}}^{-1}\Psi_{2,aq^{-2}}\,\Si_{1,aq^{-4}},\\
\QQ_{s_2(\varpi^\vee_2),\,a} &=& \Psi_{2,aq^{-2}}^{-1}\Psi_{1,aq^{-2}}\Psi_{1,a}\,\Si_{2,aq^{-2}},\\
\QQ_{s_2s_1(\varpi^\vee_1),\,a} &=& \Psi_{2,aq^{-4}}^{-1}\Psi_{1,aq^{-2}}\,\Si_{21,aq^{-4}}, \\
\QQ_{s_1s_2(\varpi^\vee_2),\,a} &=& \Psi_{1,aq^{-6}}^{-1}\Psi_{1,aq^{-4}}^{-1}\Psi_{2,aq^{-4}}\,\Si_{12,aq^{-2}}, \\
\QQ_{s_1s_2s_1(\varpi^\vee_1),\,a} &=& \Psi_{1,aq^{-6}}^{-1}\,\Si_{121,aq^{-4}}, \\
\QQ_{s_2s_1s_2(\varpi^\vee_2),\,a} &=& \Psi_{2,aq^{-6}}^{-1}\,\Si_{212,aq^{-2}}. 
\end{eqnarray*}
Here, the elements $\Si_{21,aq^{-4}}$, $\Si_{12,aq^{-2}}$, $\Si_{121,aq^{-4}}$, and $\Si_{212,aq^{-2}}$ of $\Pi$ are defined as in \cite[\S4.6]{FH2}. For instance,
\begin{eqnarray*}
E_e'(\Si_{21,aq^{-4}}) &=& \sum_{0\le \ell \le k/2} \left(\prod_{i=0}^{k-1} A_{2,aq^{-4-2i}}^{-1} \prod_{j=0}^{\ell-1} A_{1,aq^{-4-4j}}^{-1}\right) \\ 
&=& 1 + A_{2,aq^{-4}}^{-1} + A_{2,aq^{-4}}^{-1}A_{2,aq^{-6}}^{-1} 
+ A_{2,aq^{-4}}^{-1}A_{2,aq^{-6}}^{-1}A_{1,aq^{-4}}^{-1} +\cdots 
\end{eqnarray*}

Note that the action of $W$ on $P^\vee$ gives:
\[
 s_1(\varpi^\vee_1) = -\varpi^\vee_1 + \varpi^\vee_2,\quad 
 s_2s_1(\varpi^\vee_1) = \varpi^\vee_1 - \varpi^\vee_2,\quad 
 s_1s_2s_1(\varpi^\vee_1) = -\varpi^\vee_1,
\] 
\[
 s_2(\varpi^\vee_2) = 2\varpi^\vee_1 - \varpi^\vee_2,\quad 
 s_1s_2(\varpi^\vee_2) = -2\varpi^\vee_1 + \varpi^\vee_2,\quad 
 s_2s_1s_2(\varpi^\vee_2) = -\varpi^\vee_2.
\] 
Hence, we see that if we declare $\Psi_{i,a}$ to be of degree $\varpi^\vee_i$, then $\QQ_{w(\varpi^\vee_i),\,a}$ is homogeneous of degree $w(\varpi^\vee_i)$.
\cqfd
\end{example}

\begin{Rem}
Recall the operators $\Theta_i'$ introduced in Remark~\ref{Rem3-12}. Since they satisfy the braid relations,  using an arbitrary reduced decomposition of $w\in W$ we can define a ring automorphism $\Theta_w'$ of $\Pi'$. We can establish as in \cite[Lemma 6.5]{GHL} that
\[
\QQ_{w(\varpi^\vee_i),\,a} = \Theta_w'(\Psi_{i,a}),\qquad (w\in W,\ i\in I,\ a\in\C^*). 
\]
\end{Rem}

\subsection{The $QQ$-system}

It is proved in \cite{FH3} that the elements $\QQ_{w(\varpi^\vee_i),\,a}$ satisfy the following functional relations.

\begin{Thm}[FH3]\label{Thm-QtQ-system}
For $i\in I$ and $w\in W$ such that $ws_i > w$, we have:
\begin{eqnarray*}
&&\QQ_{ws_i(\varpi^\vee_i),\,aq^{d_i}} \QQ_{w(\varpi^\vee_i),\,aq^{-d_i}} - [-w(\alpha_i)] \QQ_{ws_i(\varpi^\vee_i),\,aq^{-d_i}} \QQ_{w(\varpi^\vee_i),\,aq^{d_i}}
\\[4mm]
&=& \prod_{j:\ c_{ij}=-1} \QQ_{w(\varpi^\vee_j),\,a} 
\prod_{j:\ c_{ij}=-2} \QQ_{w(\varpi^\vee_j),\,aq}\QQ_{w(\varpi^\vee_j),\,aq^{-1}} 
\prod_{j:\ c_{ij}=-3} \QQ_{w(\varpi^\vee_j),\,aq^2}\QQ_{w(\varpi^\vee_j),\,a}\QQ_{w(\varpi^\vee_j),\,aq^{-2}}.
\end{eqnarray*}
\end{Thm}

We say that the family $\QQ_{w(\varpi^\vee_i),\,a}\ (w\in W,\, i\in I,\, a\in \mathbb{C}^*)$ is a \emph{solution of the $QQ$-system} in the ring $\Pi'$. Then
for each $v\in W$, the family $E_{v}(\QQ_{w(\varpi_i^\vee),\,a})\ (w\in W,\, i\in I,\, a\in \mathbb{C}^*)$ is a solution of the same $QQ$-system of equations in the completed ring $\widetilde{\mathcal{Y}}'_{v}$.
We will now discuss how unique this solution is. 

\begin{Thm}\label{uniqqsys} Consider a solution $S_{w(\varpi_i^\vee),\,a}\ (w\in W,\, i\in I,\, a\in \mathbb{C}^*)$
of the $QQ$-system in the ring~$\widetilde{\mathcal{Y}}'_{v}$.
Fix $\widetilde{w}\in W$ and suppose that for any $i\in I$, $a\in\mathbb{C}^*$, we have
$$S_{\widetilde{w}(\varpi_i^\vee),\,a} = E_{v}(\QQ_{\widetilde{w}(\varpi_i^\vee),\,a}).$$ 
Then for any $w\in W$, $i\in I$, $a\in\mathbb{C}^*$, we have
\begin{equation}\label{eqsol}S_{w(\varpi_i^\vee),\,a} = E_{v}(\QQ_{w(\varpi_i^\vee),\,a}).\end{equation} 
\end{Thm}

\begin{proof}
We prove (\ref{eqsol}) by induction on the length of the Weyl group element $\widetilde{w}^{-1}w$. When the length is zero, this is just the hypothesis. Assume that (\ref{eqsol}) is known when the length is at most $l$ and consider $w$ such that the length of $\widetilde{w}^{-1}w$ is $l+1$. Then there is $i_0$ such that 
$\widetilde{w}^{-1}ws_{i_0} < \widetilde{w}^{-1}w$. Hence we have (\ref{eqsol}) for $ws_{i_0}$. For $i\neq i_0$, we have $ws_{i_0}(\varpi_i^\vee) = w(\varpi_i^\vee)$, and so 
(\ref{eqsol}) holds for $i$ by the induction hypothesis. Now the $S_{w(\varpi_{i_0}^\vee),\,b}$ ($b\in\mathbb{C}^*$) are uniquely determined by the 
$S_{ws_{i_0}(\varpi_{i_0}^\vee),\,b}$ ($b\in\mathbb{C}^*$) and the $S_{w(\varpi_j^\vee),\,b} = S_{ws_{i_0}(\varpi_j^\vee),\,b}$ ($j\neq i_0,\, b\in\mathbb{C}^*$)
from the $QQ$-relations and \cite[Lemma 2.5, 6.2]{FH2}\footnote{In the statement of \cite[Lemma 6.2]{FH2} a $\mathbb{C}^*$-equivariant condition is required for the solution of the relation; however, the proof of the uniqueness of the solution does not use this condition. In fact, this $\mathbb{C}^*$-equivariant condition follows from the proof therein. This is why we do not discuss the $\mathbb{C}^*$-equivariant condition here.} (if $w s_{i_0} > w$ we consider the $QQ$-relation for $w$, otherwise we consider the $QQ$-relation for $ws_{i_0}$ instead of $w$). 
\end{proof}

\section{Cluster algebra structures on formal power series rings}
\label{sec:main-thm}

\subsection{The ring $K_\Z$}\label{Kz}

Let $K:= E_e'(\Pi')$ denote the projection of $\Pi'$ on its first factor. In other words,
\[
 K = \widetilde{\YY}'_e
 =\YY'\otimes_\YY \widetilde{\YY_e}.
\]
The ring $K$ contains in particular the following $Q$-series of \cite{FH2, FH3}:
\begin{equation}\label{Qserfh}
Q_{w(\varpi^\vee_i),\,a} := E_e\left(\QQ_{w(\varpi^\vee_i),\,a}\right),\qquad (i\in I,\ w\in W,\ a\in \C^*).
\end{equation}
Each of these elements has a unique factorization of the form
\[
Q_{w(\varpi^\vee_i),\,a} = \Psi_{w(\varpi^\vee_i),\,a}\, \Si_{w(\varpi^\vee_i),\,a}, 
\]
where $\Psi_{w(\varpi^\vee_i),\,a}$ is a Laurent monomial in the variables $\Psi_{j,b}$, and $\Si_{w(\varpi^\vee_i),\,a}$ is a formal power series in the variables $A_{j,b}^{-1}$ with constant term equal to 1. We call $\Psi_{w(\varpi^\vee_i),\,a}$ the highest weight monomial of $Q_{w(\varpi^\vee_i),\,a}$. By \cite{FH3}, this monomial matches (\ref{extl}) and can be calculated from $\Psi_{\varpi_i,{\red b}}$ using (the extended version of) Chari's braid group action (see Examples in Section \ref{reduit}). 

Following \cite{FH2, GHL}, $K$ has the structure of a complete topological ring. Moreover, $K$ is an algebra over the ring $[P]$ generated by the elements $[\lambda]\ (\lambda \in P)$.

In the sequel, we will have to use weights in $\underline{P} := \frac{1}{2}\Z\otimes_\Z P$. 
%
Recall the vertex set $V \subset I \times \Z$ of the basic quiver $\Gamma_e$. Let $K_\Z$ denote the 
$[\underline{P}]$-subalgebra of $[\underline{P}]\otimes_{[P]} K$ topologically generated by the elements:
\[
Q_{w(\varpi^\vee_i),\,q^r},\quad ((i,r)\in V,\ w\in W). 
\]
 Using the above-mentioned braid group action, we see that if $(i,r)\in V$, $\Psi_{w(\varpi_i^\vee),q^r}$ is a Laurent monomial in the $\Psi_{j,q^s}^{\pm 1}$ with $(j,s)\in V$. 

 The following proposition is proved in the same way as  
\cite[Proposition 7.1]{GHL}.

\begin{Prop} \label{prop-invertible} (1) For every $(i,r)\in V$ and $w\in W$, the element $Q_{w(\varpi_i^\vee),\,q^r}$ 
is invertible in $K_{\Z}$.

(2) As a topological $[\underline{P}]$-algebra, $K_{\Z}$ is generated by the $\Psi_{i,q^r}^{\pm 1}$ with $(i,r)\in V$.
\end{Prop}

Note that for $(i,r)\in V$, we have 
$$Y_{i,q^{r-d_i}}^{\pm 1}\in K_{\Z}, \qquad A_{i,q^r}^{\pm 1}\in K_{\Z}.$$

\begin{example}
(i) In type $B_2$, the $[\underline{P}]$-algebra $K_{\Z}$ is topologically generated by the 
$$\Psi_{1,q^{2r}}^{\pm 1}, \Psi_{2,q^{2r}}^{\pm 1},\qquad (r\in \Z).$$

(ii) In type $B_3$, the $[\underline{P}]$-algebra $K_{\Z}$ is topologically generated by the 
$$\Psi_{1,q^{2r}}^{\pm 1}, \Psi_{2,q^{2r}}^{\pm 1}, \Psi_{3,q^{2r}}^{\pm 1},\qquad (r\in \Z).$$

(iii) In type $C_3$, the $[\underline{P}]$-algebra $K_{\Z}$ is topologically generated by the 
$$\Psi_{1,q^{2r+1}}^{\pm 1}, \Psi_{2,q^{2r}}^{\pm 1}, \Psi_{3,q^{2r}}^{\pm 1},\qquad (r\in \Z).$$

(iv) In type $G_2$, the $[\underline{P}]$-algebra $K_{\Z}$ is topologically generated by the 
$$\Psi_{1,q^{2r}}^{\pm 1}, \Psi_{2,q^{2r+1}}^{\pm 1},\qquad (r\in \Z).$$
\end{example}

\subsection{Renormalized $Q$-variables}

Let $\Omega$ be the group homomorphism from the (multiplicative) group of Laurent monomials in the variables $\Psi_{i,q^r}\ ((i,r)\in V)$ and $[\la] (\la\in P)$ to $P$, defined by 
\[
 \Omega(\Psi_{i,q^r}) := \frac{r}{2d_i}\, \varpi_i,\quad \Omega([\la]) := \la, \qquad ((i,r)\in V,\ \la\in P).
\]
Note that 
\begin{equation}\label{Eq-Omega-1}
\Omega(Y_{i,q^{r+d_i}}) = \Omega\left([\varpi_i]\frac{\Psi_{i,q^{r}}}{\Psi_{i,q^{r+2d_i}}}\right) = 0, \qquad ((i,r)\in V). 
\end{equation}
We set: 
\begin{equation}\label{Qsrfhr}
 \bQ_{w(\varpi^\vee_i),\,q^r} := \left[-\Omega(\Psi_{w(\varpi^\vee_i),\,q^r})\right]\,Q_{w(\varpi^\vee_i),\,q^r},
 \qquad ((i,r)\in V,\ w\in W). 
\end{equation}
In terms of these renormalized elements $\bQ_{w(\varpi^\vee_i),\,q^r}$, the $QQ$-system simplifies slightly and becomes the \emph{renormalized $QQ$-system}:
\begin{Prop}\label{renqq}
For $(i,r)\in V$ and $w\in W$ such that $ws_i > w$, 
we have:
\begin{eqnarray*}
&&{\bQ}_{ws_i(\varpi^\vee_i),\,q^{r}} \bQ_{w(\varpi^\vee_i),\,q^{r-2d_i}} -\ {\bQ}_{ws_i(\varpi^\vee_i),\,q^{r-2d_i}} \bQ_{w(\varpi^\vee_i),\,q^{r}}
\\[4mm]
&=& \prod_{j:\ c_{ij}=-1} \bQ_{w(\varpi^\vee_j),\,q^{r-d_i}} 
\prod_{j:\ c_{ij}=-2} \bQ_{w(\varpi^\vee_j),\,q^{r}}\bQ_{w(\varpi^\vee_j),\,q^{r-2}} 
\prod_{j:\ c_{ij}=-3} \bQ_{w(\varpi^\vee_j),\,q^{r+1}}\bQ_{w(\varpi^\vee_j),\,q^{r-1}}\bQ_{w(\varpi^\vee_j),\,q^{r-3}}.
\end{eqnarray*} 
\end{Prop}

\begin{proof} The proof is similar to \cite[Proposition 7.2]{GHL}, but some technical point have to be modified in the non simply-laced types.
We have to check that, using this new normalization, the factor $[-w(\alpha_i)]$ in the second term of the left hand-side of Theorem~\ref{Thm-QtQ-system} cancels out.

Let us first check it for $w = e$. The highest weight monomial of ${\bQ}_{s_i(\varpi^\vee_i),\,q^{r}} \bQ_{\varpi^\vee_i,\,q^{r-2d_i}}$ is equal to 
\[
\widetilde{\Psi}_{i,q^{r-2d_i}}\Psi_{i,q^{r-2d_i}} =
\prod_{j:\ c_{ij}=-1} \Psi_{\varpi^\vee_j,\,q^{r-d_i}} 
\prod_{j:\ c_{ij}=-2} \Psi_{\varpi^\vee_j,\,q^{r}}\Psi_{\varpi^\vee_j,\,q^{r-2}} 
\prod_{j:\ c_{ij}=-3} \Psi_{\varpi^\vee_j,\,q^{r+1}}\Psi_{\varpi^\vee_j,\,q^{r-1}}\Psi_{\varpi^\vee_j,\,q^{r-3}},
\]
which is nothing else than the highest weight-monomial of the right-hand side. Therefore  
applying $\Omega$ to the highest weight-monomial of the right-hand side, we get exactly 
$\Omega(\widetilde{\Psi}_{i,q^{r-2d_i}}\Psi_{i,q^{r-2d_i}})$.
Now the quotient of the highest weight monomial of ${\bQ}_{s_i(\varpi^\vee_i),\,q^{r-2d_i}} \bQ_{\varpi^\vee_i,\,q^{r}}$ by the highest weight monomial of ${\bQ}_{s_i(\varpi^\vee_i),\,q^{r}} \bQ_{\varpi^\vee_i,\,q^{r-2d_i}}$ is equal to
\[
\frac{\widetilde{\Psi}_{i,q^{r-4d_i}}\Psi_{i,q^{r}}}{\widetilde{\Psi}_{i,q^{r-2d_i}}\Psi_{i,q^{r-2d_i}}} 
=
[\alpha_i-\varpi_i]Y_{i,q^{r-d_i}} A_{i,q^{r-2d_i}}^{-1} [\varpi_i]Y_{i,q^{r-d_i}}^{-1}
=
[\alpha_i] A_{i,q^{r-2d_i}}^{-1},
\]
by Lemma~\ref{Lem31}. Hence, since by Eq.~(\ref{Eq-Omega-1}) we have $\left[\Omega\left(A_{i,q^{r-2d_i}}^{-1}\right)\right] = 1$, we get:  
\[
 \left[-\Omega\left(\widetilde{\Psi}_{i,q^{r-4d_i}}\Psi_{i,q^{r}}\right)\right] =
   \left[-\Omega\left(\widetilde{\Psi}_{i,q^{r-2d_i}}\Psi_{i,q^{r-2d_i}}\right)\right]
   [-\alpha_i] 
\]
and the Proposition is proved in the case $w=e$.

The case of an arbitrary $w$ can be checked in a similar way using results from \cite{FH3}. First, one  can see that for every $w$, the highest weight monomial of the right-hand side coincides with the highest weight monomial of the first term of the left-hand side. Therefore we only have to compare the result of applying $\Omega$ to the highest weight monomials of the two terms of the left hand-side. Now, the highest weight monomial $\Lambda(\Theta_w(Y_{i,a}))$ of $\Theta_w(Y_{i,a})$, where $\Lambda$ is defined in \cite[§6.4]{FH2}, is known to be equal to
\[
\Lambda(\Theta_w(Y_{i,a})) = [w(\varpi_i)] \Psi_{w(\varpi^\vee_i),aq^{-d_i}}\Psi_{w(\varpi^\vee_i),aq^{d_i}}^{-1},
\quad (w\in W,\ a\in\C^*).
\]
Since $\Lambda(\Theta_w(Y_{i,a}))$ is a Laurent monomial in the variables $Y_{j,b}\ (j\in I,\ b\in\C^*)$, it follows from Eq.~(\ref{Eq-Omega-1}), that:
\[
\left[\Omega\left(\Lambda(\Theta_w(Y_{i,q^{r-d_i}}))\right)\right] = [w(\varpi_i)]\, \left[\Omega\left(\Psi_{w(\varpi^\vee_i),q^{r-2d_i}}\Psi_{w(\varpi^\vee_i),q^r}^{-1}\right)\right] = 1,
\]
and similarly
\[
\left[\Omega\left(\Lambda(\Theta_{ws_i}(Y_{i,q^{r-3d_i}}))\right)\right] = [ws_i(\varpi_i)]\, \left[\Omega\left(\Psi_{ws_i(\varpi^\vee_i),q^{r-4d_i}}\Psi_{ws_i(\varpi^\vee_i),q^{r-2d_i}}^{-1}\right)\right] = 1.
\]
Therefore,
\[
\left[\Omega\left(\frac{\Psi_{ws_i(\varpi^\vee_i),q^{r-4d_i}}\Psi_{w(\varpi^\vee_i),q^{r}}}{\Psi_{ws_i(\varpi^\vee_i),q^{r-2d_i}}\Psi_{w(\varpi^\vee_i),q^{r-2d_i}}} 
\right)\right] 
= [w(\varpi_i) - ws_i(\varpi_i)] = [w(\alpha_i)], 
\]
hence
\[
 \left[-\Omega\left(\Psi_{ws_i(\varpi^\vee_i),q^{r-4d_i}}\Psi_{w(\varpi^\vee_i),q^{r}}\right)\right] =
   \left[-\Omega\left(\Psi_{ws_i(\varpi^\vee_i),q^{r-2d_i}}\Psi_{w(\varpi^\vee_i),q^{r-2d_i}}\right)\right]
   [-w(\alpha_i)]. 
\]
\cqfd
\end{proof}

\subsection{Cluster structure on $K_\Z$}\label{ssubsec-main-th}

Recall the cluster algebra $\AA_{w_0}$ of Section \ref{SecAw}, defined by an initial seed $\mathcal{S}_{\bi,\br} = (\G_{\bi, \br},(z_{(i,r)})_{(i,r)\in V)})$, where  $\bi:=(i_1,\ldots, i_l)$ is a reduced word for $w_0$, and $\br := (r_1,\ldots, r_l)\in \Z^r$ is an adapted tuple of integers. By Theorem \ref{Prop2-10}, the cluster algebra $\AA_{w_o}$ is
independent (up to isomorphism) of the choice of $(\bi,\br)$.

By \S\ref{sgvect}, the initial cluster variables $z_{(i,r)}$ ($(i,r) \in V$) of 
$\mathcal{S}_{\bi,\br}$ are parametrized by their stabilized g-vectors
\[
\bg_{(i,r)}^\infty = t_{i_1} t_{i_2}\cdots t_{i_m}(\be_{i, r+s_m}),
\]
where  $m$ is determined by the condition $(i,r)\in V_m$,
and $s_m:=\sig(i_m,r_m)$,  
see Remark~\ref{comps1} for more details and Example~\ref{exam-5.3} for an illustration.

\begin{figure}[t]
\[
\def\objectstyle{\scriptstyle}
\def\lablestyle{\scriptstyle}
\def\objectstyle{\scriptstyle}
\def\lablestyle{\scriptstyle}
\xymatrix@-1.0pc{
&{}\save[]+<0cm,0ex>*{\vdots}\restore&{}\save[]+<0cm,0ex>*{\vdots}\restore
& \save[]+<0cm,0ex>*{\vdots}\restore
\\
&{\bf e}_{1,7}\ar[rd] & {\bf e}_{2,7} \ar[rd]&&
\\
&&\ar[ld] {\bf e}_{2,5}\ar[u]& \ar[ld]{\bf e}_{1,5}&
\\
&{\color{rgb,255:red,255;green,0;blue,4}{\bf e}_{1,3}}\ar[uu]\ar[d] &\ar[u] {\bf e}_{2,3} \ar[rd]&&
\\
&{\color{rgb,255:red,0;green,204;blue,3} t_1({\bf e}_{1,-1+4})}\ar[rd] &&  \ar[uu]\ar[d] {\color{rgb,255:red,255;green,0;blue,4}
{\bf e}_{1,1}}
\\
&&\ar[uu]\ar[ld] {\bf e}_{2,1} &\ar[ld]   {\color{rgb,255:red,0;green,204;blue,3} t_1({\bf e}_{1,-3+4})} &
\\
&t_1({\bf e}_{1,-5+4})\ar[uu]\ar[rdd]\ar[r]   & \ar[d]\ar[u]{\color{rgb,255:red,255;green,0;blue,4} {\bf e}_{2,-1}} &&
\\
&                          & {\color{rgb,255:red,0;green,204;blue,3} t_1t_2({\bf e}_{2,-3+2})} \ar[lu]\ar[rd]
\\
&&\ar[ld] \ar[u] t_1t_2({\bf e}_{2,-5+2}) &  {\color{rgb,255:red,255;green,0;blue,4} t_1({\bf e}_{1,-7+4})}\ar[uuu]\ar[d] &
\\
&{\color{rgb,255:red,255;green,0;blue,4}
 t_1({\bf e}_{1,-9+4})}\ar[uuu]\ar[d] & &\ar[ld] {\color{rgb,255:red,0;green,204;blue,3} t_1t_2t_1({\bf e}_{1,-11+8})}  &
\\
& {\color{rgb,255:red,0;green,204;blue,3} t_1t_2t_1({\bf e}_{1,-13+8})\ar[dr]}& t_1t_2({\bf e}_{2,-7+2}) \ar[uu]\ar[rd]&&
\\
&&{\color{rgb,255:red,255;green,0;blue,4} t_1t_2({\bf e}_{2,-9+2})} \ar[u] \ar[d] &t_1t_2t_1({\bf e}_{1,-15+8})\ar[l]\ar[ldd]  \ar[uu]  &
\\
&&\ar[ld] \ar[ru] {\color{rgb,255:red,0;green,204;blue,3}t_1t_2t_1t_2({\bf e}_{2,-11+4})}&&
\\
&t_1t_2t_1({\bf e}_{1,-17+8})\ar[uuu]
\ar[rd]   
& \ar[u] t_1t_2t_1t_2({\bf e}_{2,-13+4})
\ar[rd] &&
\\
&&\ar[u] & \ar[uuu]
}
\qquad
\xymatrix@-1.0pc{
&{}\save[]+<0cm,0ex>*{\vdots}\restore&{}\save[]+<0cm,0ex>*{\vdots}\restore
& \save[]+<0cm,0ex>*{\vdots}\restore
\\
&\bQ_{\varpi^\vee_1,q^7}\ar[rd] & \bQ_{\varpi^\vee_2,q^7} \ar[rd]&&
\\
&&\ar[ld] \bQ_{\varpi^\vee_2,q^5}\ar[u]& \ar[ld]\bQ_{\varpi^\vee_1,q^5}
\\
&\mathbf{\color{rgb,255:red,255;green,0;blue,4} \bQ_{\varpi^\vee_1,q^3}}\ar[uu]\ar[d] &\ar[u] \bQ_{\varpi^\vee_2,q^3} \ar[rd]&&
\\
&\mathbf{\color{rgb,255:red,0;green,204;blue,3} \bQ_{s_1(\varpi^\vee_1),q^3}}\ar[rd] && \mathbf{\color{rgb,255:red,255;green,0;blue,4} \bQ_{\varpi^\vee_1,q}} \ar[uu]\ar[d] 
\\
&&\ar[uu]\ar[ld] \bQ_{\varpi^\vee_2,q} &\ar[ld] \mathbf{\color{rgb,255:red,0;green,204;blue,3} \bQ_{s_1(\varpi^\vee_1),q}}   &
\\
&\ar[uu]\bQ_{s_1(\varpi^\vee_1),q^{-1}}\ar[rdd]\ar[r]   & \ar[d]\mathbf{\color{rgb,255:red,255;green,0;blue,4} \bQ_{\varpi^\vee_2,q^{-1}}}\ar[u] &&
\\
&                          & \ar[lu]\mathbf{\color{rgb,255:red,0;green,204;blue,3} \bQ_{s_1s_2(\varpi^\vee_2),q^{-1}}}\ar[rd]
\\
&&\ar[ld] \bQ_{s_1s_2(\varpi^\vee_2),q^{-3}} \ar[u] & \mathbf{\color{rgb,255:red,255;green,0;blue,4} \bQ_{s_1(\varpi^\vee_1),q^{-3}}} \ar[uuu]\ar[d] &
\\
&\mathbf{\color{rgb,255:red,255;green,0;blue,4} \bQ_{s_1(\varpi^\vee_1),q^{-5}}}\ar[uuu]\ar[d] & &\ar[ld] \mathbf{\color{rgb,255:red,0;green,204;blue,3} \bQ_{s_1s_2s_1(\varpi^\vee_1),q^{-3}}}  &
\\
&\mathbf{\color{rgb,255:red,0;green,204;blue,3}\ar[rd] \bQ_{s_1s_2s_1(\varpi^\vee_1),q^{-5}}} & \ar[uu]\ar[rd] \bQ_{s_1s_2(\varpi^\vee_2),q^{-5}}&&
\\
&&\ar[u] \mathbf{\color{rgb,255:red,255;green,0;blue,4} \bQ_{s_1s_2(\varpi^\vee_2),q^{-7}}}\ar[d] &\ar[l]\ar[ldd] \bQ_{s_1s_2s_1(\varpi^\vee_1),q^{-7}} \ar[uu]  &
\\
&&\ar[ld] \mathbf{\color{rgb,255:red,0;green,204;blue,3} \bQ_{s_1s_2s_1s_2(\varpi^\vee_2),q^{-7}}}\ar[ru] &&
\\
&\ar[uuu] \bQ_{s_1s_2s_1(\varpi^\vee_1),q^{-9}} 
& \ar[u]\bQ_{s_1s_2s_1s_2(\varpi^\vee_2),q^{-9}} 
}
\]
\caption{\label{Fig9}  An initial seed of $\AA_{w_0}$ and its image by $F$ in type $B_2$.} 
\end{figure}
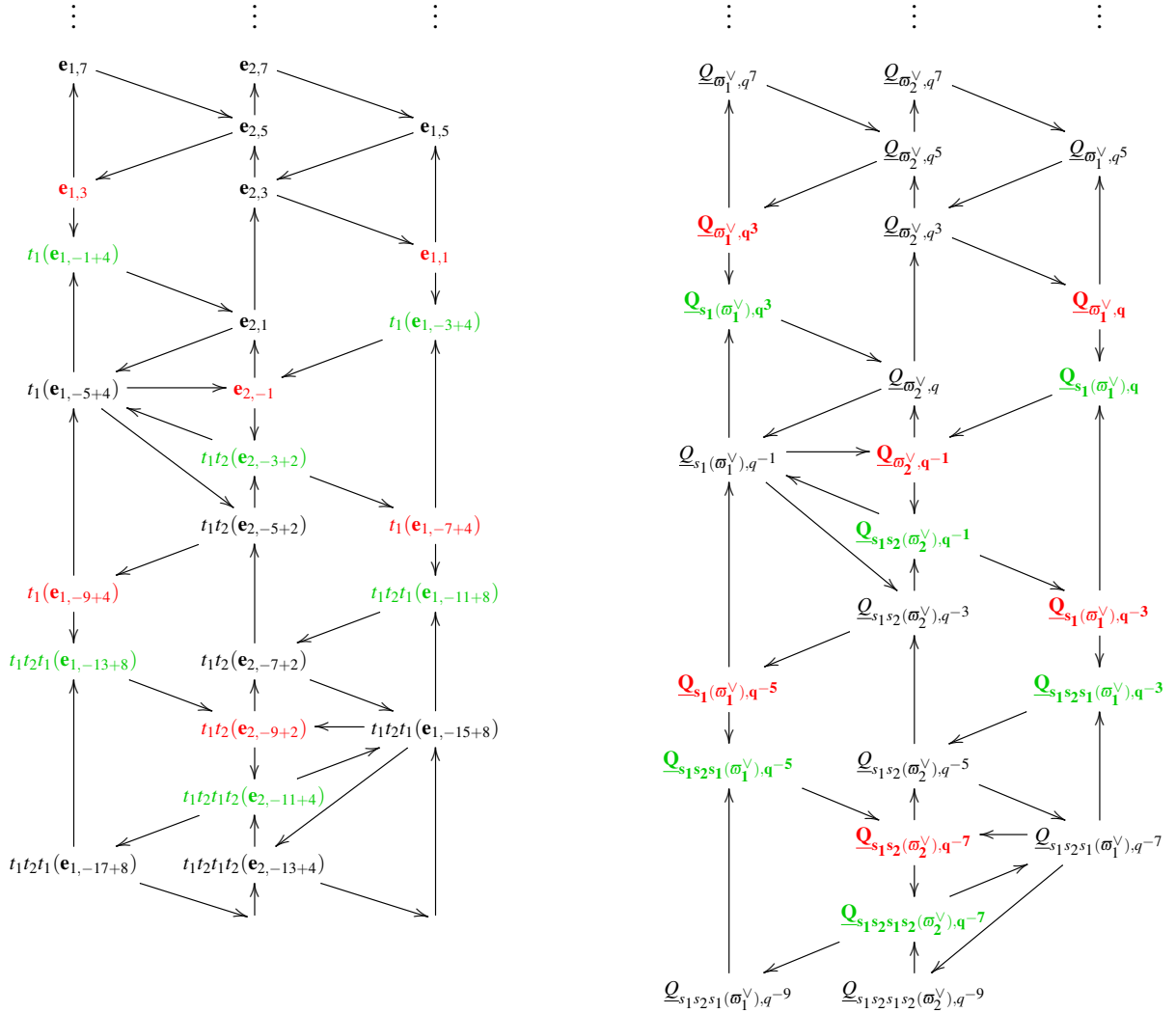

Let $L$ denote the Laurent polynomial ring in the variables $z_{(i,r)}$ ($(i,r) \in V$). By the Laurent phenomenon, the cluster algebra $\AA_{w_0}$
embeds into $L$.

\begin{Thm}\label{main-Thm}
The assignment 
\[
z_{(i,r)} \mapsto 
\bQ_{s_{i_1}s_{i_2}\cdots s_{i_m}(\varpi_i^\vee),q^{r+s_m}},\qquad
((i,r)\in V),
\]
extends to an injective ring homomorphism $F : L \to K_{\Z}$. 
Moreover the topological closure of $ [\underline{P}]\otimes_\Z F(\AA_{w_0})$ is equal to $K_{\Z}$.
\end{Thm}

The seed $(\G_{\bi, \br}, (F(z_{(i,r)})_ {(i,r)\in V})$ in $K_{\Z}$ will be denoted by $\Sigma_{\bi, \br}$.

\begin{Rem} (i) We will establish in Theorem \ref{uniF} that the embedding $F$ does not depend on the choice of a \emph{compatible} seed, in the sense of Definition \ref{comps}. 

(ii) Note that using Lemma~\ref{Lem28} it is easy to see that one passes from a seed $\Sigma_{\bi,\br}$ to a seed $\Sigma_{\bi,\br'}$ with the same $\bi$ but different $\br,\br'$ via a sequence of mutations whose exchange relations are realized in $K_\mathbb{Z}$ as renormalized $QQ$-relations. But this is not true for different $\bi$, $\bi'$. 
Since the proof of the uniqueness of $F$ given in \cite[Remark 7.6]{GHL} for simply-laced type does not extend to general types, we will here prove Theorem \ref{uniF} by a new argument. \end{Rem}

\begin{example}\label{exam-5.3}
Let $\g$ be of type $B_2$. 
In Figure~\ref{Fig9} we display the initial seed of $\AA_{w_0}$ with quiver 
$\Gamma_{(1,2,1,2),\ (3,-1,-3,-7)}$, and its image under $F$.
Mutating the initial cluster variable $\bQ_{\varpi^\vee_1,q^3}$ we get the new cluster variable $\bQ_{s_1(\varpi^\vee_1),q^7}$, and the corresponding exchange relation is
\[
\bQ_{\varpi^\vee_1,q^3} \bQ_{s_1(\varpi^\vee_1),q^7} = \bQ_{\varpi^\vee_1,q^7} \bQ_{s_1(\varpi^\vee_1),q^3}, 
+ \, \bQ_{\varpi^\vee_2,q^5}, 
\]
an example of relation of the $QQ$-system. Mutating the initial cluster variable $\bQ_{s_1(\varpi^\vee_2),q^{-1}}$ we get the new  variable $\bQ_{s_1s_2(\varpi^\vee_2),q}$, and the corresponding exchange relation is also a relation of the $QQ$-system
\[
\bQ_{s_1(\varpi^\vee_2),q^{-1}} \bQ_{s_1s_2(\varpi^\vee_2),q} = \bQ_{\varpi^\vee_2,q} \bQ_{s_1s_2(\varpi^\vee_2),q^{-1}} 
+ \,\bQ_{s_1(\varpi^\vee_1),q^{-1}} \bQ_{s_1(\varpi^\vee_1),q}, 
\]

On the other hand, mutating the initial cluster variable $\bQ_{\varpi^\vee_2,q^{3}}$, we get a new cluster variable, say $x$, given by the exchange relation
\[
x\, \bQ_{\varpi^\vee_2,q^{3}} = \bQ_{\varpi^\vee_2,q^{5}} \bQ_{\varpi^\vee_1,q} +\, \bQ_{\varpi^\vee_2,q}\bQ_{\varpi^\vee_1,q^{5}} 
\]
which is \emph{not} a relation of the $QQ$-system, but a $QQ^*$-relation as studied in \cite{HL2} and \cite{FH}.
\cqfd
\end{example}

 The injectivity of the homomrphism $F$ is proved as in \cite[Proposition 8.1]{GHL}. The crucial point is the existence of stabilized $g$-vectors of an initial seed (Proposition \ref{uniquegv}) and their linear independence. To establish the surjectivity of $F$ after completion, we will need the following uniqueness result.

\subsection{Uniqueness of the homomorphism $F$ } \label{ssec:uniqueCS}

In Proposition \ref{uniquegv} we have established that each cluster variable $x$ of $\AA_{w_0}$ has a well-defined stabilized $g$-vector $\bg^\infty(x)$ which can be calculated by performing a finite number of mutations starting in a compatible initial seed.

We say that a cluster variable $x$ is \emph{unital} if $\bg^\infty(x)$ is unital (that is, it has a single non-zero coefficient, and this coefficient is $1$).
By Proposition~\ref{uniquegv} we have: 

\begin{Lem}\label{unital} The collection of unital cluster variables does not depend on the choice of a compatible initial seed. 
Moreover, 
for any $(i,r)\in V$, there is a unique unital cluster variable $u_{i,r}$ whose stabilized $g$-vector is ${\bf e}_{i,r}$.
\end{Lem}

\begin{Rem} 
Heuristically, one can regard $\bg^\infty(x)$ as the $g$-vector of $x$ with respect to the \emph{limit seed} consisting of the 
collection of all unital cluster variables $u_{(i,r)}$ sitting at the vertices $(i,r)$ of 
the basic infinite quiver $\Gamma$. Although this limit seed is very important for our purposes, in a strict sense, this is not a seed of the cluster algebra $\mathcal{A}_{w_0}$, since it 
can not be reached with a finite number of mutations. Note however that any monomial in the unital cluster variables, that is any cluster monomial of the limit seed, is a cluster monomial of $\mathcal{A}_{w_0}$.
\end{Rem}

We can now prove:

\begin{Thm}\label{uniF} The homomorphism $F$ does not depend on the  choice of a compatible initial seed.
\end{Thm}

In simply-laced types, the argument in \cite[Remark 7.6]{GHL} was partly based on the proof of 
exchange relations in type $A_2$ that are not $QQ$-relations. For general types, we propose here a different argument, which is uniform and does not involve specific computations in rank $2$.  This also gives an alternative new proof for the simply-laced types.

\begin{proof} Consider two homomorphisms 
$$ F:\mathcal{A}\rightarrow K_\mathbb{Z},\quad F' : \mathcal{A}'\rightarrow K_\mathbb{Z}$$
corresponding to two  compatible initial seeds $\mathcal{S}$ and $\mathcal{S}'$ as in Theorem \ref{main-Thm}. 
We have already established that the corresponding cluster algebras $\mathcal{A}$ and $\mathcal{A}'$ are isomorphic, via a finite sequence of mutations. One issue is that this sequence of mutations involves exchange relations which are not necessarily $QQ$-relations. Hence it is not clear a priori that $F$ and $F'$ coincide. We want to establish that we have a commutative diagram
\[\xymatrix{\mathcal{A} \ar[r]^{\simeq} \ar[d]_{F}& \mathcal{A}' \ar[d]_{F'}
\\ K_{\mathbb{Z}}\ar[r]^{=}& K_{\mathbb{Z}}}
\]

Recall the solutions ${\bQ}_{w(\varpi_i^\vee),\,q^{s}}$ of the renormalized $QQ$-system in $K_{\mathbb{Z}}$.

Let $V_u\subset V$ (resp. $V_u'\subset V$) be the set of indices $(i,s)\in V$ so that the unital cluster variable $u_{i,s}$ (as defined in Lemma \ref{unital}) is in  $\mathcal{S}$ (resp. in $\mathcal{S}'$). Note that Lemma \ref{unital} shows that the cluster variable $u_{i,s}$ is independent on the choices : it is the same cluster variable in the cluster algebras $\mathcal{A}$ and  $\mathcal{A}'$ identified through the isomorphism.

Then we have
$$F(u_{i,s}) = {\bQ}_{\varpi_i^\vee,\,q^s}\quad (\mbox{resp.}\quad F'(u_{i,s}) = {\bQ}_{\varpi_i^\vee,\,q^s}).$$
Note that for $(i,s)\in V$ with large enough $s$, we have $(i,s)\in V_u\cap V'_u$. Hence, for $s$ large enough we have
 \begin{equation}\label{init}F(u_{i, s}) = F'(u_{i, s}) = {\bQ}_{\varpi_i^\vee,\,q^s} = \left[ -\frac{s}{2d_i}\varpi_i\right][\Psi_{i,q^s}].\end{equation}

Let us prove that this formula (\ref{init}) is actually satisfied for any $(i,s)\in V$. Consider a general $(i,s)\in V$. Then starting from the initial seed $\mathcal{S}$ (resp. $\mathcal{S}'$), and  using only mutations whose corresponding exchange relations are $QQ$-relations, one can obtain the cluster variable $u_{i,s}$ of $\mathcal{A}$ such that $F(u_{i,s}) = {\bQ}_{\varpi_i^\vee,\,q^s}$ (resp. $F'(u_{i,s}) = {\bQ}_{\varpi_i^\vee,\,q^s})$). 
Hence we have established $(\ref{init})$ for any $(i,s)\in V$.

Recall (Theorem \ref{uniqqsys}) that the solution of the (non-renormalized) $QQ$-system in $K_\mathbb{Z}$ is uniquely determined by the \emph{initial conditions} corresponding to $w = e$, that is by the collection of the 
$$Q_{\varpi_i^\vee,\,q^s} = \left[ \frac{s}{2d_i}\varpi_i^\vee\right]{\bQ}_{\varpi_i^\vee,\,q^s} = [\Psi_{i,q^s}].$$ 
We have denoted this unique solution $(Q_{w(\varpi_i),\,q^s})_{w,i,s}$, and we have the corresponding normalized elements $(\bQ_{w(\varpi_i),\,q^s})_{w,i,s}$.

Consider all the cluster variables belonging to all the compatible initial seeds as in Theorem \ref{main-Thm}. To each such cluster variable is associated a $Q$-variable ${\bQ}_{w(\varpi_i^\vee),\,q^s}$. But the couple $(w(\varpi_i^\vee),q^s)$ does not depend on the initial seed, only on the cluster variable. Indeed, by applying $F$ which is injective, we obtain a solution of the renormalized $QQ$-system, which is is unique as discussed above, provided we have imposed the initial conditions  
$\left[ \frac{s}{2d_i}\varpi_i^\vee\right]F(u_{i,s}) = {\bQ}_{\varpi_i^\vee,\,q^s}$.
Hence, we obtain cluster variables \begin{equation}z_{w(\varpi_i^\vee),q^s}\in\mathcal{A}_{w_0},\end{equation}
which are well-defined (that is independent on the choice of an initial seed), and we have \begin{equation}F(z_{w(\varpi_i^\vee),q^s}) = {\bQ}_{w(\varpi_i),\,q^s}\end{equation} for any $w$, $i$, $s$. 

As we have the same for $F'$, we obtain that $F = F'$.
\cqfd
\end{proof}

By construction, the cluster variables $z_{w(\varpi_i^\vee),q^s}\in \mathcal{A}_{w_0}$ satisfy the renormalized $QQ$-system as in Proposition \ref{renqq}. They 
will be called \emph{$Q$-cluster variables}.
\begin{Rem} 
We have considered different families of cluster variables in this section: the unital cluster variables $u_{i,r}$, the initial cluster variables $z_{(i,r)}$ and the $Q$-cluster variables $z_{w(\varpi_i^\vee), q^s}$. The family of $Q$-cluster variables contains all initial cluster variables and all unital cluster variables. Wheareas the unital cluster variables and the $Q$-cluster variables do not depend on choices, the initial cluster variables depend on the choice of the initial seed.
\end{Rem}

\subsection{End of the proof of Theorem~\ref{main-Thm}}

Once we have established Theorem \ref{uniF} (independence of the homomorphism $F$ on the choices) and the existence of the stabilized $g$-vectors of $Q$-variables (Section \ref{sgvect}), one can conclude the proof of Theorem \ref{main-Thm} as in \cite[Section 8]{GHL}. The surjectivity of $F$ after completion follows from the fact that all $Q$-variables are images of cluster variables under $F$.

\section{Limits of cluster expansions}

\label{sect-limitF}

In this section, we establish the existence of a positive cluster expansion of any cluster monomial with respect to the limit reference seed (Theorem \ref{devtp}). This result, not included in \cite{GHL}, is also new for $A,D,E$ types. In the case of $Q$-cluster variables, it implies certain positivity conjectures of \cite{FH2} (Theorem \ref{posq}).

\subsection{Formal expansions of $Q$-variables}

Let $w\in W$, $i\in I$ and $r\in \mathbb{Z}$ such that $ws_i > w$. We denote by $P_{w,i,q^{r - d_i}}$ the right-hand side of the $QQ$-relation
\begin{equation}\label{useqq}
{\bQ}_{ws_i(\varpi_i^\vee),\,q^r} \bQ_{w(\varpi_i^\vee),\,q^{r-2d_i}} -\ {\bQ}_{ws_i(\varpi_i^\vee),\,q^{r-2d_i}} \bQ_{w(\varpi_i^\vee),\,q^r}
=P_{w,i,q^{r -d_i}}.\end{equation}
Thus, in the $ADE$-case, 
$$P_{w,i,q^r} = \prod_{j:\ c_{ij}=-1} \bQ_{w(\varpi_j^\vee),\,q^r}$$
and in general $P_{w,i,q^r}$ is equal to 
$$\prod_{j:\ c_{ij}=-1} \bQ_{w(\varpi^\vee_j),\,q^{r}} 
\prod_{j:\ c_{ij}=-2} \bQ_{w(\varpi^\vee_j),\,q^{r+1}}\bQ_{w(\varpi^\vee_j),\,q^{r-1}} 
\prod_{j:\ c_{ij}=-3} \bQ_{w(\varpi^\vee_j),\,q^{r+2}}\bQ_{w(\varpi^\vee_j),\,q^{r}}\bQ_{w(\varpi^\vee_j),\,q^{r-2}}.$$

Recall from Section \ref{Kz} that ${\bQ}_{w(\varpi_i^\vee),\,q^r}$ has a highest weight term. Hence, any Laurent monomial in the ${\bQ}_{w(\varpi_i^\vee),\,q^r}$ has a highest weight term. Let us compute the highest weight for some Laurent monomials that will be important in the sequel.

\begin{Lem}\label{hwt}
For $r\in \mathbb{Z}$, the highest weight of 
\begin{equation}\label{fqp}\frac{\bQ_{w(\varpi_i^\vee),\,q^{r}}P_{w,i,q^{r-3d_i}}}{\bQ_{w(\varpi_i^\vee),\,q^{r - 4d_i}}P_{w,i,q^{r-d_i}}}\end{equation}
does not depend on $r$ and is equal to 
$w(-\alpha_i)$
.\end{Lem}

\begin{proof} We have established in the proof of Proposition \ref{renqq} that $-w(\varpi_i)$ is the highest weight of 
$\bQ_{w(\varpi_i^\vee),q^r}\bQ_{w(\varpi_i^\vee),q^{r - 2d_i}}^{-1}$
and that $-ws_i(\varpi_i)$ is the highest weight of 
$\bQ_{ws_i(\varpi_i^\vee),q^r}\bQ_{ws_i(\varpi_i^\vee),q^{r - 2d_i}}^{-1}$. Besides, the $QQ$-relation implies that the highest weight of $P_{w,i,q^{r -d_i}}$ is the highest weight of 
$${\bQ}_{ws_i(\varpi_i^\vee),\,q^r} \bQ_{w(\varpi_i^\vee),\,q^{r-2d_i}}.$$ 
Hence the highest weight of (\ref{fqp}) is the highest weight of 
$$\frac{\bQ_{ws_i(\varpi_i^\vee),q^{r - 2d_i}}\bQ_{w(\varpi_i^\vee)),q^{r}}}{\bQ_{ws_i(\varpi_i^\vee),q^{r}}\bQ_{w(\varpi_i^\vee),q^{r - 2d_i}}}$$
that is
$$ws_i(\varpi_i) - w(\varpi_i) =  w(-\alpha_i).$$
\qed
\end{proof}

\begin{Lem}\label{hwt2} For $r\in\mathbb{Z}$, the highest weight of
$$\frac{\bQ_{w(\varpi_i^\vee),\,q^{r}}}{P_{w,i,q^{r - d_i}}}{\bQ}_{ws_i(\varpi_i^\vee),\,q^{r-2d_i}}$$
does not depend on $r$ and is equal to $w(-\alpha_i)$.
\end{Lem}

\begin{proof} The highest weight of this term is the difference of the highest weight terms of the two left terms of the $QQ$-relation, which 
is known to be $-w(\alpha_i)$ by \cite{FH3}.\qed
\end{proof}

\begin{Prop} We have the following relation for $w\in W$, $i\in I$ and $r\in\mathbb{Z}$
\begin{equation}\label{subst}{\bQ}_{ws_i(\varpi_i^\vee),\,q^{r}}
= \frac{P_{w,i,q^{r - d_i}}}{\bQ_{w(\varpi_i^\vee),\,q^{r-2d_i}}}
 \left(1 +  \frac{\bQ_{w(\varpi_i^\vee),\,q^{r}}
 P_{w,i,q^{r - 3d_i}}}
{\bQ_{w(\varpi_i^\vee),\,q^{r-4d_i}}P_{w,i,q^{r - d_i}}}
 \left(1 +  \frac{\bQ_{w(\varpi_i^\vee),\,q^{r-2}}P_{w,i,q^{r - 5d_i}}}{\bQ_{w(\varpi_i^\vee),\,q^{r-6}}P_{w,i,q^{r - 3d_i}}}
 \left(1 +  \cdots \right)\right)\right)\end{equation}
where the infinite sum converges in the topological ring $K_\Z$.
\end{Prop}

\begin{proof} The $QQ$-relation (\ref{useqq}) can be written as
\begin{equation}\label{exr}{\bQ}_{ws_i(\varpi_i^\vee),\,q^r} = \frac{P_{w,i,q^{r - d_i}}}{\bQ_{w(\varpi_i^\vee),\,q^{r-2d_i}}}
 \left(1 +  \frac{\bQ_{w(\varpi_i^\vee),\,q^{r}}}{P_{w,i,q^{r-d_i}}}{\bQ}_{ws_i(\varpi_i^\vee),\,q^{r-2d_i}}  \right).\end{equation}
 By iterating this relation, we obtain for $r\in\mathbb{Z}$
\begin{equation}\label{exri}{\bQ}_{ws_i(\varpi_i^\vee),\,q^{r}}
= \frac{P_{w,i,q^{r-d_i}}}{\bQ_{w(\varpi_i^\vee),\,q^{r-2d_i}}}
 \left(1 +  \frac{\bQ_{w(\varpi_i^\vee),\,q^{r}}P_{w,i,q^{r-3d_i}}}{\bQ_{w(\varpi_i^\vee),\,q^{r-4d_i}}P_{w,i,q^{r - d_i}}}
 \left(1 +  \frac{\bQ_{w(\varpi_i^\vee),\,q^{r-2d_i}}}{P_{w,i,q^{r - 3d_i}}}{\bQ}_{ws_i(\varpi_i^\vee),\,q^{r-4d_i}} \right)\right)\end{equation}
$$= \frac{P_{w,i,q^{r-d_i}}}{\bQ_{w(\varpi_i^\vee),\,q^{r-2d_i}}}
 \left(1 +  \frac{\bQ_{w(\varpi_i^\vee),\,q^{r}}P_{w,i,q^{r - 3d_i}}}{\bQ_{w(\varpi_i^\vee),\,q^{r-4d_i}}P_{w,i,q^{r-d_i}}}
 \left(1 +  \frac{\bQ_{w(\varpi_i^\vee),\,q^{r-2d_i}}P_{w,i,q^{r-5d_i}}}{\bQ_{w(\varpi_i^\vee),\,q^{r-6d_i}}P_{w,i,q^{r-3d_i}}}
 \left(1 +  \frac{\bQ_{w(\varpi_i^\vee),\,q^{r-4d_i}}}{P_{w,i,q^{r-5d_i}}}{\bQ}_{ws_i(\varpi_i^\vee),\,q^{r-6d_i}}  \right)\right)\right)$$
$$= \cdots$$
$$=\left(\sum_{k=0}^s \frac{\bQ_{w(\varpi_i^\vee),q^r}P_{w,i,q^{r - d_i-2kd_i}}}{\bQ_{w(\varpi_i^\vee),q^{r-2kd_i}}\bQ_{w(\varpi_i^\vee),q^{r-2kd_i-2d_i}}}\right) + \frac{\bQ_{w(\varpi_i^\vee),q^r}\bQ_{ws_i(\varpi_i^\vee),q^{r-2sd_i-2d_i}}}{\bQ_{w(\varpi_i^\vee),q^{r-2d_i-2sd_i}}}
.$$
For $k\in \mathbb{Z}$, by Lemma \ref{hwt}
the highest weight of a factor
$$\frac{\bQ_{w(\varpi_i^\vee),\,q^{r-2kd_i}}P_{w,i,q^{r-2kd_i-3d_i}}}{\bQ_{w(\varpi_i^\vee),\,q^{r-2kd_i - 4d_i}}P_{w,i,q^{r-2kd_i-d_i}}}$$ 
is $w(-\alpha_i)$ which is a negative root as $ws_i < w$. 

Also, by Lemma \ref{hwt2} the highest weight of the factor $$\frac{\bQ_{w(\varpi_i^\vee),\,q^{r-2sd_i}}}{P_{w,i,q^{r - 2sd_i-d_i}}}{\bQ}_{ws_i(\varpi_i^\vee),\,q^{r-2sd_i-2d_i}}$$
is $w(-\alpha_i)$ and does not depend on $r$, $s$.

In the decomposition of ${\bQ}_{ws_i(\varpi_i^\vee),\,q^{r}}$ at step $s$, the monomial containing
 ${\bQ}_{ws_i(\varpi_i^\vee),\,q^{r-2sd_i-2d_i}}$ has highest weight  $sw(-\alpha_i)$ plus a constant (independent of $s$). So this monomial converges to $0$ as $s$ tends to infinity. Since this monomial is the rest of the series, this proves that the series converges and that the substitution (\ref{subst}) makes sense and is satisfied.
 \cqfd
\end{proof}

\begin{example}\label{exqsl2} Let us consider type $A_1$ with $w = e$. We remove the indices of the Dynkin diagram. We recover the known development in \cite{FH2} :
$${\bQ}_{-\varpi^\vee,\,q^{r}}$$
$$= \frac{1}{\bQ_{\varpi^\vee,\,q^{r-2}}}
 \left(1 +  \frac{\bQ_{\varpi^\vee,\,q^{r}}}{\bQ_{\varpi^\vee,\,q^{r-4}}}
 \left(1 +  \frac{\bQ_{\varpi^\vee,\,q^{r-2}}}{\bQ_{\varpi^\vee,\,q^{r-6}}}
 \left(1 +  \frac{\bQ_{\varpi^\vee,\,q^{r-4}}}{\bQ_{\varpi^\vee,\,q^{r-8}}}
 \left(1 + \cdots \right)\right)\right)\right.
= \sum_{k\geq 0} \frac{\bQ_{\varpi^\vee,\,q^{r}}}{\bQ_{\varpi^\vee,\,q^{r-2k}}\bQ_{\varpi^\vee,\,q^{r-2k-2}}}.$$
\end{example}

\begin{example} Let us consider type $A_2$ with $w = e$ and $i= 1$. Then for $r\in \mathbb{Z}$ :
$${\bQ}_{\varpi_2^\vee - \varpi_1^\vee,\,q^{r}}
 = \frac{\bQ_{\varpi_2^\vee,\,q^{r-1}}}{\bQ_{\varpi_1^\vee,\,q^{r-2}}}
 \left(1 +  \frac{\bQ_{\varpi_1^\vee,\,q^{r}}\bQ_{\varpi_2^\vee,\,q^{r-3}}}{\bQ_{\varpi_1^\vee,\,q^{r-4}}\bQ_{\varpi_2^\vee,\,q^{r-1}}}
 \left(1 +  \frac{\bQ_{\varpi_1^\vee,\,q^{r-2}}\bQ_{\varpi_2^\vee,\,q^{r-5}}}{\bQ_{\varpi_1^\vee,\,q^{r-6}} \bQ_{\varpi_2^\vee,\,q^{r-3}}}
 \left(1 +  \cdots \right)\right)\right)$$
$$= \sum_{k\geq 0} \frac{\bQ_{\varpi_1^\vee,\,q^{r}}\bQ_{\varpi_2^\vee,\,q^{r-2k-1}}}{\bQ_{\varpi_1^\vee,\,q^{r-2k}}\bQ_{\varpi_1^\vee,\,q^{r-2k-2}}}.$$
In the same way, for $w = e$ and $i = 2$, one gets 
$${\bQ}_{ \varpi_1^\vee - \varpi_2^\vee,\,q^{r}} = \sum_{k\geq 0} \frac{\bQ_{\varpi_2^\vee,\,q^{r}}\bQ_{\varpi_1^\vee,\,q^{r-2k-1}}}{\bQ_{\varpi_2^\vee,\,q^{r-2k}}\bQ_{\varpi_2^\vee,\,q^{r-2k-2}}}.$$
\end{example}

\begin{example}\label{lex} We consider type $A_2$. Then for $w = s_2$, $i= 1$, one obtains :
$${\bQ}_{-\varpi_2^\vee,\,q^{r}}$$
$$= \frac{\bQ_{\varpi_1^\vee - \varpi_2^\vee,\,q^{r-1}}}{\bQ_{\varpi_1^\vee,\,q^{r-2}}}
 \left(1 +  \frac{\bQ_{\varpi_1^\vee,\,q^{r}} \bQ_{\varpi_1^\vee - \varpi_2^\vee,\,q^{r-3}}}{\bQ_{\varpi_1^\vee,\,q^{r-4}} \bQ_{\varpi_1^\vee - \varpi_2^\vee,\,q^{r-1}}}
 \left(1 +  \frac{\bQ_{\varpi_1^\vee,\,q^{r-2}}\bQ_{\varpi_1^\vee - \varpi_2^\vee,\,q^{r-5}}}{\bQ_{\varpi_1^\vee,\,q^{r-6}} \bQ_{\varpi_1^\vee - \varpi_2^\vee,\,q^{r-3}}}
 \left(1 +  \cdots \right)\right)\right)$$
$$= \sum_{k\geq 0} \frac{\bQ_{\varpi_1^\vee,\,q^{r}}}{\bQ_{\varpi_1^\vee,\,q^{r-2k}}\bQ_{\varpi_1^\vee,\,q^{r-2k-2}}}\bQ_{\varpi_1^\vee - \varpi_2^\vee,\,q^{r-2k-1}}
.$$
\end{example}

\medskip
Using the substitution (\ref{subst}) with a decreasing induction on the length of $w$, one obtains a development of any $Q$-variable ${\bQ}_{w(\varpi_i^\vee),q^r}$ in terms of the 
${\bQ}_{\varpi_j^\vee,q^s}^{\pm 1}$.
As the ${\bQ}_{\varpi_j^\vee,q^s}^{\pm 1}$ generate freely the topological ring $K_\Z$, this development matches the developments of (normalized) $Q$-series (\ref{Qserfh}), (\ref{Qsrfhr}).

\begin{example} We continue Example \ref{lex} :
$${\bQ}_{-\varpi_2^\vee,\,q^{r}}
= \sum_{k\geq 0} \frac{\bQ_{\varpi_1^\vee,\,q^{r}}}{\bQ_{\varpi_1^\vee,\,q^{r-2k}}\bQ_{\varpi_1^\vee,\,q^{r-2k-2}}}
\sum_{k'\geq k} \frac{\bQ_{\varpi_2^\vee,\,q^{r-2k-1}}\bQ_{\varpi_1^\vee,\,q^{r-2k'-2}}}{\bQ_{\varpi_2^\vee,\,q^{r-2k'-1}}\bQ_{\varpi_2^\vee,\,q^{r-2k'-3}}}.$$
\end{example}

\medskip

 As a result, we have obtained a recursive process in $K_\Z$ to compute the development of a $Q$-series in terms of the ${\bQ}_{\varpi_j^\vee,q^s}^{\pm 1}$. Each step of this process is a $QQ$-relation.

\subsection{Interpretation as limits of cluster expansions}\label{lcle}

Let us interpret the above recursive process from the point of view of the cluster structure.

We will consider the same infinite sequences of mutations as in Section \ref{ssec:g-vect}. We will check that each exchange relation in these sequences corresponds to a $QQ$-relation in the recursive process. (For columns corresponding to $d_i > 1$, the sequence of mutations changes column at each step but the index $i$ of the columns remains the same. This means that we handle $d_i
$ recursive processes from the previous section in parallel. However, this is not a problem since the mutation for such different columns commute.)

We can work with the seed $\Sigma_{\bi,\br}$ in $K_{\Z}$ that we have associated to $\mathcal{S}_{\bi,\br}$. Recall the infinite sequence of mutations going from $\Sigma^{(l)} = \Sigma_{\bi,\br}$
to $\Sigma^{(0)}$, via the intermediate steps 
\[
\Sigma^{(m)} = \Sigma_{(i_1,\ldots,i_{m}),(r_1,\ldots,r_{m})} \ (1\le m \le l-1).
\]
To go from $\Sigma^{(m)}$ to $\Sigma^{(m-1)}$, we have a sequence of intermediate seeds
\[\Sigma^{(m,s)} := \Sigma_{(i_1,\ldots,i_{m}),(r_1,\ldots,r_{m}- 2s)}\] 
with $s\geq 0$ and $\Sigma^{(m,0)} = \Sigma^{(m)}$.

The exchange relation to go from $\Sigma^{(m,s)}$ to $\Sigma^{(m,s+1)}$ 
is a $QQ$-relation (\ref{useqq}), with index $i_m$. The Weyl group element $w_m = s_{i_1}\dots s_{i_{m-1}}$ depends only on $m$, not on $s$. The previous cluster variable is ${\bQ}_{w_ms_{i_m}(\varpi_{i_m}^\vee),\,q^{N - 2s}}$ and the new cluster variable is $\bQ_{w_m(\varpi_{i_m}^\vee),\,q^{N -2s - 2d_{i_m}}}$ for an integer $N$ which depends only on $m$ (and not on $s$). Besides, the exchange relation can be rewritten as (\ref{exr}).

For $s = 0$, this is the cluster expansion of 
${\bQ}_{w_ms_{i_m}(\varpi_{i_m}^\vee),\,q^{N}}$ in the seed $\Sigma^{m,1}$. We can iterate this relation as explained in the previous section. Indeed, the next $QQ$-relation in the iteration, where ${\bQ}_{w_ms_{i_m}(\varpi_{i_m}^\vee),\,q^{N - 2d_{i_m}}}$ and $\bQ_{w_m(\varpi_{i_m}^\vee),\,q^{N - 3d_{i_m}}}$ appear, corresponds to the mutation from 
$\Sigma^{(m,d_{i_m})}$ to $\Sigma^{(m,d_{i_m}+1)}$. Note that if $d_{i_m} > 1$, the intermediate mutation from 
$\Sigma^{(m,1)}$ to $\Sigma^{(m,2)}$
does not affect the arrows going from or to the vertices attached to the cluster variable ${\bQ}_{w_ms_{i_m}(\varpi_{i_m}^\vee),\,q^{N - 2d_{i_m}}}$ in the seed $\Sigma^{(m,0)}$ (for $G_2$ type, this is also true for the mutation from $\Sigma^{(m,2)}$ to $\Sigma^{(m,3)}$).

Then we obtain the cluster expansion of ${\bQ}_{w_ms_{i_m}(\varpi_{i_m}^\vee),\,q^{N}}$ in the seed $\Sigma^{(m,d_{i_m}+1)}$. Continuing the mutation sequence, we iterate the $QQ$-relation and we obtain the expansion of ${\bQ}_{w_ms_{i_m}(\varpi_{i_m}^\vee),\,q^{N}}$ in the seed $\Sigma^{(m,s)}$ for $s\geq 0$. The limit $s\rightarrow +\infty$ of the cluster expansions makes sense as we obtain  the substitution (\ref{subst}). It expresses each
${\bQ}_{w_ms_{i_m}(\varpi_{i_m}^\vee),\,q^{N}}$ in terms of the ${\bQ}_{w_m(\varpi_{j}^\vee),\,q^{N'}}$ with $j\in I$ and $N' < N$, that is cluster variables of the seed $\Sigma^{(m-1)}$.
In parallel, the mutations on the other columns have allowed to express ${\bQ}_{w_ms_{i_m}(\varpi_{i_m}^\vee),\,q^{N-2}}$ in terms of the ${\bQ}_{w_m(\varpi_{j}^\vee),\,q^{N'}}$ with $j\in I$ and $N' < N - 2$, which are cluster variables of the seed $\Sigma^{(m-1)}$. Then we can iterate to express 
${\bQ}_{w_ms_{i_m}(\varpi_{i_m}^\vee),\,q^{N}}$ in terms of the cluster variables of the seed 
$\Sigma^{(0)}$, that is we obtain an expression involving only the variables ${\bQ}_{\varpi_j^\vee,\,q^{N'}}$ with $j\in I$ and $N' < N$.
As the ${\bQ}_{\varpi_j^\vee,q^s}$ generate freely the topological ring $K_\Z$, as discussed above this matches the developments of $Q$-series (\ref{Qserfh}, \ref{Qsrfhr}).

Thus, we see that we can interpret the $Q$-series as a limit of cluster expansions with respect to the successive seeds. In particular, it has positive integral coefficients by \cite{LS}.

\begin{Thm}\label{posq} The $Q$-series (\ref{Qserfh}) have positive integral coefficients.
\end{Thm}

\begin{Rem} (i) This is a strong evidence for the conjecture in \cite{FH2} stating that $Q$-series (\ref{Qserfh}) are $q$-characters of representations of shifted quantum affine algebras (see the next section). 

(ii) This Theorem implies the positivity of the series $\chi_{w(s_i)}$ defined combinatorially in \cite[Lemma 5.3]{FH2} (which can be also obtained from $Q$-variables by forgetting the spectral parameter). This posivity was conjectured therein, see \cite[Remark 6.10]{FH2}.
\end{Rem}

\subsection{Limit cluster expansion of cluster monomials}

The result of the previous subsection can be generalized to cluster monomials in the following way.

\begin{Thm}\label{devtp} Every cluster monomial in $K_\Z$ admits a development as a Laurent formal power series in the ${\bQ}_{\varpi_j^\vee,q^s}$ with positive integral coefficients.
\end{Thm}

This development will be called the {\it limit cluster expansion} of the cluster monomial.

\begin{proof} It suffices to prove the statement for cluster variables.
The reasoning is the same as for $Q$-variables, considering the cluster expansion in each seed $\Sigma^{(m,s)}$. 

More precisely, consider a cluster variable $\chi$. It has a positive cluster expansion in our initial seed $\Sigma^{(m)}$. There is a finite number of $Q$-variables $\bQ_t$, $1\leq t\leq T$, that occur in this expansion :  
$$\chi = L(\bQ_1,\bQ_2, \cdots, \bQ_T)$$
for a Laurent polynomial $L$. There might be some $t$ so that $\bQ_t$ is a cluster variable of $\Sigma^{(m,s)}$ for $s$ large enough (that is, it is a cluster variable of $\Sigma^{(m-1)}$). For the other $t$, for $s$ large enough, we have written (\ref{exri}) the explicit expansion of $\bQ_t$ in $\Sigma^{(m,s)}$. It is of the form 
$$\bQ_t = P_t^{(s)} + R_t^{(s)}$$
where $P_t^{(s)}$ is a Laurent polynomial in cluster variables of $\Sigma^{(m-1)}$ of highest weight $h_t$ that does not depend on $s$, and $R_t^{(s)}$ is a Laurent monomial of highest weight $h_t + s \alpha_t$ where $\alpha_t$ is a negative root that does not depend on $s$. If $\bQ_t$ is already a cluster variable of $\Sigma^{(m-1)}$, we will just write $P_t^{(s)} = \bQ_t$.

As we will consider Laurent polynomials, we note that we have also
$$\bQ_t^{-1} 
= \frac{1}{P_t^{(s)} + R_t^{(s)}}
= (P_t^{(s)})^{-1}\sum_{l\geq 0} (-(P_t^{(s)})^{-1}R_t^{(s)})^l$$
$$=(P_t^{(s)})^{-1} + \sum_{l\geq 1} (P_t^{(s)})^{-l-1}(-R_t^{(s)})^l
=(P_t^{(s)})^{-1} +  S_t^{(s)}$$
where $S_t^{(s)}$ has highest weight $-h_t + s \alpha_t$ as each $(P_t^{(s)})^{-l-1}(-R_t^{(s)})^l$
has highest weight 
$$-(l+1)h_t + l (h_t + s\alpha_t) = -h_t + ls\alpha_t.$$
Hence, for a fixed weight $\omega$ and $s$ large enough, the terms of weight $\omega$ are the same in $\chi$ and in 
$$P^{(s)} = L(P_1^{(s)},P_2^{(s)},\cdots , P_T^{(s)}).$$
Each $P_t^{(s)}$ converges to $\bQ_t$ as $s$ tends to infinity and $P^{(s)}$ converges to $\chi$. 
Hence, for a fixed weight $\omega$, the terms in $P^{(s)}$ of weight $\omega$ become stationary as $s$ tends to infinity. This implies that, for a given  Laurent monomials of cluster variables of $\Sigma^{(m)}$, its multiplicity in $P^{(s)}$ is stationary. From the positivity \cite{LS} of the cluster expansion in the seed $\Sigma^{(m,s)}$, this implies that $\chi$ is a formal power series of cluster variables in $\Sigma^{(m-1)}$ with positive coefficients. We continue inductively to establish that  is a formal power series with positive coefficients of cluster variables in $\Sigma^{(m-2)}$, $\Sigma^{(m-3)}$ and finally in $\Sigma^{(0)}$. 
\qed
\end{proof}

\begin{Rem}
(i) The cluster expansion of a cluster variable can be written in terms of its $F$-polynomial and its $g$-vector. We have already established the convergency of $g$-vectors. This result can be seen as a convergence property of $F$-polynomials.

(ii) We will explain in the next section how to extend the statement of the conjecture of \cite{FH2} to reachable objects.
\end{Rem}

\begin{example}\label{excl} We continue Example \ref{exqsl2}. We consider a cluster variable $\chi$ obtained after one mutation of a cluster variable of the form 
$\bQ_{-\varpi^\vee,q^{-2N}}$ in the seed $\Sigma^{(1)}$ with exchange relation :
$$\chi\bQ_{-\varpi^\vee,q^{-2N}} = \bQ_{-\varpi^\vee,q^{-2N - 2}} + \bQ_{-\varpi^\vee,q^{-2N + 2}}.$$
The expansion in $\Sigma^{(1)}$ is 
\begin{equation}\label{exp1}\chi = \frac{\bQ_{-\varpi^\vee,q^{-2N - 2}}}{\bQ_{-\varpi^\vee,q^{-2N}}} + \frac{\bQ_{-\varpi^\vee,q^{-2N + 2}}}{\bQ_{-\varpi^\vee,q^{-2N}}}.\end{equation}
We have seen above the expansion in $\Sigma^{(1,s)}$ for $s$ large enough :  
$$\bQ_{-\varpi^\vee,q^{-2N}} 
= \sum_{k=0}^s \frac{\bQ_{\varpi^\vee,q^{-2N}}}{\bQ_{\varpi^\vee,q^{-2N-2k}}\bQ_{\varpi^\vee,q^{-2N-2k-2}}} + \frac{\bQ_{\varpi^\vee,q^{-2N}}}{\bQ_{\varpi^\vee,q^{-2N-2-2s}}}
\bQ_{-\varpi^\vee,q^{-2N-2s-2}},$$
where the sum will be denoted
$$P_{q^{-2N}}^{(s)} = \sum_{k=0}^s \frac{\bQ_{\varpi^\vee,q^{-2N}}}{\bQ_{\varpi^\vee,q^{-2N-2k}}\bQ_{\varpi^\vee,q^{-2N-2k-2}}}.$$
It converges to $\bQ_{-\varpi,q^{-2N}}$ as $s$ tends to infinity. We can substitute in 
(\ref{exp1}) and we obtain 
$$P^{(s)} = \frac{P_{q^{-2N - 2}}^{(s)}}{P_{q^{-2N}}^{(s)}} + \frac{P_{q^{-2N + 2}}^{(s)}}
{P_{q^{-2N}}^{(s)}}$$
$$= \frac{\bQ_{\varpi^\vee,q^{-2N-2}}}{\bQ_{\varpi^\vee,q^{-2N}}} + \frac{\bQ_{\varpi^\vee,q^{-2N+2}}}{\bQ_{\varpi^\vee,q^{-2N}}}
+ \left(\frac{-\frac{\bQ_{\varpi^\vee,q^{-2N+2}}\bQ_{\varpi^\vee,q^{-2N-2}}}{\bQ_{\varpi^\vee,q^{-2N-2s}}\bQ_{\varpi^\vee,q^{-2N-2s-2}}} + \frac{\bQ_{\varpi^\vee,q^{-2N-2}}^2}{\bQ_{\varpi^\vee,q^{-2N-2s-2}}\bQ_{\varpi^\vee,q^{-2N-2s-4}}} }{\sum_{k = 0}^s\frac{\bQ_{\varpi^\vee,q^{-2N}}\bQ_{\varpi^\vee,q^{-2N-2}}}{\bQ_{\varpi^\vee,q^{-2N-2k}}\bQ_{\varpi^\vee,q^{-2N-2k-2}}}}\right).
$$
The highest weight of the third term of the right-hand side is $-\omega -s\alpha$. Hence the expression converges to 
\begin{equation}\label{excle}\chi = \frac{\bQ_{\varpi^\vee,q^{-2N-2}}}{\bQ_{\varpi^\vee,q^{-2N}}} + \frac{\bQ_{\varpi^\vee,q^{-2N+2}}}{\bQ_{\varpi^\vee,q^{-2N}}}.
\end{equation}
\end{example}

\section{Shifted quantum affine algebras and refined conjectures}
\label{sec:shifted-quantum}

In this section, we give a rapid overview of shifted quantum affine algebras and their
category~$\O^{\rm sh}$. The reader is referred to \cite{FT, H, GHL} for more details and references.
The main motivations for this study come from 
quantized $K$-theoretic Coulomb branches, 
and from representation theory of Borel subalgebras of quantum affine algebras.

We explain that the ring $K_\Z$ coincides with the $q$-character ring of a subcategory $\O_\Z$ of $\O^{\rm sh}$. 
This allows us to translate Theorem~\ref{main-Thm} into Theorem~\ref{isomgv}, and to give an interpretation of the cluster 
algebra $\AA_{w_0}$ in terms of the category $\O_\Z$. 

We also prove that the $q$-character of a simple object corresponding to a cluster monomial (a reachable object) is given by its limit cluster expansion (Theorem \ref{qc}). This allows us to state a refined Conjecture \ref{main-conj2}, stating that the limit cluster expansion of any cluster monomial is the $q$-character of a simple object. These last points are also new in the $A,D,E$-types.

\subsection{Representations of shifted quantum affine algebras}\label{sub-sec-shift-quantum}

We fix a non-zero complex number $q$, not a root of unity. 
 
\begin{Def}[\cite{FT}] Let $\mu$ be an integral coweight of $\mathfrak{g}$.
The shifted quantum affine algebra $U_q^\mu(\hat{\mathfrak{g}})$
  is the $\C$-algebra with the same Drinfeld generators $\phi_{i,r}^\pm$, $x_{i,r}^\pm$ ($i\in I$, $r\in\mathbb{Z}$) as the
  ordinary quantum affine algebra $U_q(\hat{\mathfrak{g}})$ and the same relations, except that
  the definition of the series $\phi_i^-(z)$ is
  modified as follows:
\begin{equation}\label{phrelshm}\phi_i^-(z) =
  \sum_{m\in\mathbb{Z}}\phi_{i, m}^- z^{ m} =
  \phi_{i,\alpha_i(\mu)}^-z^{\alpha_i(\mu)} \exp\left(- (q^{d_i} -
  q^{-d_i})\sum_{r > 0} h_{i,-r } z^{- r} \right).\end{equation}
In addition, we impose the condition that $\phi_{i,0}^+$ and
$\phi_{i,\alpha_i(\mu)}^-$ are invertible, but not necessarily inverse to each other.
\end{Def}

The complete list of relations of $U_q^\mu(\hat{\mathfrak{g}})$ can be found in \cite[Section 5]{FT}, see also \cite[Section 3.1]{H}.
Recall the basis $(h_i)_{i\in I}$ of the Cartan subalgebra $\mathfrak{h}$. For a $U_q^\mu(\hat{\mathfrak{g}})$-module $V$ and $\lambda\in P$, we define
\[
V_{\lambda}:=\left\{v\in V \mid \phi_{i,0}^+\, v = q^{d_i\lambda(h_i)} v,\ \mbox{ for every } i\in I\right\},
\]
and call it the weight space of $V$ of weight $\lambda$.
We say that $V$ is Cartan-diagonalizable
if
\[
V = \bigoplus_{\lambda\in P}V_{\lambda}.
\]

For $\lambda\in P$, recall the cone $C_\lambda = \lambda - Q_+\subset P$ defined in Section \ref{sect3-2}.

\begin{Def}[\cite{H}] A $U_q^\mu(\hat{\mathfrak{g}})$-module $V$ is said to be in category $\mathcal{O}^\mu$ if
\begin{itemize} 
\item[i)] $V$ is Cartan-diagonalizable;
\item[ii)] for all $\lambda\in P$ we have 
$\dim (V_{\lambda})<\infty$;
\item[iii)] there exist finitely many weights 
$\lambda_1,\ldots,\lambda_s\in P$ 
such that all weights of $V$ are contained in 
${\bigcup}_{j=1,\ldots, s}\,C_{\lambda_j}$.
\end{itemize}
\end{Def}

\begin{Def} A collection $\Psib=(\Psi_{i, m})_{i\in I, m\geq 0}$
of complex numbers $\Psi_{i,m}$ such that $\Psi_{i,0} \in q^\Z$ for all $i\in I$ is called an $\ell$-weight. 
\end{Def}
 
We will represent every $\ell$-weight as an
$I$-tuple of formal power series in $z$:
\[
\Psib = (\Psi_i(z))_{i\in I},
\qquad
\Psi_i(z) := \sum_{m\geq 0} \Psi_{i,m} z^m.
\]
Then the set $\tb^\times_\ell$ of $\ell$-weights has a natural group structure by pointwise multiplication. We have a surjective group homomorphism $\varpi :
\tb^\times_\ell\rightarrow P$ given by
\begin{equation}
\label{spep}
 \varpi(\Psib) := \sum_i a_i\varpi_i, \qquad (\Psib\in \tb^\times_\ell),
\end{equation}
where for $i\in I$ we have written $\Psib_i(0) = q^{d_ia_i}$.
This is well-defined because $q$ is not a root of unity.

\begin{Def}
We say that an $\ell$-weight $\Psib\in \tb^\times_\ell$ is rational if 
$\Psi_i(z)$ is the expansion of a rational function in $z$ for every $i\in I$.
We denote by $\mathfrak{r}$ the subgroup of $\tb^\times_\ell$ consisting of all
rational $\ell$-weights.
\end{Def}

\begin{Thm}[\cite{H}] \label{Th-classification}
The simple representations in $\mathcal{O}^\mu$ are parameterized by the subset $\mathfrak{r}_\mu$ of rational $\ell$-weights of degree $\mu$, that is, the rational $\ell$-weights $\Psib = (\Psi_i(z))_{i\in I}$  such that: 
$$ \sum_{i\in I} \deg(\Psi_i(z)) \varpi_i^\vee = \mu.$$ 
\end{Thm}
The simple representation associated with $\Psib\in\mathfrak{r}_\mu$ will be denoted by $L(\Psib)$. 

For instance, we have the prefundamental representations defined as simple representations from the simplest rational $\ell$-weights of degree $\pm\varpi_i^\vee$ as follows \cite{HJ, H}. For $i\in I$ and $a\in\C^\times$, let 
\begin{equation}
L_{i,a}^\pm := L\left(\Psib_{i,a}^{\pm 1}\right)
\quad \mbox{where}\quad 
(\Psib_{i,a})_j(z) := 
\left\{ 
\begin{array}{cc}
1 - za & (j=i)\,,\\
1 & (j\neq i)\,.\\
\end{array}
\right.
\label{fund-rep}
\end{equation}
Also, for $\lambda \in P$ and $\Psib_\lambda = (q^{d_i\lambda(h_i)})_{i\in I}$, $L(\Psib_\lambda)$ is a $1$-dimensional invertible representation in $\mathcal{O}^0$. 


\begin{Thm}[\cite{H}] \label{thmH}
The algebra $U_q^\mu(\hat{\mathfrak{g}})$ has a non-zero finite-dimensional representation if and only if $\mu$ is a dominant coweight.
\end{Thm}
To illustrate, the $U_q^{\varpi_i^\vee}(\hat{\mathfrak{g}})$-module $L_{i,a}^+$ is $1$-dimensional.

\subsection{Grothendieck ring}
 
Let $K_0(\mathcal{O}^\mu)$ be the Grothendieck group of the category $\mathcal{O}^\mu$ and $\mathcal{O}^{\,\mathrm{sh}}$ the direct sum of the abelian categories $\mathcal{O}^\mu$.
It is shown in \cite{H} that the direct sum of groups
$$K_0(\mathcal{O}^{\,\mathrm{sh}}) := \bigoplus_{\mu}K_0(\mathcal{O}^\mu)$$
has a ring structure, coming from a fusion product and the deformed Drinfeld coproduct.

Recall the ring $[P]$ in Section \ref{Kz}. For $\lambda \in P$, the class of $L(\lambda)$ will be denoted by $[\lambda]$, which makes $K_0(\mathcal{O})$ as a $[P]$-algebra.

Let $\mathcal{E}_\ell$ be the additive group of maps $\gamma : \mathfrak{r} \rightarrow \Z$
so that $\varpi(\mathrm{supp}(\gamma))$ is contained in
a finite union of sets of the form $C_\mu$, where
$$\mathrm{supp}(\gamma) := \{\Psib\in \mathfrak{r}\mid \gamma(\Psib) \neq 0\}.$$
Then $\mathcal{E}_\ell$ has a natural ring structure.
For $\Psib\in \tb^\times_\ell$, we define $[\Psib] = \delta_{\Psib,-}\in\mathcal{E}_\ell$. In particular, for $\lambda\in P$ we set $[\lambda] := [\Psib_\lambda]$ and $\mathcal{E}_\ell$ is also a $[P]$-algebra.

We say that $\Psib$ is an $\ell$-weight of a representation $V$ if we have a non-zero subspace 
$$V_\Psib = \{v\in V |(\phi_{i,m}^+ - \Psib_{i,m})^r.v = 0,\text{ for } i\in I, m\geq 0 \text{ and }r >> 0\}.$$

\begin{Thm}[\cite{H}] 
The $\ell$-weights of a representation in $\mathcal{O}^\mu$ belong to $\mathfrak{r}_\mu$.
\end{Thm}

It follows that every representation $V$ in the category $\mathcal{O}^\mu$ has a $q$-character
$$\chi_q(V) = \sum_{\Psib\in
\mathfrak{r_\mu}}\dim(V_{\Psib})[\Psib]\in\mathcal{E}_\ell.$$
It gives rise to an injective morphism of $[P]$-algebra \cite[Corollary 5.1]{H}:
$$\chi_q : K_0(\mathcal{O}^{\,\mathrm{sh}})\rightarrow \mathcal{E}_\ell.$$
In fact, $\chi_q$ is a ring isomorphism 
(the argument is the same as in \cite[Theorem 4.19]{W}).

Since $\mathcal{E}_\ell$ has the structure of a topological ring, $K_0(\mathcal{O}^{\,\mathrm{sh}})$ itself inherits the structure of a topological ring. The classes $[L(\Psib)]\ (\Psib \in \mathfrak{r})$ of simple modules form a topological basis of $K_0(\mathcal{O}^\mathrm{sh})$. 

\begin{Def}
Let $\O_\Z$ denote the full subcategory of $\mathcal{O}^{\,\mathrm{sh}}$ consisting of the objects $M$ which satisfy the following property:   
every simple subquotient of $M$ is of the form $L(\Psib)$ where $\Psib$ is a Laurent monomial in the variables $\Psib_\lambda\ (\la\in P)$ and $\Psib_{i,q^r}\ (i,r)\in V$.
\end{Def}

The Grothendieck group $K_0(\mathcal{O}_{\mathbb{Z}})$ is the $[P]$-topological subspace of $K_0(\mathcal{O}^{\,\mathrm{sh}})$
generated by the simple classes $[L(\Psib)]$ such that
$\Psib\in \mathbb{Z}[\Psib_{i,q^r}^{\pm 1}]_{(i,r)\in V}$.

The following was established in \cite{GLS} for simply-laced types, the proof is analog in general.

\begin{Prop}\label{KE} $K_0(\mathcal{O}_{\mathbb{Z}})$ is a subring of $K_0(\mathcal{O}^{\,\mathrm{sh}})$ and $\chi_q$ restricts to an isomorphism
$$\chi_q : K_0(\mathcal{O}_{\mathbb{Z}}) \rightarrow \mathcal{E}_{\ell, \mathbb{Z}}$$
where $\mathcal{E}_{\ell,\mathbb{Z}}$ is the $[P]$-topological subring of $\mathcal{E}_\ell$ generated by the $[\Psib_{i,q^r}^{\pm 1}]$, $(i,r)\in V$.
\end{Prop}

\subsection{Isomorphism and first conjecture}

Recall the subring $K_\mathbb{Z}\subset [\uP]\otimes_{[P]} K$.
As each generator of $K_\mathbb{Z}$ belongs to $\mathcal{E}_{\ell,\mathbb{Z}}$, we can identify $K_\mathbb{Z}$ with $ [\uP]\otimes_{[P]}\mathcal{E}_{\ell,\mathbb{Z}}$. Combining Proposition \ref{KE}, Theorem~\ref{main-Thm} now implies the following:

\begin{Thm}\label{isomgv} 
We have an injective ring homomorphism
$$I : \mathcal{A}_{w_0}\rightarrow [\uP]\otimes_{[P]}K_0(\mathcal{O}_{\mathbb{Z}})$$
and the topological closure of $ [\uP]\otimes_{[P]} I(\mathcal{A}_{w_0})$ is the entire topological ring $ [\uP]\otimes_{[P]} K_0(\mathcal{O}_{\mathbb{Z}})$.
\end{Thm}

As established above, this homomorphism does not depend on the choice of a compatible initial seed.

We say that a simple class $[L(\Psib)]$ in $K_0(\mathcal{O}^\mathrm{sh})$ is real if $[L(\Psib)]^2 = [L(\Psib^2)]$ is a simple class. The main conjecture of \cite{GHL} can now be formulated for general types.

\begin{Conj}\label{main-conj} 
The image of a cluster monomial of $\mathcal{A}_{w_0}$ is a real simple class in $ [\uP]\otimes_{[P]}K_0(\mathcal{O}_{\mathbb{Z}})$.
\end{Conj}

\begin{Rem} (i) The results in \cite{HN} allow to reformulate these statements in terms of representations of the Borel subalgebra of an ordinary quantum affine algebra.

(ii) It is proved in \cite{HZ} that the subcategory $\mathcal{O}_{\mathbb{Z}}^f$ of finite length representations in $\mathcal{O}_{\mathbb{Z}}$ is stable by fusion product. Hence this Conjecture implies that the image of $\mathcal{A}_{w_0}$ lies in $[\uP]\otimes_{[P]}K_0(\mathcal{O}_{\mathbb{Z}}^f)$. One expects that this last ring is isomorphic to $[\uP]\otimes \mathcal{A}_{w_0}$. This is known if one restricts to finite-dimensional modules by \cite{H}, and it has been recently proved in $ADE$-types in \cite{KLLLPW}.
\end{Rem}

\subsection{Examples}

The following evidences were given in \cite{GHL} and also hold in the non simply-laced types.

\medskip

\subsubsection{Finite-dimensional representations}

Let $\mathcal{C}_{\mathbb{Z}}$ denote the full subcategory of
$\mathcal{O}_{\mathbb{Z}}$ whose objects are the
finite-dimensional modules. Its simple objects have been classified in \cite{H}.
For example, the positive prefundamental representations $L_{i,q^r}^+ \ ((i,r)\in V)$ are contained in 
$\mathcal{C}_{\mathbb{Z}}$. The category $\mathcal{C}_{\mathbb{Z}}$ also contains the fundamental representations $V_{i,q^{r+d_i}} = L(Y_{i,q^{r+d_i}})$ of the ordinary quantum affine algebra $U_q(\hg)$.

Consider the corresponding subring $K_0(\mathcal{C}_{\mathbb{Z}})$ of $K_0(\mathcal{O}_{\mathbb{Z}})$. Recall the cluster algebra $\AA_\G = \mathcal{A}_e$ associated with the quiver $\G=\Gamma_e$. 
Building on results of \cite{HJ} and \cite{KKOP}, the following theorem was established in \cite{H}.

\begin{Thm}\label{fdrep} There is an isomorphism of algebras 
\[
[\uP]\otimes_{[P]} \mathcal{A}_e \to  [\uP]\otimes_{[P]} K_0(\mathcal{C}_{\mathbb{Z}}) 
\]
mapping the initial cluster variable of $\mathcal{A}_e$ attached to vertex $(i,r)$ of $\G_e$ to the renormalized positive prefundamental representation $ [-\Omega(\Psi_{i,q^r})]L(\Psi_{i,q^r})$.
Moreover, all cluster monomials in $\mathcal{A}_e$ are mapped to simple real classes in 
$[\uP]\otimes_{[P]}K_0(\mathcal{C}_{\mathbb{Z}})$.
\end{Thm}

The same discussion as in \cite{GHL} establishes that Theorem \ref{fdrep} implies Conjecture \ref{main-conj} for all cluster monomials of $\AA_{e}$.

\subsubsection{$Q$-variables}

\medskip

For $i\in I$, $w\in W$, and $a\in \C^*$, we have the $Q$-seires 
\[
Q_{w(\varpi_i^\vee),\,a} \in K,
\]
whose leading term $\Psi_{w(\varpi_i),\,a}$ is a Laurent monomial in the variables $\Psi_{j,b}$, explicitly computed in \cite{FH3} in terms of Chari's braid group action. It is the image of a $Q$-variable in the cluster algebra.

The following conjecture was formulated in \cite{FH3}. 

\begin{Conj}\label{cjq} We have the $q$-character formula :
$$\chi_q(L(\Psi_{w(\varpi_i^\vee),a})) = Q_{w(\varpi_i^\vee),a}.$$
\end{Conj}

\begin{Rem}
(i) The argument as in \cite{GHL} establishes that this Conjecture is compatible with Conjecture \ref{main-conj} for $Q$-variables. 

(ii) By \cite{FH3} and by the arguments in \cite{GHL}, this Conjecture holds for $Q$-variables associated with $w=e$, $w=s_i\ (i\in I)$ and $w = w_0$.

(iii) In $ADE$-types, Conjecture~\ref{cjq} was recently established in \cite{KLLLPW}.
\end{Rem}

\subsection{Refined conjecture and limit cluster expansions}

Recall the limit cluster expansion of a cluster monomial introduced in Section \ref{lcle}. By replacing each $Q$-variable $\bQ_{\varpi_i^\vee, q^r}$ by $[\Omega(\Psib_{i,q^r})^{-1}][\Psib_{i,q^r}]$, we can regard this limit cluster expansion as an element of $\mathcal{E}_\ell$.

A simple module in $\mathcal{O}_\mathbb{Z}$ is said be reachable if its class in $K_0(\mathcal{O}_\mathbb{Z})$ is the image of a cluster monomial of $\mathcal{A}_{w_0}$.
We have the following remarkable consequence of our constructions since the positive prefundamental representations are 1-dimensional.

\begin{Thm}\label{qc} The $q$-character of a reachable simple object is given by the limit cluster expansion of the corresponding cluster monomial.
\end{Thm}

\begin{example} In Example \ref{excl}, we recognize in the limit cluster expansion (\ref{excle}) the $q$-character of a $2$-dimensional fundamental representation of the ordinary quantum affine algebra $U_q(\hat{\mathfrak{sl}_2})$.
\end{example}

This leads to the following general conjecture, whose statement is already new in simply-laced types.

\begin{Conj}\label{main-conj2} 
The limit cluster expansion of a cluster monomial of $\mathcal{A}_{w_0}$ is the $q$-character of a real simple object in  $\mathcal{O}_{\mathbb{Z}}$. 
\end{Conj}

\begin{Rem} (i) Conjecture~\ref{main-conj2} is more precise than Conjecture \ref{main-conj}. Indeed, it yields a formula for the $q$-character, and it allows to specify the simple module in question by looking at the highest $\ell$-weight of this
$q$-character.

(ii) Conjecture~\ref{main-conj2} generalizes  Conjecture \ref{cjq} to all reachable objects.

(iii) The positivity of limit cluster expansion (Theorem \ref{devtp}) gives a strong evidence. 

(iv) By Theorem \ref{qc}, all the previous examples give evidences supporting Conjecture~\ref{main-conj2}.
\end{Rem}
We are planning to come back to these questions in a forthcoming paper.

\appendix
\section{Appendix: The $\sfG_2$-braid relation for the initial quivers} 
\label{App:G2-braid}
In order to describe the promised sequence of 163 mutations in type $\sfG_2$ (see the end of the proof of Theorem~\ref{thm:mut_equivQ}), we need to introduce a larger class of (infinite) quivers and describe them in a compact way. In fact, we need 20 such intermediate quivers, and for each of those quivers the relevant part has about 50 vertices and would occupy a full page. 
Similarly to Section~\ref{ssec:B2-F4},
those quivers are parametrized by reduced expression for the longest element
in the unfolded Weyl group $W(\tsfG_2)\cong W(\sfD_4)$. 

We consider the Dynkin diagram of type $\sfG_2$ and its “unfolding"
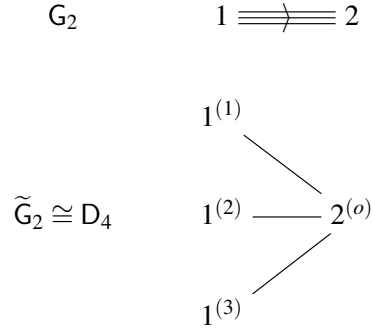
\[\begin{tikzcd}[cramped, sep=small]
	{\mathsf{G}_2} && 1 && 2 \\
	\\
	&& {1^{(1)}} \\
	\\
	{\widetilde{\mathsf{G}}_2\cong\mathsf{D}_4} && {1^{(2)}} && {2^{(o)}} \\
	\\
	&& {1^{(3)}}
    	\arrow["{\mathbf{\rangle}}"{marking, allow upside down}, Rightarrow, scaling nfold=3, no head, from=1-3, to=1-5]
	\arrow[no head, from=3-3, to=5-5]
	\arrow[no head, from=5-3, to=5-5]
	\arrow[no head, from=7-3, to=5-5]
\end{tikzcd}\]
together with the projection $\overline{?}\df\{1^{(1)}, 1^{(2)}, 1^{(3)}, 2^{(0)}\}\mapsto\{1,2\}, i^{(*)}\mapsto i$.  
Thus, we are dealing with the Cartan matrix 
\[
C=\begin{pmatrix}
  2 & -1\\ -3 & 2  
\end{pmatrix}
\quad\text{with symmetrizer}\quad
D=\begin{pmatrix}
 3 & 0  \\ 0 & 1  
\end{pmatrix}
\quad\text{and}\quad
B:= D\cdot C=\begin{pmatrix}
    6 & -3\\ -3 & 2
\end{pmatrix}.
\]
For typographic reasons we agree that the basic quiver $\Gam$ has here the vertex set $V'=\Z$ and the arrow set
\[
\Gam_1=\{\alp_k^+\mid k\in\Z\}\cup\{\alp_k^o\mid k\in\Z\}
\]
with 
\begin{align*}
  s(\alp^+_k)&=k, &  t(\alp^+_k) &= k+2d_k,\\
  s(\alp^o_k)&=k, &  t(\alp^o_k) &= k-3,
\end{align*}
for all $k\in\Z$. Here we have used the notation 
\[
d_k:=
\begin{cases}
  1 &\text{if } k\equiv 1 \pmod{2},\\
  3 &\text{if } k\equiv 0 \pmod{2},
\end{cases}
\qquad 
(k\in\Z).
\]
Thus, compared with our general convention,
we identify here  $V=(\{1\}\times(2\Z+1))\,\cup\,(\{2\}\times 2\Z)$
with $\Z$ by projecting onto the second component. As for the
arrows, we replaced accordingly $\alp_{i,i}^{(k)}$ with
$(i,k)\in V$ by $\alp^+_k$, and $\alp_{ij}^{(k)}$ with
$i\neq j$ and $(j,k)\in V$ by $\alp^o_k$.
Note, that this identification is indeed bijective on vertices and on  arrows.

Consider the map 
\[
w\df\Z\ra\{1^{(1)}, 1^{(2)},1^{(3)}, 2^{(0)}\},\;
k\mapsto\begin{cases}
1^{(j)} &\text{if } k\equiv 2j-1 \pmod{6}\quad\text{for some } j\in\{1,2,3\},\\     
2^{(0)} &\text{if } k\equiv 0 \pmod{2}.
\end{cases}
\]
\begin{Def}
We say, that a strictly decreasing sequence of integers $\br=(r_1, r_2,\ldots, r_l)$ is
 \emph{$\tsfG_2$-adapted} if the following conditions are fulfilled:
\begin{itemize}
\item  If $r_p-r_{p+1}\equiv 1\pmod{2}$ for some
$p\in L:=\{1,2,\ldots,l\}$, then $r_p\geq r_{p+1}+3$.
We agree from here on that 
$r_0=\infty$ and $r_{l+1}=-\infty$.
\item $w(\br):=(w(r_1),w(r_2),\ldots, w(r_l))$ is
a reduced expression for some element $\tilde{w}\in W(\tsfG_2)$.
\end{itemize}
\end{Def}
Note, that in this case $l\leq 12$, since $12=l(\tilde{w}_0)$ is the length of the longest element of $W(\tsfG_2)$. Moreover, if $w(r_p)=w(r_q)$ for $p\neq q$, then $\abs{r_p-r_q}\geq 6$. 

Associated to a $\tsfG_2$-adapted sequence $\br$ we introduce the auxiliary functions
\begin{align*}
\htau\df\Z\ra\Z, r &\mapsto\begin{cases}
r- 2\cdot\abs{\{p\in L\mid w(r_p)=2^{(o)} \land r_p>r\}} &\text{if } r\equiv 0 \pmod{2},\\
r- 6\cdot\abs{\{p\in L\mid w(r_p)=1^{(1)} \land r_p>r\}} &\text{if } r\equiv 1 \pmod{6},\\
r- 6\cdot\abs{\{p\in L\mid w(r_p)=1^{(3)} \land r_p>r\}} &\text{if } r\equiv 3 \pmod{6},\\
r- 6\cdot\abs{\{p\in L\mid w(r_p)=1^{(5)} \land r_p>r\}} &\text{if } r\equiv 5 \pmod{6},\\
\end{cases}\\
\hsig\df \Z \ra\Z, r &\mapsto \htau(r-2d_r)+2d_r,\\
\rho_{p}&:=\htau(r_p)-2d_{r_p}=\hsig(r_p)\quad\text{for}\quad p\in L:=\{1,2,\ldots,l\}.
\end{align*}
Note that for $k\in\Z\setminus R$ we have $\hsig(k)=\htau(k)$.
Now, we can introduce the quiver $\tGam_{\br}$ with the same vertex set $\Z=\Gam_0$
and a new set of arrows
\[
\tGam_1:=\Gam_1\cup \{\bet^-_{r-d_r}\mid r\in R\}
\cup\{\bet^o_r, \bet^o_{r+1},\bet^o_{r+2}, \bet^o_{r+3}\mid
r\in R_0\},
\]
where we agree that $R=\{r_1,r_2,\ldots, r_l\}$,
$R_0:=\{r\in R\mid r\equiv 0\pmod{2}\}$ and $R_1:=R\setminus R_0$.

Next, we define the source and target of each arrow of  $\tGam_{\br}$. Namely,
for all $k\in\Z$ and $r\in R$ we set:
\begin{align*}
s(\alp_k^{+}) &:= \htau(k), &  
t(\alp_k^+) &:= \htau(k)+ 2d_k, \\
s(\alp_{k}^{o}) &:= \hsig(k), &
t(\alp_{k}^{o}) &:= \htau(k-3),\\
s(\bet_r^{-}) &:=  \htau(r), &
t(\bet_r^-)   &:=  \hsig(r).\\
\end{align*}
In particular, $t(\bet^-_{r_p})=\rho_p$ for all $1\leq p\leq l$.
Moreover, for all $r\in R_0$ we set:
\begin{align*}
s(\bet_r^o)&:= \hsig(r), 
&t(\bet_r^o)&:= \htau(r-1),\\
s(\bet_{r+1}^o)&:= \htau(r-1), &
t(\bet_{r+1}^o)&:= \htau(r),\\
s(\bet_{r+2}^o)&:= \hsig(r), &
t(\bet_{r+2}^o)&:= \htau(r+1), \\
s(\bet_{r+3}^o)&:=  \htau(r+1),&
t(\bet_{r+3}^o)&:= \htau(r)).
\end{align*}

\begin{example}
Let $\tilde{\br}=(39,32,29,24,21,18,13,10,5,3,0,-5)$. It is easy to see that
$w(\tilde{\br})=(2,0,3,0,2,0,1,0,2,3,0,1)$ is a reduced expression for
$w_0\in W(\tsfG_2)$, and accordingly, $\tilde{\br}$ is $\sfG_2$-adapted. 
Thus, our function $\htau$ has a priori 13 regions of linearity:
\begin{align*}
      &                     &     &                  & Z_1 &=\{33,27,21\},\\
      &                     & Z_2 &=\{30,28,26,24\}, &     &               & Z_3 &=\{23,17,11,5\},\\
      &                     & Z_4 &=\{22,20,18\},    & Z_5 &=\{15,9,3\}   \\
      &                     & Z_6 &=\{16,14,12,10\},                    \\
  Z_7 &=\{7,1,-5\},         & Z_8 &=\{8,6,4,2,0\},   &     &               & Z_9 &=\{-1,-7,\ldots\}, \\ 
      &                    &Z_{11}&=\{-2,-4,\ldots\},&Z_{10}&=\{-3,-9,\ldots\},\\
Z_{12}&=\{-11,-17,\ldots\},      
\end{align*}
and $Z_0=\Z\setminus (\cup_{i=1}^{12} Z_i)$
In fact, we find now
\[
\htau(z)= \begin{cases}
    z  & \text{if } z\in Z_0,\\
    z-2& \text{if } z\in Z_2,\\ 
    z-4& \text{if } z\in Z_4,\\
    z-6& \text{if } z\in Z_1 \cup Z_3 \cup Z_6 \cup Z_7,\\
   z-8 & \text{if } z\in Z_8,\\
   z-10& \text{if } z\in Z_{11}\\ 
   z-12& \text{if } z\in Z_5\cup Z_9\cup Z_{12},\\
   z-18& \text{if } z\in Z_{10},\\
\end{cases}
\]
We display the relevant part of the resulting quiver $\tGam_{\tilde{\br}}$ in Figure~\ref{fig:G2-Expl1}
 \begin{figure}
     \centering
     \scalebox{0.9}{\input{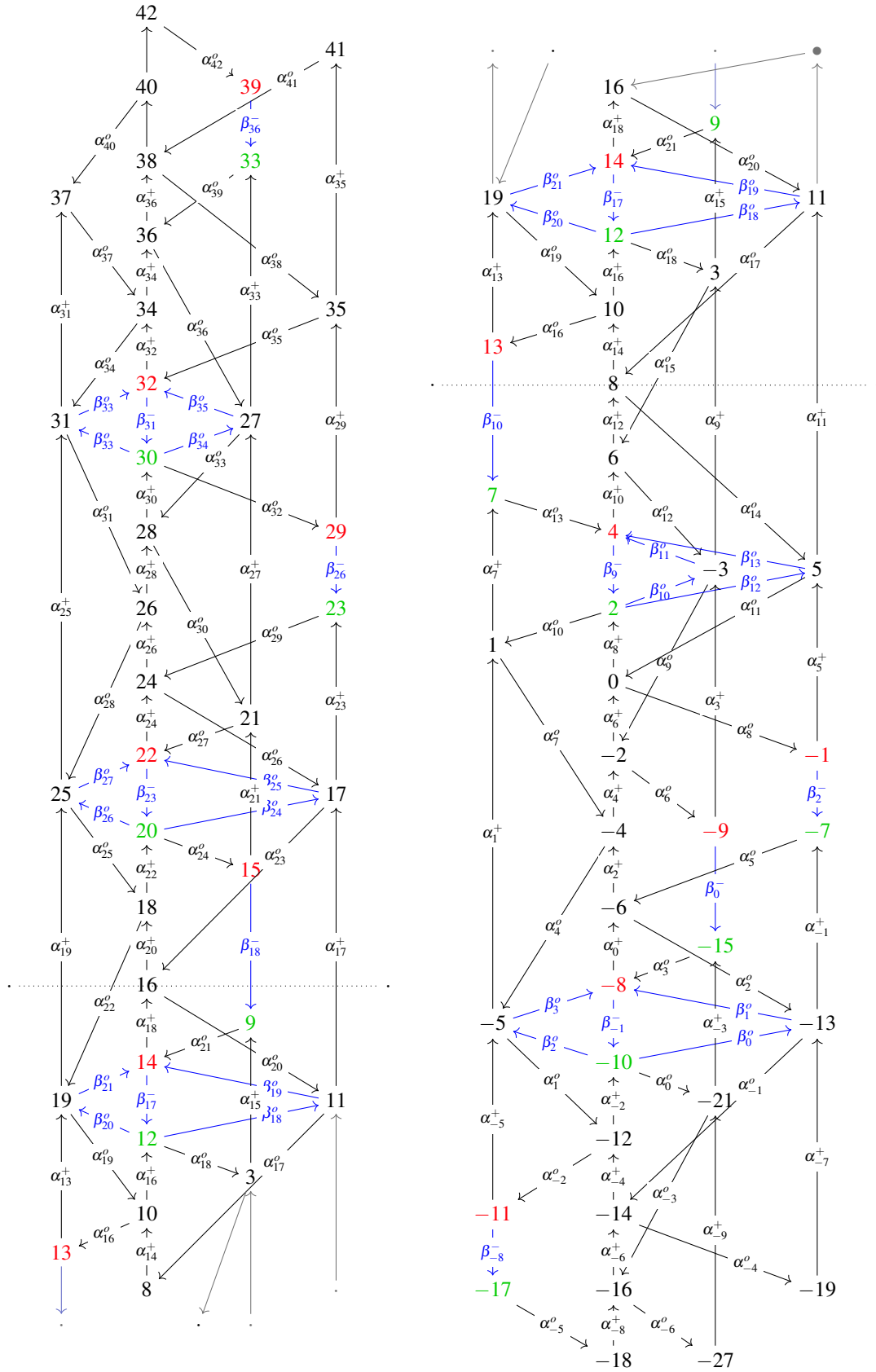}}
     \caption{$\tGam_{(39,32,29,24,21,18,13,10,5,3,0,-5)}$}
     \label{fig:G2-Expl1}
 \end{figure}   
\end{example}

\begin{Rem} 
Let $\bi=(2,1,2,1,2,1)$ be one of the two reduced expression
for $w_0\in W(\sfG_2)$, and let $\br=(r_1,r_2,\ldots, r_6)$ be adapted to $\bi$. Then
\[
\tGam_{(r_1, r_2, r_2-2, r_2-4,r_3, r_4, r_4-2,r_4-4,r_5, r_6,r_6-2,r_6-4)}=\Gam_{\bi,\br}.
\]
Similarly, for $\bi'=(1,2,1,2,1,2)$
and $\br'=(r'_1,\ldots,r'_6)$ adapted to
 $\bi'$, we see that
 \[
 \tilde{\br}'=(r'_1,r'_1-2,r'_1-4,r'_2,r'_3,
 r'_3-2,\ldots, r'_5-4,r'_6)
 \]
 is $\tsfG_2$-adapted and $\tGam_{\tilde{\br}'}=\Gam_{\bi',\br'}$.

\end{Rem}

\begin{Def} 
Set $\del(r):=\begin{cases} 5 &\text{if } r\equiv 0\pmod{2},\\8 &\text{if } r\equiv 1\pmod{2}.
\end{cases}$\\
Let $\br=(r_1,r_2,\ldots,r_l)$ be a decreasing sequence which is $\tsfG_2$-adapted, and suppose that for some $1\leq p\leq l-2$ we have 
\begin{equation} \label{eq:G2-braid}
\qquad r_{p+j}=r_p-3j \quad\text{for}\quad j=1,2.
\end{equation}
If additionally to~\eqref{eq:G2-braid} there holds $r_{p-2}-\del(r_p)\geq r_{p+3}$, we define 
$b_{r_p}^-(\br)$ as the sequence which is obtained from $\br$ by replacing $(r_p,r_{p+1},r_{p+2})$ with $(r_{p+1},r_{p+2}, r_{p+2}-3)$.
If additionally to~\eqref{eq:G2-braid} there holds $r_{p-1}\geq r_p+\del(r_p)$, we define
$b_{r_p}^+(\br)$ as the sequence which is obtained
from $\br$ by 
replacing $(r_p,r_{p+1},r_{p+2})$ with $(r_p+3,r_{p},r_{p+1})$.
\end{Def}

We leave it as an exercise to show that,
whenever $b_s^\pm(\br)$ is defined, it is again
$\tsfG_2$-adapted, and $w(b_s^\pm(\br))$ is obtained
from $\br$ by a $\sfD_4$-braid move. 
The proof of the following Lemma is a straightforward verification.

\begin{Lem}[Braid moves] \label{lem:G2-braid}
Let $\br=(r_1,r_2,\ldots,r_l)$ be a $\tsfG_2$-adapted decreasing sequence of integers, and 
$1\leq p\leq l-2$ such that $r_{p+j}=r_p-3j$ for $j=1,2$. In what follows
we will sometimes write $r(p):=r_p$ for typographical reasons. 
\begin{itemize}
\item[(a1)]
Suppose  that $r_p\equiv 1\pmod2$ and that the operation
$b_{r(p)}^-(\br)$ is defined. Then
\[
\mu_{\rho_{p+1}-4}\circ\mu_{\rho_{p+1}-2}\circ\mu_{\rho_{p+1}}\circ\mu_{\rho_p-6}\circ\mu_{\rho_p}(\tGam_{\br})
\cong \tGam_{b_{r(p)}^-(\br)}.
\]
\item[(a2)]
Suppose  that $r_p\equiv 0\pmod2$ and that the operation
$b_{r(p)}^-(\br)$ is defined. Then
\[
\mu_{\rho_{p+1}}\circ\mu_{\rho_p-6}\circ\mu_{\rho_p-4}
\circ\mu_{\rho_p-2}\circ\mu_{\rho_p}(\tGam_{\br})\cong\tGam_{b_{r(p)}^-(\br)}.
\]
\item[(b1)]
Suppose  that $r_p\equiv 0\pmod2$ and that the operation
$b_{r(p)}^+(\br)$ is defined. Then
\[
\mu_{\rho_{p+1}+6}\circ \mu_{\rho_p}\circ\mu_{\rho_p-2}
\circ \mu_{\rho_p-4}\circ\mu_{\rho_p-6}(\tGam_{\br})\cong
\tGam_{b_{r(p)}^+(\br)}.
\]
\item[(b2)]
Suppose  that $r_p\equiv 1\pmod2$ and that the operation
$b_{r(p)}^+(\br)$ is defined. Then
\[
\mu_{\rho_{p+1}+6}\circ\mu_{\rho_{p+1}+4}\circ\mu_{\rho_{p+1}+2}\circ\mu_{\rho_p}\circ\mu_{\rho_p-6}(\tGam_{\br})
\cong \tGam_{b_{r(p)}^+(\br)}.
\]
\end{itemize}
\end{Lem}
See Figure~\ref{fig:7-braid1} for an illustration. 

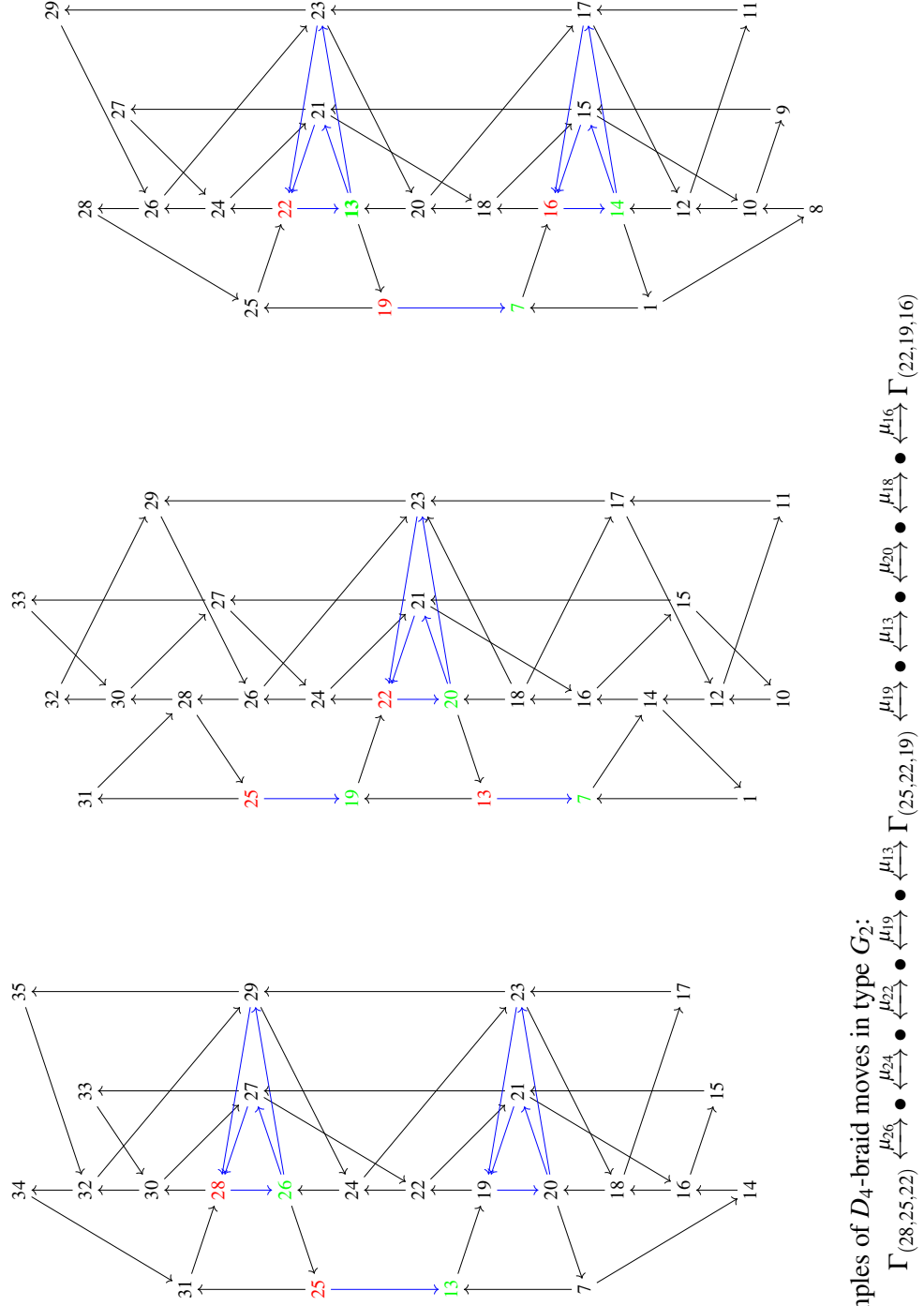
\begin{sidewaysfigure}[ht]
 \begin{minipage}[b][8cm][b]{23cm}  
  \vspace*{3cm} \hspace*{2.5cm}   
    \scalebox{0.7}{\begin{tikzcd}[cramped,column sep=scriptsize,row sep=tiny]
	&& 34 &&&& 35 &&&&&&&& 33 \\
	&&&&&&&&&&&& 32 &&&&&&&&&&&&&& 29 \\
	&& 32 && 33 &&&&&& 31 &&&&&&&&&&&& 28 \\
	&&&&&&&&&&&& 30 &&&&&&&&&&&& 27 \\
	&& 30 &&&&&&&&&&&&&& 29 &&&&&& 26 \\
	31 &&&&&&&&&&&& 28 \\
	&& \textcolor{rgb,255:red,255;green,4;blue,0}{28} &&&&&&&&&&&& 27 &&&&&&&& 24 \\
	&&&& 27 && 29 &&&& \textcolor{rgb,255:red,255;green,4;blue,0}{25} && 26 &&&&&&&& 25 \\
	&& \textcolor{rgb,255:red,0;green,255;blue,4}{26} &&&&&&&&&&&&&&&&&&&& \textcolor{rgb,255:red,255;green,4;blue,0}{22} \\
	\textcolor{rgb,255:red,255;green,4;blue,0}{25} &&&&&&&&&&&& 24 &&&&&&&&&&&& 21 && 23 \\
	&& 24 &&&&&&&& \textcolor{rgb,255:red,0;green,255;blue,4}{19} &&&&&&&&&&&& \textcolor{rgb,255:red,0;green,255;blue,4}{{\mathbf{13}}} \\
	&&&&&&&& {\phantom{X}} &&&& \textcolor{rgb,255:red,255;green,4;blue,0}{22} &&&&&& {\phantom{X}} && \textcolor{rgb,255:red,255;green,4;blue,0}{19} \\
	&& 22 &&&&&&&&&&&& 21 && 23 &&&&&& 20 \\
	\textcolor{rgb,255:red,0;green,255;blue,4}{13} &&&&&&&&&&&& \textcolor{rgb,255:red,0;green,255;blue,4}{20} \\
	&& 19 &&&&&&&& \textcolor{rgb,255:red,255;green,4;blue,0}{13} &&&&&&&&&&&& 18 \\
	&&&& 21 && 23 &&&&&& 18 &&&&&&&& \textcolor{rgb,255:red,0;green,255;blue,4}{7} \\
	&& 20 &&&&&&&&&&&&&&&&&&&& \textcolor{rgb,255:red,255;green,4;blue,0}{16} \\
	7 &&&&&&&&&& \textcolor{rgb,255:red,0;green,255;blue,4}{7} && 16 &&&&&&&&&&&& 15 && 17 \\
	&& 18 &&&&&&&&&&&&&& 17 &&&&&& \textcolor{rgb,255:red,0;green,255;blue,4}{14} \\
	&&&&&&&&&&&& 14 &&&&&&&& 1 \\
	&& 16 &&&& 17 &&&&&&&& 15 &&&&&&&& 12 \\
	&&&& 15 &&&&&&&& 12 \\
	&& 14 &&&&&&&& 1 &&&&&&&&&&&& 10 &&&& 11 \\
	&&&&&&&&&&&& 10 &&&& 11 &&&&&&&& 9 \\
	&&&&&&&&&&&&&&&&&&&&&& 8
	\arrow[from=1-3, to=6-1]
	\arrow[from=1-7, to=3-3]
	\arrow[from=1-15, to=4-13]
	\arrow[from=2-13, to=5-17]
	\arrow[from=2-27, to=5-23]
	\arrow[from=3-3, to=1-3]
	\arrow[from=3-3, to=8-7]
	\arrow[from=3-5, to=5-3]
	\arrow[from=3-11, to=6-13]
	\arrow[from=3-23, to=8-21]
	\arrow[from=4-13, to=2-13]
	\arrow[from=4-13, to=7-15]
	\arrow[from=4-25, to=7-23]
	\arrow[from=5-3, to=3-3]
	\arrow[from=5-3, to=8-5]
	\arrow[from=5-17, to=8-13]
	\arrow[from=5-23, to=3-23]
	\arrow[from=5-23, to=10-27]
	\arrow[from=6-1, to=7-3]
	\arrow[from=6-13, to=4-13]
	\arrow[from=6-13, to=8-11]
	\arrow[from=7-3, to=5-3]
	\arrow[draw={rgb,255:red,0;green,4;blue,255}, from=7-3, to=9-3]
	\arrow[from=7-15, to=1-15]
	\arrow[from=7-15, to=10-13]
	\arrow[from=7-23, to=5-23]
	\arrow[from=7-23, to=10-25]
	\arrow[from=8-5, to=3-5]
	\arrow[draw={rgb,255:red,0;green,4;blue,255}, from=8-5, to=7-3]
	\arrow[from=8-5, to=13-3]
	\arrow[from=8-7, to=1-7]
	\arrow[draw={rgb,255:red,0;green,4;blue,255}, from=8-7, to=7-3]
	\arrow[from=8-7, to=11-3]
	\arrow[from=8-11, to=3-11]
	\arrow[draw={rgb,255:red,0;green,4;blue,255}, from=8-11, to=11-11]
	\arrow[from=8-13, to=6-13]
	\arrow[from=8-13, to=13-17]
	\arrow[from=8-21, to=9-23]
	\arrow[draw={rgb,255:red,0;green,4;blue,255}, from=9-3, to=8-5]
	\arrow[draw={rgb,255:red,0;green,4;blue,255}, from=9-3, to=8-7]
	\arrow[from=9-3, to=10-1]
	\arrow[from=9-23, to=7-23]
	\arrow[draw={rgb,255:red,0;green,4;blue,255}, from=9-23, to=11-23]
	\arrow[from=10-1, to=6-1]
	\arrow[draw={rgb,255:red,0;green,4;blue,255}, from=10-1, to=14-1]
	\arrow[from=10-13, to=8-13]
	\arrow[from=10-13, to=13-15]
	\arrow[from=10-25, to=4-25]
	\arrow[draw={rgb,255:red,0;green,4;blue,255}, from=10-25, to=9-23]
	\arrow[from=10-25, to=15-23]
	\arrow[from=10-27, to=2-27]
	\arrow[draw={rgb,255:red,0;green,4;blue,255}, from=10-27, to=9-23]
	\arrow[from=10-27, to=13-23]
	\arrow[from=11-3, to=9-3]
	\arrow[from=11-3, to=16-7]
	\arrow[from=11-11, to=12-13]
	\arrow[draw={rgb,255:red,0;green,4;blue,255}, from=11-23, to=10-25]
	\arrow[draw={rgb,255:red,0;green,4;blue,255}, from=11-23, to=10-27]
	\arrow[from=11-23, to=12-21]
	\arrow[from=12-13, to=10-13]
	\arrow[draw={rgb,255:red,0;green,4;blue,255}, from=12-13, to=14-13]
	\arrow[from=12-21, to=8-21]
	\arrow[draw={rgb,255:red,0;green,4;blue,255}, from=12-21, to=16-21]
	\arrow[from=13-3, to=11-3]
	\arrow[from=13-3, to=16-5]
	\arrow[from=13-15, to=7-15]
	\arrow[draw={rgb,255:red,0;green,4;blue,255}, from=13-15, to=12-13]
	\arrow[from=13-15, to=18-13]
	\arrow[from=13-17, to=5-17]
	\arrow[draw={rgb,255:red,0;green,4;blue,255}, from=13-17, to=12-13]
	\arrow[from=13-23, to=11-23]
	\arrow[from=13-23, to=18-27]
	\arrow[from=14-1, to=15-3]
	\arrow[draw={rgb,255:red,0;green,4;blue,255}, from=14-13, to=13-15]
	\arrow[draw={rgb,255:red,0;green,4;blue,255}, from=14-13, to=13-17]
	\arrow[from=14-13, to=15-11]
	\arrow[from=15-3, to=13-3]
	\arrow[draw={rgb,255:red,0;green,4;blue,255}, from=15-3, to=17-3]
	\arrow[from=15-11, to=11-11]
	\arrow[draw={rgb,255:red,0;green,4;blue,255}, from=15-11, to=18-11]
	\arrow[from=15-23, to=13-23]
	\arrow[from=15-23, to=18-25]
	\arrow[from=16-5, to=8-5]
	\arrow[draw={rgb,255:red,0;green,4;blue,255}, from=16-5, to=15-3]
	\arrow[from=16-5, to=21-3]
	\arrow[from=16-7, to=8-7]
	\arrow[draw={rgb,255:red,0;green,4;blue,255}, from=16-7, to=15-3]
	\arrow[from=16-7, to=19-3]
	\arrow[from=16-13, to=13-17]
	\arrow[from=16-13, to=14-13]
	\arrow[from=16-13, to=19-17]
	\arrow[from=16-21, to=17-23]
	\arrow[draw={rgb,255:red,0;green,4;blue,255}, from=17-3, to=16-5]
	\arrow[draw={rgb,255:red,0;green,4;blue,255}, from=17-3, to=16-7]
	\arrow[from=17-3, to=18-1]
	\arrow[from=17-23, to=15-23]
	\arrow[draw={rgb,255:red,0;green,4;blue,255}, from=17-23, to=19-23]
	\arrow[from=18-1, to=14-1]
	\arrow[from=18-1, to=23-3]
	\arrow[from=18-11, to=20-13]
	\arrow[from=18-13, to=16-13]
	\arrow[from=18-13, to=21-15]
	\arrow[from=18-25, to=10-25]
	\arrow[draw={rgb,255:red,0;green,4;blue,255}, from=18-25, to=17-23]
	\arrow[from=18-25, to=23-23]
	\arrow[from=18-27, to=10-27]
	\arrow[draw={rgb,255:red,0;green,4;blue,255}, from=18-27, to=17-23]
	\arrow[from=18-27, to=21-23]
	\arrow[from=19-3, to=17-3]
	\arrow[from=19-3, to=21-7]
	\arrow[from=19-17, to=13-17]
	\arrow[from=19-17, to=22-13]
	\arrow[draw={rgb,255:red,0;green,4;blue,255}, from=19-23, to=18-25]
	\arrow[draw={rgb,255:red,0;green,4;blue,255}, from=19-23, to=18-27]
	\arrow[from=19-23, to=20-21]
	\arrow[from=20-13, to=18-13]
	\arrow[from=20-13, to=23-11]
	\arrow[from=20-21, to=16-21]
	\arrow[from=20-21, to=25-23]
	\arrow[from=21-3, to=19-3]
	\arrow[from=21-3, to=22-5]
	\arrow[from=21-7, to=16-7]
	\arrow[from=21-15, to=13-15]
	\arrow[from=21-15, to=24-13]
	\arrow[from=21-23, to=19-23]
	\arrow[from=21-23, to=23-27]
	\arrow[from=22-5, to=16-5]
	\arrow[from=22-13, to=20-13]
	\arrow[from=22-13, to=24-17]
	\arrow[from=23-3, to=21-3]
	\arrow[from=23-11, to=18-11]
	\arrow[from=23-23, to=21-23]
	\arrow[from=23-23, to=24-25]
	\arrow[from=23-27, to=18-27]
	\arrow[from=24-13, to=22-13]
	\arrow[from=24-17, to=19-17]
	\arrow[from=24-25, to=18-25]
	\arrow[from=25-23, to=23-23]
\end{tikzcd}
}
\end{minipage}    
   \caption{Examples of  $\sfD_4$-braid moves in type $\sfG_2$:\\
\hspace*{3cm}$\tGam_{(28,25,22)}
\xleftrightarrow{\mu_{26}}\bullet
\xleftrightarrow{\mu_{24}}\bullet
\xleftrightarrow{\mu_{22}}\bullet
\xleftrightarrow{\mu_{19}}\bullet
\xleftrightarrow{\mu_{13}}
\tGam_{(25,22,19)}
\xleftrightarrow{\mu_{19}}\bullet
\xleftrightarrow{\mu_{13}}\bullet
\xleftrightarrow{\mu_{20}}\bullet
\xleftrightarrow{\mu_{18}}\bullet
\xleftrightarrow{\mu_{16}}
\tGam_{(22,19,16)}$}
\label{fig:7-braid1}
\end{sidewaysfigure}

\begin{Def} {\rm
Let $\br=(r_1,r_2,\ldots, r_l)$ be a decreasing, $\tsfG_2$-adapted sequence of integers.  
For $r\in R$ we define 
\begin{align*}
 r^+&:=\min\{q\in R\mid q>r \text{ and } q-r\equiv 1\pmod 2\},\\
r^-&:=\max\{s\in R\mid r>s \text{ and } r-s\equiv 1\pmod 2\}.
\end{align*}
We agree that $r^+=\infty$ (\resp $r^-=\infty$)
if the corresponding sets are empty ($r^+$ should
not be confused with $r_{p^+}$ from Section~\ref{sub-sec-311}). 

Suppose that $s\in R$ and $s^+-3\geq s + 2d_s$. Then
we define  $c_s^+(\br)$ as the sequence which is obtained from $\br$ by replacing $s$ by $s+2d_{s}$ and rearranging the resulting sequence in decreasing order. 

Suppose that $s - 2d_{s}\geq s^-+3$. Then
we define  $c_s^-(\br)$ as the sequence which is obtained from $\br$ by replacing $s$ by $s-2d_{s}$ and rearranging the resulting sequence in decreasing order. 
}
\end{Def}

Note that, under the given hypothesis,  
$c_s^\pm(\br)$ is again $\tsfG_2$-adapted. Moreover,
$w(c^\pm_s(\br))$ is obtained from $w(\br)$ by a (possibly empty) sequence of $\sfD_4$ commutation operations  involving the letters $i^{(k)}$ with
$k\in\{1,2,3\}$. 
In other words, $w(\br)$ and $w(c_s^\pm(\br))$ belong to the same commutation class
of reduced expressions for elements of $W(\tsfG_2)$.
The proof of the following Lemma 
follows directly from the definitions.

\begin{Lem}[Shift moves] \label{lem:G2-shift}
Let $\br=(r_1,r_2,\ldots, r_l)$ be a $\tsfG_2$-adapted, decreasing integer sequence. 
\begin{itemize}
\item[(a)] Suppose that for $s\in R$ the sequence
$c_s^-(\br)$ is defined. Then 
\[
\mu_{\hsig(s)}(\tGam_\br)=\tGam_{c_s^-(\br)}.
\]
\item[(b)] Suppose that for $s\in R$ the sequence
$c_s^+(\br)$ is defined. Then 
\[
\mu_{\htau(s)}(\tGam_\br)=\tGam_{c_s^+(\br)}.
\]
\end{itemize}
\end{Lem}
Recall that for $s\in R$ we have a vertical down arrow 
$\bet_{s-d_s}^-\df \htau(s)\ra\hsig(s)$.

\begin{example} 
Let $\br$ and $\br'$ be two $\tsfG_2$-adapted decreasing sequences with $w(\br)=w(\br')$.
Then $\br'$ can be obtained from $\br$ by a sequence of operations of the form
$c_p^\pm$, where at each step \emph{no} rearranging is required. 

On the other hand, if $\bi$ is a reduced expression for $w_o\in W(\sfG_2)$ and $\br= (r_1,r_2,\ldots,r_6)$ is adapted to $\bi$, we
may consider the $\tsfG_2$-adapted sequence  $\tilde{\br}$ such that
$\tGam_{\tilde{\br}}=\Gam_{\bi,\br}$. Then 
\[
\tilde{\br}':=b^+_{r_6}\circ b^+_{r_5}\circ\cdots\circ b^+_{r_1}(\tilde{\br})
\]
is defined, and we find
\[
\tGam_{\tilde{\br}'}=\Gam_{\bi, (r_1+2,r_2+2,\ldots,r_6)}.
\]
Note however, that $w(\tilde{\br}')\neq w(\tilde{\br})$. 
\end{example}

\begin{Lem} \label{lem:G2-key}
Consider $\bi=(2,1,2,1,2,1)$ and the $\bi$-adapted sequence
$\br=(26,23,16,13,6,3)$. 
Consider also $\bi'=(1,2,1,2,1,2)$ and the
$\bi'$-adapted sequence $\br'=(35, 28, 25, 18, 15, 8)$. Then there
is an explicit sequence of \textbf{163} mutations, which transforms
$\Gam_{\bi,\br}$ into a quiver, which is isomorphic to 
$\Gam_{\bi',\br'}$. 
\end{Lem}

\begin{proof}
To begin with, consider the $\tsfG_2$-adapted sequences
\begin{align*}
\tilde{\br}  &:=(26, 23,21,19,16,13,11,9,6,3,1,-1)\quad\text{with}\\
w(\tilde{\br}) &= (2^{(0)},1^{(3)},1^{(2)},1^{(1)},2^{(0)},1^{(1)},1^{(3)},1^{(2)},2^{(0)},1^{(2)},1^{(1)},1^{(3)}), \quad\text{and}\\
\tilde{\br}' &=(35,33,31,28,25,23,21,18,15,13,11,8)\quad\text{with}\\
w(\tilde{\br}') &=(1^{(3)},1^{(2)},1^{(1)},2^{(0)},1^{(1)},1^{(3)},1^{(2)},2^{(0)},1^{(2)},1^{(1)},1^{(3)},2^{(0)}).
\end{align*}
Thus, 
\[
\tGam_{\tilde{\br}}=\Gam_{\bi,\br}\quad\text{and}\quad\tGam_{\tilde{\br}'}=\Gam_{\bi',\br'}.
\]
Thus, in view of Lemma~\ref{lem:G2-braid} and Lemma~\ref{lem:G2-shift} it is sufficient
to produce a sequence of operations $b_*^\pm$ and $c_*^\pm$, which transforms
$\tilde{\br}=\br(0)$ into $\tilde{\br}'=\br(21)$. Accordingly, we will carry this  out in 21 steps.
We display at each step the present sequence $\br(i)$ together with (in most cases) $w(i):=w(\br(i))$, where we
simplify each vertex $a^{(j)}$ of $\tsfG_2$ to $j\in\{0,1,2,3\}$.
Under the two sequences $\br(i)$ and $w(i)$, we display a sequence  
of shift (\resp braid moves) which transform, according to the above-mentioned lemmas, $\br(i)$ into $\br(i+1)$. Those moves are composed from right to left. 

\medskip
\begin{tabular}{c|rrrrrrrrrrrr}
$r\setminus p$  &   1,&  2,&  3,&  4,&  5,&  6,&  7,&  8,&  9,& 10,& 11,& 12\phantom{)}\\ \hline
$\br(0)$    & (26,& 23,& 21,& 19,& 16,& 13,& 11,&  9,&  6,&  3,&  1,& -1)\\ 
          $w(0)$    & ( 0,&  3,&  2,&  1,&  0,&  1,&  3,&  2,&  0,&  2,&  1,&  3)\\
    & & & \multicolumn{8}{c}{$c^+_{21}\circ c^+_{23} \circ (c^+_{30}\circ c^+_{28}\circ c^+_{26})\circ c^-_{1} \circ c^-_{-1} $}\\[.5ex] \hline
 $\br(1)$   & (32,& 29,& 27,& 19,& 16,& 13,& 11,&  9,&  6,&  3,& -5,& -7)\\
    &&& \multicolumn{8}{c}{$b^+_{19}\circ b^-_9$}\\[.5ex]\hline
 $\br(2)$   & (32,& 29,& 27,& 22,& 19,& 16,& 11,&  6,&  3,&  0,& -5,& -7)\\
           $w(2)$   & ( 0,&  3,&  2,&  0,&  1,&  0,&  3,&  0,&  2,&  0,&  1,&  3)\\
    &&& \multicolumn{8}{c}{$c^-_{16}\circ c^+_6$}\\[.5ex]\hline
 $\br(3)$   & (32,& 29,& 27,& 22,& 19,& 14,& 11,&  8,&  3,&  0,& -5,& -7)\\    
    &&& \multicolumn{8}{c}{$b^+_{14}$}\\\hline
 $\br(4)$   & (32,& 29,& 27,& 22,& 19,& 17,& 14,& 11,&  3,&  0,& -5,& -7)\\ 
          $w(4)$    & ( 0,&  3,&  2,&  0,&  1,&  3,&  0,&  3,&  2,&  0,&  1,&  3)\\ 
     &\multicolumn{11}{l}{$(c^+_{33}\circ c^+_{27})\circ  (c^+_{40}\circ c^+_{38}\circ c^+_{36}\circ c^+_{34}\circ c^+_{32})\circ  (c^+_{24}\circ c^+_{22})$}\\
     &&\multicolumn{11}{r}{$\circ c^+_{17}\circ(c^-_{11}\circ c^+_3)\circ c^+_0\circ (c^-_{-5}\circ c^+_{-7})$}\\[.5ex]\hline
$\br(5)$   & (42,& 39,& 29,& 26,& 23,& 19,& 14,&  9,&  5,&  2,& -1,& -11)\\
          $w(5)$   & ( 0,&  2,&  3,&  0,&  3,&  1,&  0,&  2,&  3,&  0,&  3,&   1)\\ 
\end{tabular}\\[1cm]
\begin{tabular}{c|rrrrrrrrrrrr}
$r\setminus p$     &   1,&  2,&  3,&  4,&  5,&  6,&  7,&  8,&  9,& 10,&  11,&  12\phantom{)}\\ \hline
 $\br(5)$  & (42,& 39,& 29,& 26,& 23,& 19,& 14,&  9,&  5,&  2,&  -1,& -11)\\  
   &&&\multicolumn{8}{c}{$b^+_{29}\circ b^-_5$}\\\hline  
$\br(6)$   & (42,& 39,& 32,& 29,& 26,& 19,& 14,&  9,&  2,&  -1,& -4,& -11)\\  
          $w(6)$   & ( 0,&  2,&  0,&  3,&  0,&  1,&  0,&  2,&  0,&   3,&  0,&    1)\\ 
 &&&\multicolumn{8}{c}{$(c^+_{34}\circ c^+_{32})\circ c^-_{14}\circ(c^+_4\circ c^+_2)$}\\[.5ex]\hline
$\br(7)$    & (42,& 39,& 36,& 29,& 26,& 19,& 12,& 9,&  6,&  -1,& -4,& -11)\\
    &&&\multicolumn{8}{c}{$b^-_{42}\circ b^+_{12}$}\\[.5ex]\hline  
$\br(8)$    & (39,& 36,& 33,& 29,& 26,& 19,& 15,& 12,&  9,&  -1,& -4,& -11)\\ 
          $w(8)$    & ( 2,&  0,&  2,&  3,&  0,&  1,&  2,&  0,&  2,&   3,&  0,&   1)\\ 
  &&\multicolumn{10}{c}{$(c^-_{34}\circ c^-_{36})\circ c^-_{33}\circ c^-_{26}\circ (c^-_{19}\circ c^+_{15})\circ c^-_{12}\circ c^-_9$} \\[.5ex] \hline
$\br(9)$    & (39,& 32,& 29,& 27,& 24,& 21,& 13,& 10,&  3,&  -1,& -4,& -11)\\
          $w(9)$    & ( 2,&  0,&  3,&  2,&  0,&  2,&  1,&  0,&  2,&   3,&  0,&   1)\\ 
  &&&\multicolumn{8}{c}{$b^-_{27}$}\\[.5ex] \hline
$\br(10)$    & (39,& 32,& 29,& 24,& 21,& 18,& 13,& 10,&  3,&  -1,& -4,& -11)\\
          $w(10)$    & ( 2,&  0,&  3,&  0,&  2,&  0,&  1,&  0,&  2,&   3,&  0,&   1)\\ 
  &&&\multicolumn{8}{c}{$c^+_{24}\circ c^-_{18}$}\\[.5ex] \hline
$\br(11)$     & (39,& 32,& 29,& 26,& 21,& 16,& 13,& 10,&  3,&  -1,& -4,& -11)\\
\end{tabular}\\[1cm]
\begin{tabular}{c|rrrrrrrrrrrr}
$r\setminus p$  &   1,&  2,&  3,&  4,&  5,&  6,&  7,&  8,&  9,& 10,& 11,& 12\phantom{)}\\ \hline
 $\br(11)$  & (39,& 32,& 29,& 26,& 21,& 16,& 13,& 10,&  3,&  -1,& -4,& -11)\\
           $w(11)$  & ( 2,&  0,&  3,&  0,&  2,&  0,&  1,&  0,&  2,&   3,&  0,&   1)\\
  &&&\multicolumn{8}{c}{$b^-_{32}\circ b^+_{16}$} \\[.5ex] \hline
$\br(12)$   & (39,& 29,& 26,& 23,& 21,& 19,& 16,& 13,&  3,&  -1,& -4,& -11)\\
           $w(12)$  & ( 2,&  3,&  0,&  3,&  2,&  1,&  0,&  1,&  2,&   3,&  0,&   1)\\
&&\multicolumn{9}{l}{$c^+_{29}\circ c^+_{26}\circ (c^-_{23}\circ c^+_{19})\circ c^-_{16}\circ$}\\
&&&\multicolumn{9}{r}{$((c^-_{7}\circ c^-_{13})\circ (c^+_5\circ c^+_{-1}))\circ c^+_{-4}\circ c^+_{-11}$}\\[.5ex] \hline
$\br(13)$   & (39,& 35,& 28,& 25,& 21,& 17,& 14,& 11,&  3,&  1,& -2,& -5)\\
           $w(13)$  & ( 2,&  3,&  0,&  1,&  2,&  3,&  0,&  3,&  2,&  1,&  0,&  1)\\
&&&\multicolumn{8}{c}{$b^-_{17}\circ b^-_1$} \\[.5ex] \hline   
$\br(14)$   & (39,& 35,& 28,& 25,& 21,& 14,& 11,&  8,&  3,& -2,& -5,& -8)\\
           $w(14)$  & ( 2,&  3,&  0,&  1,&  2,&  0,&  3,&  0,&  2,&  0,&  1,&  0)\\
&&&\multicolumn{8}{c}{$c^-_8\circ c^+_{-2}$}\\[.5ex] \hline   
 $\br(15)$  & (39,& 35,& 28,& 25,& 21,& 14,& 11,&  6,&  3,&  0,& -5,& -8)\\
&&&\multicolumn{8}{c}{$b^+_6$}\\[.5ex] \hline
$\br(16)$   & (39,& 35,& 28,& 25,& 21,& 14,& 11,&   9,& 6,&  3,& -5,& -8)\\
          $w(16)$   & ( 2,&  3,&  0,&  1,&  2,&  0,&  3,&  2,&  0,&  2,&  1,&  0)\\[.5ex]\hline
\end{tabular}\\[1cm]
\begin{tabular}{c|rrrrrrrrrrrr}
$r\setminus p$  
   &   1,&  2,&  3,&  4,&  5,&  6,&  7,&  8,&  9,& 10,& 11,& 12\phantom{)}\\ \hline
$\br(16)$   & (39,& 35,& 28,& 25,& 21,& 14,& 11,&  9,&  6,&  3,& -5,& -8)\\
&&\multicolumn{10}{c}{$c^+_{25}\circ (c^+_{32}\circ c^+_{30}\circ c^+_{28})\circ c^+_{35}\circ c^+_9\circ (c^+_{16}\circ c^+_{14})$} \\[.5ex] \hline     
 $\br(17)$  & (41,& 39,& 34,& 31,& 21,& 18,& 15,& 11,&  6,&  3,& -5,& -8)\\
           $w(17)$  & ( 3,&  2,&  0,&  1,&  2,&  0,&  2,&  3,&  0,&  2,&  1,&  0)\\
&&&\multicolumn{8}{c}{$b^+_{21}$}\\[.5ex] \hline
 $\br(18)$  & (41,& 39,& 34,& 31,& 24,& 21,& 18,& 11,&  6,&  3,& -5,& -8)\\
           $w(18)$  & ( 3,&  2,&  0,&  1,&  0,&  2,&  0,&  3,&  0,&  2,&  1,&  0)\\
&&&\multicolumn{8}{c}{$(c^+_{26}\circ c^+_{24})\circ (c^-_{16}\circ c^-_{18})\circ c^+_{6}$}\\[.5ex] \hline
 $\br(19)$  & (41,& 39,& 34,& 31,& 28,& 21,& 14,& 11,&  8,&  3,& -5,& -8)\\
&&&\multicolumn{8}{c}{$b^-_{34}\circ b^+_{14}$}\\[.5ex] \hline  
$\br(20)$   & (41,& 39,& 31,& 28,& 25,& 21,&  17,& 14,& 11,&  3,& -5,& -8)\\
          $w(20)$   & ( 3,&  2,&  1,&  0,&  1,&  2,&   3,&  0,&  3,&  2,&  1,&  0)\\
&\multicolumn{10}{l}{$(c^+_6\circ c^+_4\circ c^+_2\circ c^+_0\circ c^+_{-2}\circ c^+_{-4}\circ c^+_{-6}\circ c^+_{-8})\circ $}\\
&&\multicolumn{11}{r}{$c^-_{41}\circ c^-_{39}\circ (c^+_9\circ c^+_3)\circ (c^+_7\circ c^+_1\circ c^+_{-5})\circ(c^+_{16}\circ c^+_{14})\circ c^+_{17}$}\\[.5ex] \hline
 $\br(21)$  & (35,& 33,& 31,& 28,& 25,& 23,& 21,& 18,& 15,& 13,& 11,&  8)\\ 
  $w(21)$           & ( 3,&  2,&  1,&  0,&  1,&  3,&  2,&  0,&  2,&  1,&  3,&  0)\\ \hline
\end{tabular}\\[2ex]
For typographical reasons we displayed this sequence of transformations in four blocks.

\end{proof}

\begin{Rem}  
Our strategy for the proof of Lemma~\ref{lem:G2-key} was to find first an \emph{explicit}
sequence of braid and commutation relations which connect the two reduced expressions
$w(0)$ and $w(21)$ for the longest element in $W(\sfD_4)=W(\tsfG_2)$. Then it became rather straightforward
to "lift" this sequence to our sequence of transformations from $\tilde{\br}$ to $\tilde{\br}'$.
We refrain from displaying the relevant parts of the quivers $\tGam_{\tilde{\br}(i)}$ explicitly, since each such quiver would occupy a full page, see for example Figure~\ref{fig:G2-Expl1}.  
\end{Rem}

\subsection*{Acknowledgments}
Ch.\@ G. and B.\@ L. both acknowledge gratefully that several visits to Université Paris Cité were (partially) funded by IMJ-PRG. In June 2026, Ch.\@ G. could stay at Université Paris Cité as a visiting Professor. 
Ch.\@G. also acknowledges partial support by PAPIIT-DGAPA grant IN113527. 
We heavily used Bernhard Keller's Quiver Mutation App, which can be found at
\url{https://webusers.imj-prg.fr/~bernhard.keller/quivermutation/}\\

Finally, the authors declare that Artificial Intelligence was not used 
at \emph{any stage} of this work . 

\end{document}